\documentclass[12pt, a4paper, twoside, openright, leqno]{book}

\usepackage[latin9]{inputenc} 
\usepackage[T1]{fontenc} 
\usepackage[english, francais]{babel} 
\usepackage{lmodern} 
	\rmfamily 
	\DeclareFontShape{T1}{lmr}{b}{sc}{<->ssub*cmr/bx/sc}{}
	\DeclareFontShape{T1}{lmr}{bx}{sc}{<->ssub*cmr/bx/sc}{}

\usepackage{a4wide}
\usepackage[francais]{layout} 
\usepackage{cdlr} 
\usepackage[french]{minitoc} 
\usepackage[nottoc]{tocbibind} 
\usepackage{makeidx} \makeindex 
\usepackage{fancyhdr} 
\usepackage[color, final]{showkeys} 
	\definecolor{refkey}{rgb}{0,1,0}
	\definecolor{labelkey}{rgb}{1,0,0}

\usepackage{graphicx} 
\graphicspath{ {../Images/} } 
\usepackage{xcolor, colortbl} 
\usepackage[absolute]{textpos} 
\usepackage{rotating} 
\usepackage{lscape} 
\usepackage{subcaption} 
\usepackage{enumerate} 
\usepackage{tabularx, array, multirow} 

\usepackage{amsmath, amsfonts, amssymb, mathtools, mathrsfs, xfrac, mathperso} 
\usepackage{amsthm} 
\usepackage{etex} 
\usepackage[all]{xy} 
\usepackage{tikz} 
	\newenvironment{tikzbox}%
		{\vcenter \bgroup \hbox \bgroup \begin{tikzpicture}}%
		{\end{tikzpicture} \egroup \egroup}
\usepackage{braids} 
\usepackage[vcentermath]{youngtab} 
	\newcommand{\lyng}[1]{\vcenter{\hbox{\scalebox{0.5}{\yng(#1)}}}}
	\newcommand{\lyoung}[1]{\vcenter{\hbox{\scalebox{0.5}{\young(#1)}}}}
\usepackage{figperso} 
\allowdisplaybreaks 

\theoremstyle{plain}
	\newtheorem{myThm}{Théorème}
	
	\newtheorem{myCor}[myThm]{Corollaire}
\theoremstyle{plain}
	\newtheorem{Thm}[equation]{Théorème}
	\newtheorem{Prop}[equation]{Proposition}
	\newtheorem{Lemme}[equation]{Lemme}
	\newtheorem{Cor}[equation]{Corollaire}
	\newtheorem{Pb}{Problème}
\theoremstyle{definition}
	\newtheorem*{Def}{Définition}
	
\theoremstyle{remark}
\newcommand{\qedrem}{\hfill \ensuremath{\vartriangle}}
	\newtheorem{Rem}[equation]{Remarque}
	\let \OldendRem \endRem \def \endRem{\qedrem \OldendRem}
	\newtheorem{Cas}{$\bullet$ Cas}[equation]

\numberwithin{equation}{section}

\begin{document}

\title{ Idempotents de Jones-Wenzl évaluables aux racines de l'unité et représentation modulaire sur le centre de $\Uq$ }
\author{Elsa Ibanez}
\email{ibanez@enim.fr}
\date{vendredi 04 décembre 2015}
\titletimestamp{}

\university{Université de Montpellier}
\univlogo{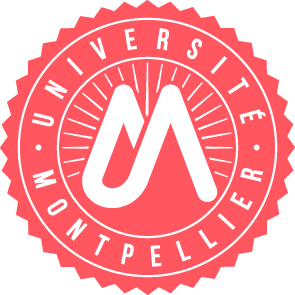}

\doctoral{Information Structures Systèmes}
\researchunit{Institut Montpelliérain Alexander Grothendieck}

\specialisation{Mathématiques et modélisation}

\jury{M. Stéphane {\sc Baseilhac}}{Université de Montpellier}{Directeur}
\jury{M. François {\sc Costantino}}{Université Toulouse III - Paul Sabatier}{Rapporteur}
\jury{M. Charles {\sc Frohman}}{University of Iowa}{Rapporteur}
\jury{M. Damien {\sc Calaque}}{Universté de Montpellier}{Examinateur}
\jury{M. Louis {\sc Funar}}{Université Grenoble I}{Examinateur}

\maketitle 

\newpage 
\thispagestyle{empty}


\frontmatter 
\chapter*{Résumé}
\markboth{Résumé}{Résumé}


\section*{Résumé}

Soit $p \in \N^*$. On définit une famille d'idempotents (et de nilpotents) des algèbres de Temperley-Lieb aux racines $4p$-ième de l'unité qui généralisent les idempotents de Jones-Wenzl usuels. Ces nouveaux idempotents sont associés aux représentations simples et indécomposables projectives de dimension finie du groupe quantique restreint $\Uq$, où $q$ est une racine $2p$-ième de l'unité. A l'instar de la théorie des champs quantique topologique (TQFT) de \cite{BHMV95}, ils fournissent une base canonique de classes d'écheveaux coloriés pour définir des représentations des groupes de difféotopie des surfaces. Dans le cas du tore, cette base nous permet d'obtenir une correspondance partielle entre les actions de la vrille négative et du bouclage, et la représentation de $\SL_2(\Z)$ de \cite{LM94} induite sur le centre de $\Uq$, qui étend non trivialement de la représentation de $\SL_2(\Z)$ obtenue par la TQFT de \cite{RT91}.

\subsubsection*{Mots-clefs}

Groupes quantiques, théorie des écheveaux, algèbres de Temperley-Lieb, idempotents de Jones-Wenzl, représentations modulaires, TQFTs.

\begin{otherlanguage}{english}

\vspace{1cm}
\begin{center} \rule{\textwidth/3}{1pt} \end{center}
\vspace{1cm}

\newpage
\begin{center} \Large \bf 
Evaluable Jones-Wenzl idempotents at root of unity and modular representation on the center of $\Uq$ 
\end{center}

\section*{Abstract}
  
Let $p \in \N^*$. We define a family of idempotents (and nilpotents) in the Temperley - Lieb algebras at $4p$-th roots of unity which generalize the usual Jones-Wenzl idempotents. These new idempotents correspond to finite dimentional simple and projective indecomposable representations of the restricted quantum group $\Uq$, where $q$ is a $2p$-th root of unity. In the manner of the \cite{BHMV95} topological quantum field theorie (TQFT), they provide a canonical basis in colored skein modules to define mapping class groups representations. In the torus case, this basis allows us to obtain a partial match between the negative twist and the buckling actions, and the \cite{LM94} induced representation of $\SL_2(\Z)$ on the center of $\Uq$, which extends non trivially the \cite{RT91} representation of $\SL_2(\Z)$.

\subsection*{Keywords} 

Quantum groups, skein theory, temperley-Lieb algebras, Jones-Wenzl idempotents, modular representations, TQFT.

\end{otherlanguage}

\dominitoc 
\tableofcontents

\mainmatter
\chapter*{Introduction}
\addstarredchapter{Introduction}
\markboth{Introduction}{Introduction}

Cette thèse étudie quelques aspects des théories des champs quantiques topologiques (TQFTs) de type Reshetikhin-Turaev pour $sl_2$ :
\begin{itemize}
	\item une partie de la catégorie des représentations de dimension finie du groupe quantique restreint $\Uq$, où $q$ est une racine primitive paire de l'unité,
	\item l'espace d'écheveaux du tore solide colorié par des nouveaux idempotents (et nilpotents) de Jones-Wenzl aux racines de l'unité, 
	\item les représentations linéaires projectives des groupes de difféotopie des surfaces, et en particulier celles du groupe de difféotopie du tore (épointé).
\end{itemize}

\subsection*{Présentation du sujet}

Les invariants de 3-variété de type Reshetikhin-Turaev motivent de nombreux sujets de recherche en topologie de basse dimension : constructions effectives par la théorie des groupes quantiques, constructions effectives par les modules d'écheveaux coloriés, extensions en des TQFTs, représentations des groupes de difféotopies des surfaces. La première construction éponyme est donnée en 1991 dans \cite{RT91} et repose principalement sur deux outils. Le premier est la \emph{chirurgie} qui permet de passer de l'étude des 3-variétés compactes orientées à celle des entrelacs enrubannés de la sphère $\mathbb{S}^3$ de dimension 3. En effet, toute 3-variété compacte orientée est obtenue par chirurgie le long d'un entrelacs enrubanné dans $\mathbb{S}^3$. De plus, deux telles 3-variétés $M$ et $M'$, obtenues par chirurgie le long des entrelacs enrubannés $L$ et $L'$ respectivement, sont homéomorphes si et seulement si $L$ et $L'$ sont liés par une séries de \emph{mouvements de Kirby} (cf. par exemple \cite[§ 16, § 19]{PS97}). Le deuxième outils est la structure remarquable des catégories des représentations de dimension finie des \emph{groupes quantiques} $\overline{U}_q \mathfrak{g}$, où $\mathfrak{g}$ est une algèbre de Lie et $q$ une racine de l'unité.

Un exemple fondamental de ces catégories est donné par des groupes quantiques quotients $\Uq$, associés à l'algèbre de Lie $sl_2$ et aux racines $q$ de l'unité. Leurs catégories $\Rep^{fd}$ des représentations de dimension finie sont naturellement munies d'un produit tensoriel et d'une dualité. Conformément au théorème de Krull-Schmidt, les objets de $\Rep^{fd}$ se décomposent à isomorphisme près en somme directe de modules indécomposables. Parmi ceux-ci, Reshetikhin et Turaev utilisent uniquement les modules \emph{simples}. Ils obtiennent ainsi des \emph{catégories modulaires} dont les objets sont semi-simples, engendrés à isomorphisme près par un nombre fini de modules simples $\{ \X_1,...,\X_N \}$, et admettent des \emph{traces quantiques} non nulles (cf. \cite[§ 3]{RT91}). Ils construisent alors leur invariant de la façon suivante. Pour chaque racine $q$ de l'unité et pour chaque 3-variété $M$ compacte orientée, obtenue par chirurgie le long d'un entrelacs enrubanné $L = L_1 \cup ... \cup L_k$, on associe à toute composante connexe $L_i$ de $L$ la classe d'isomorphisme d'un $\Uq$-module simple $\X_{j_i}$, $j_i \in \{1,...,N\}$. \`A un scalaire près, on considère ensuite :
\[ Z_{j_1,...,j_k} := \left( \prod_{i=1}^k \mathrm{qdim} (\X_{j_i}) \right) I_L(\X_{j_1},...,\X_{j_k}), \]
où $\mathrm{qdim}$ et $I_L$ désignent respectivement la trace quantique et l'invariant de nœud construits dans \cite{RT90}. L'invariant $Z(M)$ associé à $M$ s'obtient alors en sommant les termes $Z_{j_1,...,j_k}$ sur toutes les associations $(j_1,...,j_k)$ possibles. De plus, cet invariant de 3-variété compacte orientée s'étend en une TQFT de la manière suivante. Pour chaque racine $q$ de l'unité et pour chaque surface $\Sigma = \Sigma_1 \cup ... \cup \Sigma_l$ fermée orientée, on associe à toute composante connexe $\Sigma_i$ de $\Sigma$ de genre $g_i$ un enchevêtrement enrubanné $L = L_1 \cup ... \cup L_{g_i}$, dont les composantes connexes sont coloriées comme précédemment par les classes d'isomorphisme de $\Uq$-modules simples $\{ \X_{j_1},...,\X_{j_{g_i}} \}$. \`A un scalaire près, on considère cette fois :
\[ V_{i ; j_1,...,j_{g_i}} := \bigotimes_{m=1}^{g_i} \X_{j_m} \otimes \X_{j_m}^* \mod \text{"modules non simples"}. \]
L'espace vectoriel $V(\Sigma)$ associé à $\Sigma$ s'obtient alors en sommant les espaces vectoriels $V_{i; j_1,...,j_k}$ sur $i \in \{1,...,l\}$ et toutes les associations $(j_1,...,j_{g_i})$ possibles. Pour tout couple $(\Sigma, \Sigma')$ de surfaces fermées orientées, l'application linéaire $f_M : V(\Sigma) \mapsto V(\Sigma')$ associée à un \emph{cobordisme} $(M, \Sigma, \Sigma')$ est définie de manière tautologique via l'entrelacs enrubanné qui code $M$ par chirurgie et la structure de $\Rep^{fd}$ (cf. \cite[§ 4]{RT91}). Ainsi, Reshetikhin et Turaev obtiennent un foncteur $F : (\Sigma, M) \mapsto (V(\Sigma), f_M)$ et, pour toute surface $\Sigma$ fermée orientée connexe, une représentation projective $\varphi \mapsto f_{\Sigma \times_\varphi [0,1]}$ du groupe de difféotopie de $\Sigma$ sur $V(\Sigma)$ (cf. \cite[§ 4.6]{RT91}).

Dans les articles \cite{Lic91}, \cite{Lic92} et \cite{Lic93}, Lickorish offre une construction topologique alternative de ces invariants quantiques pour $sl_2$. Dans ce cadre, les représentations de dimensions finie sont remplacées par les \emph{espaces d'écheveaux}, et les modules simples par les \emph{idempotents de Jones-Wenzl}. Il retrouve les invariants de \cite{RT91} pour les racines \emph{paires} de l'unité de la manière suivante. On considère des classes d'écheveaux coloriés par les idempotents de Jones-Wenzl $\{f_1,...,f_N\}$ qui admettent des \emph{traces} non nulles (cf. \cite[§ 2, § 4]{Lic92}). Pour chaque 3-variété $M$ compacte orientée, obtenue par chirurgie le long d'un entrelacs enrubanné $L = L_1 \cup ... \cup L_k$, on associe à toute composante $L_i$ de $L$ un idempotent de Jones-Wenzl $f_{j_i}$, $j_i \in \{1,...,N\}$. \`A un scalaire près, on considère ensuite :
\[ Z_{j_1,...,j_k}' := \left( \prod_{i=1}^k \tr ( f_{j_i} ) \right) \Phi_L (f_{i_1},...,f_{i_k}) , \]
où $\tr$ et $\Phi_L$ désignent respectivement la trace et la forme multilinéaire construites dans \cite{Lic92} à partir du \emph{crochet de Kauffman}. L'invariant $Z(M)'$ associé à $M$ s'obtient alors en sommant les termes $Z_{j_1,...,j_k}'$ sur toutes les associations $\{j_1,...,j_k\}$ possibles. Il s'avère que la catégorie des classes d'écheveaux coloriés par les idempotents de Jones-Wenzl possède une structure analogue à la catégorie modulaire construite dans \cite{RT91}. Forts de ce constat, Blanchet, Habegger, Masbaum et Vogel ont généralisé le travail de Lickorish pour toutes les racines de l'unité dans \cite{BHMV92}, puis ont étendu leurs invariants en une TQFT dans \cite{BHMV95}. Ils obtiennent ainsi des représentations projectives des groupes de difféotopie des surfaces sur les espaces d'écheveaux.

\subsection*{Objet de cette thèse}

On concentre notre étude sur les invariants quantiques pour $sl_2$ associés à une racine $q$ primitive \emph{paire} de l'unité. Ils mettent en œuvre la catégorie $\Rep^{fd}$ des représentations de dimension finie du groupe quantique \emph{restreint} $\Uq$. Les deux constructions possibles, celle algébrique de \cite{RT91} et celle topologique de \cite{Lic92} rappelées ci-dessus, ne rendent compte que des modules \emph{simples} de $\Rep^{fd}$. En effet, dans le cadre topologique, les idempotents de Jones-Wenzl correspondent bijectivement à des projecteurs sur les $\Uq$-modules simples (cf. par exemple \cite[§ 3.5]{CFS95}). Pourtant, on ne compte pas seulement des modules simples dans $\Rep^{fd}$, ni-même dans sa sous-catégorie $\Rep^{fd}_s$ engendrée sous les deux lois $\oplus$ et $\otimes$ par les modules simples. Plus précisément, les produits tensoriels des $\Uq$-modules simples de dimension finie font apparaître tous les $\Uq$-modules indécomposables projectifs (PIMs) de dimension finie. Cette impasse sur les $\Uq$-modules non simples de dimension finie repose sur deux obstructions majeures : d'une part leurs traces quantiques sont nulles, et d'autre part on ne sait pas les interpréter avec les idempotents de Jones-Wenzl.

Cette thèse propose de lever la seconde obstruction via la construction de nouveaux idempotents (et nilpotents) de Jones-Wenzl qui tiennent compte des $\Uq$-PIMs. En effet, les candidats que nous proposons généralisent les idempotents de Jones-Wenzl usuels et correspondent à des projecteurs sur les $\Uq$-modules simples et les $\Uq$-PIMs de dimension finie. Par ailleurs, notre intérêt particulier pour les $\Uq$-PIMs est motivé par la représentation linéaire du groupe de difféotopie du tore épointé de \cite{LM94}. Pour le groupe quantique restreint $\Uq$, Feigin, Gainutdinov, Semikhanov et Tipunin ont montré dans \cite{FGST06b} que cette représentation induit une représentation de $\SL_2(\Z)$ sur le centre de $\Uq$ qui étend non trivialement la représentation du groupe de difféotopie du tore obtenue par la TQFT de \cite{RT91}. A l'instar de la TQFT de \cite{BHMV95}, les nouveaux idempotents de Jones-Wenzl devraient fournir une base canonique de classes d'écheveaux coloriés pour retrouver topologiquement cette représentation de $\SL_2(\Z)$ de \cite{FGST06b}. Or, les éléments du centre de $\Uq$ correspondent aux $\Uq$-modules simples et aux $\Uq$-PIMs de dimension finie (cf. \cite[§ 4]{FGST06b}). L'interprétation des $\Uq$-modules simples par les idempotents de Jones-Wenzl usuels étant déjà connue, il est donc naturel de chercher à interpréter les $\Uq$-PIMs de manière similaire. Ainsi, on obtient une analogie entre les éléments du centre de $\Uq$, les objets de $\Rep^{fd}_s$, et les classes d'écheveaux coloriés par nos idempotents (et nilpotents) de Jones-Wenzl. Avec celle-ci, une partie de l'action de $\SL_2(\Z)$ de \cite{FGST06b} s'interprète avec les actions de la vrille négative et du bouclage sur l'espace d'écheveaux du tore solide.

\subsection*{Structure et résumé des résultats}

On fixe $p \in \N^*$ et $q = e^{\frac{i \pi}{p}}$. Les chapitres I et II reprennent et détaillent des résultats connus sur le groupe quantique restreint $\Uq$. Dans le chapitre I, on commence par rappeler le système de générateurs et de relations de $\Uq$, ainsi que sa structure d'algèbre de Hopf. On donne ensuite une description de ses modules simples $\X^\pm(s)$, $1 \leq s \leq p$, et de ses PIMs $\PIM^\pm(s)$, $1 \leq s \leq p$, de dimension finie.  En étudiant leurs produits tensoriels, on obtient une description exhaustive de la sous-catégorie $\Rep^{fd}_s$ de $\Rep^{fd}$. Dans le chapitre II, on suit l'approche de \cite[§ 4]{FGST06a} pour décrire le centre $\Zf$ de $\Uq$ de trois façons différentes. On construit d'abord une base canonique $\{ e_s \; ; \; 0 \leq s \leq p \} \cup \{ w^\pm_s \; ; \; 1  \leq s \leq p-1 \}$ de $\Zf$ à partir de $\Uq$-modules cités précédemment, puis on la décrit en fonction des générateurs de $\Uq$, et on la retrouve enfin via les morphismes de Drinfeld \cite[Prop. 3.3]{Dri90} et de Radford \cite[Prop. 3]{Dri90}. Pour ces derniers, on aura besoin de rappels préliminaires sur les algèbres de Hopf tressées et enrubannées.

Le chapitre III donne les outils et la construction de nos nouveaux idempotents (et nilpotents) de Jones-Wenzl. Comme pour les idempotents de Jones-Wenzl usuels, après quelques rappels sur les espaces d'écheveaux, on commence par étudier la structure des algèbres de Temperley-Lieb génériques $\TL_n(A^2)$, $n \in \N^*$, définies pour un paramètre $A$ \emph{formel}. Pour tout $n \in \N^*$, l'algèbre $\TL_n(A^2)$ se décompose en somme directs d'idéaux indécomposables à gauche correspondant à des idempotents orthogonaux primitifs (POIs) $p_t$, indexés par l'ensemble $T_n^{\leq2}$ des tableaux standards de taille $n$ avec au plus deux lignes. Notre construction repose sur la nouvelle structure des algèbres de Temperley-Lieb $\TL_n(q)$, $n \in \N^*$, obtenues après évaluation de $A^2$ en $q$. La partie technique de ce chapitre consiste à déterminer les POIs de ces algèbres de Temperley-Lieb évaluées à partir des POIs des algèbres de Temperley-Lieb génériques. Pour cela, on s'approprie les outils mis en place dans \cite{GW93}. Pour tout $n \in \N^*$ et tout tableau standard $t \in T_n^{\leq 2}$ de forme $\lambda = [\lambda_1, \lambda_2]$, on note :
\begin{itemize}
	\item $\gamma(t)$ le graphe orienté du treillis de Young donné par :
	\[ \gamma(t):= \xymatrix @-2ex {
	\overset{\emptyset}{\bullet} \ar[r] &
	\overset{\lambda^{(1)}}{\bullet} \ar[r] &
		\overset{\lambda^{(2)}}{\bullet} \ar[r] &
		\cdots \ar[r] &
		\overset{\lambda^{(n)}}{\bullet} } \]
	où, pour tout $i \in \{1,...,n\}$, $\lambda^{(i)}$ est le diagramme de Young contenant les étiquettes $1,...,i$ de $t$ ;
	\item $\omega(\lambda) := \lambda_1-\lambda_2+1$.
\end{itemize}
On dit qu'un tableau standard de forme $\lambda$ est \emph{critique} si $\omega(\lambda)$ est divisible par $p$. Or, sur le treillis de Young, l'entier $\omega(\lambda)$ s'interprète comme le numéro de la colonne du sommet étiqueté par $\lambda$. Ainsi, l'ensemble des tableaux critiques forme des lignes, dites \emph{critiques}, sur le treillis de Young. Lorsque le graphe $\gamma(t)$ d'un tableau standard $t$ ne passe pas deux fois consécutivement sur la même ligne critique, on dit que $t$ est \emph{régulier}. Alors, pour tout $n \in \N^*$, les POIs des algèbres $\TL_n(q)$ se scindent en trois types :
\begin{itemize}
	\item[$(R1)$] les POIs $\bar{p}_{t}$ associés à des tableaux standards $t \in T_n^{\leq2}$ réguliers;
	\item[$(R2)$] les POIs $\bar{p}_{[t]}$ associés à des couples $(t,\bar{t}) \in T_n^{\leq2} \times T_n^{\leq2}$ de tableaux standards réguliers ;
	\item[$(NR)$] les POIs associés à des familles finies de tableaux standards non réguliers, dont les expressions en fonction des POIs $\bar{p}_t$, $t \in T_n^{\leq2}$, sont compliquées.
\end{itemize}
On met alors en évidence une correspondance (pas tout à fait bijective) entre ces POIs et les objets de $\Rep^{fd}_s$ grâce aux morphismes d'algèbres injectifs de \cite[Thm 2.4]{GW93} :
\[ \theta_n : \TL_n(q) \hookrightarrow \End_{\Uq} \left( \X^+(2)^{\otimes n} \right) \quad ; \quad n \in \N^* . \]
Pour les POIs de type $(R1)$ et $(R2)$, cette correspondance prend la forme suivante.

\begin{myThm}
Soient $n \in \N^*$ et $t \in T_n^{\leq2}$ un tableau standard régulier de forme $\lambda$. On note $N \in \N$ le nombre d'intersection(s) entre le graphe $\gamma(t)$ et l'ensemble des lignes critiques privé de la première.	
\begin{itemize}
	\item[$(R1a)$] Si $t$ ne possède pas de sous-tableau critique propre, alors il existe un unique facteur direct $\X$ de $\X^+(2)^{\otimes n}$, isomorphe à $\X^+(\omega(\lambda))$, tel que $\theta_n(\bar{p}_t)$ est un projecteur sur $\X$.
	\item[$(R1b)$] Si $t$ est critique, alors il existe un unique facteur $\X$ de $\X^+(2)^{\otimes n}$, isomorphe à $2^N \PIM^{(-)^{\frac{\omega(\lambda)}{p}-1}}(p) = 2^N \X^{(-)^{\frac{\omega(\lambda)}{p}-1}}(p)$, tel que $\theta_n(\bar{p}_t)$ est un projecteur sur $\X$.
	\item[$(R2)$] Si $t$ n'est pas critique et possède un sous-tableau critique maximal $r \in T_k^{\leq2}$ de forme $\mu$, alors il existe un unique facteur $\X$ de $\X^+(2)^{\otimes n}$, isomorphe à \\ $2^N \PIM^{(-)^{\frac{\omega(\mu)}{p}-1}} \left( p-|\omega(\lambda)-\omega(\mu)| \right)$, tel que $\theta_n(\bar{p}_{[t]})$ est un projecteur sur $\X$. \\ De plus, le morphisme $\theta_n(\bar{p}_{[t]} h_k \bar{p}_{[t]})$ envoie le sous-module de $\X$ isomorphe à \\ $2^N \X^{(-)^{\frac{\omega(\mu)}{p}-1}} \left( p-|\omega(\lambda)-\omega(\mu)| \right)$ sur le module quotient isomorphe à \\ $2^N \X^{(-)^{\frac{\omega(\mu)}{p}-1}} \left( p-|\omega(\lambda)-\omega(\mu)| \right)$.
\end{itemize}
\end{myThm}
\noindent
Parmi les POIs de type $(R1)$ et $(R2)$, en considérant les tableaux standards réguliers :
\[ t(n) := {\scriptstyle \begin{array}{|c|c|c|c|} 
		\hline 1 & 2 & \cdots & n \\
		\hline \end{array}} \quad ; \quad n \in \{1,...,p-1\}, \]
on retrouve les idempotents de Jones-Wenzl usuels $f_n$, $1 \leq n \leq p-1$. A savoir :
\[ f_n = \bar{p}_{t(n)} \quad ; \quad n \in \{1,...,p-1\}. \]
Notre approche consiste à considérer \emph{tous} les POIs des algèbres de Temperley-Lieb évaluées associés aux tableaux standards réguliers :
\[ t(n) = {\scriptstyle \begin{array}{|c|c|c|c|} 
		\hline 1 & 2 & \cdots & n \\
		\hline \end{array}} \quad ; \quad n \in \N^* . \]
On obtient ainsi de nouveaux idempotents :
\[\bar{f}_n = \begin{cases} 
	\bar{p}_{t(n)} & \text{si } t(n) \text{ est de type } (R1)\\ 		
	\bar{p}_{[t(n)]} & \text{si } t(n) \text{ est de type } (R2) \end{cases} \quad ; \quad n \in \N^* . \]
Conformément au théorème 1, ces POIs de type $(R2)$ s'accompagnent d'éléments nilpotents $\bar{f}_n'$, $n \geq p$ et $n \neq -1 \mod p$, et il vérifient les assertions suivantes.
\begin{itemize}
	\item[$(R1a)$] Si $n \leq p-1$, alors il existe un unique facteur direct $\X$ de $\X^+(2)^{\otimes n}$, isomorphe à $\X^+(n+1)$, tel que $\theta_n(\bar{f}_n)$ est un projecteur sur $\X$ ;
	\item[$(R1b)$] Si $n = lp-1$ avec $l \in \N^*$, alors il existe un unique facteur direct $\X$ de $\X^+(2)^{\otimes n}$, isomorphe à $2^{l-1} \PIM^{(-)^{l-1}}(p) = 2^{l-1} \X^{(-)^{l-1}}(p)$, tel que $\theta_n(\bar{f}_n)$ est un projecteur sur $\X$ ;
	\item[$(R2)$] Si $n = lp+r$ avec $l \in \N^*$ et $r \in \{1,...,p-2\}$, alors il existe un unique facteur direct $\X$ de $\X^+(2)^{\otimes n}$, isomorphe à $2^{l-1} \PIM^{(-)^{l-1}}((l+1)p-n-1)$, tel que $\theta_n(\bar{f}_n)$ est un projecteur sur $\X$. \\ De plus, le morphisme $\theta_n(\bar{f}_n')$ envoie le sous-module de $\X$ isomorphe à \\ $2^{l-1} \X^{(-)^{l-1}}((l+1)p-n-1)$ sur le module quotient isomorphe à \\ $2^{l-1} \X^{(-)^{l-1}}((l+1)p-n-1)$.
\end{itemize}
Par ailleurs, ils généralisent les propriétés des idempotents de Jones-Wenzl usuels. D'une part, ils étendent le système de récurrence établi dans \cite{Wen87} comme suit.

\begin{myThm}
\begin{enumerate}[(i)]
	\item Les nouveaux idempotents de Jones-Wenzl vérifient le système de récurrence :
\begin{align*}
&\bar{f}_1 = 1, \\
&\bar{f}_n = \bar{f}_{n-1} + \frac{[n-1]}{[n]} \bar{f}_{n-1} h_{n-1}\bar{f}_{n-1}, && 2 \leq n \leq p-1, \\
&\bar{f}_p = \bar{f}_{p-1}, \\
&\bar{f}_{p+1} = \bar{f}_p - \left( h_p \bar{f}_p' + \bar{f}_p' h_p \right)	- [2] \bar{f}_p' h_p \bar{f}_p', \\
&\bar{f}_n = \bar{f}_{n-1} + \frac{[n-1]}{[n]} \bar{f}_{n-1} h_{n-1} \bar{f}_{n-1} - \frac{2}{[n]^2} \bar{f}_{n-1}' h_{n-1} \bar{f}_{n-1}, &&	p+2 \leq n \leq 2p-1, \\
&\bar{f}_{2p} = \bar{f}_{2p-1}, \\
&\bar{f}_{2p+1} = \bar{f}_{2p} - \left( h_{2p} \bar{f}_{2p}' + \bar{f}_{2p}' h_{2p} \right) - [2] \bar{f}_{2p}' h_{2p} \bar{f}_{2p}', \\
&\bar{f}_n = \bar{f}_{n-1} + \frac{[n-1]}{[n]} \bar{f}_{n-1} h_{n-1} \bar{f}_{n-1} - \frac{2}{[n]^2} \bar{f}_{n-1}' h_{n-1} \bar{f}_{n-1}, && 2p+2 \leq n \leq 3p-2.
\end{align*}
	\item Les nouveaux nilpotents de Jones-Wenzl vérifient le système de récurrence :
\begin{align*}
&\bar{f}_p' = \bar{f}_p h_{p-1} \bar{f}_p, \\
&\bar{f}_{p+1}' = \bar{f}_p' - \bar{f}_p' h_p \bar{f}_p', \\
&\bar{f}_n' = \bar{f}_{n-1}' + \frac{[n-1]}{[n]} \bar{f}_{n-1} h_{n-1} \bar{f}_{n-1}', && p+2 \leq p \leq 2p-2, \\
&\bar{f}_{2p}' = \bar{f}_{2p} h_{2p-1} \bar{f}_{2p}, \\
&\bar{f}_{2p+1}' = \bar{f}_{2p}' - \bar{f}_{2p}' h_{2p} \bar{f}_{2p}', \\
&\bar{f}_n' = \bar{f}_{n-1}' + \frac{[n-1]}{[n]} \bar{f}_{n-1} h_{n-1} \bar{f}_{n-1}', && 2p+2 \leq p \leq 3p-2. \\
\end{align*}
\end{enumerate}
\end{myThm}
\noindent
D'autre part, ils fournissent une base de l'espace d'écheveaux du tore solide similaire à celle décrite dans \cite{Lic92}.

\begin{myThm}
Soient $n \in \N^*$ et $l \in \N$ le quotient de $n$ dans sa division euclidienne par $p$. On note $1$ la classe d'écheveau de l'entrelacs vide, $\alpha$ la classe d'écheveau associée à l'âme $\{0\} \times \mathbb{S}^1$ du tore, et $(U_s(x))_{s \in \mathbb{N}}$ les polynômes de Chebychev de seconde espèce. Alors, on a :
\begin{align*}
	\insertion{\bar{f}_n}{}{Tore} &= \begin{cases} 
		U_{n+1}(\alpha) & \text{ si } n \leq p-1 \text{ ou } n = -1 \mod p, \\
		U_{n+1}(\alpha)+U_{2lp-n-1}(\alpha) & \text{ sinon, }
		\end{cases} \\
	\frac{1}{[lp]} \insertion{\bar{f}_n'}{}{Tore} &= (-1)^l U_{2lp-n-1}(\alpha) \quad \text{ si } n \geq p \text{ et } n \neq -1 \mod p. 
\end{align*}
\end{myThm}

\begin{myCor}
Soient $n \in \N^*$ et $l \in \N$ le quotient de $n$ dans sa division euclidienne par $p$. On note $([s])_{s \in \mathbb{N}}$ les coefficients $q$-entiers. Alors, on a :
\begin{align*}
	\insertion{\bar{f}_n}{}{Empty} &= \begin{cases} 
		(-1)^n [n+1] & \text{ si } n \leq p-1 \text{ ou } n = -1 \mod p, \\
		0 & \text{ sinon, }
		\end{cases} \\
	\frac{1}{[lp]} \; \insertion{\bar{f}_n'}{}{Empty} &= (-1)^{l+n+1} [n+1] \quad \text{ si } n \geq p \text{ et } n \neq -1 \mod p. 
\end{align*}
\end{myCor}

Le chapitre IV propose un lien entre les classes d'écheveaux coloriés par les nouveaux idempotents (et nilpotents) de Jones-Wenzl, et la représentation de $\SL_2(\Z)$ sur $\Zf$ de \cite{FGST06b}. Dans un premier temps, on détaille cette représentation de $\SL_2(\Z)$ sur des éléments de la base canonique $\{ e_s \; ; \; 0 \leq s \leq p \} \cup \{ w^\pm_s \; ; \; 1  \leq s \leq p-1 \}$ de $\Zf$ grâce aux résultats du chapitre II. On calcule ensuite les actions de la vrille et du bouclage sur l'espace d'écheveaux du tore solide colorié par nos nouveaux idempotents et nilpotents de Jones-Wenzl. Elles sont données par les :

\begin{myThm}
Soient $n \in \N^*$ et $l \in \N$ le quotient de $n$ dans sa division euclidienne par $p$. Alors, on a :
\begin{align*}
\begin{tikzbox}
		\draw (0.5,0) rectangle (2.5,0.6);
		\node at (1.5,0.25) {$\bar{f}_n$} ;
		\draw (1,0) -- (1,-0.5) ;
		\braid at (1,-0.5) s_1^{-1} ;
		\draw (2,-0.5) to[bend left=150] (2,-2) ;
		\draw (1,-2) -- (1,-2.5) ;
		\draw (0.5,-2.5) rectangle (2.5,-3.1);
		\node at (1.5,-2.85) {$\bar{f}_n$} ;
	\end{tikzbox} 
	&= \begin{cases}
		(-1)^n q^{\frac{n(n+2)}{2}} \; \rect{\bar{f}_n}{} \quad \text{ si } n \leq p-1 \text{ ou } n = -1 \mod p, \\
		(-1)^n q^{\frac{n(n+2)}{2}} \; \rect{\bar{f}_n}{}
		+ (-1)^n q^{\frac{n(n+2)}{2}} \left( q - q^{-1} \right) (n-lp+1) \; \rect{\bar{f}_n'}{} \quad \text{ sinon, }
		\end{cases} \\
\begin{tikzbox}
		\draw (0.5,0) rectangle (2.5,0.6);
		\node at (1.5,0.25) {$\bar{f}_n'$} ;
		\draw (1,0) -- (1,-0.5) ;
		\braid at (1,-0.5) s_1^{-1} ;
		\draw (2,-0.5) to[bend left=150] (2,-2) ;
		\draw (1,-2) -- (1,-2.5) ;
		\draw (0.5,-2.5) rectangle (2.5,-3.1);
		\node at (1.5,-2.85) {$\bar{f}_n$} ;
	\end{tikzbox}
	&= \begin{tikzbox}
		\draw (0.5,0) rectangle (2.5,0.6);
		\node at (1.5,0.25) {$\bar{f}_n$} ;
		\draw (1,0) -- (1,-0.5) ;
		\braid at (1,-0.5) s_1^{-1} ;
		\draw (2,-0.5) to[bend left=150] (2,-2) ;
		\draw (1,-2) -- (1,-2.5) ;
		\draw (0.5,-2.5) rectangle (2.5,-3.1);
		\node at (1.5,-2.85) {$\bar{f}_n'$} ;
	\end{tikzbox}
	=  (-1)^n q^{\frac{n(n+2)}{2}} \; \rect{\bar{f}_n'}{} \quad \text{ si } n \geq p \text{ et } n \neq -1 \mod p.
\end{align*}
\end{myThm}

\begin{myThm}
Soient $n \in \N^*$, $l \in \N$ le quotient de $n$ dans sa division euclidienne par $p$, et $i \in \N^*$. Alors, on a :
\begin{align*}
\begin{tikzbox}
		\draw (0.5,0) rectangle (2.5,0.6);
		\node at (1.5,0.25) {$\bar{f}_n$} ;
		\draw (1,0) -- (1,-0.5) ;
		\braid[height=0.5cm] at (1,-0.5) s_1 s_1 ;
		\draw (2,-0.5) to[bend left=150] (2,-2) node[below right] {$\scriptstyle i$};
		\draw (1,-2) -- (1,-2.5) ;
		\draw (0.5,-2.5) rectangle (2.5,-3.1);
		\node at (1.5,-2.85) {$\bar{f}_n$} ;
	\end{tikzbox}
	&= \begin{cases}
		(-1)^i \left( q^{(n+1)} + q^{-(n+1)} \right)^i \; \rect{\bar{f}_n}{} \quad \text{ si } n \leq p-1 \text{ ou } n = -1 \mod p, \\
		\begin{multlined}[t]
		(-1)^i \left( q^{(n+1)} + q^{-(n+1)} \right)^i \; \rect{\bar{f}_n}{}  \\
		+ (-1)^{i-1} \left( q^{(n+1)} q^{-(n+1)} \right)^{i-1} i \left( q - q^{-1} \right)^2 [n+1] \; \rect{\bar{f}_n'}{} \quad \text{ sinon, }
		\end{multlined} 
		\end{cases} \\
\begin{tikzbox}
		\draw (0.5,0) rectangle (2.5,0.6);
		\node at (1.5,0.25) {$\bar{f}_n'$} ;
		\draw (1,0) -- (1,-0.5) ;
		\braid[height=0.5cm] at (1,-0.5) s_1 s_1 ;
		\draw (2,-0.5) to[bend left=150] (2,-2) node[below right] {$\scriptstyle i$};
		\draw (1,-2) -- (1,-2.5) ;
		\draw (0.5,-2.5) rectangle (2.5,-3.1);
		\node at (1.5,-2.85) {$\bar{f}_n$} ;
	\end{tikzbox}
	&= \begin{tikzbox}
		\draw (0.5,0) rectangle (2.5,0.6);
		\node at (1.5,0.25) {$\bar{f}_n$} ;
		\draw (1,0) -- (1,-0.5) ;
		\braid[height=0.5cm] at (1,-0.5) s_1 s_1 ;
		\draw (2,-0.5) to[bend left=150] (2,-2) node[below right] {$\scriptstyle i$};
		\draw (1,-2) -- (1,-2.5) ;
		\draw (0.5,-2.5) rectangle (2.5,-3.1);
		\node at (1.5,-2.85) {$\bar{f}_n'$} ;
	\end{tikzbox}
	= \begin{multlined}[t]
		(-1)^i \left( q^{(n+1)} + q^{-(n+1)} \right)^i \; \rect{\bar{f}_n'}{} \\
		\quad \text{ si } n \geq p \text{ et } n \neq -1 \mod p.
		\end{multlined}
\end{align*}
\end{myThm}
\noindent
Avec ces calculs d'écheveaux, on construit enfin les prémisses d'un analogue topologique de la représentation de $\SL_2(\Z)$ de \cite{FGST06b}. En effet, on a les correspondances suivantes entre les classes d'écheveaux coloriés, les objets de $\Rep^{fd}_s$, et les éléments du centre $\Zf$.
\[ \begin{array}{ccccc} 
	\text{Classes d'écheveaux} && \Uq-\text{représentations} && \text{Eléments du centre} \\
	\rect{\bar{f}_{p-1}}{} & \leftrightarrow & \X^+(p) & \leftrightarrow & e_p, \\
	\rect{\bar{f}_{2p-s-1}}{} + \rect{\bar{f}_{2p+s-1}}{} & \leftrightarrow & \PIM^+(s) \oplus (2) \PIM^-(p-s) & \leftrightarrow & e_s, 1 \leq s \leq p-1, \\
	\rect{\bar{f}_{2p-s-1}'}{} & \leftrightarrow & \X^+(s) & \leftrightarrow & \frac{w^+_s}{[s]}, 1 \leq s \leq p-1, \\
	\rect{\bar{f}_{2p+s-1}'}{} & \leftrightarrow & (2) \X^-(p-s) & \leftrightarrow & - \frac{w^-_s}{[s]}, 1 \leq s \leq p-1, \\
	\rect{\bar{f}_{2p-1}}{} & \leftrightarrow & (2) \X^-(p) & \leftrightarrow & e_0
\end{array} \]
Sous ces identifications, une partie de l'action de $\SL_2(\Z)$ de \cite{FGST06b} s'interprète alors avec les actions de la vrille négative et du bouclage.

\subsection*{Questions ouvertes et perspectives}

Pour la représentation de $\SL_2(\Z)$ obtenue par la TQFT de \cite{RT91}, l'interprétation topologique est entièrement donnée par les actions de la vrille négative et du bouclage via les idempotents de Jones-Wenzl usuels. C'est pourquoi, après interprétation totale de l'action de $\SL_2(\Z)$ de \cite{FGST06b}, on espère définir une nouvelle représentation de $\SL_2(\Z)$ sur l'espace d'écheveaux du tore solide qui étende celle de \cite{RT91}.

On espère également que ces nouveaux idempotents (et nilpotents) de Jones-Wenzl permettent de construire des invariants de 3-variété à la manière de \cite{Lic92}. Cette perspective soulève les problèmes ouverts suivants.

\begin{Pb}
Définir une trace modifiée, analogue à celle de \cite{GPMT09}, sur l'espace d'écheveaux du tore solide qui ne s'annule pas sur les nouveaux idempotents (et nilpotents) de Jones-Wenzl.
\end{Pb}

\begin{Pb}
Définir une couleur de Kirby à partir des nouveaux idempotents (et nilpotents) de Jones-Wenzl à la manière de \cite{Lic91}.
\end{Pb}

\begin{Pb}
Etudier les coupons :
\[ \begin{tikzbox}
	\draw (-0.75,0.5) -- (-0.75,0) ;
	\draw (0.75,0.5) -- (0.75,0) ;
	\draw[thick] (-1.25,0) node[above] {$\scriptstyle a$} -- (-0.25,0) ;
	\draw[thick] (0.25,0) -- (1.25,0) node[above] {$\scriptstyle b$} ;
	\draw (-0.5,0) to[bend right=90] (0.5,0) ; 
	\draw (-1,0) to[out=-90, in=90] (-0.25,-1.5) ;
	\draw (1,0) to[out=-90, in=90] (0.25,-1.5) ;
	\draw[thick] (-0.5,-1.5) -- (0.5,-1.5) node[below] {$\scriptstyle j$} ;
	\draw (0,-1.5)  -- (0,-2) ;
\end{tikzbox} \; ; \quad a,b,j \in \N^*, \]
correspondant aux opérateurs de Clebsch-Gordan dans $\Rep^{fd}_s$, et définir des systèmes de couleurs admissibles pour étendre le coloriage par les nouveaux idempotents (et nilpotents) de Jones-Wenzl du tore aux corps à anses à la manière de \cite{Lic93}.
\end{Pb}

\begin{Pb}
Avec les nouveaux idempotents (et nilpotents) de Jones-Wenzl, est-il possible de retrouver les invariants construits dans \cite{CGPM14} pour la classe cohomologique nulle, à la manière de \cite{Lic92} ? 
\end{Pb}

\setcounter{Pb}{0}
\chapter{Le groupe quantique restreint et ses représentations}

Introduit au début des années 1980 par Kulish et Reshetikhin, le groupe quantique $U_h sl(2)$ de l'algèbre de Lie $sl_2$, associé à un paramètre $h$ \emph{formel}, constitue un exemple phare d'algèbre de Hopf tressée et enrubannée (cf. par exemple \cite[§ VI-IX, § XIII-XIV]{Kas95}). De ce fait, sa catégorie des représentations de dimension finie est naturellement munie d'un produit tensoriel remarquable et d'une bidualité. Conformément à l'article \cite{RT90}, cette structure permet de construire des invariants de nœud de la sphère de dimension 3, qui généralisent le célèbre polynôme de Jones. 

Il existe également une forme intégrale $U_q sl(2)$ du groupe quantique $U_h sl(2)$, où $q$ est un paramètre évaluable aux racines de l'unité. Ainsi, lorsque $q$ s'évalue en une racine de l'unité, on obtient un nouveau groupe quantique $U_q sl(2)$ dont la structure n'est plus la même : c'est une algèbre de Hopf non semi-simple, à priori ni tressée ni enrubannée (cf. par exemple \cite[§ VI]{Kas95} et \cite[Cor. 3.7.4]{KS11}). Aussi, sa catégorie des représentations de dimension finie est toujours naturellement munie d'un produit tensoriel et d'une dualité, mais s'enrichit de représentations \emph{non semi-simples}. Pour chaque racine $q$ de l'unité, on considère un groupe quantique \emph{quotient} $\Uq$ de dimension finie de $U_q sl(2)$, pour lequel il y a un nombre fini de modules simples et indécomposables projectifs de dimension finie. Sa catégorie $\Rep^{fd}$ des représentations de dimension finie admet une sous-catégorie intéressante $\Rep^{fd}_s$ engendrée sous les deux lois $\oplus$ et $\otimes$ par les modules \emph{simples}. Une partie de la structure de $\Rep^{fd}_s$ est formalisée dans \cite{RT91} et permet cette fois de construire des invariants de 3-variété via le procédé de chirurgie et l'étude des mouvements de Kirby (cf. par exemple \cite[§16, § 19]{PS97}). Plus précisément, pour chaque racine $q$ de l'unité, on associe à toute surface fermée orientée la \emph{partie semi-simple} d'un produit de $\Uq$-modules simples ; ce qui revient à le quotienter par les $\Uq$-modules indécomposables projectifs dont les traces quantiques sont nulles (cf. par exemple \cite[§ XIV]{Kas95}). 

Récemment, Costantino, Geer, et Paturaud-Mirand ont généralisé la construction de \cite{RT91} dans \cite{CGPM14} en tenant compte, entre autres, de \emph{toute} la structure de $\Rep^{fd}_s$ (voire de $\Rep^{fd}$) ; les modules indécomposables projectifs sont pris en compte grâce à des \emph{traces modifiées} introduites dans \cite{GPMT09}. Dans ce travail, on s'intéresse à tous les objets de la catégorie $\Rep^{fd}_s$. Leur étude exige  une différentiation de cas : le cas des racines impaires de l'unité, et celui des racines paires de l'unité. On se concentre sur le second cas qui met en œuvre les groupes quantiques \emph{restreints}. Toutefois, on pourrait reproduire une grande partie de ce travail (chapitres suivants compris) avec des racines impaires de l'unité, lesquelles mettent en œuvre les \emph{petits} groupes quantiques. Dans ce premier chapitre, on fixe une racine primitive $q$ \emph{paire} de l'unité et on commence par des rappels succincts sur le groupe quantique restreint $\Uq$ associé à cette racine. On détaille ensuite ses représentations simples et indécomposables projectives de dimension finie, ainsi que leurs produits tensoriels. On obtient ainsi une description exhaustive de la sous-catégorie $\Rep^{fd}_s$ de $\Uq$. Avec des approches différentes, on pourra trouver cette description dans \cite{Sut94}, \cite{Ari10a} ou \cite{FGST06a}.

\minitoc

\section{Rappels sur le groupe quantique restreint}
\label{section: Uq}

On se place sur le corps $\C$. On fixe un entier $p \in \N^*$ \index{p1@ $p$} et on pose $q = e^{\frac{i \pi}{p}}$ \index{q1@ $q$} une racine $2p$-ième de l'unité. On définit le groupe quantique restreint, noté $\Uq$, et on donne quelques rappels. On pourra trouver ces éléments dans le chapitre VI de \cite{Kas95}.

\begin{Def}[{\cite[Def. VI.1.1, Def. VI.5.6, Prop. VII.1.1]{Kas95}}]
Le \emph{groupe quantique restreint} $\Uq$ \index{Uq@ $\Uq$} est la $\C$-algèbre engendrée par $E, F, K, K^{-1}$ sous les relations :
\begin{equation} 
\label{Eqn: comm1}
\begin{gathered}
K E K^{-1} =q^2 E,
	\qquad K F K^{-1} =q^{-2} F ,
	\qquad [E,F] = \frac{K-K^{-1}}{q-q^{-1}},  \\ 
E^p = 0,
	\qquad F^p = 0,
	\qquad K^{2p}=1.
\end{gathered}
\end{equation}
Elle est munie d'une structure d'algèbre de Hopf, dont le coproduit $\Delta$, la co-unité $\varepsilon$, et l'antipode $S$ sont donnés par :
\begin{equation} 
\label{Eqn: coprod1}
\begin{gathered} 
\Delta(E) = 1 \otimes E + E \otimes K,
	\; \Delta(F) = K^{-1} \otimes F + F \otimes 1,
	\; \Delta(K) = K \otimes K, \\	
\varepsilon(E) = 0, 
	\quad \varepsilon(F) = 0, 
	\quad \varepsilon(K) = 1,  \\
S(E) = -EK^{-1},
	\quad S(F) = -KF,
	\quad S(K)=K^{-1}.
\end{gathered}
\end{equation}
\end{Def}

\begin{Thm}[{\cite[Prop. VI.5.8]{Kas95}, \cite[§ 3]{Sut94}}] \label{Thm: base PBW}
Les ensembles suivants sont des $\C$-bases de $\Uq$ :
\begin{enumerate}[(a)]
	\item $\{ F^m K^j E^n \; ; \; 0 \leq m,n \leq p-1, \; 0 \leq j \leq 2p-1\}$ ;
	\item $\{ E^n K^j F^m \; ; \; 0 \leq m,n \leq p-1, \; 0 \leq j \leq 2p-1\}$.
\end{enumerate}
\end{Thm}

On utilise les notations standards pour les coefficients $q$-entiers, $q$-factoriels ainsi que $q$-binomiaux. Pour tout $n \in \N$ : \index{n1@ $[n]$} \index{n2@ $[n]"!"$} \index{n3@ $"{" n \brack m "}"$}
\begin{equation} 
\label{Eqn: qcoeff1}
[n] := \frac{q^n-q^{-n}}{q-q^{-1}},
	\qquad [n]! := [n][n-1]...[1],
	\qquad { n \brack m } := \frac{[n]!}{[m]![n-m]!},
\end{equation}
avec la convention $[0]!=1$. Les coefficients $q$-entiers apparaissent dans tous les calculs de commutateurs. Par exemple, pour tous $m, n \in \N$ :
\begin{equation} 
\label{Eqn: comm2}
\begin{aligned} 
\left[ E,F^m \right] &= [m] F^{m-1} \frac{q^{-(m-1)}K - q^{m-1}K^{-1}}{q-q^{-1}} \\
	&= [m] \frac{q^{m-1}K - q^{-(m-1)}K^{-1}}{q-q^{-1}} F^{m-1},  \\
\left[ E^n,F \right] &= [n] E^{n-1} \frac{q^{n-1}K - q^{-(n-1)}K^{-1}}{q-q^{-1}} \\
	&= [n] \frac{q^{-(n-1)}K - q^{n-1}K^{-1}}{q-q^{-1}} E^{n-1}.
\end{aligned}
\end{equation}
Les coefficients $q$-factoriels et $q$-binomiaux permettent de généraliser la formule du binôme pour des éléments $x, y \in \Uq$ tels que $yx = q^2 xy$. Par exemple, pour tous $m, n \in \N$ :
\begin{equation} 
\label{Eqn: coprod2}
\begin{aligned}
\Delta(F^m) =& \Delta(F)^m
= \sum_{r=0}^m { m \brack k } q^{r(m-r)} F^r K^{r-m} \otimes F^{m-r},  \\
\Delta(E^n) =& \Delta(E)^n
= \sum_{s=0}^n {n \brack s } q^{s(n-s)} E^{n-s} \otimes E^s K^{n-s}.
\end{aligned}
\end{equation}
Enfin, comme $q$ est une racine $2p$-ième de l'unité, on a pour tout $n \in \N$ :
\begin{equation} 
\label{Eqn: qcoeff2}
[p+n] = -[n] = [-n],
	\quad [n]!=[p-n][p-n+1]...[p-1],
	\quad {p-1 \brack n} = 1.
\end{equation}

\section{Modules simples et PIMs}
\label{section: reps}

On note $\Rep^{fd}$ \index{Rep@ $\Rep^{fd}$} la catégorie des $\Uq$-représentations à gauche de dimension finie, c'est-à-dire dont :
\begin{itemize}
	\item les objets sont les $\Uq$-modules à gauche de dimension finie,
	\item les morphismes sont les morphismes $\Uq$-linéaires de $\Uq$-modules à gauche de dimension finie.
\end{itemize} 
Grâce au coproduit $\Delta$ de $\Uq$, la catégorie $\Rep$ est naturellement munie d'un produit. En effet, le produit tensoriel $\X \otimes \Y$ de deux $\Uq$-modules à gauche $\X, \Y$ possède encore une structure de $\Uq$-module à gauche donnée par :
\[ \forall a \in \Uq \quad \forall x \in \X \quad \forall y \in \Y \quad
	a (x \otimes y) = \Delta(a) (x \otimes y) . \]
On s'intéresse à la sous-catégorie $\Rep^{fd}_s$\index{Rep@ $\Rep^{fd}_s$} de $\Rep^{fd}$ engendrée sous les deux lois $\oplus$ et $\otimes$ par les $\Uq$-modules à gauche \emph{simples} de dimension finie. Dans la suite, sauf mention contraire, tous les modules considérés seront :
\begin{equation}
\begin{aligned}
& - \text{ des modules à gauches, } \\
& - \text{ des modules de dimension finie. }
\end{aligned}
\end{equation}

On rappelle que la $\C$-algèbre $\Uq$ est de dimension finie (cf. le théorème \ref{Thm: base PBW}). Dans ce cas, d'après le théorème de Krull-Schmidt (cf. par exemple \cite[Thm 14.5]{CR62}), tout $\Uq$-module $\X$ se décompose de manière unique, à isomorphisme et ordre des facteurs près, en somme directe de $\Uq$-modules indécomposables appelés \emph{facteurs directs} de $\X$. Ainsi, la catégorie $\Rep^{fd}_s$ est engendrée sous $\oplus$ par les facteurs directs de l'ensemble des produits tensoriels possibles de $\Rep^{fd}_s$. Il s'agit donc de décrire (la classe d'isomorphisme de) ces facteurs directs.

En outre, les facteurs directs de la représentation régulière $\Reg$ \index{Reg@ $\Reg$} de $\Uq$, appelés \emph{modules indécomposables principaux} (PIMs) \index{PIMs1@ PIMs} de $\Uq$, jouent un rôle particulier. D'une part, un raffinement du théorème de Krull-Schmidt affirme que tout $\Uq$-module \emph{projectif} se décompose de manière unique, à isomorphisme et ordre des facteurs près, en somme directe de $\Uq$-PIMs (cf. par exemple \cite[Thm 56.6]{CR62}). De ce fait, les classes d'isomorphisme de $\Uq$-modules indécomposables projectifs coïncident avec celles de $\Uq$-PIMs. D'autre part, on verra dans la section \ref{section: reps prod} que les facteurs directs des $\Uq$-modules de $\Rep^{fd}_s$ sont des modules simples ou des PIMs, lesquels s'obtiennent par extension non triviale \emph{maximale} de modules simples. Par conséquent, on commence par décrire les $\Uq$-modules simples et leurs extensions non triviales.

\begin{Rem}
\label{Rem: droite}
Parlons des $\Uq$-modules \emph{à droite} de dimension finie. On note ${\Rep'}^{fd}$ la catégorie qu'ils définissent. Pour les mêmes raisons que précédemment, elle est naturellement munie d'un produit et on s'intéresse à la sous-catégorie ${\Rep'}^{fd}_s$ engendrée sous $\oplus$ et $\otimes$ par les $\Uq$-modules \emph{simples} à droite de dimension finie. Alors, la catégorie ${\Rep'}^{fd}_s$ est la version duale de $\Rep^{fd}_s$ car $\Uq$ est une algèbre de \emph{Frobenius} (cf. par exemple \cite[Thm 2.1.3]{Mon93}). En effet, pour tout $\Uq$-module à gauche $\X$, le dual $\X^* := \Hom(\X,\C)$ est muni d'une structure de $\Uq$-module à droite donnée par :
\[ \forall a \in \Uq \quad \forall \alpha \in \X^* \qquad \alpha \cdot a = \alpha( a \cdot ? ) \]
où le symbol $?$ désigne la place de la variable. Dans le cas des algèbres de Frobenius, le dual de tout module à gauche simple est un module simple à droite (cf. par exemple \cite[Thm 58.6]{CR62}), et le dual de tout PIM à gauche est un PIM à droite (cf. par exemple \cite[Thm 60.6, Thm 62.11]{CR62}). En outre, les structures de modules simples à droite (resp. des PIMs à droite) sont analogues aux structures de modules simples à gauche (resp. de PIMs à gauche, cf. par exemple \cite[§ 61]{CR62}). C'est pourquoi on se contente d'étudier les $\Uq$-modules à gauche de dimension finie.
\end{Rem}

\subsection{Modules simples} 

A isomorphisme près, il y a $2p$ modules simples $\X^\alpha(s)$, indexés par $\alpha \in \{+,-\}$ et $s \in \{1,...,p\}$ (cf. par exemple \cite[§ VI.3, § VI.5]{Kas95}).

\begin{Def}[{\cite[Def. VI.3.2]{Kas95}}]
Soit $\X$ un $\Uq$-module. Soit $x \in \X$.
\begin{enumerate}[(i)]
	\item On dit que $x$ est un vecteur \emph{de poids $\lambda$} s'il est vecteur propre sous l'action de $K$, de valeur propre $\lambda$.
	\item On dit que $x$ est un vecteur \emph{de plus haut poids} (resp. \emph{de plus bas poids}) s'il est vecteur propre sous l'action de $K$ et nul sous l'action de $E$ (resp. de $F$).
\end{enumerate}
\end{Def}

\begin{Rem} 
\label{Rem: php/pbp}
Illustrons ces notions sur la représentation régulière $\Reg$. Pour tout $h \in \N$, il existe un vecteur de poids $q^h$ dans $\Reg$ :
\[ x_h:= \sum_{j=0}^{2p-1} q^{-hj} K^j . \]
On obtient ainsi $2p-1$ vecteurs de poids différents tels que, pour tout $h \in \N$ :
\[ K x_h = q^h x_h = x_h K, \quad E x_h = x_{h+2} E, \quad F x_h = x_{h-2} F \] 
(cf. les équations \eqref{Eqn: comm1}). Sous $\Uq$, ils engendrent tous les vecteurs de poids de $\Reg$. En effet, en effectuant un changement de variables sur les $\C$-bases de $\Uq$ données dans le théorème \ref{Thm: base PBW}, les ensembles :
\begin{enumerate}[(a)]
	\item $\{ F^m x_h E^n \; ; \; 0 \leq m,n \leq p-1, \; 0 \leq h \leq 2p-1\}$,
	\item $\{ E^n x_h F^m \; ; \; 0 \leq m,n \leq p-1, \; 0 \leq h \leq 2p-1\}$,
\end{enumerate}
sont aussi des $\C$-bases de $\Uq$, formées de vecteurs de poids de $\Reg$. 

On en déduit que :
\begin{enumerate}[(a)]
	\item les vecteurs $F^{p-1} x_h E^n$, $0 \leq n \leq p-1$ et $0 \leq h \leq 2p-1$, sont les vecteurs de plus \emph{bas} poids de $\Reg$,
	\item les vecteurs $E^{p-1} x_h F^m$, $0 \leq m \leq p-1$ et $0 \leq h \leq 2p-1$, sont les vecteurs de plus \emph{haut} poids de $\Reg$.
\end{enumerate}
\end{Rem}

\begin{Prop} 
\label{Prop: simples}
Soient $\alpha \in \{-,+\}$ et $s \in \{1,...,p\}$. Le module simple $\X^\alpha(s)$ \index{X@ $\X^\alpha(s)$} est engendré sous $\Uq$ par un vecteur de plus haut poids $x_0^\alpha(s)$ (resp. de plus bas poids $x_{s-1}^\alpha(s)$), et il admet une $\C$-base $\{ x_n^\alpha(s) \; ; \; 0 \leq n \leq s-1 \}$ telle que :
\begin{align*}
& K x_n^\alpha(s) = \alpha q^{s-1-2n} x_n^\alpha(s), \quad 0 \leq n \leq s-1, \\
& E x_0^\alpha(s) = 0, 
	\qquad E x_n^\alpha(s) = \alpha [n] [s-n] x_{n-1}^\alpha(s), \quad 1 \leq n \leq s-1, \\
& F x_n^\alpha(s) = x_{n+1}^\alpha(s), \quad 0 \leq n \leq s-2,
	\qquad F x_{s-1}^\alpha(s) = 0.
\end{align*}
On l'appelle la \emph{base canonique de $\X^\alpha(s)$}.
\end{Prop}

\begin{Rem} 
\label{Rem: dim simples}
La notation est choisie de sorte que, pour tous $\alpha \in \{-,+\}$ et $s \in \{1,...,p\}$, on a $\dim \X^\alpha(s) = s$.
\end{Rem}

Il est pratique d'utiliser une représentation diagrammatique de ces structures de modules. Pour cela, on représente les vecteurs de poids de la $\C$-base par des sommets, ordonnées de haut en bas suivant leurs indices croissants. On utilise les flèches verticales pour représenter les actions de $E$ et de $F$ : montantes pour l'action de $E$, et descendantes pour l'action de $F$. On ne met pas de flèche lorsque l'action correspondante est nulle. On en donne deux exemples dans la figure \ref{Fig: simples}.

\begin{figure}[!ht] 
\caption{Représentation diagrammatique de modules simples pour $p=5$.}
\label{Fig: simples}
\[ \xymatrix @-2ex { 
	\text{Poids} && 
		\X^+(3) && 
		\X^-(4) \\
	q^8 &&
		&& 
		\overset{x_0^-(4)}{\bullet} \ar@<2pt>[d]^-F &  \\
	q^6 &&
		&&
		\overset{x_1^-(4)}{\bullet} \ar@<2pt>[d] \ar@<2pt>[u]^-E  \\
	q^4 &&
		&&
		\overset{x_2^-(4)}{\bullet} \ar@<2pt>[d] \ar@<2pt>[u] &  \\
	q^2 & &
		\overset{x_0^+(3)}{\bullet} \ar@<2pt>[d]^-F && 
		\overset{x_3^-(4)}{\bullet} \ar@<2pt>[u] \\
	q^0 &&
		\overset{x_1^+(3)}{\bullet} \ar@<2pt>[d] \ar@<2pt>[u]^-E &&
		\\
	q^8 && 
		\overset{x_2^+(3)}{\bullet} \ar@<2pt>[u] &&
		}\]
\end{figure}
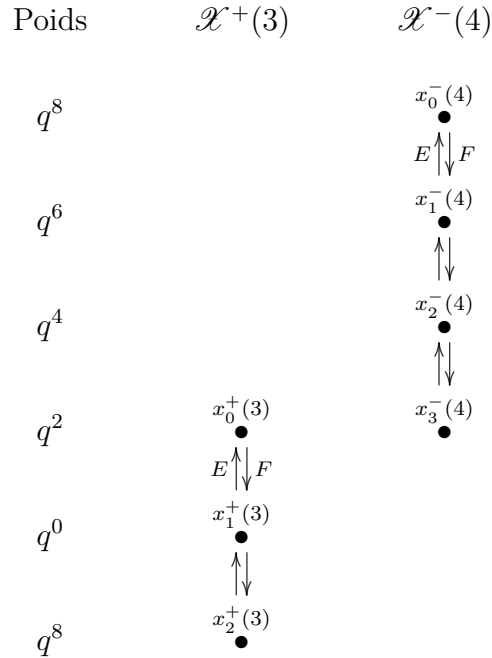

\begin{Lemme} 
\label{Lemme: Verma}
Soient $\alpha,\alpha' \in \{-,+\}$ et $s,s' \in \{1,...,p\}$. On a :
\[ \mathrm{Ext_{\Uq}^1} \left( \X^\alpha(s),\X^{\alpha'}(s') \right) \cong
	\begin{cases}
	\C^2 & \text{ si } \alpha'=-\alpha, \; s'=p-s \\
	0 & \text{ sinon }
	\end{cases} , \]
où $\mathrm{Ext_{\Uq}^1} \left( \X^\alpha(s), \X^{\alpha'}(s') \right)$ est le $\C$-espace vectoriel des classes d'extensions de $\X^\alpha(s)$ par $\X^{\alpha'}(s')$.
\end{Lemme}

\begin{proof}
Soit $M$ une classe d'extension non triviale de $\X^\alpha(s)$ par $\X^{\alpha'}(s')$. On note $N$ le sous-module de $M$ isomorphe à $\X^{\alpha'}(s')$, et $\{x_m^{\alpha'}(s') \; ; \; 0 \leq m \leq s'-1\}$ sa $\C$-base canonique (cf. la proposition \ref{Prop: simples}). On note $a_0^\alpha(s)$ un représentant dans $M$ du vecteur de plus haut poids du quotient $M/N$, isomorphe à $\X^\alpha(s)$. Alors $\{F^n a_0^\alpha(s)+N \; ; \; 0 \leq n \leq s-1\}$ est la $\C$-base canonique de $M/N$ (cf. la proposition \ref{Prop: simples}). Par construction, l'ensemble : 
\begin{equation} \tag{1}
\{x_m^{\alpha'}(s') \; ; \; 0 \leq m \leq s'-1\} \cup \{F^n a_0^\alpha(s) \; ; \; 0 \leq n \leq s-1\}
\end{equation}
est une $\C$-base de $M$. 

Reste à connaître l'action de $\Uq$ sur $M$. Pour cela, il suffit de connaître les vecteurs $K a_0^\alpha(s)$, $F^s a_0^\alpha(s)$ et $E a_0^\alpha(s)$. Pour le premier, on a :
\[ K a_0^\alpha(s) + N 
= K ( a_0^\alpha(s) + N) 
= \alpha q^{s-1} (a_0^\alpha(s) + N) 
= \alpha q^{s-1} a_0^\alpha(s) + N. \]
Il existe donc $y \in N$ tel que :
\begin{equation} \tag{2}
K a_0^\alpha(s) = \alpha q^{s-1} a_0^\alpha(s) + y.
\end{equation}
Pour les autres, on a :
\[ F^s a_0^\alpha(s)+N = F^s (a_0^\alpha(s)+N) = 0+N,
	\qquad E a_0^\alpha(s)+N = E (a_0^\alpha(s)+N) = 0+N. \]
Autrement dit, $F^s a_0^\alpha(s), E a_0^\alpha(s) \in N$. Quitte à translater $a_0^\alpha(s)$ par un vecteur adéquat de $N$, on peut supposer que $E a_0^\alpha(s) \in \Vect_\C \left( x_{s'-1}^{\alpha'}(s') \right)$. On distingue alors deux cas.

\begin{figure}[!ht]  
\caption{Représentation diagrammatique de $\V^\alpha(s)$ et  $\Vb^\alpha(s)$.}
\label{Fig: Lemme: Verma}
\[ \xymatrix @-2ex {
	\text{Poids} && 
		\V^\alpha(s) && 
		\Vb^\alpha(s) \\
	-\alpha q^{p-s-1} &&
		&& 
		\overset{x_0^{-\alpha}(p-s)}{\bullet} \ar@<2pt>[d]  \\
	\vdots &&
		&& 
		\vdots \ar@<2pt>[d] \ar@<2pt>[u] \\
	-\alpha q^{-p+s+1} &&
		&& 
		\overset{x_{p-s-1}^{-\alpha}(p-s)}{\bullet} \ar@<2pt>[u] \\
	\alpha q^{s-1} && 
		\overset{a_0^\alpha(s)}{\bullet} \ar@<2pt>[d]^-F && 
		\overset{a_0^\alpha(s)}{\bullet} \ar@<2pt>[d]^-F \ar@<2pt>[u] \\
	\vdots && 
		\vdots \ar@<2pt>[d] \ar@<2pt>[u]^-E && 
		\vdots \ar@<2pt>[d] \ar@<2pt>[u]^-E \\
	\alpha q^{-s+1} && 
		\overset{F^{s-1} a_0^\alpha(s)}{\bullet} \ar@<2pt>[d] \ar@<2pt>[u] && 
		\overset{F^{s-1} a_0^\alpha(s)}{\bullet} \ar@<2pt>[u] \\
	-\alpha q^{p-s-1} &&
		\overset{x_0^{-\alpha}(p-s)}{\bullet} \ar@<2pt>[d] &&
		\\
	\vdots &&
		\vdots \ar@<2pt>[d] \ar@<2pt>[u] &&
		\\
	-\alpha q^{-p+s+1} &&
		\overset{x_{p-s-1}^{-\alpha}(p-s)}{\bullet} \ar@<2pt>[u] &&
 		} \]
\end{figure}
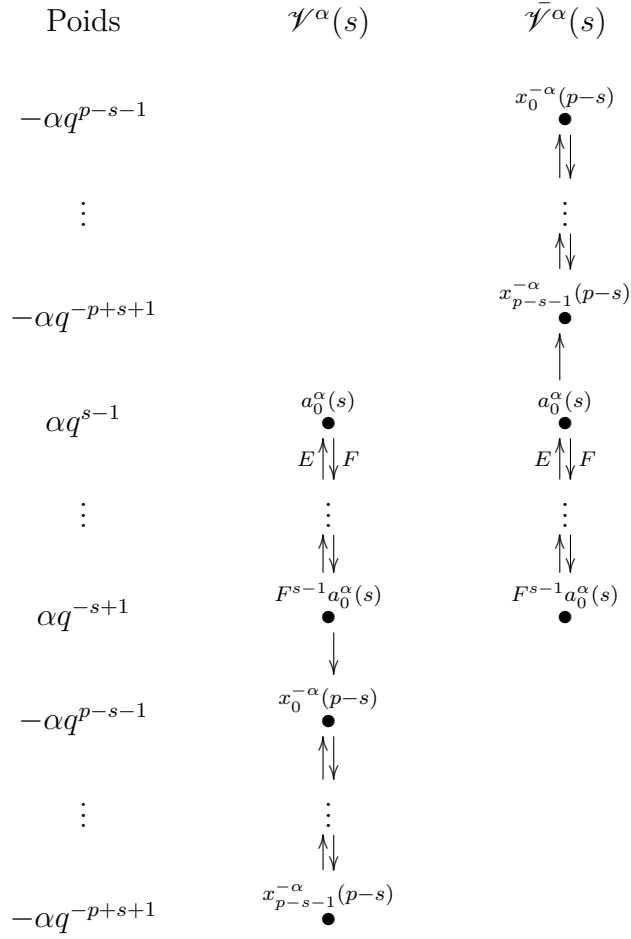

\begin{Cas} 
On suppose que $E a_0^\alpha(s) = 0$. En multipliant l'égalité $(2)$ par $E$, on obtient :
\[ E y \overset{\eqref{Eqn: comm1}}{=} q^{-2}K E a_0^\alpha(s) - \alpha q^{s-1} E a_0^\alpha(s)
= 0. \]
Or $y \in N$, donc $y \in \Vect_\C \left( x_0^{\alpha'}(s') \right)$. Quitte à translater $a_0^\alpha(s)$ par un vecteur adéquat de $\Vect_\C \left( x_0^{\alpha'}(s') \right)$, on peut supposer que $y = 0$. On en déduit que :
\begin{align*}
& K F^n a_0^\alpha(s) \overset{\eqref{Eqn: comm1}}{=} q^{-2n} F^n K a_0^\alpha(s) = \alpha q^{s-1-2n} F^n a_0^\alpha(s), \quad 1 \leq n \leq s, \\
& E F^s a_0^\alpha(s) \overset{\eqref{Eqn: comm2}}{=} F^s E a_0^\alpha(s) + [s] F^{s-1} \frac{q^{-(s-1)}K - q^{s-1}K^{-1}}{q-q^{-1}} a_0^\alpha(s) = 0.
\end{align*}
Or $F^s a_0^\alpha(s) \in N$, donc $F^s a_0^\alpha(s) \in \Vect_\C \left( x_0^{\alpha'}(s') \right)$. Ce qui impose que $s'=p-s$ et $\alpha'=-\alpha$. Par conséquent, la $\C$-base $(1)$ de M vérifie :
\begin{equation} \tag{3a}
\begin{aligned}
& \{ x_m^{-\alpha}(p-s) \; ; \; 0 \leq m \leq p-s-1 \} \text{ est la base canonique de } \X^{-\alpha}(p-s), \\
& K F^n a_0^\alpha(s) = \alpha q^{s-1-2n} F^n a_0^\alpha(s), \quad 0 \leq n \leq s, \\
& E F^n a_0^\alpha(s) = \alpha [n] [s-n] F^{n-1} a_0^\alpha(s), \quad 1 \leq n \leq s-1, \\
& E a_0^\alpha(s) = 0, \qquad F^s a_0^\alpha(s) \in \Vect_\C \left( x_0^{-\alpha}(p-s) \right).
\end{aligned}
\end{equation}
On note $\V^\alpha(s)$ le module correspondant au cas où $F^s a_0^\alpha(s)=x_0^{-\alpha}(p-s)$ (cf. la figure \ref{Fig: Lemme: Verma}).
\end{Cas} \setcounter{Cas}{1}

\begin{Cas} 
On suppose que $E a_0^\alpha(s) \not = 0$. Quitte à multiplier $a_0^\alpha(s)$ par un scalaire adéquat, on peut supposer que $E a_0^\alpha(s) = x_{s'-1}^{\alpha'}(s')$. En multipliant l'égalité $(2)$ par $E$, on obtient :
\[ E y \overset{\eqref{Eqn: comm1}}{=} q^{-2}K E a_0^\alpha(s) - \alpha q^{s-1} E a_0^\alpha(s) = (\alpha' q^{-s'-1} - \alpha q^{s-1}) x_{s'-1}^{\alpha'}(s'). \]
Or $y \in N$, donc $E y=0$. Ce qui impose que :
\[ s'=p-s, \quad \alpha'=-\alpha, \quad y \in \Vect_\C \left( x_0^{\alpha'}(s') \right) . \]
Quitte à translater $a_0^\alpha(s)$ par un vecteur adéquat de $\Vect_\C \left( x_0^{\alpha'}(s') \right)$, on peut supposer que $y = 0$. On en déduit que :
\begin{align*}
& K F^n a_0^\alpha(s) \overset{\eqref{Eqn: comm1}}{=} q^{-2n} F^n K a_0^\alpha(s) = \alpha q^{s-1-2n} F^n a_0^\alpha(s), \quad 1 \leq n \leq s, \\
& E F^s a_0^\alpha(s) \overset{\eqref{Eqn: comm2}}{=} F^s E a_0^\alpha(s) + [s] F^{s-1} \frac{q^{-(s-1)}K - q^{s-1}K^{-1}}{q-q^{-1}} a_0^\alpha(s) = 0.
\end{align*}
Or $F^s a_0^\alpha(s) \in N$, donc $F^s a_0^\alpha(s) \in \Vect_\C \left( x_0^{-\alpha}(p-s) \right)$. Par conséquent, la $\C$-base $(1)$ de M vérifie :
\begin{equation} \tag{3b}
\begin{aligned}
& \{ x_m^{-\alpha}(p-s) \; ; \; 0 \leq m \leq p-s-1 \} \text{ est la base canonique de } \X^{-\alpha}(p-s), \\
& K F^n a_0^\alpha(s) = \alpha q^{s-1-2n} F^n a_0^\alpha(s), \quad 0 \leq n \leq s-1, \\
& E F^n a_0^\alpha(s) = \alpha [n] [s-n] F^{n-1} a_0^\alpha(s), \quad 1 \leq n \leq s-1, \\
& E a_0^\alpha(s) = x_{p-s-1}^{-\alpha}(p-s), \qquad F^s a_0^\alpha(s) \in \Vect_\C \left( x_0^{-\alpha}(p-s) \right).
\end{aligned}
\end{equation}
On note $\Vb^\alpha(s)$ le module correspondant au cas où $F^s a_0^\alpha(s)=0$ (cf. la figure \ref{Fig: Lemme: Verma}).
\end{Cas}

L'extension $M$ n'étant pas triviale, soit on est dans le premier cas $(3a)$ avec $F^s a_0^\alpha(s) \not = 0$, soit on est dans le second cas $(3b)$. Dans chacun de ces cas, la $\C$-base $(1)$ de M vérifie :
\begin{align*}
& \{ x_m^{-\alpha}(p-s) \; ; \; 0 \leq m \leq p-s-1 \} \text{ est la base canonique de } \X^{-\alpha}(p-s), \\
& K F^n a_0^\alpha(s) = \alpha q^{s-1-2n} F^n a_0^\alpha(s), \quad 0 \leq n \leq s-1, \\
& E F^n a_0^\alpha(s) = \alpha [n] [s-n] F^{n-1} a_0^\alpha(s), \quad 1 \leq n \leq s-1, \\
& E a_0^\alpha(s) = \lambda x_{p-s-1}^{-\alpha}(p-s), \qquad F^s a_0^\alpha(s) = \mu \left( x_0^{-\alpha}(p-s) \right),
\end{align*}
où $\lambda$ et $\mu$ ne sont pas tout deux nuls. Autrement dit, l'extension $M$ s'exprime comme somme de Baer des classes de $\V^\alpha(s)$ et $\Vb^\alpha(s)$. Par ailleurs, les modules $\V^\alpha(s)$ et $\Vb^\alpha(s)$ ne sont pas isomorphes ; ils forment donc une famille libre dans le $\C$-espace vectoriel des classes d'extensions de $\X^\alpha(s)$ par $\X^{\alpha'}(s')$. D'où le résultat.
\end{proof}

\subsection{Modules de Verma : premières extensions}

Soient $\alpha \in \{-,+\}$ et $s \in \{1,...,p-1\}$. D'après le lemme \ref{Lemme: Verma}, il y a deux classes d'extensions $\V^\alpha(s)$ et $\Vb^\alpha(s)$  qui engendrent le $\C$-espace vectoriel des classes d'extensions de $\X^\alpha(s)$ par $\X^{-\alpha}(p-s)$. On les appelle respectivement \emph{module de Verma} et \emph{module de Verma contragrédient}.

\begin{Prop} 
\label{Prop: Verma}
Soient $\alpha \in \{-,+\}$ et $s \in \{1,...,p-1\}$. 
\begin{enumerate}[(i)]
	\item Le module de Verma $\V^\alpha(s)$ \index{V@ $\V^\alpha(s)$} est engendré sous $\Uq$ par un vecteur de plus haut poids $x_0^\alpha(s)$, et il admet une $\C$-base $\{ x_m^{-\alpha}(p-s) \; ; \; 0 \leq m \leq p-s-1 \} \cup \{ a_n^\alpha(s) \; ; \; 0 \leq n \leq s-1 \}$ telle que :
\begin{align*}
& \{ x_m^{-\alpha}(p-s) \; ; \; 0 \leq m \leq p-s-1 \} \text{ est la base canonique de } \X^{-\alpha}(p-s), \\
& K a_n^\alpha(s) = \alpha q^{s-1-2n} a_n^\alpha(s), \quad 0 \leq n \leq s-1, \\
& E a_0^\alpha(s) = 0, 
	\qquad E a_n^\alpha(s) = \alpha [n] [s-n] a_{n-1}^\alpha(s), \quad 1 \leq n \leq s-1, \\
& F a_n^\alpha(s) = a_{n+1}^\alpha(s), \quad 0 \leq n \leq s-2,
	\qquad F a_{s-1}^\alpha(s) = x_0^{-\alpha}(p-s).
\end{align*}
On l'appelle la \emph{base canonique de $\V^\alpha(s)$}.
	\item Le module de Verma contragrédient $\Vb^\alpha(s)$ \index{V@ $\Vb^\alpha(s)$} est engendré sous $\Uq$ par un vecteur de plus bas poids $x_{s-1}^\alpha(s)$, et il admet une $\C$-base $\{ x_m^{-\alpha}(p-s) \; ; \; 0 \leq m \leq p-s-1 \} \cup \{ a_n^\alpha(s) \; ; \; 0 \leq n \leq s-1 \}$	telle que :
\begin{align*}
& \{ x_m^{-\alpha}(p-s) \; ; \; 0 \leq m \leq p-s-1 \} \text{ est la base canonique de } \X^{-\alpha}(p-s), \\
& K a_n^\alpha(s) = \alpha q^{s-1-2n} a_n^\alpha(s), \quad 0 \leq n \leq s-1, \\
& E a_0^\alpha(s) = x_{p-s-1}^{-\alpha}(p-s), 
	\qquad E a_n^\alpha(s) = \alpha [n] [s-n] a_{n-1}^\alpha(s), \quad 1 \leq n \leq s-1, \\
& F a_n^\alpha(s) = a_{n+1}^\alpha(s), \quad 0 \leq n \leq s-2,
	\qquad F a_{s-1}^\alpha(s) = 0.
\end{align*}
On l'appelle la \emph{base canonique de $\Vb^\alpha(s)$}.
\end{enumerate}
\end{Prop}

\begin{proof}
La proposition découle directement de la preuve du lemme \ref{Lemme: Verma}.
\end{proof}

\begin{Rem} 
\label{Rem: dim Verma}
\begin{enumerate}[(i)]
	\item Pour tout $\alpha \in \{-,+\}$ et $s=p$, on étend la définition précédente à $\V^\alpha(p)$ et $\Vb^\alpha(p)$ avec le système de relations :
	\begin{align*}
	& K a_n^\alpha(s) = \alpha q^{s-1-2n} a_n^\alpha(s), \quad 0 \leq n \leq s-1, \\
	& E a_0^\alpha(s) = 0, 
		\qquad E a_n^\alpha(s) = \alpha [n] [s-n] a_{n-1}^\alpha(s), \quad 1 \leq n \leq s-1, \\
	& F a_n^\alpha(s) = a_{n+1}^\alpha(s), \quad 0 \leq n \leq s-2,
		\qquad F a_{s-1}^\alpha(s) = 0.
\end{align*}
	Ainsi, on a $\V^\alpha(p) = \X^\alpha(p) = \Vb^\alpha(p)$.
	\item Pour tous $\alpha \in \{-,+\}$ et $s \in \{1,...,p\}$, on a $\dim \V^\alpha(s) = \dim \Vb^\alpha(s) = p$.
\end{enumerate}
\end{Rem}

On utilise la même représentation diagrammatique que précédemment, illustrée dans la figure \ref{Fig: Verma}, ou plus succinctement :
\[ \V^\alpha(s) = \vcenter{ \xymatrix @-2ex {
		\overset{\X^\alpha(s)}{\bullet} \ar[d]^-F \\
		\overset{\X^{-\alpha}(p-s)}{\bullet} }}, \quad
	\Vb^\alpha(s) = \vcenter{ \xymatrix @-2ex { 
		\overset{\X^\alpha(s)}{\bullet} \ar[d]^-E \\
		\overset{\X^{-\alpha}(p-s)}{\bullet} }} . \]

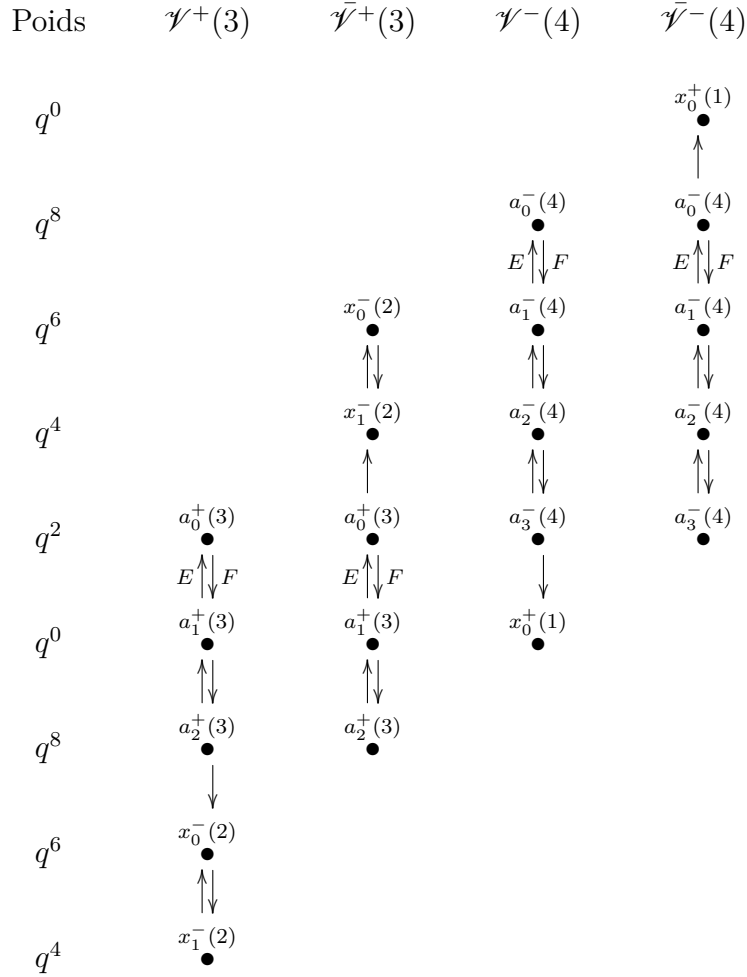
\begin{figure}[!ht] 
\caption{Représentation diagrammatique de modules de Verma pour $p=5$.}
\label{Fig: Verma}
\[ \xymatrix @-2ex @C-1ex { 
	\text{Poids} &&
		\V^+(3) &&
		\Vb^+(3) &&
		\V^-(4) &&
		\Vb^-(4) \\
	q^0 &&
		&&
		&&
		&&
		\overset{x_0^+(1)}{\bullet} \\
	q^8 &&
		&&
		&&
		\overset{a_0^-(4)}{\bullet} \ar@<2pt>[d]^-F &&
		\overset{a_0^-(4)}{\bullet} \ar@<2pt>[d]^-F \ar@<2pt>[u] \\
	q^6 && 
		&&
		\overset{x_0^-(2)}{\bullet} \ar@<2pt>[d] && 
		\overset{a_1^-(4)}{\bullet} \ar@<2pt>[d] \ar@<2pt>[u]^-E &&
		\overset{a_1^-(4)}{\bullet} \ar@<2pt>[d] \ar@<2pt>[u]^-E \\
	q^4 && 
		&&
		\overset{x_1^-(2)}{\bullet} \ar@<2pt>[u] && 
		\overset{a_2^-(4)}{\bullet} \ar@<2pt>[d] \ar@<2pt>[u] &&
		\overset{a_2^-(4)}{\bullet} \ar@<2pt>[d] \ar@<2pt>[u] \\
	q^2 && 
		\overset{a_0^+(3)}{\bullet} \ar@<2pt>[d]^-F && 
		\overset{a_0^+(3)}{\bullet} \ar@<2pt>[d]^-F \ar@<2pt>[u] && 
		\overset{a_3^-(4)}{\bullet} \ar@<2pt>[d] \ar@<2pt>[u] &&
		\overset{a_3^-(4)}{\bullet} \ar@<2pt>[u] \\
	q^0 && 
		\overset{a_1^+(3)}{\bullet} \ar@<2pt>[d] \ar@<2pt>[u]^-E && 
		\overset{a_1^+(3)}{\bullet} \ar@<2pt>[d] \ar@<2pt>[u]^-E && 
		\overset{x_0^+(1)}{\bullet} &&
		 \\
	q^8 && 
		\overset{a_2^+(3)}{\bullet}  \ar@<2pt>[d] \ar@<2pt>[u] &&
		\overset{a_2^+(3)}{\bullet}  \ar@<2pt>[u] &&
		&&
		 \\
	q^6 &&
		\overset{x_0^-(2)}{\bullet} \ar@<2pt>[d] && 
		&&
		&&
		\\
	q^4 &&
		\overset{x_1^-(2)}{\bullet} \ar@<2pt>[u] &&
		&&
		&&
		}\]
\end{figure}

\begin{Cor} 
\label{Cor: php/pbp}
Soient $\alpha \in \{+,-\}$ et $s \in \{1,...,p\}$.
\begin{enumerate}[(i)]
	\item Un vecteur de plus haut poids $\alpha q^{s-1}$ engendre sous $\Uq$ un module isomorphe à $\X^\alpha(s)$ ou à $\V^\alpha(s)$.
	\item Un vecteur de plus bas poids $\alpha q^{-s+1}$ engendre sous $\Uq$ un module isomorphe à $\X^\alpha(s)$ ou à $\Vb^\alpha(s)$.
\end{enumerate}
 \end{Cor}
 
\begin{proof}
Montrons par exemple $(i)$. Soit $x^\alpha(s)$ un vecteur de plus haut poids $\alpha q^{s-1}$. On note $M$ le module engendré par $x^\alpha(s)$ sous $\Uq$. D'après le théorème \ref{Thm: base PBW}, on a :
\begin{align*}
M &= \Vect_\C \{ F^m K^j E^n x^\alpha(s) \; ; \; 0 \leq m,n \leq p-1, 0 \leq j \leq 2p-1 \} \\
&= \Vect_\C \{ F^m K^j x^\alpha(s) \; ; \; 0 \leq m \leq p-1, 0 \leq j \leq 2p-1 \} \\
&= \Vect_\C \{ F^m x^\alpha(s) \; ; \; 0 \leq m \leq p-1 \},
\end{align*}
car $x^\alpha(s)$ est vecteur propre sous l'action de $K$. Le structure de module de $M$ est donc entièrement déterminée par l'action de $\Uq$ sur les vecteurs $F^m x^\alpha(s)$, $1 \leq m \leq p-1$. D'après les règles de commutativité \eqref{Eqn: comm1}, on a :
 \begin{align*}
& K F^m x^\alpha(s) = \alpha q^{s-1-2m} F^m x^\alpha(s), \quad 0 \leq m \leq p-1, \\
& E F^m x^\alpha(s) = \alpha [m] [s-m] F^{m-1} x^\alpha(s), \quad 1 \leq m \leq p-1.
\end{align*}
L'ensemble $\left\{ m \in \N^* \; ; \; F^m x^\alpha(s) = 0 \right\}$ étant majoré par $p$, il possède possède un minimum $r \in \{1,...,p\}$. On distingue deux cas.

\begin{Cas} On suppose que $r \not = p$.
Pour cet entier $r$, on a :
\[ E \underbrace{F^{r} x^\alpha(s)}_{=0} = \alpha [r] [s-r] \underbrace{F^{r-1} x^\alpha(s)}_{\not = 0} . \]
Or $[r] \not = 0$ car $r \in \{1,...,p-1\}$, donc $[s-r]=0$. Il s'ensuit que $r = s$ et que $M$ est isomorphe à $\X^\alpha(s)$ via les identifications :
\[ F^m x^\alpha(s) \longmapsto x_m^\alpha(s), \quad 0 \leq m \leq s-1. \]
\end{Cas}

\begin{Cas} On suppose que $r = p$.
Alors $M$ est isomorphe à $\V^\alpha(s)$ via les identifications :
\[ \begin{cases}
	F^m x^\alpha(s) \longmapsto a_m^\alpha(s), & 0 \leq m \leq s-1, \\
	F^m x^\alpha(s) \longmapsto x_{m-s}^{-\alpha}(p-s), & s \leq m \leq p.
	\end{cases} \]
\end{Cas}
\end{proof}

\subsection{PIMs : extensions maximales}
\label{subsection: PIMs}

Soient $\alpha \in \{-,+\}$ et $s \in \{1,...,p\}$. On note $\PIM^\alpha(s)$ l'extension non triviale \emph{maximale} de $\X^\alpha(s)$. A isomorphisme près, elle est unique (cf. par exemple \cite[Cor. 54.14]{CR62}) et s'identifie à un PIM $\Uq$ (voir l'introduction de la section \ref{section: reps}).

\begin{Rem} 
\label{Rem: nbre PIMs}
On obtient ainsi tous les PIMs (cf. par exemple \cite[Cor. 54.14]{CR62}). Autrement dit, à isomorphisme près, il y a $2p$ PIMs $\PIM^\alpha(s)$, indexés par $\alpha \in \{+,-\}$ et $s \in \{1,...,p\}$.
\end{Rem}

\begin{Prop} 
\label{Prop: PIMs}
Soient $\alpha \in \{-,+\}$ et $s \in \{1,...,p\}$. 
\begin{enumerate}[(a)]
	\item Si $s=p$, alors le PIM $\PIM^\alpha(p)$ s'identifie à $\X^\alpha(p)$ ;
	\item Sinon, le PIM $\PIM^\alpha(s)$ \index{PIMs2@ $\PIM^\alpha(s)$} est engendré sous $\Uq$ par un vecteur $y_0^\alpha(s)$, et il admet une $\C$-base $\{ x_n^\alpha(s), y_n^\alpha(s) \; ; \; 0 \leq n \leq s-1 \} \cup \{ a_m^{-\alpha}(p-s), b_m^{-\alpha}(p-s) \; ; \; 0 \leq m \leq p-s-1 \}$ telle que :
	\begin{align*}
	& \begin{multlined}[t][0.95\textwidth]
		\{ x_n^\alpha(s) \; ; \; 0 \leq n \leq s-1 \} \cup \{ a_m^{-\alpha}(p-s) \; ; \; 0 \leq m \leq p-s-1 \} \\
		\text{ est la base canonique de } \V^{-\alpha}(p-s), 
		\end{multlined} \\
	& \begin{multlined}[t][0.95\textwidth]
		\{ x_n^\alpha(s) \; ; \; 0 \leq n \leq s-1 \} \cup \{ b_m^{-\alpha}(p-s) \; ; \; 0 \leq m \leq p-s-1 \} \\
		\text{ est la base canonique de } \Vb^{-\alpha}(p-s), 
		\end{multlined} \\
	& K y_n^\alpha(s) = \alpha q^{s-1-2n} y_n^\alpha(s), \quad 0 \leq n \leq s-1, \\
	& \begin{multlined}[t][0.95\textwidth]
		E y_0^\alpha(s) = a_{p-s-1}^{-\alpha}(p-s), \\
		E y_n^\alpha(s) = \alpha [n][s-n] y_{n-1}^\alpha(s) + x_{n-1}^\alpha(s), \quad 1 \leq n \leq s-1, 
		\end{multlined} \\
	& F y_n^\alpha(s) = y_{n+1}^\alpha(s), \quad 0 \leq n \leq s-2,
		\qquad F y_{s-1}^\alpha(s) = b_0^{-\alpha}(p-s).
	\end{align*}
	On l'appelle la \emph{base canonique de $\PIM^\alpha(s)$}.
\end{enumerate}
\end{Prop}

\begin{proof}
Il s'agit de construire l'extension non triviale maximale du module simple $\X^\alpha(s)$. Pour cela, on construit une suite de modules $(M_n)_{n \geq 0}$ telle que, pour tout $n \geq 0$, $M_{n+1}$ est obtenu par extension non triviale des sous-modules simples de $M_n$ par un nombre maximal de modules simples. On utilisera régulièrement la description des modules simples de la proposition \ref{Prop: simples} et leurs classes d'extensions détaillées dans le lemme \ref{Lemme: Verma}. Les extensions successives seront illustrées dans la figure \ref{Fig: Prop: PIMs}.

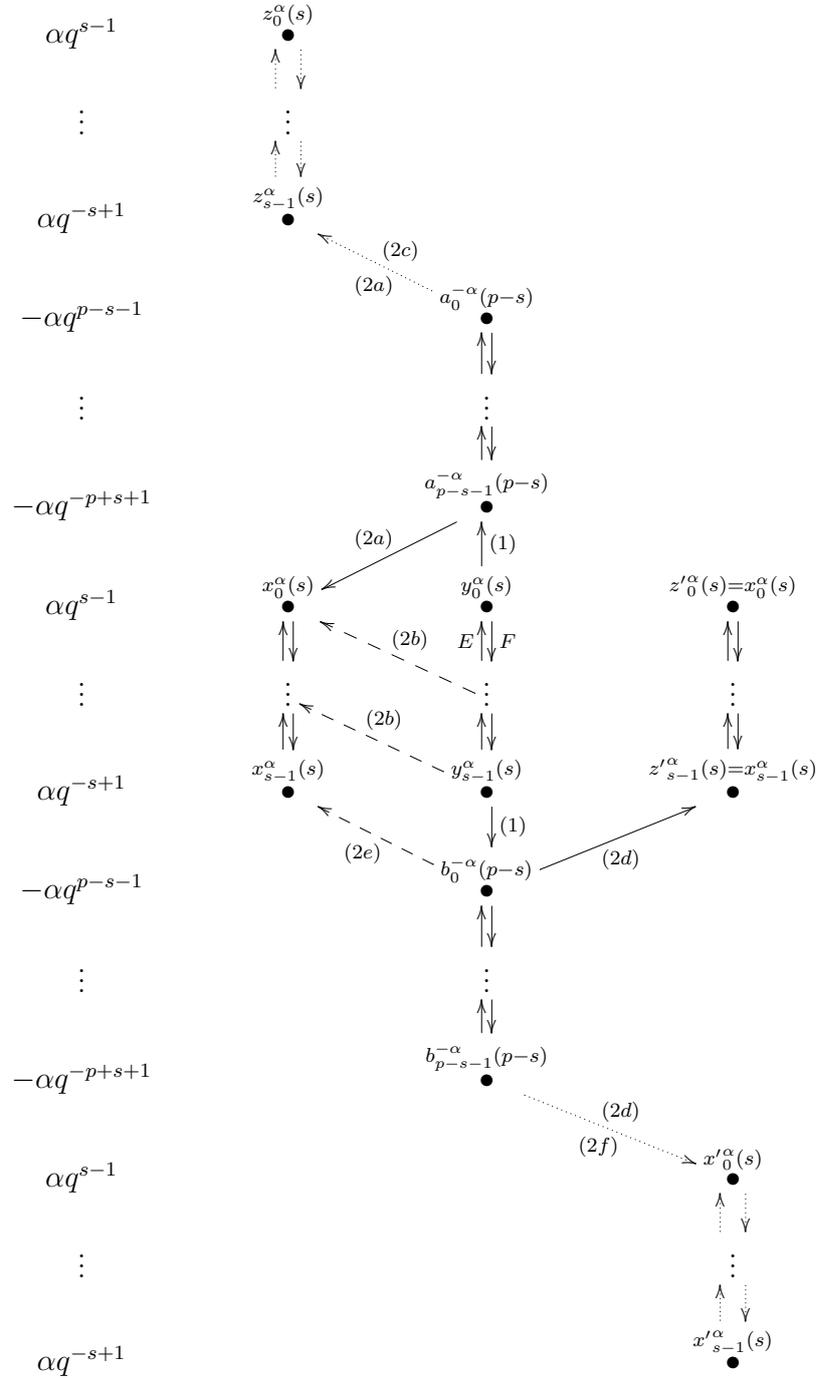
\begin{figure}[p]  
\caption{Extensions successives à partir de $\X^\alpha(s)$.}
\label{Fig: Prop: PIMs}
\[ \scalebox{0.95}{ \xymatrix @-2ex { 
	\alpha q^{s-1} &&
		\overset{z_0^\alpha(s)}{\bullet} \ar@{.>}@<5pt>[d] &&
		&&
		\\
	\vdots &&
		\vdots \ar@{.>}@<5pt>[d] \ar@{.>}@<5pt>[u] &&
		&&
		\\
	\alpha q^{-s+1} &&
		\overset{z_{s-1}^\alpha(s)}{\bullet} \ar@{.>}@<5pt>[u] &&
		&&
		\\
	-\alpha q^{p-s-1} &&
		&&
		\overset{a_0^{-\alpha}(p-s)}{\bullet} \ar@<2pt>[d] \ar@{.>}[llu]^{(2a)}_{(2c)} &&
		\\
	\vdots &&
		&&
		\vdots \ar@<2pt>[d] \ar@<2pt>[u] &&
		\\
	-\alpha q^{-p+s+1} &&
		&&
		\overset{a_{p-s-1}^{-\alpha}(p-s)}{\bullet} \ar[lld]_{(2a)} \ar@<2pt>[u] &&
		\\
	\alpha q^{s-1} &&
		\overset{x_0^\alpha(s)}{\bullet} \ar@<2pt>[d] &&
		\overset{y_0^\alpha(s)}{\bullet} \ar@<2pt>[d]^-F \ar@<2pt>[u]_-{(1)} && 
		\overset{{z'}_0^\alpha(s) = x_0^\alpha(s)}{\bullet} \ar@<2pt>[d] \\
		\vdots &&
		\vdots \ar@<2pt>[d] \ar@<2pt>[u] &&
		\vdots \ar@<2pt>[d] \ar@<2pt>[u]^-E \ar@{-->}[llu]_-{(2b)} && 
		\vdots \ar@<2pt>[d] \ar@<2pt>[u] \\
	\alpha q^{-s+1} &&
		\overset{x_{s-1}^\alpha(s)}{\bullet} \ar@<2pt>[u] && 
		\overset{y_{s-1}^\alpha(s)}{\bullet} \ar@<2pt>[d]^-{(1)} \ar@{-->}[llu]_-{(2b)} \ar@<2pt>[u] && 
		\overset{{z'}_{s-1}^\alpha(s) = x_{s-1}^\alpha(s)}{\bullet} \ar@<2pt>[u] \\
	-\alpha q^{p-s-1} &&
		&& 
		\overset{b_0^{-\alpha}(p-s)}{\bullet} \ar@<2pt>[d] \ar[rru]_{(2d)} \ar@{-->}[llu]^-{(2e)} &&
		\\
	\vdots &&
		&& 
		\vdots \ar@<2pt>[d] \ar@<2pt>[u] &&
		\\
	-\alpha q^{-p+s+1} &&
		&& 
		\overset{b_{p-s-1}^{-\alpha}(p-s)}{\bullet} \ar@{.>}[rrd]^{(2d)}_{(2f)} \ar@<2pt>[u] &&
		\\
	\alpha q^{s-1} &&
		&&
		&&
		\overset{{x'}_0^\alpha(s)}{\bullet} \ar@{.>}@<5pt>[d] \\
	\vdots &&
		&&
		&&
		\vdots \ar@{.>}@<5pt>[d] \ar@{.>}@<5pt>[u] \\
	\alpha q^{-s+1} &&
		&&
		&&
		\overset{{x'}_{s-1}^\alpha(s)}{\bullet} \ar@{.>}@<5pt>[u] }}\]
\end{figure}

Le cas $(a)$ découle directement du lemme \ref{Lemme: Verma} : $\X^\alpha(p)$ est l'extension non triviale maximale du module simple $\X^\alpha(p)$. D'où $\PIM^\alpha(p) = \X^\alpha(p)$. On suppose désormais que $s \neq p$ et on démontre la cas $(b)$.

On commence avec $M_0=\X^\alpha(s)$. On construit une extension non triviale de $\X^\alpha(s)$ par deux copies de $\X^{-\alpha}(p-s)$. On obtient ainsi un module indécomposable :
\[ M_1 = \vcenter{ \xymatrix @-2ex {
		& \overset{\X^\alpha(s)}{\bullet} \ar[ld]_-E \ar[rd]^-F \\
		\overset{\X^{-\alpha}(p-s)}{\bullet} && \overset{\X^{-\alpha}(p-s)}{\bullet} }} . \]
Quitte à le normaliser, il admet une $\C$-base $\{ a_m^{-\alpha}(p-s), b_m^{-\alpha}(p-s) \; ; \; 0 \leq m \leq p-s-1 \} \cup \{ y_n^\alpha(s) \; ; \; 0 \leq n \leq s-1 \}$ telle que :
\begin{equation} \tag{1}
\begin{aligned}
& \{ a_m^{-\alpha}(p-s) \; ; \; 0 \leq m \leq p-s-1 \} \\
& \qquad \text{ est la base canonique de la première copie de } \X^{-\alpha}(p-s), \\
& \{ b_m^{-\alpha}(p-s) \; ; \; 0 \leq m \leq p-s-1 \} \\
& \qquad \text{ est la base canonique de la seconde copie de } \X^{-\alpha}(p-s), \\
& K y_n^\alpha(s) = \alpha q^{s-1-2n} y_n^\alpha(s), \quad 0 \leq n \leq s-1, \\
& E y_0^\alpha(s) = a_{p-s-1}^{-\alpha}(p-s),
	\qquad E y_n^\alpha(s) = \alpha [n][s-n] y_{n-1}^\alpha(s), \quad 1 \leq n \leq s-1, \\
& F y_n^\alpha(s) = y_{n+1}^\alpha(s), \quad 0 \leq n \leq s-2,
	\qquad F y_{s-1}^\alpha(s) = b_0^{-\alpha}(p-s).
\end{aligned}
\end{equation}

On réitère le procédé avec $M_1$. Il possède deux sous-modules simples, en somme directe, tout deux isomorphes à $\X^{-\alpha}(p-s)$. On en construit des extensions non triviales par deux fois deux copies de $\X^\alpha(s)$. On obtient ainsi un module indécomposable :
\[ M_2 = \vcenter{ \xymatrix @-2ex {
		&&& \overset{\X^\alpha(s)}{\bullet} \ar[lld]_-E \ar[rrd]^-F \\
		& \overset{\X^{-\alpha}(p-s)}{\bullet} \ar[ld]_-E \ar[rd]^-F &&&& \overset{\X^{-\alpha}(p-s)}{\bullet} \ar[ld]_-E \ar[rd]^-F \\
		\overset{\X^\alpha(s)}{\bullet} && \overset{\X^\alpha(s)}{\bullet} && \overset{\X^\alpha(s)}{\bullet} && \overset{\X^\alpha(s)}{\bullet}}} . \]
On note $N_1$ le sous-module de $M_1$ $\C$-engendré par $\{ a_m^{-\alpha}(p-s) \; ; \; 0 \leq m \leq p-s-1 \}$, et $\tilde{N}_1$ l'extension normalisée de $N_1$ par deux copies de $\X^\alpha(s)$. Il admet une $\C$-base $\{ x_n^\alpha(s), z_n^\alpha(s) \; ; \; 0 \leq n \leq s-1 \} \cup \{ a_m^{-\alpha}(p-s) \; ; \; 0 \leq m \leq p-s-1 \}$ telle que :
\begin{equation} \tag{2a}
\begin{aligned}
& \begin{multlined}[t][0.9\textwidth]
	\{ z_n^\alpha(s) \; ; \; 0 \leq n \leq s-1 \} \\
	\text{ est la base canonique de la première copie de } \X^\alpha(s), 
	\end{multlined} \\
& \begin{multlined}[t][0.9\textwidth]
	\{ x_n^\alpha(s) \; ; \; 0 \leq n \leq s-1 \} \\
	\text{ est la base canonique de la seconde copie de } \X^\alpha(s), 
	\end{multlined} \\
& K a_m^{-\alpha}(p-s) = -\alpha q^{p-s-1-2m} a_m^{-\alpha}(p-s), \quad 0 \leq m \leq p-s-1, \\
& E a_0^{-\alpha}(p-s) = z_{s-1}^\alpha(s), \\
	& \qquad E a_m^{-\alpha}(p-s) = -\alpha [m][s-m] a_{m-1}^{-\alpha}(p-s), \quad 1 \leq m \leq p-s-1, \\
& F a_m^{-\alpha}(p-s) = a_{m+1}^{-\alpha}(p-s), \quad 0 \leq m \leq p-s-2, \\
	& \qquad F a_{p-s-1}^{-\alpha}(p-s) = x_0^\alpha(s).
\end{aligned}
\end{equation}
Il s'ensuit que, pour tout $n \in \{1,...,s-1\}$ :
\begin{equation} \tag{2b}
\begin{aligned}
E y_n^\alpha(s) &= EF^n y_0^\alpha(s)
\overset{\eqref{Eqn: comm2}}{=} F^nE y_0^\alpha(s) + \alpha [n][s-n] F^{n-1} y_0^\alpha(s) \\
&= F^{n-1} + \alpha [n][s-n] F^{n-1} y_0^\alpha(s)  x_0^\alpha(s) \\
&=  x_{n-1}^\alpha(s) + \alpha [n][s-n] y_{n-1}^\alpha(s).
\end{aligned}
\end{equation}
D'autre part, comme $E^p = 0$, on a :
\begin{equation} \tag{2c}
\begin{aligned}
0  &= E^p y_{s-1}^\alpha(s)
 = \alpha^{s-1} ([s-1]!)^2 E^{p-s+1} y_0^\alpha(s) \\
&= \alpha^{s-1} ([s-1]!)^2 E^{p-s} a_{p-s-1}^{-\alpha}(p-s) \\
&= \alpha^{s-1} (-\alpha)^{p-s-1} ([s-1]!)^2 ([p-s-1]!)^2 E a_0^{-\alpha}(p-s) \\ 
&= \alpha^{s-1} (-\alpha)^{p-s-1} ([s-1]!)^2 ([p-s-1]!)^2 z_{s-1}^\alpha(s).
\end{aligned}
\end{equation}
Donc $z_{s-1}^\alpha(s)=0$ ; la famille $\{z_n^\alpha \; ; \; 0 \leq n \leq s-1 \}$ est identiquement nulle dans $M_2$. \\
De même, on note $N_1'$ le sous-module de $M_1$ $\C$-engendré par $\{ b_m^{-\alpha}(p-s) \; ; \; 0 \leq m \leq p-s-1 \}$, et $\tilde{N}_1'$ l'extension normalisée de $N_1'$ par deux copies de $\X^\alpha(s)$. Il admet une $\C$-base $\{ {x'}_n^\alpha(s), {z'}_n^\alpha(s) \; ; \; 0 \leq n \leq s-1 \} \cup \{ b_m^{-\alpha}(p-s) \; ; \; 0 \leq m \leq p-s-1 \}$ telle que :
\begin{equation} \tag{2d}
\begin{aligned}
& \begin{multlined}[t][0.9\textwidth]
	\{ {z'}_n^\alpha(s) \; ; \; 0 \leq n \leq s-1 \} \\
	\text{ est la base canonique de la première copie de } \X^\alpha(s), 
	\end{multlined} \\
& \begin{multlined}[t][0.9\textwidth]
	\{ {x'}_n^\alpha(s) \; ; \; 0 \leq n \leq s-1 \} \\
	\text{ est la base canonique de la seconde copie de } \X^\alpha(s), 
	\end{multlined} \\
& K b_m^{-\alpha}(p-s) = -\alpha q^{p-s-1-2m} b_m^{-\alpha}(p-s), \quad 0 \leq m \leq p-s-1, \\
& E b_0^{-\alpha}(p-s) = {z'}_{s-1}^\alpha(s), \\
	& \qquad E b_m^{-\alpha}(p-s) = -\alpha [m][s-m] b_{m-1}^{-\alpha}(p-s), \quad 1 \leq m \leq p-s-1, \\
& F b_m^{-\alpha}(p-s) = b_{m+1}^{-\alpha}(p-s), \quad 0 \leq m \leq p-s-2, \\
	& \qquad F b_{p-s-1}^{-\alpha}(p-s) = {x'}_0^\alpha(s).
\end{aligned}
\end{equation}
Il s'ensuit que  :
\begin{equation} \tag{2e}
\begin{aligned}
{z'}_{s-1}^\alpha(s) & \; = EF y_{s-1}^\alpha(s)
\overset{\eqref{Eqn: comm1}}{=} FE y_{s-1}^\alpha(s) - \alpha [s-1] y_{s-1}^\alpha(s) \\
& \overset{(2b)}{=} F x_{s-2}^\alpha(s) + \alpha [s-1] F y_{s-2}^\alpha(s) - \alpha [s-1] y_{s-1}^\alpha(s) \\
& \; = x_{s-1}^\alpha(s) + \alpha [s-1] y_{s-1}^\alpha(s) - \alpha [s-1] y_{s-1}^\alpha(s) \\
& \; = x_{s-1}^\alpha(s).
\end{aligned}
\end{equation}
Donc ${z'}_n^\alpha(s) = x_n^\alpha(s)$, $0 \leq n \leq s-1$. D'autre part, comme $F^p = 0$, on a :
\begin{equation} \tag{2f}
0 = F^p y_0^\alpha(s) 
= F^{p-s+1} y_{s-1}^\alpha(s) 
= F^{p-s} b_0^{-\alpha}(p-s)
= F b_{p-s-1}^{-\alpha}(p-s) 
= {x'}_0^\alpha(s).
\end{equation}
Donc ${x'}_0^\alpha(s)=0$ ; la famille $\{{x'}_n^\alpha \; ; \; 0 \leq n \leq s-1 \}$ est identiquement nulle dans $M_2$. \\
Par conséquent, on a :
\[ M_2 = \vcenter{ \xymatrix @-2ex {
		& \overset{\X^\alpha(s)}{\bullet} \ar[ld]_-E \ar@{-->}[dd]_-E \ar[rd]^-F \\
		\overset{\X^{-\alpha}(p-s)}{\bullet} \ar[rd]_-F && \overset{\X^{-\alpha}(p-s)}{\bullet} \ar[ld]^-E  \\
		& \overset{\X^\alpha(s)}{\bullet} }} . \]
Il admet une $\C$-base $\{ x_n^\alpha(s), y_n^\alpha(s) \; ; \; 0 \leq n \leq s-1 \} \cup \{ a_m^{-\alpha}(p-s), b_m^{-\alpha}(p-s) \; ; \; 0 \leq m \leq p-s-1 \}$ telle que :
\begin{align*}
& \begin{multlined}[t][0.95\textwidth]
	\{ x_n^\alpha(s) \; ; \; 0 \leq n \leq s-1 \} \cup \{ a_m^{-\alpha}(p-s) \; ; \; 0 \leq m \leq p-s-1 \} \\
	\text{ est la base canonique de } \V^{-\alpha}(p-s),
	\end{multlined} \\
& \begin{multlined}[t][0.95\textwidth]
	\{ x_n^\alpha(s) \; ; \; 0 \leq n \leq s-1 \} \cup \{ b_m^{-\alpha}(p-s) \; ; \; 0 \leq m \leq p-s-1 \} \\
	\text{ est la base canonique de } \Vb^{-\alpha}(p-s), 
	\end{multlined} \\
& K y_n^\alpha(s) = \alpha q^{s-1-2n} y_n^\alpha(s), \quad 0 \leq n \leq s-1, \\
& \begin{multlined}[t][0.95\textwidth]
	E y_0^\alpha(s) = a_{p-s-1}^{-\alpha}(p-s), \\
	E y_n^\alpha(s) = \alpha [n][s-n] y_{n-1}^\alpha(s) + x_{n-1}^\alpha(s), \quad 1 \leq n \leq s-1, 
	\end{multlined} \\
& F y_n^\alpha(s) = y_{n+1}^\alpha(s), \quad 0 \leq n \leq s-2,
	\qquad F y_{s-1}^\alpha(s) = b_0^{-\alpha}(p-s).
\end{align*}

On réitère le procédé avec $M_2$. Il possède un sous-module simple isomorphe à $\X^\alpha(s)$. On en construit une extension non triviale par deux copies de $\X^{-\alpha}(p-s)$. On obtient ainsi un module indécomposable :
\[ M_3 = \vcenter{ \xymatrix @-2ex {
		& \overset{\X^\alpha(s)}{\bullet} \ar[ld]_-E \ar@{-->}[dd]_-E \ar[rd]^-F \\
		\overset{\X^{-\alpha}(p-s)}{\bullet} \ar[rd]_-F && \overset{\X^{-\alpha}(p-s)}{\bullet} \ar[ld]^-E  \\
		& \overset{\X^\alpha(s)}{\bullet} \ar[ld]_-E \ar[rd]^-F \\
		\overset{\X^{-\alpha}(p-s)}{\bullet} && \overset{\X^{-\alpha}(p-s)}{\bullet}  }} . \]
On note $N_2$ le sous-module de $M_2$ $\C$-engendré par $\{ x_n^\alpha(s) \; ; \; 0 \leq n \leq s-1 \}$. Pour les mêmes raisons que précédemment, les conditions $E^p=0$ et $F^p=0$ annulent les extensions de $N_2$ par $\X^{-\alpha}(p-s)$ dans $M_3$. Donc $M_2 = M_3$ et, par définition, $\PIM^\alpha(s) = M_2$.
\end{proof}

\begin{Rem} 
\label{Rem: dim PIMs}
\begin{enumerate}[(i)]
	\item Pour tout $\alpha \in \{-,+\}$, on a $\PIM^\alpha(p) = \X^\alpha(p) = \V^\alpha(p) = \Vb^\alpha(p)$.
	\item Pour tous $\alpha \in \{-,+\}$ et $s \in \{1,...,p-1\}$, on a $\dim \PIM^\alpha(s) = 2p$.
\end{enumerate}
\end{Rem}

On utilise la même représentation diagrammatique que précédemment, illustrée dans la figure \ref{Fig: PIMs}, ou plus succinctement :
\[ \PIM^\alpha(s) = \vcenter{ \xymatrix @-2ex {
	& \overset{\X^\alpha(s)}{\bullet} \ar[ld]_-E \ar@{-->}[dd]_-E \ar[rd]^-F \\
	\overset{\X^{-\alpha}(p-s)}{\bullet} \ar[rd]_-F && \overset{\X^{-\alpha}(p-s)}{\bullet} \ar[ld]^-E  \\
	& \overset{\X^\alpha(s)}{\bullet} }} 
= \vcenter{ \xymatrix @-2ex {
	\overset{\V^\alpha(s)}{\bullet} \ar[dd]_-E \\ \\
	\overset{\V^{-\alpha}(p-s)}{\bullet} }}
= \vcenter{ \xymatrix @-2ex {
	\overset{\Vb^\alpha(s)}{\bullet} \ar[dd]_-E^-F \\ \\
	\overset{\Vb^{-\alpha}(p-s)}{\bullet} }} . \]

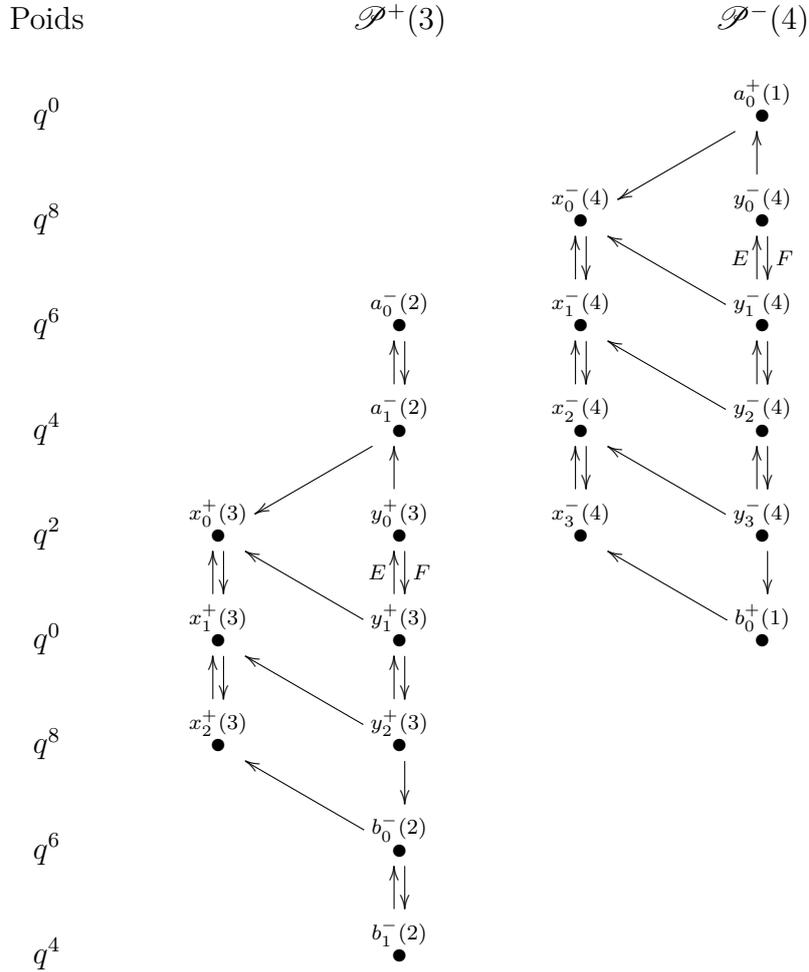
\begin{figure}[!ht]
\caption{Représentation diagrammatique de PIM pour p=5.}
\label{Fig: PIMs}
\[ \xymatrix @-2ex {
	\text{Poids} &&
		&&
		\PIM^+(3) &&
		&&
		\PIM^-(4) \\
	q^0 &&
		&&
		&&
		&&
		\overset{a_0^+(1)}{\bullet} \ar[lld] \\
	q^8 &&
		&&
		&&
		\overset{x_0^-(4)}{\bullet} \ar@<2pt>[d] && 
		\overset{y_0^-(4)}{\bullet} \ar@<2pt>[d]^-F \ar@<2pt>[u] \\
	q^6 &&
		&& 
		\overset{a_0^-(2)}{\bullet} \ar@<2pt>[d] &&
		\overset{x_1^-(4)}{\bullet} \ar@<2pt>[d] \ar@<2pt>[u] && 
		\overset{y_1^-(4)}{\bullet} \ar@<2pt>[d] \ar@<2pt>[u]^-E \ar[llu] \\
	q^4 &&
		&& 
		\overset{a_1^-(2)}{\bullet} \ar[lld] \ar@<2pt>[u] &&
		\overset{x_2^-(4)}{\bullet} \ar@<2pt>[u] \ar@<2pt>[d] && 
		\overset{y_2^-(4)}{\bullet} \ar@<2pt>[d] \ar@<2pt>[u] \ar[llu] \\
	q^2 && 
		\overset{x_0^+(3)}{\bullet} \ar@<2pt>[d] && 
		\overset{y_0^+(3)}{\bullet} \ar@<2pt>[d]^-F \ar@<2pt>[u] &&
		\overset{x_3^-(4)}{\bullet} \ar@<2pt>[u] && 
		\overset{y_3^-(4)}{\bullet} \ar@<2pt>[d] \ar@<2pt>[u] \ar[llu] \\
	q^0 && 
		\overset{x_1^+(3)}{\bullet} \ar@<2pt>[d] \ar@<2pt>[u] && 
		\overset{y_1^+(3)}{\bullet} \ar@<2pt>[d] \ar@<2pt>[u]^-E \ar[llu] &&
		&&
		\overset{b_0^+(1)}{\bullet} \ar[llu] \\
	q^8 && 
		\overset{x_2^+(3)}{\bullet} \ar@<2pt>[u] && 
		\overset{y_2^+(3)}{\bullet} \ar@<2pt>[d] \ar@<2pt>[u] \ar[llu] &&
		&&
		\\
	q^6 &&
		&& 
		\overset{b_0^-(2)}{\bullet} \ar@<2pt>[d] \ar[llu] &&
		&&
		\\
	q^4 &&
		&& 
		\overset{b_1^-(2)}{\bullet} \ar@<2pt>[u] &&
		&&
		}\]
\end{figure}

\begin{Thm} 
\label{Thm: Reg}
La représentation régulière $\Reg$ de $\Uq$ se décompose de manière unique, à isomorphisme et ordre des facteurs près, en la somme directe :
\[ \Reg \cong \bigoplus_{s=1}^{p-1} \left[ s \PIM^+(s) \oplus s \PIM^-(s) \right] \oplus p \X^+(p) \oplus p \X^-(p). \]
\end{Thm}

\begin{proof}
On sait que $\Reg$ se décompose de manière unique, à isomorphisme et ordre des facteurs près, en somme directe des PIMs (cf. par exemple \cite{CR62}[Thm 14.5, Thm 56.6]). Dans cette décomposition, la multiplicité de la classe d'isomorphisme d'un PIM est donnée par la dimension de son module quotient simple (cf. par exemple \cite[(26.8), Thm 54.5, Thm 54.11]{CR62}). D'après la remarque \ref{Rem: nbre PIMs}, à isomorphisme près, il y a $2p$ PIMs $\PIM^\pm(s)$, $1 \leq s \leq p$. Leurs modules quotients simples sont respectivement $\X^\pm(s)$, $1 \leq s \leq p$, et vérifient $\dim( \X^\pm(s) )=s$ (cf. la proposition \ref{Prop: PIMs} et la remarque \ref{Rem: dim simples})). Par conséquent :
\[ \Reg \cong \bigoplus_{s=1}^{p-1} \left[ s \PIM^+(s) \oplus s \PIM^-(s) \right] \oplus p \underbrace{\PIM^+(p)}_{\X^+(p)} \oplus p \underbrace{\PIM^-(p)}_{\X^-(p)} . \]
\end{proof}

\begin{Rem} 
\label{Rem: ideaux}
En particulier, tous les modules précédemment décrits sont isomorphes à des idéaux à gauche de $\Uq$.
\end{Rem}

\section{Produits de modules simples et de PIMs} 
\label{section: reps prod}

On étudie à présent l'ensemble des produits tensoriels  possibles à isomorphisme près dans la catégorie $\Rep^{fd}_s$. Naturellement, on commence par la description des produits tensoriels de $\Uq$-modules simples. Plus précisément, on donne leurs décompositions en facteurs directs annoncées dans l'introduction de la section \ref{section: reps}. \`A cette occasion, on verra que ces facteurs directs sont des $\Uq$-modules simples ou des $\Uq$-PIMs. Pour compléter l'étude de $\Rep^{fd}_s$, on décrit donc également les produits tensoriels de $\Uq$-modules simples-PIMs, PIMs-simples et PIMs-PIMs. \`A cette fin, on utilisera la notation de Kronecker : \index{delta @$\delta_{i,j}$} \index{delta @$\delta_{i \in E}$}
\[ \forall i,j \in \N \qquad \delta_{i,j} = \begin{cases} 1 & \text{ si } i=j, \\ 0 & \text{ sinon. } \end{cases} \]
Par extension, pour tout sous-ensemble $E$ de $\N$, on notera :
\[ \forall i \in \N \qquad \delta_{i \in E} = \begin{cases} 1 & \text{ si } i \in E, \\ 0 & \text{ sinon. } \end{cases} \]

Ces quatre types de produits tensoriels (simples-simples, simples-PIMs, PIMs-simples, et PIMs-PIMs) suffisent à donner une description exhaustive de $\Rep^{fd}_s$ pour les raisons suivantes.

\begin{Lemme} 
\label{Lemme: proj produit}
Soit $\X$ un $\Uq$-module. Le produit tensoriel de $\X$ par $\Uq$ à gauche (resp. à droite) induit deux structures de $\Uq$-module isomorphes données par :
\begin{gather*}
\forall a,b \in \Uq \quad \forall x \in \X \quad 
	a (b \otimes x) \overset{(1)}{:=} \Delta(a) (b \otimes x)
	\quad \text{ et } \quad
	a (b \otimes x)\overset{(2)}{:=} ab \otimes x \\
\left( \text{ resp. } \quad  
	a (x \otimes b) \overset{(1)}{:=} \Delta(a) (x \otimes b)
	\quad \text{ et } \quad
	a (x \otimes b)\overset{(2)}{:=} x \otimes ab \right).
\end{gather*}
\end{Lemme}

\begin{proof}
On utilise les notations de Sweedler (cf. par exemple \cite[Not. III.1.6]{Kas95}) :
\[ \forall a \in \Uq \qquad \Delta(a) = \sum_{(a)} a' \otimes a'' . \]
\begin{enumerate}[(i)]
	\item Les deux structures de $\Uq$-module sur $\Uq \otimes \X$ sont isomorphes via les morphismes $\Uq$-linéaires réciproques :
\begin{align*}
& \alpha : \left\{ \begin{array}{l}
	\Uq \otimes_1 \X \longrightarrow \Uq \otimes_2 \X \\
	a \otimes x \longmapsto \sum_{(a)} a' \otimes S(a'') x \\
	\end{array} \right. , \\
& \beta : \left\{ \begin{array}{l}
	\Uq \otimes_2 \X \longrightarrow \Uq \otimes_1 \X \\
	a \otimes x \longmapsto \sum_{(a)} a' \otimes a'' x \\
	\end{array} \right. .
\end{align*}
En effet, pour tous $a \in \Uq$ et $x \in \X$, on a :
\begin{align*}
\alpha \circ \beta (a \otimes x) =& \sum_{(a)} \alpha (a' \otimes a'' x)
= \sum_{(a)} a' \otimes S(a'') a''' x \\
=& \sum_{(a)} a' \otimes \varepsilon(a'') x
= \sum_{(a)} a' \varepsilon(a'') \otimes  x
= a \otimes x, \\
\beta \circ \alpha (a \otimes x) =& \sum_{(a)} \alpha (a' \otimes S(a'') x)
= \sum_{(a)} a' \otimes a'' S(a''') x \\
=& \sum_{(a)} a' \otimes \varepsilon(a'') x
= \sum_{(a)} a' \varepsilon(a'') \otimes  x
= a \otimes x.
\end{align*}
	\item Les deux structures de $\Uq$-module sur $\X \otimes \Uq$ sont isomorphes via les morphismes $\Uq$-linéaires réciproques :
\begin{align*}
& \alpha : \left\{ \begin{array}{l}
	\X \otimes_1 \Uq \longrightarrow \X \otimes_2 \Uq \\
	x \otimes a \longmapsto \sum_{(a)} S^{-1}(a') x \otimes a'' \\
	\end{array} \right. \\
& \beta : \left\{ \begin{array}{l}
	\X \otimes_2 \Uq \longrightarrow \X \otimes_1 \Uq \\
	x \otimes a \longmapsto \sum_{(a)} a' x \otimes a'' \\
	\end{array} \right. .
\end{align*}
En effet, pour tous $a \in \Uq$ et $x \in \X$, on a :
\begin{align*}
\alpha \circ \beta (x \otimes a) =& \sum_{(a)} \alpha (S^{-1}(a') x \otimes a'')
= \sum_{(a)} a'' S^{-1}(a') x \otimes a''' \\
=& \sum_{(a)} \varepsilon(a') x \otimes a''
= x \otimes \sum_{(a)} \varepsilon(a') a''
= x \otimes a, \\
\beta \circ \alpha (x \otimes a) =& \sum_{(a)} \alpha (a' x \otimes a'')
= \sum_{(a)} S^{-1}(a'') a' x \otimes a''' \\
=& \sum_{(a)} \varepsilon(a') x \otimes a''
= x \otimes \sum_{(a)} \varepsilon(a') a''
= x \otimes a.
\end{align*}
L'existence de l'inverse de l'antipode est assurée par le théorème \cite[Thm 2.1.3]{Mon93}, car $\Uq$ est une $\C$-algèbre de dimension finie.
\end{enumerate}
\end{proof}

\begin{Prop} 
\label{Prop: proj produit}
Soit $\X$ un $\Uq$-module. Le produit tensoriel de $\X$ par un $\Uq$-module projectif est encore un $\Uq$-module projectif. 
\end{Prop}

\begin{proof}
D'après le lemme \ref{Lemme: proj produit}, le produit tensoriel de $\X$ par un PIMs est un $\Uq$-module projectif. Tout module projectif se décomposant en somme directe de PIMs (cf. par exemple \cite[Thm 56.6]{CR62}), ce résultat s'étend à tous les modules projectifs.
\end{proof}

Tout module projectif se décomposant en somme directe de PIMs (cf. par exemple \cite[Thm 56.6]{CR62}), les facteurs directs des produits tensoriels de $\Uq$-modules simples-PIMs, PIMs-simples et PIMs-PIMs sont des PIMs.

\subsection{Produits de modules simples} 
\label{subsection: produits simples}

Soient $\alpha, \alpha' \in \{+,-\}$ et $s, s' \in \{1,...,p\}$. On commence par étudier le produit tensoriel $\X^\alpha(s) \otimes \X^{\alpha'}(s')$. Pour cela, on définit une loi multiplicative sur l'ensemble $\{+,-\}$ telle que :
\[ + + = +, \quad + - = -, \quad - + = -, \quad - - = +. \]
Compte-tenu de la structure de cogèbre de $\Uq$ (donnée parmi les équations \eqref{Eqn: coprod1} et de la description des modules simples donnée dans la proposition \ref{Prop: simples}, on a :
\[ \X^\alpha(s) \otimes \X^\pm(1) 
\cong \X^{\pm \alpha}(s) 
\cong \X^\pm(1) \otimes  \X^\alpha(s). \]
On remarque alors que :
\begin{equation} 
\label{Eqn: produits simples}
\X^\alpha(s) \otimes \X^-(s')
\cong \X^\alpha(s) \otimes \X^-(1) \otimes \X^+(s') 
\cong \X^{-\alpha}(s) \otimes \X^+(s').
\end{equation}
Il suffit donc d'étudier le produit $\X^\alpha(s) \otimes \X^+(s')$.

\begin{Lemme} 
\label{Lemme: produits simples1}
Soient $\alpha \in \{+,-\}$ et $s \in \{1,...,p\}$. On a :
\[ \X^\alpha(s) \otimes \X^+(2) \cong 
	\begin{cases}
		\X^\alpha(2) & \text{ si } s=1, \\
		\X^\alpha(s-1) \oplus \X^\alpha(s+1) & \text{ si } 2 \leq s \leq p-1, \\
		\PIM^\alpha(p-1) & \text{ si } s=p.
	\end{cases} \]
\end{Lemme}

\begin{proof}
On commence par déterminer les vecteurs de poids de $\X^\alpha(s) \otimes \X^+(2)$. Soient $\{ x_n^\alpha(s) \; ; \; 0 \leq n \leq s-1 \}$ la $\C$-base canonique de $\X^\alpha(s)$ et $\{ x_0^+(2), x_1^+(2) \}$ celle de $\X^+(2)$ (cf. la proposition \ref{Prop: simples}). Sachant que $\Delta(K) = K \otimes K$ \eqref{Eqn: coprod1}, les vecteurs de poids de $\X^\alpha(s) \otimes \X^+(2)$ sont de la forme :
\begin{equation}
x(s+1-2n):= C_{n,0} \; x_n^\alpha(s) \otimes x_0^+(2) + \delta_{n \geq 1} C_{n-1,1} \; x_{n-1}^\alpha(s) \otimes x_1^+(2), 
	\quad 0 \leq n \leq s-1,
\end{equation}
où $C_{n,0}, C_{n-1,1}$ sont des constantes pour tout $n \in \{0,...,s-1\}$. 

Cherchons lesquels sont de plus haut poids. On fixe $n \in \{0,...,s-1\}$ et on pose $s''=s+1-2n$. On a :
\begin{align*}
\Delta(K) x(s'') \overset{\eqref{Eqn: coprod1}}{=}& K \otimes K \; x(s'') = \alpha q^{s''-1} x(s''), \\
\Delta(E) x(s'') \overset{\eqref{Eqn: coprod1}}{=}& (1 \otimes E + E \otimes K) \; x(s'') \\
	= \; \; & C_{n,0} \alpha q [n] [s-n] \; x_{n-1}^\alpha(s) \otimes x_0^+(2) + \delta_{n \geq 1} C_{n-1,1} \; x_{n-1}^\alpha \otimes x_0^+(2) \\
	& + \delta_{n \geq 1} C_{n-1,1} \alpha q^{-1} [n-1][s-n+1] \; x_{n-2}^\alpha(s) \otimes x_1^+(2) .
\end{align*}
L'équation $\Delta(E) x(s'')=0$ donne un vecteur $x(s'')$ non nul si $n \in \{0,1\}$.  Dans ce cas, $s'' \in \{s+1,s-1\}$ et l'équation $\Delta(E) x(s'')=0$ est équivalente à :
\[ C_{0,1} = - \alpha q [s-1] C_{1,0} . \]
Pour tout $n \in \{0,1\}$, on choisit un vecteur $x(s'')$ solution de cette équation tel que :
\[ C_{n,0}=1  ; \]
donc $x(s'')$ est un vecteur de plus haut poids $\alpha q^{s''-1}$. Par conséquent, il y a au plus deux vecteurs de plus haut poids :
\begin{equation} \tag{1}
\begin{aligned}
& x_0^\alpha(s+1) := x(s+1) && \text{ si } s \leq p-1, \\
& x_0^{-\alpha}(1) := x(s+1) && \text{ si } s=p, \\
& x_0^\alpha(s-1) := x(s-1) && \text{ si } s \geq 2,
\end{aligned}
\end{equation}
à multiplication par un scalaire non nul près.

D'après le corollaire \ref{Cor: php/pbp}, l'action de $\Delta(F)$ sur ces vecteurs de plus haut poids engendre des modules respectivement isomorphes à :
\begin{align*}
& \X^\alpha(s+1) \text{ ou } \V^\alpha(s+1) && \text{ si } s \leq p-1, \\
& \X^{-\alpha}(1)\text{ ou } \V^{-\alpha}(1) && \text{ si } s=p, \\
& \X^\alpha(s-1) \text{ ou } \V^\alpha(s-1) && \text{ si } s \geq 2.
\end{align*}
Afin de départager ces structures possibles, on compte le nombre de vecteurs de plus haut poids et on identifie les couples qui appartiennent à un même module de Verma. On distingue trois cas.

\begin{Cas} On suppose que $s = 1$.
D'après l'inventaire $(1)$, il y a exactement un vecteur de plus haut poids à multiplication par un scalaire non nul près : $x_0^\alpha(2)$. Sous $\Uq$, il engendre donc  un sous-module isomorphe à $\X^\alpha(2)$. Les $\C$-espaces vectoriels $\X^\alpha(1) \otimes \X^+(2)$ et $\X^\alpha(2)$ étant de même dimension, ils sont isomorphes.
\end{Cas}

\begin{Cas} On suppose que $s \in \{2,...,p-1\}$.
D'après l'inventaire $(1)$, il y a exactement deux vecteurs de plus haut poids à multiplication par un scalaire non nul près : $x_0^\alpha(s+1)$ et $x_0^\alpha(s-1)$. Sous $\Uq$, ils sont indépendants et engendrent donc deux sous-modules respectivement isomorphes à $\X^\alpha(s+1)$ et à $\X^\alpha(s-1)$. Il s'ensuit que :
\[  \X^\alpha(s-1) \oplus \X^\alpha(s+1) \hookrightarrow \X^\alpha(s) \otimes \X^+(2). \]
Enfin, les dimensions des $\C$-espaces vectoriels correspondants sont égales. Ils sont donc isomorphes.
\end{Cas}

\begin{Cas} On suppose que $s = p$.
D'après l'inventaire $(1)$, il y a exactement deux vecteurs de plus haut poids à multiplication par un scalaire non nul près : $x_0^{-\alpha}(1)$ et $x_0^\alpha(p-1)$. De plus :
\begin{align*}
\Delta(F) x_0^{-\alpha}(1) 
\overset{(1)}{=} \; & \Delta(F) x_0^\alpha(p) \otimes x_0^+(2) 
\overset{\eqref{Eqn: coprod1}}{=} \left( K^{-1} \otimes F + F \otimes 1 \right) x_0^\alpha(p) \otimes x_0^+(2) \\
= \, \; & \alpha q^{-p+1} \; x_0^\alpha(p) \otimes x_1^+(2) + x_1^\alpha(p) \otimes x_0^+(2) \\
= \, \; & x_1^\alpha(p) \otimes x_0^+(2) - \alpha q \; x_0^\alpha(p) \otimes x_1^+(2) \\
\overset{(1)}{=} \; & x_0^\alpha(p-1).
\end{align*}
Il s'ensuit que les deux vecteurs de plus haut poids $x_0^{-\alpha}(1)$ et $x_0^\alpha(p-1)$ engendrent un sous-module isomorphe à  $\V^{-\alpha}(1)$. Donc :
\[  \V^{-\alpha}(1) \hookrightarrow \X^\alpha(p) \otimes \X^+(2). \]
Par ailleurs, comme $\X^\alpha(p) = \PIM^\alpha(p)$ est projectif, le produit tensoriel $\X^\alpha(p) \otimes \X^+(2)$ est aussi projectif (cf. la remarque \ref{Prop: proj produit}). Par conséquent, il se décompose en somme directe de PIMs (cf. par exemple \cite[Thm 56.6]{CR62}). D'après la remarque \ref{Rem: nbre PIMs} et la proposition \ref{Prop: PIMs}, il y a $2p$ classes d'isomorphisme de PIMs données par :
\[ \PIM^\pm(s'') = \vcenter{ \xymatrix @-2ex {
	\overset{\V^\pm(s'')}{\bullet} \ar[d]_-E \\
	\overset{\V^{\mp}(p-s'')}{\bullet} }} , 
	\quad 1 \leq s'' \leq p. \]
On en déduit que :
\[ \PIM^\alpha(p-1) \hookrightarrow \X^\alpha(p) \otimes \X^+(2) . \]
Enfin, les dimensions des $\C$-espaces vectoriels correspondants sont égales. Ils sont donc isomorphes.
\end{Cas}
\end{proof}

\begin{Rem} 
\label{Rem: facteurs simples}
Pour tout $s \in \{1,...,p\}$, le module simple $\X^+(s)$ est facteur direct de $\X^+(2)^{\otimes(s-1)}$ (récurrence immédiate).
\end{Rem}

\begin{Lemme} 
\label{Lemme: produits simples2}
Soient $\alpha \in \{+,-\}$ et $s,s' \in \{1,...,p\}$. Les facteurs directs du $\Uq$-module $\X^\alpha(s) \otimes \X^+(s')$ sont des modules simples ou des PIMs.
\end{Lemme}

\begin{proof}
D'après la remarque \ref{Rem: facteurs simples}, on sait que  $\X^+(s')$ est un facteur direct de $\X^+(2)^{\otimes (s'-1)}$. Il existe donc un $\Uq$-module $\mathcal{Y}$ tel que :
\[ \X^\alpha(s) \otimes \X^+(2)^{\otimes (s'-1)}
\cong \left( \X^\alpha(s) \otimes \X^+(s') \right) \oplus \mathcal{Y}. \] 
Il suffit alors de montrer le résultat pour $\X^\alpha(s) \otimes \X^+(2)^{\otimes (s'-1)}$. 

Pour cela, on procède par induction sur la puissance $k \in \{0,...,s'-1\}$ de $\X^+(2)$. Le résultat est immédiat pour $k=0$ car $\X^\alpha(s) \otimes \C \cong \X^\alpha(s)$ est un module simple. Soit $k \in \{0,...,s'-1\}$. On suppose que les facteurs directs de $\X^\alpha(s) \otimes \X^+(2)^{\otimes k}$ sont soit des modules simples, soit des PIMs. Pour tout $s'' \in \{1,...,p\}$, on note $\lambda^\pm_{s''}$ la multiplicité du facteur $\X^\pm(s'')$ et $\mu^\pm(s'')$ celle de $\PIM^\pm(s'')$ (cf. la remarque \ref{Rem: nbre PIMs}). Alors, on a :
\[ \X^\alpha(s) \otimes \X^+(2)^{\otimes k} \cong \bigoplus_{\substack{1 \leq s'' \leq p \\ \alpha'' \in \{+,-\} \\ \lambda_{s'}^{\alpha''} \not = 0}} \lambda^{\alpha''}_{s''} \X^{\alpha''}(s'') \oplus \bigoplus_{\substack{1 \leq s'' \leq p-1 \\ \alpha'' \in \{+,-\} \\ \mu_{s''}^{\alpha''} \not = 0}} \mu^{\alpha''}_{s''} \PIM^{\alpha''}(s''). \]
Il s'ensuit que :
\begin{align*}
\X^\alpha(s) \otimes \X^+(2)^{\otimes k+1} \cong& \bigoplus_{\substack{1 \leq s'' \leq p \\ \alpha'' \in \{+,-\} \\ \lambda_{s''}^{\alpha''} \not = 0}} \lambda^{\alpha''}_{s''} \left( \X^{\alpha''}(s'') \otimes \X^+(2) \right) \\
&\qquad \oplus \bigoplus_{\substack{1 \leq s'' \leq p-1 \\ \alpha'' \in \{+,-\} \\ \mu_{s''}^{\alpha''} \not = 0}} \mu^{\alpha''}_{s''} \left( \PIM^{\alpha''}(s'') \otimes \X^+(2) \right).
\end{align*}
D'après le lemme \ref{Lemme: produits simples1}, pour tout $s'' \in \{1,...,p\}$, les facteurs directs du module $\X^\pm(s'') \otimes \X^+(2)$ sont soit des modules simples, soit des PIMs. D'autre part, pour tout $s'' \in \{1,...,p\}$, le module $\PIM^\pm(s'') \otimes \X^+(2)$ est projectif car $\PIM^\pm(s'')$ l'est (cf. remarque \ref{Prop: proj produit}). Or, d'après le théorème de Krull-Schmidt, tout module projectif se décompose en somme directe de PIMs (cf. par exemple \cite[Thm 56.6]{CR62}). D'où le résultat.
\end{proof}

\begin{Thm} 
\label{Thm: produits simples}
Soient $\alpha \in \{+,-\}$ et $s,s' \in \{1,...,p\}$. On a :
\[ \X^\alpha(s) \otimes \X^+(s') \cong 
	\bigoplus_{\substack{s''=|s-s'|+1 \\ \mathrm{pas}=2}}^{p-1-|p-s-s'|} \X^\alpha(s'') 
	\oplus \bigoplus_{\substack{s''=2p-s-s'+1 \\ \mathrm{pas}=2}}^p \PIM^\alpha(s''). \]
\end{Thm}

\begin{proof}
A l'aide du lemme \ref{Lemme: produits simples2}, on généralise le procédé utilisé dans la preuve du lemme \ref{Lemme: produits simples1} donnant la décomposition pour $s'=2$. 

On commence par déterminer les vecteurs de poids de $\X^\alpha(s) \otimes \X^+(s')$. On se donne $\{ x_n^\alpha(s) \; ; \; 0 \leq n \leq s-1 \}$ la $\C$-base canonique de $\X^\alpha(s)$ et $\{ x_m^+(s') \; ; \; 0 \leq m \leq s'-1 \}$ celle 
$\X^+(s')$ (cf. la proposition \ref{Prop: simples}). Sachant que $\Delta(K) = K \otimes K$ \eqref{Eqn: coprod1}, les vecteurs de poids de $\X^\alpha(s) \otimes \X^+(s')$ sont de la forme :
\[ x(s+s'-1-2k) := \sum_{\substack{0 \leq n \leq s-1 \\ 0 \leq m \leq s'-1 \\ n+m=k}} C_{n,m} \; x_n^\alpha(s) \otimes x_m^+(s'), 
	\quad 0 \leq k \leq s+s'-2, \]
où $C_{n,m}$ est une constante pour tous $n \in \{0,...,s-1\}$ et $m \in \{0,...,s'-1\}$. 

Cherchons lesquels sont de plus haut poids. On fixe $k \in \{0,...,s+s'-2\}$ et on pose $s'' = s+s'-1-2k$. On a :
\begin{align*}
\Delta(K) x(s'') \overset{\eqref{Eqn: coprod1}}{=}& K \otimes K \; x(s'') = \alpha q^{s''-1} x(s''), \\
\Delta(E) x(s'') 
\overset{\eqref{Eqn: coprod1}}{=}& \sum_{\substack{0 \leq n \leq s-1 \\ 0 \leq m \leq s'-1 \\ n+m=k}} C_{n,m} \left( 1 \otimes E + E \otimes K \right) x_n^\alpha(s) \otimes x_m^+(s') \\
= \; \; & \sum_{\substack{0 \leq n \leq s-1 \\ 1 \leq m \leq s'-1 \\ n+m=k}} C_{n,m} [m] [s'-m] x_n^\alpha(s) \otimes x_{m-1}^+(s') \\
& + \sum_{\substack{1 \leq n \leq s-1 \\ 0 \leq m \leq s'-1 \\ n+m=k}} C_{n,m} \alpha q^{s'-1-2m} [n] [s-n] x_{n-1}^\alpha(s) \otimes x_m^+(s') \\
= \; \; & \sum_{\substack{0 \leq n \leq s-1 \\ 0 \leq m \leq s'-2 \\ n+m=k-1}} C_{n,m+1} [m+1] [s'-m-1] x_n^\alpha(s) \otimes x_m^+(s') \\
& + \sum_{\substack{0 \leq n \leq s-2 \\ 0 \leq m \leq s'-1 \\ n+m=k-1}} C_{n+1,m} \alpha q^{s'-1-2m} [n+1] [s-n-1] x_n^\alpha(s) \otimes x_m^+(s').
\end{align*}
L'équation $\Delta(E)x(s'')=0$ donne un vecteur $x(s'')$ non nul si $k \in \{0,...,\min(s,s')-1\}$.  Dans ce cas, $s'' \in \{|s-s'|+1,|s-s'|+3,...,s+s'-1\}$ et l'équation $\Delta(E)x(s'')$ est équivalente à :
\[ C_{n,m} = (-\alpha)^m q^{m(s'-m)} \frac{[k]! [s-1-n]! [s'-1-m]!}{[n]! [s-1-k]! [m]! [s'-1]!} C_{k,0} , \]
pour tous $n \in \{0,...,s-2\}$ et $m \in \{0,...,s'-2\}$ tels que $n+m=k$. Pour tout $k \in \{0,...,\min(s,s')-1\}$, on choisit un vecteur $x(s'')$ solution de cette équation tel que :
\[ C_{k,0}=1 ; \]
donc $x(s'')$ est un vecteur de plus haut poids $\alpha q^{s''-1}$. Par conséquent, il y a $\min(s,s')$ vecteur(s) de plus haut poids :
\begin{equation} \tag{1}
\begin{aligned}
&\begin{multlined}[t][0.9\textwidth]
	x_0^\alpha(s'') := x(s''), \quad s'' \in \{|s-s'|+1, |s-s'|+3, ..., s+s'-1\}, \\ 	\text{ si } s+s' \leq p+1, 
	\end{multlined} \\
&\begin{multlined}[t][0.9\textwidth]
	x_0^\alpha(s'') := x(s''), \quad s'' \in \{|s-s'|+1, |s-s'|+3, ..., b(s,s')\}, \\
	\text{ si } s+s' \geq p+2,
	\end{multlined} \\
&\begin{multlined}[t][0.9\textwidth]
	x_0^{-\alpha}(s''-p) := x(s''), \quad s'' \in \{b(s,s')+2, b(s,s')+4, ..., s+s'-1\}, \\ 
	\text{ si } s+s' \geq p+2.
	\end{multlined} \\
\end{aligned}
\end{equation}
à multiplication par un scalaire non nul près, où :
\[ b(s,s') = \begin{cases}
	p-1 & \text{ si } s+s' = p \mod 2, \\
	p & \text{ si } s+s' = p-1 \mod 2. 
	\end{cases} \]

D'après le corollaire \ref{Cor: php/pbp}, l'action de $\Delta(F)$ sur ces vecteurs de plus haut poids engendre des modules respectivement isomorphes à :
\begin{align*}
&\begin{multlined}[t][0.9\textwidth]
	\X^\alpha(s'') \text{ ou } \V^\alpha(s''), \quad s'' \in \{|s-s'|+1, |s-s'|+3, ..., s+s'-1\}, \\
	\text{ si } s+s' \leq p+1, 
	\end{multlined} \\
&\begin{multlined}[t][0.9\textwidth]
	\X^\alpha(s'') \text{ ou } \V^\alpha(s''), \quad s'' \in \{|s-s'|+1, |s-s'|+3, ..., b(s,s')\}, \\
	\text{ si } s+s' \geq p+2, 
	\end{multlined} \\
&\begin{multlined}[t][0.9\textwidth]
	\X^{-\alpha}(p-s'') \text{ ou } \V^{-\alpha}(p-s''), \\
	s'' \in \{2p-s-s'+1, 2p-s-s'+3, ..., 2p-b(s,s')-2\},
	\quad \text{ si } s+s' \geq p+2.
	\end{multlined}
\end{align*}
Or, d'après le lemme \ref{Lemme: produits simples2}, les facteurs directs de $\X^\alpha(s) \otimes \X^+(s')$ sont soit des modules simples, soit des PIMs. De plus, d'après la remarque \ref{Rem: nbre PIMs} et la proposition \ref{Prop: PIMs}, il y a $2p$ classes d'isomorphisme de PIMs données par :
\[ \PIM^\pm(s'') = \vcenter{ \xymatrix @-2ex {
	\overset{\V^\pm(s'')}{\bullet} \ar[d]_-E \\
	\overset{\V^{\mp}(p-s'')}{\bullet} }} , 
	\quad 1 \leq s'' \leq p. \]
On en déduit que les facteurs directs de $\X^\alpha(s) \otimes \X^+(s')$ sont isomorphes à :
\begin{align*}
&\begin{multlined}[t][0.98\textwidth]
	\X^\alpha(s'') \text{ ou } \PIM^{-\alpha}(p-s''), \quad s'' \in \{|s-s'|+1, |s-s'|+3, ..., s+s'-1\}, \\
	\text{ si } s+s' \leq p+1, 
	\end{multlined} \\
&\begin{multlined}[t][0.98\textwidth]
	\X^\alpha(s'') \text{ ou } \PIM^{-\alpha}(p-s''), \quad s'' \in \{|s-s'|+1, |s-s'|+3, ..., b(s,s')\}, \\
	\text{ si } s+s' \geq p+2, 
	\end{multlined} \\
&\begin{multlined}[t][0.9\textwidth]
	\X^{-\alpha}(p-s'') \text{ ou } \PIM^\alpha(s''), \\
	s'' \in \{2p-s-s'+1, 2p-s-s'+3, ..., 2p-b(s,s')-2\},
	\quad \text{ si } s+s' \geq p+2.
	\end{multlined}
\end{align*}
Afin de départager ces structures possibles, on compte le nombre de vecteurs de plus haut poids et on identifie les couples liés sous $\Uq$. On distingue deux cas.

\begin{Cas} On suppose que $s+s' \leq p+1$.
D'après l'inventaire $(1)$, le(s) vecteur(s) de plus haut poids sont : 
\[ x^\alpha(s''), \quad s'' \in \{|s-s'|+1, |s-s'|+3,...,s+s'-1\} , \]
à multiplication par un scalaire non nul près. Sous $\Uq$, ils sont deux à deux indépendants et engendrent donc des sous-modules respectivement isomorphes à :
\[ \X^\alpha(s''), \quad s'' \in \{|s-s'|+1, |s-s'|+3,...,s+s'-1\} . \]
Il s'ensuit que :
\[ \X^\alpha(s) \otimes \X^+(s') \cong 
	\bigoplus_{\substack{s''=|s-s'|+1 \\ \mathrm{pas}=2}}^{s+s'-1} \X^\alpha(s'') . \]
\end{Cas}

\begin{Cas} On suppose que $s+s' \geq p+2$.
D'après l'inventaire $(1)$, les vecteurs de plus haut poids sont : 
\begin{equation} \tag{2}
\begin{aligned}
& x_0^\alpha(s''), \quad s'' \in \{|s-s'|+1, |s-s'|+3, ..., b(s,s')\}, \\
& x_0^{-\alpha}(p-s''), \quad s'' \in \{2p-s-s'+1, 2p-s-s'+3, ..., 2p-b(s,s')-2\},
\end{aligned}
\end{equation}
à multiplication par un scalaire non nul près, où :
\[ |s-s'|+1 \leq 2p-s-s'+1 \leq 2p-b(s,s')-2 \leq b(s,s') . \]
Parmi ceux-ci, les vecteurs de plus hauts poids :
\begin{align*}
& x_0^\alpha(s''), \quad s'' \in \{|s-s'|+1, |s-s'|+3, ..., 2p-s-s'-1\}, && \\
& x_0^\alpha(p) && \text{ si } b(s,s')=p,
\end{align*}
sont deux à deux indépendants et ne sont pas liés aux autres sous $\Uq$. Ils engendrent donc des facteurs directs respectivement isomorphes à :
\begin{align*}
& \X^\alpha(s''), \quad s'' \in \{|s-s'|+1, |s-s'|+3, ..., 2p-s-s'-1\}, && \\
& \X^\alpha(p) && \text{ si } b(s,s')=p.
\end{align*}
Reste à étudier la structure engendrée par les couples :
\[ \left( x_0^\alpha(s''), x_0^{-\alpha}(p-s'') \right), \quad s'' \in \{2p-s-s'+1, 2p-s-s'+3, ..., 2p-b(s,s')-2\}. \]
On fixe $s'' \in \{2p-s-s'+1, |s-s'|+3, ..., s+s'-1\}$. D'après les équations\eqref{Eqn: comm1} et \eqref{Eqn: comm2} sur les commutateurs, les vecteurs $\Delta(F^{s''}) x_0^\alpha(s'')$ et $\Delta(F^{p-s''}) x_0^{-\alpha}(p-s'')$ sont nuls ou de plus haut poids $\alpha q^{s''-1}$. Or, d'après les équations \eqref{Eqn: coprod2}, pour tout $M \in \N$ :
\begin{align*}
\Delta(F^{M}) x(s'')
&= \sum_{\substack{0 \leq n \leq s- 1 \\ 0 \leq m \leq s'-1 \\ 2(n+m)=s+s'-1-s''}} C_{n,m} \sum_{r=0}^{M}   { M \brack r } q^{r(M-r)} F^r K^{r-M} x_n^\alpha(s) \otimes F^{M-r} x_m^+(s') \\
&= \sum_{\substack{0 \leq n \leq s-1 \\ 0 \leq m \leq s'-1 \\ 2(n+m)=s+s'-1-s''}}
	\begin{multlined}[t]
	 C_{n,m} \sum_{r=0}^{M} {M \brack r } \alpha^{r-M} q^{(r-s+1+2n)(M-r)} \\
	\times \delta_{m+M-s'+1 \leq r \leq s-1-n} \; x_{n+r}^\alpha(s) \otimes x_{m+M-r}^+(s').
	\end{multlined}
\end{align*}
Donc $\Delta(F^{s''}) x_0^\alpha(s'')$ est nul, et $\Delta(F^{p-s''}) x_0^{-\alpha}(p-s'')$ est un vecteur de plus haut poids $\alpha q^{s''-1}$. D'après l'inventaire $(2)$, il existe $\lambda \in \C^*$ tel que :
\[ \Delta(F^{p-s''}) x_0^{-\alpha}(p-s'') = \lambda x_0^\alpha(s''). \]
Il s'ensuit que les deux vecteurs de plus haut poids $x_0^\alpha(s'')$ et $x_0^{-\alpha}(p-s'')$ engendrent un sous-module isomorphe à $\V^{-\alpha}(p-s'')$ dans $\PIM^\alpha(s'')$.
Par conséquent, les couples de vecteurs de plus haut poids étudiés contribuent à des facteurs directs isomorphes à :
\[ \PIM^\alpha(s''), \quad s'' \in \{2p-s-s'+1, 2p-s-s'+3, ..., 2p-b(s,s')-2\}. \]
En regroupant les facteurs directs obtenus, on obtient :
\begin{align*}
\X^\alpha(s) \otimes \X^+(s') \cong& \bigoplus_{\substack{s''=|s-s'|+1 \\ \mathrm{pas}=2}}^{2p-s-s'-1} \X^\alpha(s'') \oplus \delta_{b(s,s'),p} \X^\alpha(p) \oplus \bigoplus_{\substack{s''=2p-s-s'-1 \\ \mathrm{pas}=2}}^{2p-b(s,s')-2} \PIM^\alpha(s'') \\
\cong& \bigoplus_{\substack{s''=|s-s'|+1 \\ \mathrm{pas}=2}}^{2p-s-s'-1} \X^\alpha(s'') \oplus \delta_{b(s,s'),p} \X^\alpha(p) \oplus \bigoplus_{\substack{s''=2p-s-s'-1 \\ \mathrm{pas}=2}}^{p-1} \PIM^\alpha(s'') \\
\cong& \bigoplus_{\substack{s''=|s-s'|+1 \\ \mathrm{pas}=2}}^{2p-s-s'-1} \X^\alpha(s'') \oplus \bigoplus_{\substack{s''=2p-s-s'-1 \\ \mathrm{pas}=2}}^{p} \PIM^\alpha(s'') ,
\end{align*}
car $\X^\pm(p) = \PIM^\pm(p)$.
\end{Cas}
La formule donnée dans l'énoncé regroupe les cas 1 et 2. D'où le résultat.
\end{proof}

\begin{Cor} 
\label{Cor: produits simples}
Soient $\alpha \in \{+,-\}$ et $s,s' \in \{1,...,p\}$. On a :
\[ \X^\alpha(s) \otimes \X^{\alpha'}(s') \cong 
	\bigoplus_{\substack{s''=|s-s'|+1 \\ \mathrm{pas}=2}}^{p-1-|p-s-s'|} \X^{\alpha \alpha'}(s'') 
	\oplus \bigoplus_{\substack{s''=2p-s-s'+1 \\ \mathrm{pas}=2}}^p \PIM^{\alpha \alpha'}(s''). \]
\end{Cor}

\begin{proof}
Le résultat découle directement du théorème \ref{Thm: produits simples} et des identifications \eqref{Eqn: produits simples}.
\end{proof}

\begin{Rem} 
\label{Rem: comm produit}
Le résultat du corollaire \ref{Cor: produits simples} est symétrique en les couples de variables $(\alpha,s)$ et $(\alpha',s')$, donc :
\[ \X^{\alpha'}(s') \otimes \X^\alpha(s) \cong \X^\alpha(s) \otimes \X^{\alpha'}(s'). \]
Par construction, il en sera de même pour tous les produits tensoriels prochainement décrits.
\end{Rem}

\subsection{Produits de modules simples-PIMs}

Soient $\alpha, \alpha' \in \{+,-\}$ et $s, s' \in \{1,...,p\}$. On étudie maintenant le produit tensoriel $\X^\alpha(s) \otimes \PIM^{\alpha'}(s')$. Pour cela, on construit d'abord une extension de $\X^\alpha(s) \otimes \X^{\alpha'}(s')$ par $\X^\alpha(s) \otimes \X^{-\alpha'}(p-s')$, puis une extension de $\X^\alpha(s) \otimes \V^{\alpha'}(s')$ par $\X^\alpha(s) \otimes \V^{-\alpha'}(p-s')$.

\begin{Lemme} 
\label{Lemme: produits simples-Verma}
Soient $\alpha, \alpha' \in \{+,-\}$ et $s,s' \in \{1,...,p\}$. On a :
\begin{align*}
\X^\alpha(s) \otimes \V^{\alpha'}(s') \cong& \bigoplus_{\substack{s''=|s-s'|+1 \\ \mathrm{pas}=2}}^{p-1-|p-s-s'|} \V^{\alpha \alpha'}(s'')
\oplus \bigoplus_{\substack{s''=p-s+s'+1 \\ \mathrm{pas}=2}}^p \PIM^{-\alpha \alpha'}(s'')
\oplus \bigoplus_{\substack{s''=2p-s-s'+1 \\ \mathrm{pas}=2}}^p \PIM^{\alpha \alpha'}(s'').
\end{align*}
\end{Lemme}

\begin{proof}
Par construction (cf. la proposition \ref{Prop: Verma}), le module de Verma $\V^{\alpha'}(s')$ est une extension de $\X^{\alpha'}(s')$ par $\X^{-\alpha'}(p-s')$ telle que :
\[ \V^{\alpha'}(s') = \vcenter{ \xymatrix @-2ex {
	\overset{\X^{\alpha'}(s')}{\bullet} \ar[d]^-F \\
	\overset{\X^{-\alpha'}(p-s')}{\bullet} }} . \]
A la vue des expressions des coproduits $\Delta(E)$, $\Delta(F)$ et $\Delta(K)$ \eqref{Eqn: coprod1} et de la structure de module de $\X^\alpha(s)$ (cf. la proposition \ref{Prop: simples}), le produit tensoriel $\X^\alpha(s) \otimes \V^{\alpha'}(s')$ est une extension de $\X^\alpha(s) \otimes \X^{\alpha'}(s')$ par $\X^\alpha(s) \otimes \X^{-\alpha}(p-s)$ telle que :
\begin{equation} \tag{1}
\X^\alpha(s) \otimes \V^{\alpha'}(s') = \vcenter{ \xymatrix @-2ex {
	\overset{\X^\alpha(s) \otimes \X^{\alpha'}(s')}{\bullet} \ar[d]^-{\Delta(F)} \\
	\overset{\X^\alpha(s) \otimes \X^{-\alpha}(p-s)}{\bullet} }} .
\end{equation}

D'après le corollaire \ref{Cor: produits simples}, on connait les facteurs directs de $\X^\alpha(s) \otimes \X^{\alpha'}(s')$ et $\X^\alpha(s) \otimes \X^{-\alpha}(p-s)$ :
\begin{align}
\X^\alpha(s) \otimes \X^{-\alpha'}(p-s') 
	\cong & \bigoplus_{\substack{s''=|s+s'-p|+1 \\ \mathrm{pas}=2}}^{p-1-|s'-s|} \X^{-\alpha \alpha'}(s'') 
	\oplus \bigoplus_{\substack{s''=p-s+s'+1 \\ \mathrm{pas}=2}}^p \PIM^{-\alpha \alpha'}(s'') \\
	\cong & \bigoplus_{\substack{s''=|s'-s|+1 \\ \mathrm{pas}=2}}^{p-1-|s+s'-p|} \X^{-\alpha \alpha'}(p-s'')
	\oplus \bigoplus_{\substack{s''=p-s+s'+1 \\ \mathrm{pas}=2}}^p \PIM^{-\alpha \alpha'}(s''), \tag{2} \\ 
\X^\alpha(s) \otimes \X^{\alpha'}(s') 
\cong & \bigoplus_{\substack{s''=|s-s'|+1 \\ \mathrm{pas}=2}}^{p-1-|p-s-s'|} \X^{\alpha \alpha'}(s'') 
\oplus \bigoplus_{\substack{s''=2p-s-s'+1 \\ \mathrm{pas}=2}}^p \PIM^{\alpha\alpha'}(s''). \tag{3}
\end{align}
D'après le théorème de Krull-Schmidt (cf. par exemple \cite[Thm 14.5]{CR62}), il s'agit donc de déterminer des classes d'extension indécomposables des facteurs directs de $(2)$ par des facteurs directs de $(3)$ qui vérifient $(1)$. 

Or, par construction, les PIMs n'ont pas d'extension non triviale, et les classes des modules de Verma $\V^{\alpha \alpha'}(s'')$ et $\Vb^{\alpha \alpha'}(s'')$, $1 \leq s'' \leq p$, forment une base du $\C$-espace vectoriel des classes d'extensions de $\X^{\alpha \alpha'}(s'')$ par $\X^{- \alpha \alpha'}(p-s'')$ (cf. lemme \ref{Lemme: Verma}) :
\[ \V^{\alpha \alpha'}(s'') = \vcenter{ \xymatrix @-2ex {
	\overset{\X^{\alpha \alpha'}(s'')}{\bullet} \ar[d]^-F \\
	\overset{\X^{-\alpha \alpha'}(p-s'')}{\bullet} }}, \quad
\Vb^{\alpha \alpha'}(s'') = \vcenter{ \xymatrix @-2ex {
	 \overset{\X^{\alpha \alpha'}(s'')}{\bullet} \ar[d]^-E \\
	 \overset{\X^{-\alpha \alpha'}(p-s'')}{\bullet} }}. \]
On en déduit que :
\begin{align*}
\X^\alpha(s) \otimes \V^{\alpha'}(s') 
\cong & \bigoplus_{\substack{s''=|s-s'|+1 \\ \mathrm{pas}=2}}^{p-1-|p-s-s'|} \V^{\alpha \alpha'}(s'')
\bigoplus_{\substack{s''=p-s+s'+1 \\ \mathrm{pas}=2}}^p \PIM^{-\alpha \alpha'}(s'')
\oplus \bigoplus_{\substack{s''=2p-s-s'+1 \\ \mathrm{pas}=2}}^p \PIM^{\alpha \alpha'}(s'').
\end{align*}
\end{proof}

\begin{Rem} 
Le lemme \ref{Lemme: produits simples-Verma} est aussi valable en remplaçant le module de Verma $\V^{\alpha'}(s')$ par le module de Verma contragrédient $\Vb^{\alpha'}(s')$.
\end{Rem}

\begin{Thm} 
\label{Thm: produits simples-PIMs}
Soient $\alpha, \alpha' \in \{+,-\}$ et $s,s' \in \{1,...,p\}$. On a :
\begin{align*}
\X^\alpha(s) \otimes \PIM^{\alpha'}(s')
\cong&  \bigoplus_{\substack{s''=|s-s'|+1 \\ \mathrm{pas}=2}}^{p-1-|p-s-s'|} \PIM^{\alpha \alpha'}(s'')
\oplus \bigoplus_{\substack{s''=2p-s-s'+1 \\ \mathrm{pas}=2}}^p 2 \PIM^{\alpha \alpha'}(s'') \\
&\qquad \oplus \bigoplus_{\substack{s''=p-s+s'+1 \\ \mathrm{pas}=2}}^p 2 \PIM^{-\alpha \alpha'}(s'').
\end{align*}
\end{Thm}

\begin{proof}
Par construction (cf. la proposition \ref{Prop: PIMs}), le PIM $\PIM^{\alpha'}(s')$ est une extension de $\V^{\alpha'} (s')$ par $\V^{-\alpha'}(p-s')$ telle que :
\[ \PIM^{\alpha'}(s') = \vcenter{ \xymatrix @-2ex {
	\overset{\V^{\alpha'}(s')}{\bullet} \ar[d]^-E \\
	\overset{\V^{-\alpha'}(p-s')}{\bullet} }} . \]
A la vue des expressions des coproduits $\Delta(E)$, $\Delta(F)$ et $\Delta(K)$ \eqref{Eqn: coprod1} et de la structure de module de $\X^\alpha(s)$, le produit tensoriel $\X^\alpha(s) \otimes \PIM^{\alpha'}(s')$ est une extension de $\X^\alpha(s) \otimes \V^{\alpha'}(s')$ par $\X^\alpha(s) \otimes \V^{-\alpha'}(p-s')$ telle que :
\begin{equation} \tag{1}
\X^\alpha(s) \otimes \PIM^{\alpha'}(s') = \vcenter{ \xymatrix @-2ex {
	 \\
	\overset{\X^\alpha(s) \otimes \V^{\alpha'}(s')}{\bullet} \ar[d]^-{\Delta(E)} \\
	\overset{\X^\alpha(s) \otimes \V^{-\alpha}(p-s)}{\bullet} }} .
\end{equation}
	
D'après le lemme \ref{Lemme: produits simples-Verma}, on connaît les facteurs directs de $\X^\alpha(s) \otimes \V^{\alpha'}(s')$ et $\X^\alpha(s) \otimes \V^{-\alpha'}(p-s')$ :
\begin{align}
& \begin{aligned}
	\X^\alpha(s) \otimes \V^{-\alpha'}(p-s') 
\cong & \bigoplus_{\substack{s''=|s+s'-p|+1 \\ \mathrm{pas}=2}}^{p-1-|s'-s|} \V^{-\alpha \alpha'}(s'') 
	\oplus \bigoplus_{\substack{s''=2p-s-s'+1 \\ \mathrm{pas}=2}}^p \PIM^{\alpha \alpha'}(s'') \\
	&\qquad \oplus \bigoplus_{\substack{s''=p-s+s'+1 \\ \mathrm{pas}=2}}^p \PIM^{-\alpha \alpha'}(s'') \\
	\cong & \bigoplus_{\substack{s''=|s'-s|+1 \\ \mathrm{pas}=2}}^{p-1-|s+s'-p|} \V^{- \alpha \alpha'}(p-s'') 
	\oplus \bigoplus_{\substack{s''=2p-s-s'+1 \\ \mathrm{pas}=2}}^p \PIM^{\alpha \alpha'}(s'') \\
	&\qquad \oplus \bigoplus_{\substack{s''=p-s+s'+1 \\ \mathrm{pas}=2}}^p \PIM^{-\alpha \alpha'}(s''),
	\end{aligned} \tag{2} \\
& \begin{aligned}
	\X^\alpha(s) \otimes \V^{\alpha'}(s') 
	\cong& \bigoplus_{\substack{s''=|s-s'|+1 \\ \mathrm{pas}=2}}^{p-1-|p-s-s'|} \V^{\alpha \alpha'}(s'') 
	\oplus \bigoplus_{\substack{s''=p-s+s'+1 \\ \mathrm{pas}=2}}^p \PIM^{-\alpha \alpha'}(s'') \\
	&\qquad \oplus \bigoplus_{\substack{s''=2p-s-s'+1 \\ \mathrm{pas}=2}}^{p-1} \PIM^{\alpha \alpha'}(s''). 	\end{aligned} \tag{3}
\end{align}
De plus, comme $\PIM^{\alpha'}(s')$ est projectif, le produit tensoriel $\X^\alpha(s) \otimes \PIM^{\alpha'}(s')$ est aussi projectif (cf. la remarque \ref{Prop: proj produit}). D'après un raffinement du théorème de Krull-Schmidt (cf. par exemple \cite[Thm 56.6]{CR62}), il s'agit donc de déterminer les PIMs qui réalisent une extension des facteurs directs de $(2)$ par des facteurs directs de $(3)$ suivant $(1)$.

Or, d'après la remarque \ref{Rem: nbre PIMs}, il y a $2p$ classes d'isomorphisme de PIMs données par :
\[ \PIM^{\alpha \alpha'}(s'') = \vcenter{ \xymatrix @-2ex {
		\overset{\V^{\alpha \alpha'}(s'')}{\bullet} \ar[d]_-E \\
		\overset{\V^{-\alpha \alpha'}(p-s'')}{\bullet} }}, 
	\quad 1 \leq s'' \leq p \; ; \]
lesquelles n'ont pas d'extension non triviale. On en déduit que :
\begin{align*}
\X^\alpha(s) \otimes \PIM^{\alpha'}(s') 
\cong & \bigoplus_{\substack{s''=|s-s'|+1 \\ \mathrm{pas}=2}}^{p-1-|p-s-s'|} \PIM^{\alpha \alpha'}(s'')
\oplus \bigoplus_{\substack{s''=2p-s-s'+1 \\ \mathrm{pas}=2}}^p 2 \PIM^{\alpha \alpha'}(s'') \\
&\qquad \oplus \bigoplus_{\substack{s''=p-s+s'+1 \\ \mathrm{pas}=2}}^p 2 \PIM^{-\alpha \alpha'}(s'').
\end{align*}
\end{proof}

Pour des raisons de symétrie énoncées précédemment dans la remarque \ref{Rem: comm produit}, on sait que :
\[ \PIM^{\alpha'}(s') \otimes \X^\alpha(s) \cong \X^\alpha(s) \otimes \PIM^{\alpha '}(s'). \]

\subsection{Produits de PIMs}

Soient $\alpha, \alpha' \in \{+,-\}$ et $s, s' \in \{1,...,p\}$. On étudie le produit tensoriel $\PIM^\alpha(s) \otimes \PIM^{\alpha'}(s')$. Pour cela, on construit d'abord une extension de $\X^\alpha(s) \otimes \PIM^{\alpha'}(s')$ par $\X^{-\alpha}(p-s) \otimes \PIM^{\alpha'}(s')$, puis une extension de $\V^\alpha(s) \otimes \PIM^{\alpha'}(s')$ par $\V^{-\alpha}(p-s) \otimes \PIM^{\alpha'}(s')$

\begin{Lemme} 
\label{Lemme: produits Verma-PIMs}
Soient $\alpha, \alpha' \in \{+,-\}$ et $s,s' \in \{1,...,p\}$. On a :
\begin{align*}
\V^\alpha(s) \otimes \PIM^{\alpha'}(s') 
\cong & \bigoplus_{\substack{s''=|s-s'|+1 \\ \mathrm{pas}=2}}^{p-1-|p-s-s'|} \left( \PIM^{-\alpha \alpha'}(p-s'') \oplus \PIM^{\alpha \alpha'}(s'') \right) \\
&\qquad \oplus \bigoplus_{\substack{s''=p+1-|s-s'| \\ \mathrm{pas}=2}}^p 2 \PIM^{-\alpha \alpha'}(s'')
\oplus \bigoplus_{\substack{s''=p+1-|p-s-s'| \\ \mathrm{pas}=2}}^p 2 \PIM^{\alpha \alpha'}(s'').
\end{align*}
\end{Lemme}

\begin{proof}
On procède comme dans la preuve du lemme \ref{Lemme: produits simples-Verma}, donnant les facteurs directs des produits tensoriels de modules simples par des modules de Verma.
\end{proof}

\begin{Rem} 
Le lemme \ref{Lemme: produits Verma-PIMs} est aussi valable en remplaçant le module de Verma $\V^{\alpha'}(s')$ par le module de Verma contragrédient $\Vb^{\alpha'}(s')$.
\end{Rem}

\begin{Thm} 
\label{Thm: produits PIMs}
Soient $\alpha, \alpha' \in \{+,-\}$ et $s,s' \in \{1,...,p\}$. On a :
\begin{align*}
\PIM^\alpha(s) \otimes \PIM^{\alpha'}(s') 
\cong & \bigoplus_{\substack{s''=|s-s'|+1 \\ \mathrm{pas}=2}}^{p-1-|p-s-s'|} \left( 2 \PIM^{-\alpha \alpha'}(p-s'') \oplus 2 \PIM^{\alpha \alpha'}(s'') \right) \\
&\qquad \oplus \bigoplus_{\substack{s''=p+1-|s-s'| \\ \mathrm{pas}=2}}^p 4 \PIM^{-\alpha \alpha'}(s'')
\oplus \bigoplus_{\substack{s''=p+1-|p-s-s'| \\ \mathrm{pas}=2}}^p 4 \PIM^{\alpha \alpha'}(s'').
\end{align*}
\end{Thm}

\begin{proof}
On procède comme dans la preuve du théorème \ref{Thm: produits simples-PIMs}, donnant les facteurs directs des produits tensoriels de modules simples par des PIMs.
\end{proof}
\chapter{Le centre du groupe quantique restreint}

En théorie des représentations, les PIMs d'une algèbre de dimension finie ont une étroite relation avec les éléments de son centre (cf. par exemple \cite[§ 55]{CR62}). Dans le cas des groupes quantiques quotients $\Uq$, associés à l'algèbre de lie $sl_2$ et aux racines $q$ de l'unité, chaque module simple ou couple de PIMs correspond canoniquement à un élément du centre. On obtient ainsi une \emph{base canonique} du centre dont on connaît le système de relations. La première description explicite de cette base est donnée par Kerler en 1995 pour les \emph{petits} groupes quantiques (associés aux racines \emph{impaires} de l'unité) grâce à des moyens calculatoires (cf. \cite[§ 3.2]{Ker95}). En 2005, Feigin, Gainutdinov, Semikhanov et Tipunin font le lien avec la théorie des représentations et en donnent la description pour les groupes quantiques \emph{restreints} (associés aux racines \emph{paires} de l'unité, cf. \cite[§ 4]{FGST06b}). Comme dans le chapitre précédent, on concentre notre travail sur le groupe quantique restreint associé à une racine primitive paire de l'unité. Dans ce deuxième chapitre, on suit l'approche de \cite{FGST06b} et on détaille leurs résultats. A cette occasion, on utilisera régulièrement les notations du chapitre I, et les descriptions des modules données dans la section \ref{section: reps}. Il sera alors facile de développer l'approche de Kerler en exprimant les éléments de cette base canonique à l'aide de polynômes en l'\emph{élément de Casimir}.

Toutefois, l'une ou l'autre de ces approches ne tient pas compte de la structure particulière des groupes quantiques quotients de $sl_2$. En effet, pour toute racine $q$ de l'unité, il est possible de réaliser le groupe quantique quotient $\Uq$ comme une sous-algèbre d'une algèbre de Hopf tressée et enrubannée (cf. par exemple \cite[§ VI-IX, § XIII]{Kas95}). Dans cette situation, on peut décrire une partie centrale $\Df$ à l'aide du morphisme de Drinfeld $\bchi$ introduit dans \cite[Prop. 3.3]{Dri90}. Pour le groupe quantique générique $U_h sl(2)$, cette méthode suffit à décrire le centre tout entier de $U_h sl(2)$ (cf. par exemple \cite[Thm 2]{Bau98}). Par contre, lorsque $h$ s'évalue en une racine $q$ de l'unité, on obtient le reste $\Rf$ du centre de $\Uq$ à l'aide du morphisme de Radford $\hbphi$ (inverse) introduit dans \cite[Prop. 3]{Rad90}. Cette idée, initialement proposée par Lachowska dans \cite{Lac03a} pour les \emph{petits} groupes quantiques (associés aux racines \emph{impaires} de l'unité), est également reprise dans \cite[§ 4]{FGST06b} pour les groupes quantiques \emph{restreints} (associés aux racines \emph{paires} de l'unité). On complète ce chapitre II avec des rappels généraux sur les algèbres Hopfs tressées et enrubannées, puis on détaille la réalisation du centre des groupes quantiques restreints par les morphismes de Drinfeld et de Radford à la manière de \cite{FGST06a}.

Comme on le verra ultérieurement dans le chapitre IV, les morphismes de Drinfeld et de Radford sont d'autant plus intéressants qu'ils permettent de retrouver le morphisme $\cS \cong \bchi \circ \hbphi^{-1}$ découvert dans \cite{LM94} comme le morphisme d'involution d'une représentation projective de $SL_2(\Z)$ sur le centre de groupes quantiques quotients. En outre, le sous-ensemble non-vide $\Df \cap \Rf$ réalise le plus grand sous-espace stable de $\cS$ (cf. \cite[Prop. 4.5.2]{FGST06b})

\minitoc

\section{La base canonique du centre}
\label{section: centre}

On cherche à décrire le centre $\Zf$ \index{Z@ $\Zf$} de $\Uq$. Pour cela, on considère la représentation régulière bilatère $\Regb$ \index{Regb@ $\Regb$} de $\Uq$ (en tant que $\Uq$-bimodule). L'ensemble de ses $\Uq$-endomorphismes est en bijection avec $\Zf$ via l'isomorphisme d'algèbres :
\[ \begin{cases}
	\End_{\Uq-\Uq}(\Regb) \overset{\sim}{\longrightarrow} \Zf \\
	\varphi \longmapsto \varphi(1)
	\end{cases} . \]
On commence par détailler la structure de bimodule de $\Regb$, puis on en déduit l'existence d'éléments centraux canoniques, dont on connaît le système de relations et qui forment une $\C$-base de $\Zf$. Ceux-ci sont enfin exprimés en fonction des générateurs de $\Uq$ à l'aide d'un élément particulier $C$, l'\emph{élément de Casimir}.

\subsection{Décomposition de la représentation régulière}
\label{seubsection: blocs}

Tout $\Uq$-bimodule peut être vu comme un $\Uq \otimes \Uq^{op}$-module à gauche, où $\Uq^{op}$ désigne l'algèbre opposée de $\Uq$ (cf. par exemple \cite[Prop. 10.1]{Pie82}). Dans ces conditions, d'après le théorème de Krull-Schmidt (cf. par exemple \cite[Thm 14.5]{CR62}), tout $\Uq$-bimodule de dimension finie se décompose de manière unique, à isomorphisme et ordre des facteurs près, en somme directe de $\Uq$-bimodules indécomposables. On appelle encore \emph{facteurs directs} les $\Uq$-\emph{bimodules} indécomposables de cette décomposition (cf. la section \ref{section: reps}). En particulier, on s'intéresse aux facteurs directs de la représentation régulière bilatère $\Regb$ de $\Uq$. Pour cela, on définit une (nouvelle) structure de $\Uq$-bimodule sur le produit tensoriel de $\Uq$-modules, et on utilise les facteurs directs de la représentation régulière à gauche $\Reg$ (cf. la proposition \ref{Thm: Reg}).

\begin{Def} 
Soient $\mathscr{X}$ un $\Uq$-module à gauche et $\Y$ un $\Uq$-module à droite. On note $\mathscr{X} \boxtimes \Y$ \index{\times@ $\boxtimes$} le $\C$-espace vectoriel $\mathscr{X} \otimes \Y$ muni de la structure de $\Uq$-bimodule définie par :
\[ a (x \otimes y) := (ax) \otimes y, \quad (x \otimes y)b := x \otimes (yb). \]
pour tous $a \in \Uq$, $b \in \Uq^{op}$, $x \in \mathscr{X}$ et $y \in \Y$.
\end{Def}

\begin{Rem} 
\label{Rem: simp boxtimes}
Pour tout $\Uq$-module simple à gauche $\mathscr{X}$ et tout $\Uq$-module simple à droite $\Y$, le $\Uq$-bimodule $\mathscr{X} \boxtimes \Y$ est encore simple.
\end{Rem}

D'après la remarque \ref{Rem: droite}, les structures de modules simples à droite sont analogues aux structures de modules simples à gauche. Par extension, il en va de même pour les modules de Verma à droite et les PIMs à droite. Par abus, on désignera les $\Uq$-modules à droite avec les mêmes notations que celles des $\Uq$-modules à gauche :
\begin{itemize}
	\item $\X^\pm(s)$, $1 \leq s \leq p$, pour les modules simples,
	\item $\V^\pm(s)$, $\Vb^\pm(s)$, $1 \leq s \leq p$, pour les modules de Verma,
	\item $\PIM^\pm(s)$, $1 \leq s \leq p$, pour les PIMs.
\end{itemize}
C'est pourquoi on ne précisera pas le latéralité des modules considérés dans le lemme \ref{Lemme: blocs2}.

\begin{Lemme} 
\label{Lemme: blocs1}
Soient $U$ et $V$ deux $\C$-espaces vectoriels de dimension finie. On note $(e_1,...,e_n)$ et $(\varepsilon_1,...,\varepsilon_n)$ leurs bases respectives, et $\langle , \rangle$ le produit scalaire hermitien canonique de $U$.
On considère les $\C$-espaces vectoriels $V \otimes U$ et $\Hom(U,V)$ munis des structures de $\End(V)-\End(U)$-bimodules respectivement définies par :
\begin{gather*}
g (v \otimes u) := g(v) \otimes u , \quad (v \otimes u) f := v \otimes f^*(u), \\
	gh := g \circ h, \quad hf := h \circ f, 
\end{gather*}
pour tous $g \in \End(V)$, $f \in \End(U)$, $v \in V$, $u \in U$, et $h \in \Hom(U,V)$, où $f^*$ désigne l'endomorphisme adjoint de $f$. Alors l'isomorphisme canonique de $\C$-espaces vectoriels donné par :
\[ \lambda_{U,V} : \begin{cases}
	V \otimes U \overset{\sim}{\longrightarrow} \Hom(U,V) \\
	\varepsilon_j \otimes  e_i \longmapsto \langle ? , e_i \rangle \; \varepsilon_j
	\end{cases}, \]
où le symbole $?$ désigne la place de la variable, induit un isomorphisme de $\End(V)-\End(U)$-bimodules.
\end{Lemme}

\begin{proof}
Il s'agit d'une vérification directe. \\ Soient $g \in \End(V)$, $f \in \End(U)$, $i \in \{1,...,n\}$ et $j \in \{1,...,j\}$. On a :
\begin{align*}
	\lambda_{U,V} \left( g (\varepsilon_j \otimes e_i) \right) &= \lambda_{U,V} \left( g(\varepsilon_j) \otimes e_i \right) = \langle ?,e_i \rangle \; g(\varepsilon_j) = g \left( \langle ?,e_i \rangle \; \varepsilon_j \right) \\
	&= g \circ \lambda_{U,V} ( \varepsilon_j \otimes e_i ) = g \lambda_{U,V} ( \varepsilon_j \otimes e_i ), \\
	\lambda_{U,V} \left( (\varepsilon_j \otimes e_i)f \right) &= \lambda_{U,V} \left( \varepsilon_j \otimes f^*(e_i) \right) = \langle ?,f^*(e_i) \rangle \;\varepsilon_j =  \langle f(?),e_i \rangle \; \varepsilon_j \\
	&= \lambda_{U,V} ( \varepsilon_j \otimes e_i ) \circ f = \lambda_{U,V} ( \varepsilon_j \otimes e_i ) f.
\end{align*}
\end{proof}

\begin{Lemme} 
\label{Lemme: blocs2} 
\begin{enumerate}[(a)]
	\item Si $\X$ est un $\Uq$-module simple, alors il existe un morphisme surjectif de $\Uq$-bimodules :
\[ \widetilde{\pi} : \Regb \longrightarrow \X \boxtimes \X . \]
	\item Si $\V$ est un $\Uq$-module de Verma, alors il existe un morphisme surjectif de $\Uq$-bimodules :
\[ \widetilde{\pi} : \Rad \left( \Regb \right) \longrightarrow \Rad \V \boxtimes \left( \V / \Rad \V \right) , \]
où $\Rad \V$ est le radical de $\V$.
\end{enumerate}
\end{Lemme}

\begin{proof}
\begin{enumerate}[(a)]
	\item Soit $\X$ un $\Uq$-module simple (cf. la proposition \ref{Prop: simples} et la remarque \ref{Rem: droite}). On considère le morphisme de $\C$-algèbres :
\[ \pi : \begin{cases}
	\Uq \longrightarrow \End \left( \X \right) \\
	a \longmapsto \left[ x \mapsto ax \right]
	\end{cases} . \]
	Il induit une structure de $\Uq$-bimodule sur tout $\End \left( \X \right)$-bimodule. En particulier, $\X \otimes \X$ et $\End \left( \X \right)$ sont munis de structures de $\Uq$-bimodules données par :
\begin{equation} \tag{1}
\begin{gathered}
a (x \otimes y) := \pi(a)(x) \otimes y , \quad (x \otimes y) a := x \otimes {\pi(a)}^*(y), \\
	ah := \pi(a) \circ h, \quad ha := h \circ \pi(a), 
\end{gathered}
\end{equation}
pour tous $a \in \Uq$, $x,y \in \X$, et $h \in \End \left( \X \right)$, où ${\pi(a)}^*$ désigne l'endomorphisme adjoint de $\pi(a)$ par rapport à la structure duale définie dans la remarque \ref{Rem: droite} :
	\[ \pi(a)^* : x \longmapsto xa . \]
	D'après le lemme \ref{Lemme: blocs1}, l'isomorphisme canonique de $\C$-espaces vectoriels :
\[ \lambda_{\X,\X} : \X \otimes \X \overset{\sim}{\longrightarrow} \End \left( \X \right) \]
induit donc un isomorphisme de $\Uq$-bimodules, dont les structures sont détaillées dans $(1)$.	
	D'autre part, d'après le théorème de Burnside (cf. par exemple \cite[Thm 27.4]{CR62}), le morphisme d'algèbres $\pi$ est surjectif. Par conséquent, la composée :
	\[ \widetilde{\pi} : \Regb \overset{\pi}{\twoheadrightarrow} \End \left( \X \right) \overset{\lambda_{\X,\X}^{-1}}{\underset{\sim}{\longrightarrow}} \X \boxtimes \X \]
est un morphisme surjectif de $\Uq$-bimodules.

	\item Soit $\V$ un $\Uq$-module de Verma (cf. la proposition \ref{Prop: Verma} et la remarque \ref{Rem: droite}). On considère le morphisme de $\C$-algèbres :
\[ \pi : \begin{cases}
	\Uq \longrightarrow \End \left( \V \right) \\
	a \longmapsto \left[ x \mapsto ax \right]
	\end{cases} . \]
	Comme $\V$ est une extension de modules simples, le morphisme $\pi$ se surjecte sur :
	\[ \mathrm{Im}(\pi) \cong \begin{pmatrix} 
	\End \left( \V / \Rad \V \right) & \{0\} \\
	\Hom \left( \V / \Rad \V , \Rad \V \right) & \End \left( \Rad \V \right) 
	\end{pmatrix} \]
	d'après le théorème de Burnside. On note $\bar{\pi}$ le morphisme surjectif induit sur les radicaux :
\[ \bar{\pi} :	\Rad(\Uq) \twoheadrightarrow \Hom \left( \V / \Rad \V , \Rad \V \right) \]
(cf. par exemple \cite[Lemme 4.1]{Pie82}). On procède alors comme en $(a)$ avec le morphisme $\bar{\pi}$. On obtient un morphisme surjectif de $\Uq$-bimodules :
	\[ \widetilde{\pi} : \Rad(\Regb) \overset{\bar{\pi}}{\twoheadrightarrow} \Hom \left( \V / \Rad \V , \Rad \V \right) \overset{\lambda_{\Rad \V,\V / \Rad \V}^{-1}}{\underset{\sim}{\longrightarrow}} \Rad \V \boxtimes \left( \V / \Rad \V \right) . \]
\end{enumerate}
\end{proof}

\begin{Thm} 
\label{Thm: blocs}
La représentation régulière bilatère $\Regb$ de $\Uq$ se décompose de manière unique, à isomorphisme et ordre des facteurs près, en la somme directe $\bigoplus_{s=0}^p \mathbf{Q}(s)$\index{Q2@ $\mathbf{Q}(s)$}, où : 
\begin{enumerate}[(i)]
	\item les $\Uq$-bimodules $\mathbf{Q}(0) = \X^-(p) \boxtimes \X^-(p)$ et $\mathbf{Q}(p) = \X^+(p) \boxtimes \X^+(p)$ sont simples,
	\item pour tout $s \in \{1,...,p-1\}$, le $\Uq$-bimodule $\mathbf{Q}(s)$ est indécomposable et admet une filtration de $\Uq$-bimodules $0 \subset \mathbf{R}(s)^2 \subset \mathbf{R}(s) \subset \mathbf{Q}(s)$ tels que :
	\begin{align*}
	& \mathbf{Q}(s) / \mathbf{R}(s) \cong \X^+(s) \boxtimes  \X^+(s) \oplus \X^-(p-s) \boxtimes \X^-(p-s), \\
	& \mathbf{R}(s) / \mathbf{R}(s)^2  \cong 2 \X^-(p-s) \boxtimes \X^+(s) \oplus 2 \X^+(s) \boxtimes \X^-(p-s), \\
	& \mathbf{R}(s)^2 \cong \X^+(s) \boxtimes  \X^+(s) \oplus \X^-(p-s) \boxtimes \X^-(p-s).
	\end{align*}
\end{enumerate}
\end{Thm}

Pour davantage de lisibilité, la structure des $\Uq$-bimodules $\mathbf{Q}(s)$, $1 \leq s \leq p-1$, est synthétisée dans la figure \ref{Fig: blocs}.

\begin{figure}[!ht] 
\caption{Structure des bimodules $\mathbf{Q}(s)$, $1 \leq s \leq p-1$.}
\label{Fig: blocs}
Représentation diagrammatique de l'action à gauche de $\Uq$ :
\[ \scalebox{0.65}{ \xymatrix @C=1cm {
	\mathbf{Q}(s) / \mathbf{R}(s) && \overset{\X^+(s) \boxtimes \X^+(s)}{\bullet} \ar[dl]_-E \ar[dr]^-F \ar[dd]^-E &&& \overset{\X^-(p-s) \boxtimes \X^-(p-s)}{\bullet} \ar[dl]_-E \ar[dr]^-F \ar[dd]^-E \\
	\mathbf{R}(s) / \mathbf{R}(s)^2 & \overset{\X^-(p-s) \boxtimes \X^+(s)}{\bullet} \ar[dr]_-F && \overset{\X^-(p-s) \boxtimes \X^+(s)}{\bullet} \ar[dl]^-E & \overset{\X^+(s) \boxtimes \X^-(p-s)}{\bullet} \ar[dr]_-F && \overset{\X^+(s) \boxtimes \X^-(p-s)}{\bullet} \ar[dl]^-E \\	
	\mathbf{R}(s)^2 && \overset{\X^+(s) \boxtimes \X^+(s)}{\bullet} &&& \overset{\X^-(p-s) \boxtimes \X^-(p-s)}{\bullet} }} \]
Représentation diagrammatique de l'action à droite de $\Uq$ :
\[ \scalebox{0.65}{ \xymatrix {
	\mathbf{Q}(s) / \mathbf{R}(s) && \overset{\X^+(s) \boxtimes \X^+(s)}{\bullet} \ar[drr]_<<<<<E \ar[drrrr]^<<<<<F \ar[dd]^<<<<<E &&& \overset{\X^-(p-s) \boxtimes \X^-(p-s)}{\bullet} \ar[dllll]_<<<<<E \ar[dll]^<<<<<F \ar[dd]^<<<<<E \\
	\mathbf{R}(s) / \mathbf{R}(s)^2 & \overset{\X^-(p-s) \boxtimes \X^+(s)}{\bullet} \ar[drrrr]_<<<<<F && \overset{\X^-(p-s) \boxtimes \X^+(s)}{\bullet} \ar[drr]^>>>>>E & \overset{\X^+(s) \boxtimes \X^-(p-s)}{\bullet} \ar[dll]_>>>>>F && \overset{\X^+(s) \boxtimes \X^-(p-s)}{\bullet} \ar[dllll]^<<<<<E \\	
	\mathbf{R}(s)^2 && \overset{\X^+(s) \boxtimes \X^+(s)}{\bullet} &&& \overset{\X^-(p-s) \boxtimes \X^-(p-s)}{\bullet} }} \]
\end{figure}

\begin{proof}
On commence par déterminer le nombre de facteurs directs de $\Regb$. Pour tout $\Uq$-bimodule $\mathbf{X}$, on note $\mathscr{X}$ le $\Uq$-module à gauche  obtenu à partir de $\mathbf{X}$ par oubli de l'action à droite de $\Uq$. En particulier, pour la représentation régulière bilatère $\Regb$, le $\Uq$-module à gauche correspondant est la représentation régulière à gauche $\Reg$. D'après la proposition \ref{Thm: Reg}, on sait que : 
\[ \Reg \cong \bigoplus_{s=1}^{p-1} \left[ s \PIM^+(s) \oplus s \PIM^-(s) \right] \oplus p \X^+(s) \oplus p \X^-(p) . \]
Les facteurs directs de cette décomposition, appelés $\Uq$-PIMs à gauche, réalisent les idéaux à gauche de $\Uq$. En les munissant de la structure de $\Uq$-module à droite donnée par la multiplication à droite, on obtiendra les facteurs directs de $\Regb$.

Identifions les $\Uq$-PIMs à gauche appartenant à un même facteur direct de $\Regb$ (en tant que $\Uq$-bimodule). Sous l'action de la multiplication à droite par $\Uq$, deux PIMs à gauche appartiennent à un même bimodule si et seulement si ils ont les mêmes quotients, à isomorphisme et ordre près, dans une suite de composition (cf. par exemple \cite[§ 13, Thm 55.2]{CR62}). Or, d'après la description des $\Uq$-PIMs donnée dans la proposition \ref{Prop: PIMs}, on sait que $\PIM^\pm(p) = \X^\pm(p)$ sont simples et que, pour tout $s \in \{1,...,p-1\}$, la structure de $\PIM^\pm(s)$ est donnée par :
\begin{equation} \tag{1}
\xymatrix @-3.6ex {
	& \overset{\X^\pm(s)}{\bullet} \ar[ld]_-E \ar[rd]^-F \ar@{-->}[dd]^-E & \\
	\overset{\X^{\mp}(p-s)}{\bullet} \ar[rd]_-F && \overset{\X^{\mp}(p-s)}{\bullet} \ar[ld]^-E \\
	& \overset{\X^\pm(s)}{\bullet} &} 
\end{equation}
Soient $\alpha \in \{+,-\}$ et $s \in \{1,...,p-1\}$. Une suite de composition de $\PIM^\alpha(s)$ est donc :
\[ 0 \subseteq \X^\alpha(s) \subseteq \V^\alpha(s) \subseteq \V^\alpha(s) + \Vb^\alpha(s) \subseteq \PIM^\alpha(s), \]
dont les sous-quotients sont respectivement :
\[ \X^\alpha(s), \quad \X^{-\alpha}(p-s), \quad \X^{-\alpha}(p-s), \quad \X^\alpha(s) . \]
Par conséquent, il y a $p+1$ facteurs directs $\mathbf{Q}(s)$ de $\Regb$, indexés par $0 \leq s \leq p$, dont la structure de $\Uq$-module à gauche est donnée par :
\begin{equation} \tag{2}
\begin{aligned}
&\mathcal{Q}(0) \cong p \X^-(p), \\
& \mathcal{Q}(s) \cong s \PIM^+(s) \oplus (p-s) \PIM^-(p-s), && 1 \leq s \leq p-1, \\
& \mathcal{Q}(p) \cong p \X^+(p).
\end{aligned}
\end{equation}

Il reste à expliciter les structures de $\Uq$-bimodule des facteurs $\mathbf{Q}(s)$, $0 \leq s \leq p$. Ceux-ci étant en somme directe, leurs structures de $\Uq$-bimodule coïncident avec leurs structures induites de $\mathbf{Q}(s)$-bimodule, $0 \leq s \leq p$. Dans la suite, on identifiera implicitement ces structures de bimodule.

Soit $s \in \{0,...,p\}$. On note $\mathbf{R}(s)$ le radical de Jacobson de $\mathbf{Q}(s)$, et on étudie la structure du bimodule quotient $\mathbf{Q}(s) / \mathbf{R}(s)$. On sait que le bimodule $\mathbf{Q}(s) / \mathbf{R}(s)$ est semi-simple et que ses sous-bimodules sont les bimodules simples quotients de $\mathbf{Q}(s)$ (cf. par exemple \cite[§ 2.7, § 4.1]{Pie82}). Afin de les expliciter, on se donne un sous-module à gauche simple $\X$ de $\mathbf{Q}(s)$. D'après le lemme \ref{Lemme: blocs2}, il existe un morphisme surjectif de $\Uq$-bimodules :
\[ \widetilde{\pi} : \Regb \twoheadrightarrow \X \boxtimes \X . \]
Comme $\X$ est un sous-module à gauche du facteur direct $\mathbf{Q}(s)$, le morphisme $\widetilde{\pi}$ se factorise par $\mathbf{Q}(s)$. Donc $\mathbf{Q}(s)$ admet un bimodule quotient isomorphe à $\X \boxtimes \X$. De plus, d'après les identifications $(2)$ et le diagramme $(1)$, on sait que $\X$ est isomorphe à :
\begin{align*}
& \X^-(p) && \text{ si } s=0, \\
& \X^+(s) \text{ ou } \X^-(p-s) && \text{ si } 1 \leq s \leq p-1, \\
& \X^+(p) && \text{ si } s=p.
\end{align*}
Par conséquent, le bimodule $\mathbf{Q}(s)$ admet un bimodule quotient isomorphe à :
\begin{align*}
& \X^-(p) \boxtimes \X^-(p) && \text{ si } s=0, \\
& \X^+(s) \boxtimes \X^+(s) \text{ ou } \X^-(p-s) \boxtimes \X^-(p-s) && \text{ si } 1 \leq s \leq p-1, \\
& \X^+(p) \boxtimes \X^+(p) && \text{ si } s=p.
\end{align*}
Par construction, ces quotients sont des bimodules simples (cf. la remarque \ref{Rem: simp boxtimes}) deux-à-deux distincts. D'après les propriétés du radical, il s'ensuit que :
\begin{equation} \tag{3}
\begin{aligned}
& \X^-(p) \boxtimes \X^-(p) \hookrightarrow \mathbf{Q}(0) / \mathbf{R}(0), \\
& \X^+(s) \boxtimes \X^+(s) \oplus \X^-(p-s) \boxtimes \X^-(p-s) \hookrightarrow \mathbf{Q}(s) / \mathbf{R}(s), \; 1 \leq s \leq p-1, \\
& \X^+(p) \boxtimes \X^+(p) \hookrightarrow \mathbf{Q}(p) / \mathbf{R}(p).
\end{aligned}
\end{equation}
Enfin, d'après les identifications $(2)$ et le diagramme $(1)$, on connaît la structure de $\Uq$-module à gauche du radical $\mathbf{R}(s)$ (cf. par exemple \cite[Prop. 4.1]{Pie82}) et par suite, celle du quotient $\mathbf{Q}(s) / \mathbf{R}(s)$ :
\begin{align*}
&\mathcal{Q}(0) / \mathscr{R}(0) = \mathcal{Q}(0) \cong p \X^-(p) , \\
&\mathcal{Q}(s) / \mathscr{R}(s) \cong s \X^+(s) \oplus (p-s) \X^-(p-s) , && 1 \leq s \leq p-1, \\
&\mathcal{Q}(p) / \mathscr{R}(p) = \mathcal{Q}(p) \cong p \X^+(p).
\end{align*}
En particulier, cela donne la $\C$-dimension du quotient $\mathbf{Q}(s) / \mathbf{R}(s)$. Le compte des dimensions de part et d'autre des injections de bimodules $(3)$ conduit à obtenir des isomorphismes de bimodules. D'où :
\begin{align*}
&\mathbf{Q}(0) / \mathbf{R}(0) = \mathbf{Q}(0) \cong \X^-(p) \boxtimes \X^-(p), \\
&\mathbf{Q}(s) / \mathbf{R}(s) \cong \X^+(s) \boxtimes \X^+(s) \oplus \X^-(p-s) \boxtimes \X^-(p-s), && 1 \leq s \leq p-1, \\
&\mathbf{Q}(p) / \mathbf{R}(p) = \mathbf{Q}(p) \cong \X^+(p) \boxtimes \X^+(p).
\end{align*}

On considère maintenant le radical de Jacobson $\mathbf{R}(s)^2$ de $\mathbf{R}(s)$ (cf. par exemple \cite[Ex. 4.1.2]{Pie82}), et on étudie la structure du bimodule quotient $\mathbf{R}(s) / \mathbf{R}(s)^2$. De même que précédemment, le bimodule $\mathbf{R}(s) / \mathbf{R}(s)^2$ est semi-simple et ses sous-bimodules sont les bimodules simples quotients de $\mathbf{R}(s)$. Afin de les expliciter, on se donne un sous-module à gauche de Verma $\V$ (cf. la proposition \ref{Prop: Verma}) de $\mathbf{Q}(s)$. D'après le lemme \ref{Lemme: blocs2}, il existe un morphisme surjectif de $\Uq$-bimodules :
\[ \widetilde{\pi}' : \Rad \left( \Regb \right) \twoheadrightarrow \Rad \V \boxtimes \left( \V / \Rad \V \right) . \]
De même que précédemment, celui-ci se factorise par $\mathbf{R}(s)$ car $\V$ est un sous-module à gauche du facteur direct $\mathbf{Q}(s)$. Donc $\mathbf{R}(s)$ admet un bimodule quotient isomorphe à $\Rad \V \boxtimes \left( \V / \Rad \V \right)$. De plus, d'après les identifications $(2)$ et le diagramme $(1)$, on sait que $\V$ est isomorphes à :
\begin{align*}
&\V^-(p) = \Vb^-(p) && \text{ si } s=0, \\
&\V^+(s), \Vb^+(s), \V^-(p-s) \text{ ou } \Vb^-(p-s) && \text{ si } 1 \leq s \leq p-1, \\
&\V^+(p) = \Vb^+(p) && \text{ si } s=p.
\end{align*}
Par conséquent, le bimodule $\mathbf{R}(s)$ admet un bimodule quotient isomorphe à :
\begin{align*}
&\{0\} && \text{ si } s=0, \\
&\begin{cases} 
	\X^-(p-s) \boxtimes \X^+(s), \X^-(p-s) \boxtimes \X^+(s), \\
	\X^+(s) \boxtimes \X^-(p-s) \text{ ou } \X^+(s) \boxtimes \X^-(p-s)
	\end{cases}
 	&& \text{ si } 1 \leq s \leq p-1, \\
&\{0\} && \text{ si } s=p.
\end{align*}
En procédant comme précédemment, on obtient les isomorphismes de bimodules :
\begin{align*}
&\mathbf{R}(0) / \mathbf{R}(0)^2 = \mathbf{R}(0) = \{0\}, \\
&\mathbf{R}(s) / \mathbf{R}(s)^2 \cong 2 \X^-(p-s) \boxtimes \X^+(s) \oplus 2 \X^+(s) \boxtimes \X^-(p-s), && 1 \leq s \leq p-1, \\
&\mathbf{R}(p) / \mathbf{R}(p)^2 = \mathbf{R}(p) = \{0\}.
\end{align*}

Enfin, on on considère le radical de Jacobson $\mathbf{R}(s)^3$ de $\mathbf{R}(s)^2$, et on étudie la structure du bimodule quotient $\mathbf{R}(s)^2 / \mathbf{R}(s)^3$. De même que précédemment, le bimodule $\mathbf{R}(s)^2 / \mathbf{R}(s)^3$ est semi-simple et ses sous-bimodules sont les bimodules simples quotients de $\mathbf{R}(s)^2$. D'après les identifications $(2)$ et le diagramme $(1)$, on sait que $\mathbf{R}(s)^3 = \{0\}$ (cf. par exemple \cite[Prop. 4.1]{Pie82}) et on connaît la structure de $\Uq$-module à gauche du quotient $\mathbf{R}(s)^2 / \mathbf{R}(s)^3 = \mathbf{R}(s)^2$ :
\begin{align*}
&\mathscr{R}(0)^2 = \{0\} , \\
&\mathscr{R}(s)^2 \cong s \X^+(s) \oplus (p-s) \X^-(p-s) \cong \mathcal{Q}(s) / \mathscr{R}(s), && 1 \leq s \leq p-1, \\
&\mathscr{R}(p)^2 = \{0\}.
\end{align*}
Par adjonction de l'action à droite de $\Uq$, on en déduit que :
\begin{align*}
&\mathbf{R}(0)^2 = \{0\}, \\
&\mathbf{R}(s)^2 \cong \mathbf{Q}(s) / \mathbf{R}(s), && 1 \leq s \leq p-1, \\
&\mathbf{R}(0)^2 = \{0\}.
\end{align*}
D'où le résultat.
\end{proof}

\begin{Rem}
Il est possible d'identifier les $\Uq$-PIMs à gauche qui appartiennent à un même facteur direct de $\Regb$ sans évoquer le critère \cite[Thm 55.2]{CR62}. Pour cela, on explicite la base canonique d'un $\Uq$-PIM à gauche en fonction des générateurs $F$, $K$, $E$ de $\Uq$ en tant que $\C$-algèbre (cf. par exemple \cite[§ 3]{Sut94}). On montre alors que, pour tout $\alpha \in \{+,-\}$ et $s \in \{1,...,p-1\}$, on a $\PIM^\alpha(s) F^s \subseteq \PIM^{-\alpha}(p-s)$.
\end{Rem}

\subsection{Description du centre à partir des représentations}
\label{subsection: centre1}

On utilise la description de la représentation régulière bilatère $\Regb$ de $\Uq$ donnée dans le théorème \ref{Thm: blocs} pour construire des éléments centraux canoniques de $\Uq$. Pour cela, on construit une base du $\C$-espace vectoriel des endomorphismes de $\Uq$-bimodules de $\Regb$, puis on prend son image dans le centre $\Zf$ de $\Uq$ via l'isomorphisme d'algèbres :
\[ \begin{cases}
\End_{\Uq-\Uq}(\Regb) \overset{\sim}{\longrightarrow} \Zf \\
\varphi \longmapsto \varphi(1)
\end{cases} . \]
Les éléments ainsi construits forment une $\C$-base de $\Zf$, dont on connaît le système de relations et leurs actions sur les idéaux à gauche de $\Uq$ (cf. la remarque \ref{Rem: ideaux}).

\begin{Lemme} 
\label{Lemme: centre}
Le $\C$-espace vectoriel $\End_{\Uq-\Uq}(\Regb)$ des endomorphismes de bimodules de $\Uq$ de la représentation régulière bilatère $\Regb$ de $\Uq$ admet une base $\{ \mathbf{e}_s \; ; \; 0 \leq s \leq p \} \cup \{ \mathbf{w}^\pm_s \; ; \; 1 \leq s \leq p-1 \}$ où :
\begin{enumerate}[(i)]
	\item pour tout $s \in \{0,...,p\}$, le morphisme de $\Uq$-bimodules $\mathbf{e}_s$ est la projection orthogonale sur $\mathbf{Q}(s)$ :
\[ \mathbf{e}_s : \begin{cases}
	\Regb \cong \bigoplus_{j=0}^p \mathbf{Q}(j) \longrightarrow \mathbf{Q}(s) \\
	x = \sum_{j=0}^p x_j \longmapsto x_s
	\end{cases} . \]
	\item pour tout $s \in \{1,...,p-1\}$, les morphismes de $\Uq$-bimodules $\mathbf{w}_s^\pm$ sont définis par :
	\begin{align*}
&\mathbf{w}_s^+ : \begin{cases}
	\Regb \cong \mathbf{Q}(0) \oplus \bigoplus_{j=1}^{p-1} \mathbf{Q}(j) \oplus \mathbf{Q}(p) \longrightarrow \mathbf{Q}(s) \\
	y_0^+(j) \mapsto \delta_{j,s} x_0^+(s), \quad y_0^-(j) \mapsto 0, \quad 1 \leq j \leq p,
	\end{cases} \\
&\mathbf{w}_s^- : \begin{cases}
	\Regb \cong \mathbf{Q}(0) \oplus \bigoplus_{j=1}^{p-1} \mathbf{Q}(j) \oplus \mathbf{Q}(p) \longrightarrow \mathbf{Q}(s) \\
	y_0^+(j) \mapsto 0, \quad y_0^-(j) \mapsto \delta_{j,p-s} x_0^-(p-s), \quad 1 \leq j \leq p,
	\end{cases}
\end{align*}
où, pour tout $j \in \{1,...,p\}$, $x^\pm_0(j)$ et $y^\pm_0(j)$ sont des vecteurs de poids $\pm q^{j-1}$ qui engendrent respectivement $\X^\pm(j)$ et $\PIM^\pm(j)$ sous $\Uq$ (cf. les propositions \ref{Prop: simples} et \ref{Prop: PIMs}).
\end{enumerate}
\end{Lemme}

\begin{proof}
D'après le théorème \ref{Thm: blocs}, on sait que :
\[ \End_{\Uq - \Uq} \left( \Regb \right) \cong \bigoplus_{s=0}^p \End_{\mathbf{Q}(s) - \mathbf{Q}(s)} \left( \mathbf{Q}(s) \right) . \]
Il s'agit donc de déterminer, pour tout $s \in \{0,...,p\}$, les endomorphismes induits de $\mathbf{Q}(s)$-bimodules de $\mathbf{Q}(s)$. Pour cela, on utilisera régulièrement la structure des bimodules $\mathbf{Q}(s)$, $0 \leq s \leq p$, explicités dans le théorème \ref{Thm: blocs} et la figure \ref{Fig: blocs}.

Pour $s \in \{0,p\}$, le bimodule $\mathbf{Q}(s)$ est simple. D'après le lemme de Schur (cf. par exemple \cite[§ 2.3]{Pie82}), les endomorphismes de bimodules correspondants sont multiples de l'identité. A multiplication par un scalaire près, ils se relèvent sur $\Uq$ en la projection orthogonale de $\Uq$-bimodules $\mathbf{e}_s$.
	
Pour $s \in \{1,...,p-1\}$, le bimodule $\mathbf{Q}(s)$ admet la filtration de sous-bimodules :
\[ 0 \subset \mathbf{R}(s)^2 \subset \mathbf{R}(s) \subset \mathbf{Q}(s) . \]
Soient $s \in \{1,...,p-1\}$ et $f \in \End_{\mathbf{Q}(s) - \mathbf{Q}(s)} \left( \mathbf{Q}(s) \right)$. Les bimodules $\mathbf{R}(s)^2$ et $\mathbf{R}(s)$ sont stables par $f$, et on a :
\[ \xymatrix{ 
	\mathbf{Q}(s) \ar[r]^-f & \mathbf{Q}(s) \\
	\mathbf{R}(s) \ar@{^{(}->}[u] \ar[r]^-f & \mathbf{R}(s) \ar@{^{(}->}[u] \\
	\mathbf{R}(s)^2 \ar@{^{(}->}[u] \ar[r]^-f \ar[d]^-[@]{\sim} & \mathbf{R}(s)^2 \ar@{^{(}->}[u] \\
	\X^+(s) \boxtimes \X^+(s) \oplus \X^-(p-s) \boxtimes \X^-(p-s) } \]
On distingue deux cas.

\begin{Cas} On suppose que $f \left( \mathbf{R}(s)^2 \right) \not = \{0\}$. \\
D'après le lemme de Schur, l'image par $f$ du sous-bimodule  simple isomorphe à $\X^+(s) \boxtimes \X^+(s)$ (resp. $\X^-(p-s) \boxtimes \X^-(p-s)$) est soit nulle, soit isomorphe à lui-même. \\
Par exemple, on suppose par l'absurde que l'image par $f$ du sous-bimodule isomorphe à $\X^+(s) \boxtimes \X^+(s)$ est nulle. D'après le théorème de factorisation, il existe un morphisme $\bar{f}_1$ de $\mathbf{Q}(s)$-modules à gauche et un morphisme $\bar{f}_2$ de $\mathbf{Q}(s)$-modules à droite tels que :
\[ \xymatrix{ 
	\mathscr{R}(s) \ar@{->>}[d] \ar[r]^-f & \mathscr{R}(s) & \mathscr{R}(s)' \ar@{->>}[d] \ar[r]^-f & \mathscr{R}(s)' \\
	\mathscr{R}(s) / \left( \X^+(s) \boxtimes \X^+(s) \right) \ar[ru]_-{\bar{f}_1} && \mathscr{R}(s)' / \left( \X^+(s) \boxtimes \X^+(s) \right) \ar[ru]_-{\bar{f}_2} & \\
	2 \X^-(p-s) \boxtimes \X^+(s) \ar@{^{(}->}[u] && 2 \X^+(s) \boxtimes \X^-(p-s) \ar@{^{(}->}[u]} \]
où $\mathscr{R}(s)$ (resp. $\mathscr{R}(s)'$) désigne le $\mathscr{Q}(s)$-module à gauche (resp. à droite) obtenu à partir de $\mathbf{R}(s)$ par oubli de l'action à droite (resp. à gauche). D'après le lemme de Schur, l'image par $\bar{f}_1$ (resp. par $\bar{f}_2$) de chaque sous-module simple à gauche (resp. à droite) isomorphe à $\X^-(p-s) \boxtimes \X^+(s)$ (resp. à $\X^+(s) \boxtimes \X^-(p-s)$) est nulle car $\mathscr{R}(s)$ ne contient pas de sous-module à gauche (resp. à droite) isomorphe à $\X^-(p-s) \boxtimes \X^+(s)$ (resp. à $\X^+(s) \boxtimes \X^-(p-s)$). La structure de bimodule de $\mathbf{Q}(s)$ donne alors $f \left( \mathbf{R}(s) \right) = \{0\}$. Ce qui contredit l'hypothèse initiale car $\mathbf{R}(s)^2 \subset \mathbf{R}(s)$. \\
De la même manière, on montre que l'image par $f$ du sous-bimodule isomorphe à $\X^-(p-s) \boxtimes \X^-(p-s)$ n'est pas nulle. Donc les endomorphismes de bimodules $f_{\vert \X^+(s) \boxtimes \X^+(s)}$ et $f_{\vert \X^-(p-s) \boxtimes \X^-(p-s)}$ sont multiples non nuls de l'identité. La structure de bimodule de $\mathbf{Q}(s)$ impose alors que $f$ est multiple non nul de l'identité. A multiplication par un scalaire près, il se relève en la projection orthogonale de $\Uq$-bimodules $\mathbf{e}_s$.
\end{Cas}

\begin{Cas} On suppose que $f \left( \mathbf{R}(s)^2 \right) = \{0\}$. \\
D'après le théorème de factorisation, il existe un morphisme $\bar{f}_1$ de $\mathbf{Q}(s)$-bimodules tel que :
\[ \xymatrix{ 
	\mathbf{R}(s) \ar@{->>}[d] \ar[r]^-f & \mathbf{R}(s) \\
	\mathbf{R}(s) / \mathbf{R}(s)^2 \ar[d]^-[@]{\sim} \ar[ur]_-{\bar{f}_1} \\
	2 \X^-(p-s) \boxtimes \X^+(s) \oplus 2 \X^+(s) \boxtimes \X^-(p-s) } \]
D'après le lemme de Schur, l'image par $\bar{f}_1$ de chaque sous-bimodule isomorphe à $\X^-(p-s) \boxtimes \X^+(s)$ (resp. à $\X^+(s) \boxtimes \X^-(p-s)$) est nulle car $\mathscr{R}(s)$ ne contient pas de sous-bimodule isomorphe à $\X^-(p-s) \boxtimes \X^+(s)$ (resp. à $\X^+(s) \boxtimes \X^-(p-s)$). Il s'ensuit que $f \left( \mathbf{R}(s) \right) = \{0\}$, et d'après le théorème de factorisation, il existe un morphisme $\bar{f}_2$ de $\mathbf{Q}(s)$-bimodules tel que :
\[ \xymatrix{ 
	\mathbf{Q}(s) \ar@{->>}[d] \ar[r]^-f & \mathbf{Q}(s) \\
	\mathbf{Q}(s) / \mathbf{R}(s) \ar[d]^-[@]{\sim} \ar[ur]_-{\bar{f}_2} \\
	\X^+(s) \boxtimes \X^+(s) \oplus \X^-(p-s) \boxtimes \X^-(p-s) } \]
D'après le lemme de Schur, les endomorphismes de bimodules
\[ \bar{f}_{2 \vert \X^+(s) \boxtimes \X^+(s)} \quad \text{ et } \quad \bar{f}_{2 \vert \X^-(p-s) \boxtimes \X^-(p-s)} \]
sont multiples de l'identité. A multiplication par un scalaire près, ils se relèvent respectivement sur $\Uq$ en les morphismes de $\Uq$-bimodules $\mathbf{w}_s^+$ et $\mathbf{w}_s^-$. Donc $f$ se relève sur $\Uq$ en une combinaison linéaire de $\mathbf{w}_s^+$ et $\mathbf{w}_s^-$.
\end{Cas}

Par conséquent, la famille $\{ \mathbf{e}_s \; ; \; 0 \leq s \leq p \} \cup \{ \mathbf{w}^\pm_s \; ; \; 1 \leq s \leq p-1 \}$ est une base du $\C$-espace vectoriel $\End_{\Uq - \Uq} \left( \Regb \right)$.
\end{proof}

\begin{Prop} 
\label{Prop: centre}
Le centre $\Zf$ de $\Uq$ admet une $\C$-base $\{ e_s \; ; \; 0 \leq s \leq p \} \cup \{ w^\pm_s \; ; \; 1 \leq s \leq p-1 \}$ telle que :
\begin{align*}
&\forall s, s' \in \{0,..,p\}
	&& e_s e_{s'} = \delta_{s,s'} \; e_s , \\
&\forall s \in \{0,...,p\} \quad \forall s' \in \{1,...,p-1\} 
	&& e_s w^\pm_{s'} = \delta_{s,s'} \; w^\pm_{s'} , \\
&\forall s, s' \in \{1,..,p-1\}
	&& w^\pm_s w^\pm_{s'} = 0 = w^\pm_s w^\mp_{s'}.
\end{align*}
On l'appelle la \emph{base canonique du centre}.	
\end{Prop}

\begin{proof}
D'après le lemme \ref{Lemme: centre}, la base $\{ \mathbf{e}_s \; ; \; 0 \leq s \leq p \} \cup \{ \mathbf{w}^\pm_s \; ; \; 1 \leq s \leq p-1 \}$ du $\C$-espace vectoriel $\End_{\Uq - \Uq} \left( \Regb \right)$ vérifie le système de relations :
\begin{align*}
&\forall s, s' \in \{0,..,p\}
	&& \mathbf{e}_s \mathbf{e}_{s'} = \delta_{s,s'} \; \mathbf{e}_s , \\
&\forall s \in \{0,...,p\} \quad \forall s' \in \{1,...,p-1\} 
	&& \mathbf{e}_s \mathbf{w}^\pm_{s'} = \delta_{s,s'} \; \mathbf{w}^\pm_{s'} , \\
&\forall s, s' \in \{1,..,p-1\}
	&& \mathbf{w}^\pm_s \mathbf{w}^\pm_{s'} = 0 = \mathbf{w}^\pm_s \mathbf{w}^\mp_{s'}.
\end{align*}
Son image $\{ e_s \; ; \; 0 \leq s \leq p \} \cup \{ w^\pm_s \; ; \; 1 \leq s \leq p-1 \}$ par l'isomorphisme d'algèbres :
\[ \begin{cases}
\End_{\Uq-\Uq}(\Regb) \overset{\sim}{\longrightarrow} \Zf \\
\varphi \longmapsto \varphi(1)
\end{cases} \]
fournit donc la $\C$-base souhaitée de $\Zf$.
\end{proof}

\begin{Cor} 
\label{Cor: centre}
Soient $z \in \Zf$ et $\{ a_s \; ; \; 0 \leq s \leq p \} \cup \{ b^\pm_s \; ; \; 1 \leq s \leq p-1 \}$ l'unique famille de scalaires telle que :
\[ z = \sum_{s=0}^p a_s e_s + \sum_{s=1}^{p-1} \left( b^+_s w^+_s + b^-_s w^-_s \right) . \]
Alors, on a :
\begin{align*}
& z x^+_0(s) = a_s x^+_0(s), && 1 \leq s \leq p, \\
& z x^-_0(s) = a_{p-s} x^-_0(s), && 1 \leq s \leq p, \\
& z y^+_0(s) = a_s y^+_0(s) + b^+_s x^+_0(s), && 1 \leq s \leq p-1, \\
& z y^-_0(s) = a_{p-s} y^-_0(s) + b^-_{p-s} x^-_0(s), && 1 \leq s \leq p-1,
\end{align*}
où, pour tout $s \in \{1,...,p\}$, $x^\pm_0(s)$ et $y^\pm_0(s)$ sont des vecteurs de poids $\pm q^{s-1}$ qui engendrent respectivement $\X^\pm(s)$ et $\PIM^\pm(s)$ sous $\Uq$ (cf. les propositions \ref{Prop: simples} et \ref{Prop: PIMs}).
\end{Cor}

\begin{proof}
Pour connaître l'action de $z$ sur les vecteurs $x^\pm_0(s)$ et $y^\pm_0(s)$, $1 \leq s \leq p$, il faut et il suffit de connaître celle de la base canonique du centre $\{ e_s \; ; \; 0 \leq s \leq p \} \cup \{ w^\pm_s \; ; \; 1 \leq s \leq p-1 \}$. Pour cela, on utilise la base $\{ \mathbf{e}_s \; ; \; 0 \leq s \leq p \} \cup \{ \mathbf{w}^\pm_s \; ; \; 1 \leq s \leq p-1 \}$ du $\C$-espace vectoriel $\End_{\Uq - \Uq} \left( \Regb \right)$ (cf. le lemme \ref{Lemme: centre}). Par construction, on a :
\begin{equation} \tag{1}
\begin{aligned}
&\forall a \in \Uq 
	&& e_s a = \mathbf{e}_s(1) a = \mathbf{e}_s(a) , && 0 \leq s \leq p, \\
&\forall a \in \Uq
	&& w_s^+ a = \mathbf{w}_s^+(1) a = \mathbf{w}_s^+(a), && 1 \leq s \leq p.
\end{aligned}
\end{equation}
Comme les vecteurs $x^\pm_0(s)$ et $y^\pm_0(s)$, $1 \leq s \leq p$, appartiennent à $\Uq$ (cf. la remarque \ref{Rem: ideaux}), il suffit d'appliquer les égalités $(1)$ sur ceux-ci. Le résultat s'ensuit directement.
\end{proof}

\subsection{Description du centre à partir de l'élément de Casimir} 
\label{subsection: centre2}

La base canonique $\{ e_s \; ; \; 0 \leq s \leq p \} \cup \{ w^\pm_s \; ; \; 1 \leq s \leq p-1 \}$ du centre $\Zf$ de $\Uq$ (cf. la proposition \ref{Prop: centre}) peut être explicitée en fonction des générateurs de $\Uq$ à l'aide d'un élément particulier du centre :
\begin{equation}
C := EF + \frac{q^{-1}K+qK^{-1}}{(q-q^{-1})^2} \overset{\eqref{Eqn: comm1}}{=} FE + \frac{qK+q^{-1}K^{-1}}{(q-q^{-1})^2}.
\label{Eqn: Casimir} \index{C@ $C$}
\end{equation}
On l'appelle l'\emph{élément de Casimir}.

\begin{Prop} 
\label{Prop: poly en C}
Pour tout polynôme $R \in \C[x]$, on a :
\[ R(C) = \sum_{s=0}^p R(\beta_s) e_s + \sum_{s=1}^{p-1} R'(\beta_s) \left( w^+_s + w^-_s \right), \]
où, pour tout $j \in \{0,...,p\}$, $\beta_j := \frac{q^j+q^{-j}}{(q-q^{-1})^2}$. \index{beta@ $\beta_j$}
\end{Prop}

\begin{proof}
Soit $R \in \C[x]$. Comme $C$ est central, l'élément $R(C)$ l'est aussi. D'après la proposition \ref{Prop: centre}, il existe une unique famille de scalaires $\{ a_s \; ; \; 0 \leq s \leq p \} \cup \{ b^\pm_s \; ; \; 1 \leq s \leq p-1 \}$ telle que :
\[ R(C) = \sum_{s=0}^p a_s e_s + \sum_{s=1}^{p-1} \left( b^+_s w^+_s + b^-_s w^-_s \right) . \]
Conformément au corollaire \ref{Cor: centre}, on explicite cette famille de scalaires en faisant agir $R(C)$ sur les vecteurs $x^\pm_0(s)$ et $y^\pm_0(s)$, $1 \leq s \leq p$, de poids $\pm q^{s-1}$, qui engendrent respectivement $\X^\pm(s)$ et $\PIM^\pm(s)$ sous $\Uq$ (cf. les propositions \ref{Prop: simples} et \ref{Prop: PIMs}). 

Pour tous $\alpha \in \{+,-\}$ et $s \in \{1,...,p\}$, on a :
\begin{align}
C \; y_0^\alpha(s) &\overset{\eqref{Eqn: Casimir}}{=} \left( FE + \frac{qK+q^{-1}K^{-1}}{(q-q^{-1})^2} \right) y_0^\alpha \notag \\
	&\quad = (1-\delta_{s,p}) x_0^\alpha(s) + \frac{\alpha q^s + \alpha q^{-s}}{(q-q^{-1})^2} y_0^\alpha(s) \notag \\ 
	&\quad = (1-\delta_{s,p}) x_0^\alpha(s) + \alpha \beta_s y_0^\alpha(s), \tag{1} \\
C \; x_0^\alpha(s) &\overset{\eqref{Eqn: Casimir}}{=} \left( FE + \frac{qK+q^{-1}K^{-1}}{(q-q^{-1})^2} \right) x_0^\alpha \notag \\
	&\quad = \frac{\alpha q^s + \alpha q^{-s}}{(q-q^{-1})^2} x_0^\alpha(s) = \alpha \beta_s x_0^\alpha(s). \tag{2}
\end{align}
La seconde égalité donne immédiatement :
\begin{align*}
&\forall s \in \{1,...,p\} && R(C) \; x_0^+(s) = R(\beta_s) x_0^+(s), \\
&\forall s \in \{1,...,p\} && R(C) \; x_0^-(s) = R(-\beta_s) x_0^-(s) = R(\beta_{p-s}) x_0^-(s).
\end{align*}
On en déduit que, pour tout $s \in \{0,...,p\}$, $a_s = R(\beta_s)$.

Pour exploiter la première égalité, on utilise un développement de Taylor-Young à l'ordre 2 sur le polynôme $R$ en $\alpha \beta$. Pour tous $\alpha \in \{+,-\}$ et $s \in \{1,...,p-1\}$, on obtient :
\begin{align*}
R(C) \; y_0^\alpha(s) &= R(\alpha \beta_s) y_0^\alpha(s) + R'(\alpha \beta_s) (C-\alpha \beta_s) y_0^\alpha(s) + O(C-\alpha \beta_s)^2 y_0^\alpha(s) \\
	&\overset{(1)}{=} R(\alpha \beta) y_0^\alpha(s) + R'(\alpha \beta_s) x_0^\alpha(s) + O(C-\alpha \beta_s) x_0^\alpha(s) \\
	&\overset{(2)}{=} R(\alpha \beta) y_0^\alpha(s) + R'(\alpha \beta_s) x_0^\alpha(s).
\end{align*}
Il s'ensuit que :
\begin{align*}
&\forall s \in \{1,...,p-1\} & R(C) \; y_0^+(s) &= R(\beta_s) y_0^+(s) + R'(\beta_s) x_0^+(s), \\
&\forall s \in \{1,...,p-1\} & R(C) \; y_0^-(s) &= R(-\beta_s) y_0^-(s) + R'(-\beta_s) x_0^-(s) \\ &&&= R(\beta_{p-s}) y_0^-(s) + R'(\beta_{p-s}) x_0^-(s).
\end{align*}
On en déduit que, pour tout $s \in \{1,...,p-1\}$, $b_s^+=b_s^-=R'(\beta_s)$. D'où le résultat.
\end{proof}

\begin{Rem}
\label{Rem: Casimir}
\begin{enumerate}[(i)]
	\item En particulier, la sous-algèbre $\langle C \rangle$ engendrée par $C$ est un $\C$-espace vectoriel de dimension $2p$, dont une $\C$-base est $\{ e_s \; ; \; 0 \leq s \leq p \} \cup \{ w^+_s+w^-_s \; ; \; 1 \leq s \leq p-1 \}$.
	\item On en déduit également le polynôme minimal de l'élément de Casimir $C$ \eqref{Eqn: Casimir} :
	\[ \psi_{2p}(x) := (x-\beta_0) \left( \prod_{j=1}^{p-1} \left( x-\beta_j \right)^2 \right) (x-\beta_p). \]
\end{enumerate}
\end{Rem} 

On considère maintenant les polynômes :
\begin{equation} 
\label{Eqn: poly en C}
\begin{aligned}
& \psi_0(x) := \prod_{j=1}^{p-1} \left( x-\beta_j \right)^2 (x-\beta_p), \\
& \psi_s(x) := (x-\beta_0) \left( \prod_{\substack{j=1 \\ j \not = s}}^{p-1} \left( x-\beta_j \right)^2 \right) (x-\beta_p), && 1 \leq s \leq p-1, \\
& \psi_p(x) := (x-\beta_0) \prod_{j=1}^{p-1} \left( x-\beta_j \right)^2,
\end{aligned}
\end{equation}
où, pour tout $j \in \{0,...,p\}$, $\beta_j := \frac{q^j+q^{-j}}{(q-q^{-1})^2}$. En utilisant ces polynômes dans la proposition \ref{Prop: poly en C}, on exprime les éléments de la base canonique du centre (cf. proposition \ref{Prop: centre}) comme des polynômes en $C$ \eqref{Eqn: Casimir} à multiplication par l'élément $K$ près.

\begin{Cor} 
\label{Cor: poly en C}
On utilise les notations \ref{Eqn: poly en C} et, pour tout $s \in \{1,...,p-1\}$, on pose :
\[ \pi^+_s := \frac{1}{2p} \sum_{\substack{0 \leq n \leq s-1 \\ 0 \leq j \leq 2p-1}} q^{(2n-s+1)j} K^j \quad \text{ et } \quad \pi^-_s := \frac{1}{2p} \sum_{\substack{s \leq n \leq p-1 \\ 0 \leq j \leq 2p-1}} q^{(2n-s+1)j} K^j . \]
Les éléments de la base canonique du centre vérifient :
\begin{align*}
& e_s = \frac{1}{\psi_s(\beta_s)} \psi_s(C), && s \in \{0,p\}, \\
& e_s = \frac{1}{\psi_s(\beta_s)} \psi_s(C) -\frac{\psi'_s(\beta_s)}{\psi_s(\beta_s)^2} (C-\beta_s) \psi_s(C), && 1 \leq s \leq p-1, \\
& w^\pm_s = \frac{1}{\psi_s(\beta_s)} \pi^\pm_s (C-\beta_s) \psi_s(C), && 1 \leq s \leq p-1,
\end{align*}
où, pour tout $j \in \{0,...,p\}$, $\beta_j := \frac{q^j+q^{-j}}{(q-q^{-1})^2}$.
\end{Cor}

\begin{proof}
Soit $s \in \{1,...,p-1\}$. On commence par utiliser la proposition \ref{Prop: poly en C} avec le polynôme $(x-\beta_s) \psi_s(x)$. On obtient :
\begin{equation} \tag{1}
(C-\beta_s) \psi_s(C) = \psi_s(\beta_s) (w_s^+ + w_s^-) .
\end{equation}
On projette l'équation $(1)$ sur l'espace vectoriel engendré par le vecteur $w_s^+$ et sur celui engendré par le veteur $w_s^-$. Pour cela, on multiplie l'équation $(1)$ par $\pi_s^+$ et $\pi_s^-$ respectivement. En effet, $\pi_s^+$ (resp. $\pi_s^-$) agit par multiplication comme la projection orthogonale sur les espaces propres de $K$ associés aux valeurs propres :
\begin{gather*}
\{ q^{s-1-2n} \; ; \; 0 \leq n \leq s-1 \} \\
\left( \text{resp. } \{ q^{s-1-2n} \; ; \; s \leq n \leq p-1 \} = \{ -q^{p-s-1-2n} \; ; \; 0 \leq n \leq p-s-1 \} \right).
\end{gather*}
On obtient ainsi :
\[ w_s^\pm = \frac{1}{\psi_s(\beta_s)} \pi_s^\pm (C-\beta_s) \psi_s(C) . \]

Soit $s \in \{0,...,p\}$. On utilise à présent la proposition \ref{Prop: poly en C} avec le polynôme $\psi_s(x)$. On obtient :
\begin{equation} \tag{2}
\begin{aligned}
&\psi_s(C) = \psi_s(\beta_s) e_s && \text{ si } s \in \{0,p\}, \\
&\psi_s(C) = \psi_s(\beta_s) e_s + \psi_s'(\beta_s) (w_s^+ + w_s^-)&& \text{ si } 1 \leq s \leq p-1.
\end{aligned}
\end{equation}
Or, pour $s \in \{1,...,p-1\}$, on sait que :
\[ (w_s^+ + w_s^-) \overset{(1)}{=} \frac{1}{\psi_s(\beta_s)} (C-\beta_s) \psi_s(C) . \]
Les équations $(2)$ donnent donc :
\begin{align*}
& e_s = \frac{1}{\psi_s(\beta_s)} \psi_s(C) && \text{ si } s \in \{0,p\}, \\
& e_s = \frac{1}{\psi_s(\beta_s)} \psi_s(C) - \frac{\psi'_s(\beta_s)}{\psi_s(\beta_s)^2} (C-\beta_s) \psi_s(C) && \text{ si } 1 \leq s \leq p-1.
\end{align*}
\end{proof}

\begin{Rem}
\label{Rem: poly en C}
Une autre formulation de la proposition \ref{Prop: poly en C} et du corollaire \ref{Cor: poly en C} sera utile dans le chapitre IV. A savoir, pour tout polynôme $R \in \C[x]$, on a :
\[ R(\pm \widehat{C}) = \sum_{s=0}^p R(\pm \widehat{\beta}_s) e_s \pm (q-q^{-1})^2 \sum_{s=1}^{p-1} R'(\pm \widehat{\beta}_s) \left( w^+_s + w^-_s \right), \]
où $\widehat{C}:=(q-q^{-1})^2 C$ et, pour tout $j \in \{0,...,p\}$, $\widehat{\beta}_j := q^j+q^{-j}$\index{beta@ $\widehat{\beta}_j$}. Les éléments de la base canonique du centre vérifient alors :
\begin{align*}
& e_s = \frac{1}{\widehat{\psi}_s(\pm \widehat{\beta}_s)} \widehat{\psi}_s((\pm \widehat{C}), && s \in \{0,p\}, \\
& e_s = \left( \frac{1}{\widehat{\psi}_s(\pm \widehat{\beta}_s)}-\frac{\widehat{\psi}'_s(\pm \widehat{\beta}_s)}{\widehat{\psi}_s(\pm \widehat{\beta}_s)^2} (\pm \widehat{C} \mp \widehat{\beta}_s) \right) \widehat{\psi}_s(\pm \widehat{C}), && 1 \leq s \leq p-1, \\
& w^\pm_s = \pm \frac{1}{(q-q^{-1})^2 \; \widehat{\psi}_s(\pm \widehat{\beta}_s)} \pi^\pm_s (\pm \widehat{C} \mp \widehat{\beta}_s) \widehat{\psi}_s(\pm \widehat{C}), && 1 \leq s \leq p-1,
\end{align*}
avec :
\begin{align*}
& \widehat{\psi}_0(x) := \prod_{j=1}^{p-1} \left( x \mp \widehat{\beta}_j \right)^2 (x \mp \widehat{\beta}_p), \\
& \widehat{\psi}_s(x) := (x \mp \widehat{\beta}_0) \left( \prod_{\substack{j=1 \\ j \not = s}}^{p-1} \left( x \mp \widehat{\beta}_j \right)^2 \right) (x \mp \widehat{\beta}_p), && 1 \leq s \leq p-1, \\
& \widehat{\psi}_p(x) := (x \mp \widehat{\beta}_0) \prod_{j=1}^{p-1} \left( x \mp \widehat{\beta}_j \right)^2.
\end{align*}
\end{Rem}

\section{Balancement, enrubannement et caractères}
\label{section: structures}

On détaille plus amplement la structure de l'algèbre de Hopf $\Uq$. Pour cela, on commence par quelques rappels sur les algèbres de Hopf tressées (on dit aussi quasi-triangulaires, cf. par exemple \cite[§ VIII.2]{Kas95}). Pour toute algèbre de Hopf tressée de dimension finie $(A, \mu, \eta, \Delta, \varepsilon, S, \mathbf{R})$, il est possible de construire un élément canonique $\mathbf{u} \in A$ \index{u@ $\mathbf{u}$} inversible tel que :
\[ \Delta(\mathbf{u}) = (\mathbf{R}_{12} \mathbf{R})^{-1} \mathbf{u} \otimes \mathbf{u}
\quad \text{ et } \quad 
\forall x \in A \quad S^2(x) = \mathbf{u} x \mathbf{u}^{-1} \]
(cf. par exemple \cite[§ VIII.4, p.173]{Kas95} et \cite[Prop. 2]{LS69}). On en déduit un élément $\mathbf{g} = \mathbf{u} S(\mathbf{u})^{-1}$ \index{g@ $\mathbf{g}$} inversible tel que :
\[ \Delta(\mathbf{g}) = \mathbf{g} \otimes \mathbf{g} \quad \text { et } \quad \forall x \in A \quad S^4(x) = \mathbf{g} x \mathbf{g}^{-1} \]
(cf. par exemple \cite[Thm 10.1.13]{Mon93}). Ce dernier permet de définir la notion de \emph{balancement}.

\begin{Def}
Soit $(A, \mu, \eta, \Delta, \varepsilon, S, \mathbf{R})$ une algèbre de Hopf tressée de dimension finie. On dit que $\mathbf{k} \in A$ est un \emph{élément de balancement} si :
\begin{equation}
\label{Eqn: balancement}
\mathbf{k}^2 = \mathbf{g}, 
\quad \Delta( \mathbf{k} ) = \mathbf{k} \otimes \mathbf{k},
\quad \text{ et } \quad 
\forall x \in A \quad S^2(x) = \mathbf{k} x \mathbf{k}^{-1}.
\end{equation}
Lorsqu'un tel élément existe, on dit que $A$ est \emph{balancée}.
\end{Def}

Cette notion est équivalente à celle d'\emph{enrubannement} définie ci-dessous.

\begin{Def}[{\cite[Def. XIV.6.1]{Ker95}}]
Soit $(A, \mu, \eta, \Delta, \varepsilon, S, \mathbf{R})$ une algèbre de Hopf tressée de dimension finie. On dit que $\mathbf{v} \in A$ est un \emph{élément d'enrubannement} si :
\begin{equation}
\label{Enq: enrubannement}
\forall x \in A \quad x \mathbf{v} = \mathbf{v} x, 
\quad \Delta(\mathbf{v}) = (\mathbf{R}_{12} \mathbf{R})^{-1} \mathbf{v} \otimes \mathbf{v}, 
\quad \text{ et } \quad
S(\mathbf{v}) = \mathbf{v}.
\end{equation}
Lorsqu'un tel élément existe, on dit que $A$ est \emph{enrubannée}.
\end{Def}

Détaillons la relation entre balancement et enrubannement. On se donne une algèbre de Hopf $(A, \mu, \eta, \Delta, \varepsilon, S, \mathbf{R})$ tressée de dimension finie. Si $\mathbf{k} \in A$ est un élément de balancement, alors l'élément inversible $\mathbf{v} = \mathbf{u} \mathbf{k}^{-1} = S^2(\mathbf{k}^{-1}) \mathbf{u} = \mathbf{k}^{-1} \mathbf{u}$ vérifie :
\begin{gather*}
\forall x \in A \quad x \mathbf{v} = x \mathbf{k}^{-1} \mathbf{u}^{-1} = \mathbf{k}^{-1} S^2(x) \mathbf{u} = \mathbf{k}^{-1} \mathbf{u} x = \mathbf{v} x, \\
\Delta(\mathbf{v}) = \Delta(\mathbf{u}) \Delta(\mathbf{k}^{-1}) = (\mathbf{R}_{12} \mathbf{R})^{-1} (\mathbf{u} \otimes \mathbf{u})  (\mathbf{k}^{-1} \otimes \mathbf{k}^{-1}) = (\mathbf{R}_{12} \mathbf{R})^{-1} \mathbf{v} \otimes \mathbf{v}, \\
S(\mathbf{v}) = S(\mathbf{u} \mathbf{k}^{-1}) = S(\mathbf{k}^{-1}) S(\mathbf{u}) = S^2(\mathbf{k}) S(\mathbf{u}) = \mathbf{k} S(\mathbf{u}) = \mathbf{u} \mathbf{k}^{-1} = \mathbf{v}.
\end{gather*}
Donc $\mathbf{v}$ est un élément d'enrubannement. Réciproquement, si $\mathbf{v} \in A$ est un élément d'enrubannement, alors $\mathbf{v}^2 = \mathbf{u} S(\mathbf{u})$ (cf. par exemple \cite[Cor. XIV.6.3]{Kas95}) et l'élément inversible $\mathbf{k} = \mathbf{v}^{-1} \mathbf{u} = \mathbf{u} \mathbf{v}^{-1}$ vérifie :
\begin{gather*}
\mathbf{k}^2 = \mathbf{v}^{-2} \mathbf{u}^2 = S(\mathbf{u})^{-1} \mathbf{u}^{-1} \mathbf{u}^2 = S(\mathbf{u})^{-1} \mathbf{u} = S^2(\mathbf{u}) S(\mathbf{u})^{-1} = \mathbf{u} S(\mathbf{u})^{-1}, \\
\Delta(\mathbf{k}) = \Delta(\mathbf{v}^{-1}) \Delta(\mathbf{u}) = \mathbf{v}^{-1} \otimes \mathbf{v}^{-1} (\mathbf{R}_{12} \mathbf{R}) (\mathbf{R}_{12} \mathbf{R})^{-1} \mathbf{u} \otimes \mathbf{u} = \mathbf{k} \otimes \mathbf{k}, \\
\forall x \in A \quad S^2(x) = \mathbf{u} x \mathbf{u}^{-1} = \mathbf{v}^{-1} \mathbf{u} x \mathbf{u}^{-1} \mathbf{v} = \mathbf{k} x \mathbf{k}^{-1}.
\end{gather*} 
Donc $\mathbf{k}$ est un élément de balancement.

Conformément à l'article \cite[Cor. 3.7.4]{KS11}, la groupe quantique restreint $\Uq$ n'est pas une algèbre de Hopf tressée (et par suite, elle n'est ni balancée ni enrubannée). Toutefois, il est possible de construire des éléments de balancement et d'enrubannement "généralisés". Pour cela, on considère le dual $\Uq^* := \Hom(\Uq, \C)$ de $\Uq$. On rappelle que, pour toute $\C$-algèbre de Hopf $(A, \mu, \eta, \Delta, \varepsilon, S)$ de dimension finie, le dual $A^* := \Hom \left( A , \C \right)$ est naturellement muni d'une structure d'algèbre de Hopf (cf. par exemple \cite[Prop. III.1.2, Prop III.1.3, Prop. III.3.3]{Kas95}). Dans ce qui suit, on considère l'algèbre de Hopf duale \emph{co-opposée} $(A^*, \Delta^*, \varepsilon^*, {\mu^*}^{op}, \eta^*, (S^{-1})^*)$. Donc, pour tous $\lambda, \mu \in A^*$ et pour tous $a, b \in A$, on a :
\begin{gather*}
(\lambda \mu)(a) = (\lambda \otimes \mu) (\Delta(a)), \qquad \Delta(\lambda)(a \otimes b) = \lambda (ba), \qquad S(\lambda)(a)=\lambda(S^{-1}(a)) .
\end{gather*}
De plus, on notera simplement $A$ toute algèbre de Hopf $(A, \mu, \eta, \Delta, \varepsilon, S)$.

\subsection{(Co-)intégrales et balancement}
\label{subsection: integrales}

Dans un premier temps, on munit $\Uq$ et $\Uq^*$ d'une structure de $\Uq$-bimodules par actions de la multiplication à gauche et à droite. On étudie les intégrales et co-intégrales de $\Uq$. On en déduit deux éléments de balancement "généralisés" possibles pour $\Uq$.

\begin{Def}[{\cite[Def. 2.1.1]{Mon93}}]
Soit $A$ une $\C$-algèbre de Hopf de dimension finie.
\begin{enumerate}[(i)]
	\item On appelle \emph{intégrale} à gauche (resp. à droite) toute forme linéaire $\bmu : A \rightarrow \C$ telle que :
	\[ \forall x \in A \quad (id \otimes \bmu) \Delta(x) = \bmu(x) 1 \quad \left( \text{resp. } (\bmu \otimes id) \Delta(x) = \bmu(x) 1 \right). \]
	On note $\int_{A^*}^l$ (resp. $\int_{A^*}^r$) le $\C$-espace vectoriel des intégrales à gauche (resp. à droite). De plus, on dit que $A$ est \emph{unicomodulaire} si $\int_{A^*}^l=\int_{A^*}^r$.
	\item On appelle \emph{co-intégrale} à gauche (resp. à droite) tout élément $\mathbf{c} \in A$ tel que :
	\[ \forall x \in A \quad x\mathbf{c} = \varepsilon(x) \mathbf{c} \quad \left( \text{resp. } \mathbf{c}x = \varepsilon(x) \mathbf{c} \right). \]
	On note $\int_A^l$ (resp. $\int_A^r$) le $\C$-espace vectoriel des co-intégrales à gauche (resp. à droite). De plus, on dit que $A$ est \emph{unimodulaire} si $\int_A^l=\int_A^r$.
\end{enumerate}
\end{Def}

\begin{Rem}
Soit $A$ une $\C$-algèbre de Hopf de dimension finie. Le $\C$-espace vectoriel $\int_A^l$ (resp. $\int_A^r$) coïncide avec le $\C$-espace vectoriel des éléments invariants de $A$ sous l'action de la multiplication à gauche (resp. à droite). Il en va de même pour le $\C$-espace vectoriel $\int_{A^*}^l$ (resp. $\int_{A^*}^r$) dans $A^*$. En effet, pour tous $\bmu \in A^*$, on a :
\begin{align*}
&\forall \beta \in A^* \quad \forall x \in A \quad (\beta \bmu)(x) = \varepsilon(\beta) \bmu(x) 
	\quad \left( \text{resp. } (\bmu \beta)(x) = \varepsilon(\beta) \bmu(x) \right) \\
\Longleftrightarrow \quad& \begin{multlined}[t]
	\forall \beta \in A^* \quad \forall x \in A \quad (\beta \otimes \bmu) \Delta(x) = \beta(1) \bmu(x) \\
	\quad \left( \text{resp. } (\bmu \otimes \beta) \Delta(x) = \beta(1) \bmu(x) \right) 
	\end{multlined} \\
\Longleftrightarrow \quad& \begin{multlined}[t]
	\forall \beta \in A^* \quad \forall x \in A \quad \beta \left( (id \otimes \bmu) \Delta(x) \right) = \beta(\bmu(x)1) \\
	\quad \left( \text{resp. } \beta \left( (\bmu \otimes id) \Delta(x) \right) = \beta(\bmu(x)1) \right) 
	\end{multlined} \\
\Longleftrightarrow \quad& \forall x \in A \quad  (id \otimes \bmu) \Delta(x) = \bmu(x)1
	\quad \left( \text{resp. } (\bmu \otimes id) \Delta(x) = \bmu(x)1 \right)
\end{align*}
car $A$ est de dimension finie.
\end{Rem}

\begin{Prop} 
\label{Prop: integrales}
On considère les algèbres de Hopf $\Uq$ et $\Uq^*$.
\begin{enumerate}[(a)]
	\item Il existe des intégrales à gauche (resp. à droite) et elles sont définies par : \index{mu@ $\bmu^l_{\zeta}$} \index{mu@ $\bmu^r_{\zeta}$}
	\begin{gather*}
	\bmu_\zeta^l : \begin{cases}
		\Uq \longrightarrow \C \\
		F^m K^j E^n \longmapsto \zeta \delta_{m,p-1} \delta_{j,p-1} \delta_{n,p-1}
		\end{cases} ; 
		\quad \zeta \in \C \\
	\left( \text{resp. } \bmu_\zeta^r : \begin{cases}
		\Uq \longrightarrow \C \\
		F^m K^j E^n \longmapsto \zeta \delta_{m,p-1} \delta_{j,p+1} \delta_{n,p-1}
		\end{cases} ; 
		\quad \zeta \in \C \right).
	\end{gather*} 
	pour tous $m, n \in \{0,...,p-1\}$ et $j \in \{0,...,2p-1\}$.
	\item Il existe des co-intégrales à gauche (resp. à droite) et elles sont données par : \index{c@ $\mathbf{c}_{\zeta}$}
	\[ \mathbf{c}_\zeta = \zeta F^{p-1} \sum_{j=0}^{2p-1} q^{2j} K^j E^{p-1} = \zeta E^{p-1} \sum_{j=0}^{2p-1} q^{-2j} K^j F^{p-1} \quad ; 
		\quad \zeta \in \C. \] 
\end{enumerate}
\end{Prop}

\begin{proof}
\begin{enumerate}[(a)]
	\item On utilise la $\C$-base $\{ F^m K^j E^n \; ; \; 0 \leq m,n \leq p-1, \; 0 \leq j \leq 2p-1\}$ de $\Uq$ (cf. le théorème \ref{Thm: base PBW}). Il s'agit de trouver des formes linéaires $\bmu : \Uq \rightarrow \C$ telles que, pour tous $m,n \in \{0,...,p-1\}$ et $j \in \{0,...,2p-1\}$, on ait :
	\begin{align}
	& \quad \quad (id \otimes \bmu) \Delta(F^m K^j E^n) = \bmu(F^m K^j E^n) 1 \tag{1a} \\
	& \left( \text{resp. } (\bmu \otimes id) \Delta(F^m K^j E^n) = \bmu(F^m K^j E^n) 1 \right). \notag
	\end{align}
	Or, pour tous $m,n \in \{0,...,p-1\}$ et $j \in \{0,...,2p-1\}$, on a :
\begin{align*}
\Delta( F^m K^j E^n )
&\; \; = \Delta(F)^m \Delta(K)^j \Delta(E)^n \\
	&\overset{\eqref{Eqn: coprod2}}{=} \sum_{r=0}^m \sum_{s=0}^n q^{r(m-r)+s(n-s)} {m \brack r} {n \brack s} F^r K^{r-m+j} E^{n-s} \otimes F^{m-r} K^j E^s K^{n-s} \\
	&\overset{\eqref{Eqn: comm1}}{=} \sum_{r=0}^m \sum_{s=0}^n q^{r(m-r)+s(n-s)} {m \brack r} {n \brack s} F^r K^{r-m+j} E^{n-s} \otimes F^{m-r} K^{j+n-s} E^s.
\end{align*}
	S'il existe une intégrale à gauche (resp. à droite) $\bmu$, alors la condition $(1a)$ impose qu'il existe $\zeta \in \C$ tel que :
	\begin{gather*}
	\bmu = \bmu_\zeta^l : \begin{cases}
		\Uq \longrightarrow \C \\
		F^m K^j E^n \longmapsto \zeta \delta_{m,p-1} \delta_{j,p-1} \delta_{n,p-1}
		\end{cases} \\
	\left( \text{resp. } \bmu = \bmu_\zeta^r : \begin{cases}
		\Uq \longrightarrow \C \\
		F^m K^j E^n \longmapsto \zeta \delta_{m,p-1} \delta_{j,p+1} \delta_{n,p-1}
		\end{cases}  \right).
	\end{gather*}
	pour tous $m, n \in \{0,...,p-1\}$ et $j \in \{0,...,2p-1\}$. Réciproquement, pour tout $\zeta \in \C$, $\bmu_\zeta$ est une intégrale à gauche (resp. à droite) d'après les calculs précédents.
	\item On utilise la $\C$-base $\{ F^m K^j E^n \; ; \; 0 \leq m,n \leq p-1, \; 0 \leq j \leq 2p-1\}$ (resp. $\{ E^n K^j F^m \; ; \; 0 \leq m,n \leq p-1, \; 0 \leq j \leq 2p-1\}$) de $\Uq$ (cf. le théorème \ref{Thm: base PBW}). Il s'agit de trouver des éléments $\mathbf{c} \in \Uq$ tels que :
	\begin{align}
	&F \mathbf{c} = \varepsilon(F) \mathbf{c} \overset{\eqref{Eqn: coprod1}}{=} 0
	&& \left( \text{resp. } \mathbf{c} F = \varepsilon(F) \mathbf{c} \overset{\eqref{Eqn: coprod1}}{=} 0 \right), \tag{1b} \\
	&K \mathbf{c} = \varepsilon(K) \mathbf{c} \overset{\eqref{Eqn: coprod1}}{=} \mathbf {c}
	&& \left( \text{resp. } \mathbf{c} K = \varepsilon(K) \mathbf{c} \overset{\eqref{Eqn: coprod1}}{=} \mathbf{c} \right), \tag{2b} \\
	&E \mathbf{c} = \varepsilon(E) \mathbf{c} \overset{\eqref{Eqn: coprod1}}{=} 0
	&& \left( \text{resp. } \mathbf{c} E = \varepsilon(E) \mathbf{c} \overset{\eqref{Eqn: coprod1}}{=} 0 \right). \tag{3b}
	\end{align}
	 S'il existe un co-intégrale à gauche (resp. à droite) $\mathbf{c}$, alors les conditions $(1b)$ et $(2b)$ imposent qu'il existe $\zeta_0,...,\zeta_{p-1} \in \C$ tels que :
	\[ \mathbf{c} = F^{p-1} \sum_{n=0}^{p-1} \sum_{j=0}^{2p-1} \zeta_n q^{2j} K^j E^n
	\quad \left( \text{resp. } \mathbf{c} = \sum_{n=0}^{p-1} \sum_{j=0}^{2p-1} \zeta_n q^{-2j} E^n K^j \; F^{p-1} \right). \]
	Avec cet élément $\mathbf{c}$, on a :
	\begin{align*}
	E \mathbf{c} &\; \; = E F^{p-1} \sum_{n=0}^{p-1} \sum_{j=0}^{2p-1} \zeta_n q^{2j} K^j E^n \\
	&\overset{\eqref{Eqn: comm2}}{=} \begin{multlined}[t]
		F^{p-1} E \sum_{n=0}^{p-1} \sum_{j=0}^{2p-1} \zeta_n q^{2j} K^j E^n \\
		+ [p-1] F^{p-2} \frac{q^{-(p-2)}K-q^{p-2}K^{-1}}{q-q^{-1}} \sum_{n=0}^{p-1} \sum_{j=0}^{2p-1} \zeta_n q^{2j} K^j E^n 
		\end{multlined} \\
	\Bigg( \text{resp. } \mathbf{c} E &\; \; = \sum_{n=0}^{p-1} \sum_{j=0}^{2p-1} \zeta_n q^{-2j} E^n K^j \; F^{p-1} E\\
	&\overset{\eqref{Eqn: comm2}}{=} \begin{multlined}[t]
		\sum_{n=0}^{p-1} \sum_{j=0}^{2p-1} \zeta_n q^{-2j} E^n K^j \; E F^{p-1} \\
		- [p-1] \sum_{n=0}^{p-1} \sum_{j=0}^{2p-1} \zeta_n q^{-2j} E^n K^j \; \frac{q^{p-2}K-q^{-(p-2)}K^{-1}}{q-q^{-1}} F^{p-2} \Bigg).
		\end{multlined}
	\end{align*}
	Le second terme vérifie :
	\begin{align*}
	&[p-1] F^{p-2} \frac{q^{-(p-2)}K-q^{p-2}K^{-1}}{q-q^{-1}} \sum_{n=0}^{p-1} \sum_{j=0}^{2p-1} \zeta_n q^{2j} K^j E^n \\
	&\; \overset{\eqref{Eqn: qcoeff2}}{=} \frac{-1}{q-q^{-1}} F^{p-2} \sum_{n=0}^{p-1} \sum_{j=0}^{2p-1} \zeta_n \left( q^{2(j+1)} K^{j+1}-q^{2(j-1)} K^{j-1} \right) E^n
	\overset{\eqref{Eqn: comm1}}{=} 0 \\
	\Bigg( \text{resp. } &-[p-1] \sum_{n=0}^{p-1} \sum_{j=0}^{2p-1} \zeta_n q^{-2j} E^n K^j \; \frac{q^{p-2}K-q^{-(p-2)}K^{-1}}{q-q^{-1}} F^{p-2} \\
	&\; \overset{\eqref{Eqn: qcoeff2}}{=} \frac{1}{q-q^{-1}} \sum_{n=0}^{p-1} \sum_{j=0}^{2p-1} \zeta_n E^n \left( q^{-2(j+1)} K^{j+1}-q^{-2(j-1)} K^{j-1} \right) \; F^{p-2}
	\overset{\eqref{Eqn: comm1}}{=} 0 \Bigg).
	\end{align*}
	Donc :
	\[ E \mathbf{c} \overset{\eqref{Eqn: comm1}}{=} F^{p-1} \sum_{n=0}^{p-1} \sum_{j=0}^{2p-1} \zeta_n K^j E^{n+1}
	\quad \left( \text{resp. } \mathbf{c} E \overset{\eqref{Eqn: comm1}}{=} \sum_{n=0}^{p-1} \sum_{j=0}^{2p-1} \zeta_n E^{n+1} K^j \; F^{p-2} \right). \]
	La condition $(3b)$ impose donc que, pour tout $n \in \{0,...,p-2\}$, $\zeta_n = 0$. D'où :
	\[ \mathbf{c} = \mathbf{c}_{\zeta_{p-1}} = \zeta_{p-1} F^{p-1} \sum_{j=0}^{2p-1} q^{2j} K^j E^{p-1} = \zeta_{p-1} E^{p-1} \sum_{j=0}^{2p-1} q^{-2j} K^j F^{p-1}, \]
	où la troisième égalité découle de :
	\begin{align*}
	q^{2j} F^{p-1} K^j E^{p-1} \overset{\eqref{Eqn: comm1}}{=} K^j F^{p-1} E^{p-1} \overset{\eqref{Eqn: comm2}}{=} K^j E^{p-1} F^{p-1} \overset{\eqref{Eqn: comm1}}{=} q^{-2j} F^{p-1} K^j E^{p-1}.
	\end{align*}
	Réciproquement, pour tout $\zeta \in \mathbb{C}$, $\mathbf{c}_{\zeta}$ est une co-intégrale bilatère d'après les calculs précédents.

\end{enumerate}
\end{proof}

\begin{Rem} 
\label{Rem: integrales}
\begin{enumerate}[(i)]
	\item Soit $A$ une $\C$-algèbre de Hopf de dimension finie. D'après l'article \cite{LS69} (dont les principaux résultats sont repris dans \cite[Thm 2.1.3]{Mon93}), les $\C$-espaces vectoriels $\int_{A^*}^l$ et $\int_{A^*}^r$ (resp. $\int_A^l$ et $\int_A^r$) sont de dimension 1. De plus, l'antipode $(S^{-1})^*$ (resp. $S$) induit un isomorphisme entre ces deux $\C$-espaces vectoriels :
	\[ \begin{cases}
		\int_{A^*}^l \overset{\sim}{\longrightarrow} \int_{A^*}^r \\
		\bmu \longmapsto S(\bmu )= \bmu  \circ S^{-1}
		\end{cases} \quad
	\left( \text{resp. } \begin{cases}
		\int_A^l \overset{\sim}{\longrightarrow} \int_A^r \\
		\mathbf{c} \longmapsto S(\mathbf{c})
		\end{cases} \right). \]
		Pour l'algèbre de Hopf $\Uq$, on a montré que $\int_A^l = \int_A^r$ dans la proposition \ref{Prop: integrales}.
	\item D'après le théorème de Maschke (cf. par exemple \cite[Thm 2.2.1]{Mon93}), on retrouve que $\Uq$ n'est pas semi-simple (cf. la proposition \ref{Thm: Reg}) car $\varepsilon(\mathbf{c}) = 0$.
\end{enumerate}
\end{Rem}

\begin{Cor}
\label{Cor: integrales}
L'algèbre de Hopf $\Uq$ est unimodulaire.
\end{Cor}

Soit $\bmu$ une intégrale à droite de $\Uq$. Pour tout $\beta \in \Uq^*$, la forme linéaire $\beta \bmu$ est encore une intégrale à droite. Par unicité de $\bmu$ à un scalaire près (cf. l'assertion $(i)$ de la remarque \ref{Rem: integrales}) et comme $\Uq^{**} \cong \Uq$, il existe un élément $\mathbf{a} \in \Uq$ tel que :
\[ \forall \beta \in \Uq^* \quad \forall x \in \Uq \quad (\beta \bmu)(x) = \beta(\mathbf{a}) \bmu(x) . \]
Cette propriété est équivalente à :
\begin{align*}
&\forall \beta \in \Uq^* \quad  \forall x \in \Uq \quad (\beta \otimes \bmu) \Delta(x) = \beta(\mathbf{a}) \bmu(x) \\
\Longleftrightarrow \quad& \forall \beta \in \Uq^* \quad \forall x \in \Uq \quad \beta \left( (id \otimes \bmu) \Delta(x) \right) = \beta \left( \bmu(x) \mathbf{a} \right) \\
\Longleftrightarrow \quad& \forall x \in \Uq \quad (id \otimes \bmu) \Delta(x) = \bmu(x) \mathbf{a}
\end{align*}
car $\Uq$ est de dimension finie. L'élément $\mathbf{a} \in \Uq$ est donc indépendant du choix de l'intégrale à droite $\bmu$. De plus, on peut vérifier que :
\[ \forall \beta, \gamma \in \Uq^* \quad (\beta \gamma)(\mathbf{a}) = \beta(\mathbf{a}) \gamma(\mathbf{a}) . \]
De même que précédemment, cette propriété est équivalente à :
\[ \forall \beta, \gamma \in \Uq^* \quad (\beta \otimes \gamma) \Delta(\mathbf{a}) = (\beta \otimes \gamma) (\mathbf{a} \otimes \mathbf{a}) \quad \Longleftrightarrow \quad \Delta(\mathbf{a}) = \mathbf{a} \otimes \mathbf{a} . \]
Autrement dit, $\mathbf{a}$ est un élément \emph{group-like}. De manière duale, il existe un élément group-like $\boldsymbol{\alpha} \in \Uq^*$, indépendant du choix de la co-intégrale $\mathbf{c}$, tel que :
\[ \forall x \in \Uq \quad x \mathbf{c} = \boldsymbol{\alpha}(x) \mathbf{c}. \]
Dans notre cas, on a $\boldsymbol{\alpha}=\varepsilon$ car $\Uq$ est unimodulaire. L'existence de tels éléments group-like s'applique également aux (co-)intégrales à gauche. C'est un fait général aux $\C$-algèbres de Hopf de dimension finie.

\begin{Def}[{\cite[Def. 2.2.3]{Mon93}}]
Soit $A$ une $\C$-algèbre de Hopf de dimension finie.
\begin{enumerate}[(i)]
	\item Soit $\bmu$ une intégrale à gauche (resp. à droite) de $A^*$. On appelle \emph{comodule} à gauche (resp. à droite) l'élément group-like $\mathbf{a} \in A$ tel que :
\[ \forall x \in A \quad (\bmu \otimes id) \Delta(x) = \bmu(x) \mathbf{a} \quad \left( \text{resp. } (id \otimes \bmu) \Delta(x) = \bmu(x) \mathbf{a} \right). \]
	\item Soit $\mathbf{c}$ une co-intégrale à gauche (resp. à droite) de $A$. On appelle \emph{module} à gauche (resp. à droite) l'élément group-like $\boldsymbol{\alpha} \in A^*$ tel que :
\[ \forall x \in A \quad \mathbf{c} x = \boldsymbol{\alpha}(x) \mathbf{c} \quad \left( \text{resp. } x \mathbf{c} = \boldsymbol{\alpha}(x) \mathbf{c} \right). \]
\end{enumerate}
\end{Def}

\begin{Rem} 
\label{Rem: modules}
Soit $A$ une $\C$-algèbre de Hopf de dimension finie.
\begin{enumerate}[(i)]
	\item Le comodule à droite $\mathbf{b}$ est l'inverse du comodule à gauche $\mathbf{a}$. En effet, d'après l'assertion $(i)$ de la remarque \ref{Rem: integrales}, l'antipode fournit un isomorphisme entre le $\C$-espace vectoriel des intégrales à droites et celui des intégrales à gauche. Il s'ensuit que $\mathbf{b}= S(\mathbf{a})$. Or $\mathbf{a}$ est un élément group-like, donc il est inversible et son inverse est $S(\mathbf{a})$ (cf. par exemple \cite[Prop. III.3.7]{Kas95}). Le même constat s'applique pour le module à droite $\boldsymbol{\beta}$ par rapport au module à gauche $\boldsymbol{\alpha}$.
	\item En particulier, si $A$ est unicomodulaire (resp. unimodulaire), alors :
\[ \mathbf{a} = 1 \quad \text{ et } \quad \mathbf{b}=S(\mathbf{a})=1 \quad \left( \text{resp. } \mathbf{a}=\varepsilon \quad \text{ et } \quad \mathbf{b}=S(\mathbf{a})=\varepsilon \right) \]
car $S(1)=1$ (resp. $\varepsilon \circ S^{-1}=\varepsilon$, cf. par exemple \cite[III.3.4]{Kas95}). Donc $\mathbf{b}=\mathbf{a}$.
\end{enumerate}
\end{Rem}

\begin{Cor} 
On considère les algèbres de Hopf $\Uq$ et $\Uq^*$.
\begin{enumerate}[(a)]
	\item Le comodule à droite est $\mathbf{b} = K^2$. 
	\item Le module à droite est $\boldsymbol{\alpha} = \varepsilon$.
\end{enumerate}
\end{Cor}

\begin{proof}
\begin{enumerate}[(a)]
	\item Il s'agit de vérifications directes. \\ D'après la proposition \ref{Prop: integrales}, les intégrales à droite sur $\Uq$ sont définies par : 
	\[ \bmu_\zeta : \begin{cases}
		\Uq \longrightarrow \C \\
		F^m K^j E^n \longmapsto \zeta \delta_{m,p-1} \delta_{j,p+1} \delta_{n,p-1}
		\end{cases} ; 
		\quad \zeta \in \C \]
	pour tous $m, n \in \{0,...,p-1\}$ et $j \in \{0,...,2p-1\}$. Soit $\bmu:=\bmu_1$ l'intégrale à droite correspondant à $\zeta=1$. On utilise la $\C$-base $\{ F^m K^j E^n \; ; \; 0 \leq m,n \leq p-1, \; 0 \leq j \leq 2p-1\}$ de $\Uq$ (cf. le théorème \ref{Thm: base PBW}). Il s'agit de trouver l'élément $\mathbf{b} \in \Uq$ tel que, pour tous $m,n \in \{0,...,p-1\}$ et $j \in \{0,...,2p-1\}$, on ait :
	\[ (id \otimes \bmu) \Delta(F^m K^j E^n) = \bmu(F^m K^j E^n) \mathbf{b} . \]
	Or, pour tous $m,n \in \{0,...,p-1\}$ et $j \in \{0,...,2p-1\}$, on a :
	\begin{align*}
	\Delta( F^m K^j E^n )
	&\; \;= \Delta(F)^m \Delta(K)^j \Delta(E)^n \\
	&\overset{\eqref{Eqn: coprod2}}{=} \sum_{r=0}^m \sum_{s=0}^n q^{r(m-r)+s(n-s)} {m \brack r} {n \brack s} F^r K^{r-m+j} E^{n-s} \otimes F^{m-r} K^j E^s K^{n-s} \\
	&\overset{\eqref{Eqn: comm1}}{=} \sum_{r=0}^m \sum_{s=0}^n q^{r(m-r)+s(n-s)} {m \brack r} {n \brack s} F^r K^{r-m+j} E^{n-s} \otimes F^{m-r} K^{j+n-s} E^s.
	\end{align*}
	Donc, pour tous $m,n \in \{0,...,p-1\}$ et $j \in \{0,...,2p-1\}$, on a :
	\begin{align*}
	(id \otimes \bmu) \Delta( F^m K^j E^n )
	&= \begin{multlined}[t]
		\sum_{r=0}^m \sum_{s=0}^n q^{r(m-r)+s(n-s)} {m \brack r} {n \brack s} \\
		F^r K^{r-m+j} E^{n-s} \bmu \left( F^{m-r} K^{j+n-s} E^s \right) 
		\end{multlined} \\
	&= \delta_{m,p-1} \delta_{n,p-1} \delta_{j,p+1} K^2 = \bmu ( F^m K^j E^n ) K^2.
	\end{align*}
	D'où le résultat.
	\item Cela découle de l'unimodularité. 
\end{enumerate}
\end{proof}

\begin{Def}
Soit $A$ une $\C$-algèbre de Hopf unimodulaire de dimension finie, et de comodule à droite $\mathbf{b}$. On dit que $\mathbf{k} \in A$ est un \emph{élément de balancement généralisé} si :
\[ \mathbf{k}^2 = \mathbf{b}, 
\quad \Delta( \mathbf{k} ) = \mathbf{k} \otimes \mathbf{k},
\quad \text{ et } \quad 
\forall x \in A \quad S^2(x) = \mathbf{k} x \mathbf{k}^{-1}. \]
\end{Def}

\begin{Rem}
\label{Rem: balancement}
Cette définition étend la définition d'élément de balancement aux $\C$-algèbres de Hopf $A$ unimodulaires de dimension finie. En effet, soit $A$ une telle algèbre de Hopf. Notons $\mathbf{a}$ (resp. $\mathbf{b}$) son comodule à gauche (resp. à droite) et $\boldsymbol{\alpha}$ son module à gauche. D'après \cite[Prop. 6.1]{Dri90}, dont le principal résultat est repris dans \cite[Prop. 10.1.14]{Mon93}, on a :
\[ \mathbf{g} = \mathbf{a}^{-1} \; (\boldsymbol{\alpha} \otimes id) \mathbf{R} = (\boldsymbol{\alpha} \otimes id) \mathbf{R} \; \mathbf{a}^{-1} . \]
où $\mathbf{a}^{-1}=\mathbf{b}$ d'après la remarque \ref{Rem: modules}. En particulier, si $A$ est unimodulaire, alors $\boldsymbol{\alpha}=\varepsilon$. Donc :
\[ \mathbf{g} = \mathbf{a}^{-1} = \mathbf{b} \]
car $(\varepsilon \otimes id) R = 1$ (cf. par exemple \cite[Thm VIII.2.4]{Kas95} ou \cite[Prop. 10.1.8]{Mon93}).
\end{Rem}

\begin{Cor} 
\label{Cor: balancement}
L'algèbre de Hopf $\Uq$ possède deux éléments de balancement généralisés $K^{\delta p +1}$, $\delta \in \{0,1\}$.
\end{Cor}

\begin{Rem}
L'élément de balancement (généralisé) est unique à multiplication près par un élément du centre d'ordre deux. En effet, soient $\mathbf{k}$ et $\mathbf{k}'$ deux éléments de balancement (généralisés) d'une algèbre de Hopf $A$ de dimension finie. Alors, on a :
\[ \mathbf{k}' = S^2(\mathbf{k}') = \mathbf{k} \mathbf{k'} \mathbf{k}^{-1}, \quad \text{ où } \quad \mathbf{k'} \mathbf{k}^{-1} = \mathbf{k}^{-1} S^2(\mathbf{k}') = \mathbf{k}^{-1} \mathbf{k}'. \]
Or, l'élément $\mathbf{k'} \mathbf{k}^{-1}$ vérifie :
\begin{gather*}
\forall x \in A \quad \mathbf{k'} \mathbf{k}^{-1} x = \mathbf{k}^{-1} \mathbf{k'} x = \mathbf{k}^{-1} S^2(x) \mathbf{k'} = x \mathbf{k}^{-1} \mathbf{k'} = x \mathbf{k'} \mathbf{k}^{-1} , \\
(\mathbf{k'} \mathbf{k}^{-1})^2 = \mathbf{k'} \mathbf{k}^{-1} \mathbf{k'} \mathbf{k}^{-1} = \mathbf{k'}^2 \mathbf{k}^{-2} = \mathbf{g} \mathbf{g}^{-1} = 1.
\end{gather*}
D'où le résultat.
\end{Rem}

\subsection{Double et enrubannement}
\label{subsection: enrubannement}

Grâce au double de Drinfeld (cf. par exemple \cite[§ IX.4]{Kas95}), il est possible de construire deux algèbres de Hopf tressée et enrubannée $(\bar{D}, \bar{\mathbf{R}}, \bar{\mathbf{v}}_\delta)$, $\delta \in \{0,1\}$, et des plongements canonique $\Uq \rightarrow (\bar{D}, \bar{\mathbf{R}}, \bar{\mathbf{v}}_\delta)$. Cette construction est détaillée dans l'article \cite[§ B]{FGST06b}. Les deux éléments d'enrubannement $\bar{\mathbf{v}}_0, \bar{\mathbf{v}}_1$ possibles de $\bar{D}$ appartiennent à la sous-algèbre de $\bar{D}$ canoniquement isomorphe à $\Uq$. On étend ainsi la définition d'élément d'enrubannement à $\Uq$.

\begin{Def}[{\cite[§ 4.1, § B.1]{FGST06b}}]
On note $\bar{D}$ la $\C$-algèbre engendrée par $e, f, k, k^{-1}$ sous les relations :
\begin{equation} 
\label{Eqn: commD}
\begin{gathered}
k e k^{-1} =q e,
	\qquad k f k^{-1} =q^{-1} f ,
	\qquad [e,f] = \frac{k^2-k^{-2}}{q-q^{-1}},  \\ 
e^p = 0,
	\qquad f^p = 0,
	\qquad k^{4p}=1.
\end{gathered}
\end{equation}
Elle est munie d'une structure d'algèbre de Hopf, dont le coproduit $\Delta$, la co-unité $\varepsilon$, et l'antipode $S$ sont donnés par :
\begin{equation} 
\label{Eqn: coprodD}
\begin{gathered} 
\Delta(e) = 1 \otimes e + e \otimes k^2,
	\qquad \Delta(f) = k^{-2} \otimes f + f \otimes 1,
	\qquad \Delta(k) = k \otimes k, \\	
\varepsilon(e) = 0, 
	\qquad \varepsilon(f) = 0, 
	\qquad \varepsilon(k) = 1,  \\
S(e) = -ek^{-2},
	\qquad S(f) = -k^2f,
	\qquad S(k)=k^{-1}.
\end{gathered}
\end{equation}
\end{Def}

Le groupe quantique $\Uq$ s'identifie à une sous-algèbre de Hopf de $\bar{D}$ via le morphisme injectif d'algèbres de Hopf :
\begin{equation} 
\label{Eqn: inj double}
\begin{cases}
	\Uq \longrightarrow \bar{D} \\
	F \mapsto f, \quad K \mapsto k^2, \quad E \mapsto e.
\end{cases}
\end{equation}

\begin{Thm}[{\cite[Thm 4.1.1]{FGST06b}}]
\label{Thm: double}
L'algèbre de Hopf $\bar{D}$ est tressée et enrubannée. Plus précisément :
\begin{enumerate}[(i)]
	\item il existe une $R$-matrice :
\[ \bar{\mathbf{R}} = \frac{1}{4p} \sum_{m=0}^{p-1} \sum_{n,j=0}^{4p-1} \frac{(q-q^{-1})^m}{[m]!} q^{\frac{m(m-1)}{2}-mj-\frac{nj}{2}} k^n e^m \otimes f^m k^j , \]
	\item il existe deux éléments d'enrubannement :
	\[ \bar{\mathbf{v}}_\delta = \frac{1-i}{2 \sqrt{p}} \sum_{m=0}^{p-1} \sum_{j=0}^{2p-1} \frac{(q-q^{-1})^m}{[m]!} q^{-\frac{m}{2}-mj+\frac{(j+p+1)^2}{2}} f^m k^{2j+2(1-\delta)p} e^m \; ; \quad \delta \in \{0,1\} . \]
\end{enumerate}
\end{Thm}

Le choix de $\delta$ revient à fixer $k^{2(\delta p+1)}$ comme élément de balancement pour $\bar{D}$. Soit $\delta \in \{0,1\}$. On observe que $\bar{\mathbf{v}}_{\delta}$ appartient à la sous algèbre de $\bar{D}$ isomorphe à $\Uq$ \eqref{Eqn: inj double}. Par conséquent, on pose : \index{v@ $\mathbf{v}_{\delta}$}
\begin{equation} 
\label{Eqn: enrubannement}
\mathbf{v}_{\delta} := \frac{1-i}{2 \sqrt{p}} \sum_{m=0}^{p-1} \sum_{j=0}^{2p-1} \frac{(q-q^{-1})^m}{[m]!} q^{-\frac{m}{2}-mj+\frac{(j+p+1)^2}{2}} F^m K^{j+(1-\delta)p} E^m.
\end{equation}
Là encore, le choix de $\delta$ revient à fixer $K^{\delta p+1}$ comme élément de balancement généralisé pour $\Uq$ (cf. le corollaire \ref{Cor: balancement}).

\begin{Cor} 
On a :
\[ \bar{\mathbf{R}}_{21} \bar{\mathbf{R}} = \frac{1}{2p} \sum_{m,n=0}^{p-1} \sum_{i,j=0}^{2p-1} \frac{(q-q^{-1})^{m+n}}{[m]![n]!} q^{\frac{m(m-1)}{2}+\frac{n(n-1)}{2}-m^2-mj-ij+mi} f^m k^{2j} e^n \otimes e^m k^{2i} f^n. \]
\end{Cor}

De même que précédemment, l'élément $\bar{\mathbf{M}}:=\bar{\mathbf{R}}_{21} \bar{\mathbf{R}}$ appartient à la sous algèbre de $\bar{D} \otimes \bar{D}$ isomorphe à $\Uq \otimes \Uq$ \eqref{Eqn: inj double}. Par conséquent, on pose : \index{M@ $\mathbf{M}$}
\begin{equation} 
\label{Eqn: M-matrice}
\mathbf{M} := \begin{multlined}[t]
	\frac{1}{2p} \sum_{m,n=0}^{p-1} \sum_{i,j=0}^{2p-1} \frac{(q-q^{-1})^{m+n}}{[m]![n]!} q^{\frac{m(m-1)}{2}+\frac{n(n-1)}{2}-m^2-mj-ij+mi} \\
	F^m K^j E^n \otimes E^m K^i F^n .
	\end{multlined}
\end{equation}
Ainsi, quelque soit le choix de $\delta$, l'élément inversible $\mathbf{v}_{\delta} \in \Uq$ \eqref{Eqn: enrubannement} vérifie :
\[ \forall x \in A \quad x \mathbf{v}_{\delta} = \mathbf{v}_{\delta} x, 
\quad \Delta(\mathbf{v}_{\delta}) = \mathbf{M}^{-1} \mathbf{v}_{\delta} \otimes \mathbf{v}_{\delta}, 
\quad \text{ et } \quad S(\mathbf{v}_{\delta}) = \mathbf{v}_{\delta}. \]
Cela généralise la définition d'élément d'enrubannement à $\Uq$.

\begin{Prop} 
\label{Prop: enrubannement}
Soit $\delta \in \{0,1\}$. L'expression de $\mathbf{v}_{\delta}$ \eqref{Eqn: enrubannement} dans la base canonique du centre (cf. la proposition \ref{Prop: centre}) est :
\[ \mathbf{v}_{\delta} = \sum_{s=0}^p (-1)^{\delta(s-1)} q^{-\frac{s^2-1}{2}} e_s + (q-q^{-1}) \sum_{s=1}^{p-1} (-1)^{\delta(s-1)} q^{-\frac{s^2-1}{2}} \left( \frac{p-s}{[s]} w^+_s - \frac{s}{[s]} w^-_s \right) .\]
\end{Prop}

\begin{proof}
Conformément au corollaire \ref{Cor: centre}, on étudie l'action de $\mathbf{v}_\delta$ sur $\X^\pm(s)$, $1 \leq s \leq p$, puis sur $\PIM^\pm(s)$, $1 \leq s \leq p-1$.

Soient $\alpha \in \{+,-\}$, $s \in \{1,...,p\}$, et $x_0^\alpha(s)$ le vecteur de plus haut poids de $\X^\alpha(s)$ (cf. la proposition \ref{Prop: simples}). Pour tous $m \in \{0,...,p-1\}$ et $j \in \{0,...,2p-1\}$, on a :
\begin{align*}
F^m K^{j+(1-\delta)p} E^m x_0^\alpha(s) &= \delta_{m,0} K^{j+(1-\delta)p} x_0^\alpha(s) = \delta_{m,0} \alpha^{j+p(1-\delta)} q^{(s-1)(j+p-\delta p)} x_0^\alpha(s) \\
&= \delta_{m,0} (-1)^{(1-\delta)(s-1)} \alpha^{j+p(1-\delta)} q^{j(s-1)} x_0^\alpha(s).
\end{align*}
Donc :
\begin{align*}
\mathbf{v}_\delta x_0^+(s) &= \frac{1-i}{2 \sqrt{p}}  (-1)^{(1-\delta)(s-1)} \underbrace{\sum_{j=0}^{2p-1} q^{\frac{(j+p+1)^2}{2}} q^{j(s-1)}}_{A_s} x_0^+(s), \\
\mathbf{v}_\delta x_0^-(s) &= \frac{1-i}{2 \sqrt{p}}  (-1)^{(1-\delta)(p+s-1)} \underbrace{\sum_{j=0}^{2p-1} q^{\frac{(j+p+1)^2}{2}} q^{j(p+s-1)}}_{A_{p+s}} x_0^-(s).
\end{align*}
La coordonnée de $\mathbf{v}_\delta$ suivant $e_s$ est donnée par :
\begin{equation} \tag{1}
\begin{aligned}
a_s &= \frac{1-i}{2 \sqrt{p}}  (-1)^{(1-\delta)(s-1)} A_s, && \text{ si } s \not = 0, \\
a_s &= \frac{1-i}{2 \sqrt{p}}  (-1)^{(1-\delta)(s-1)} A_{2p-s}, && \text{ si } s \not = p.
\end{aligned}
\end{equation}
Soit $n \in \N$. Le coefficient $A_n$ vérifie  :
\[ A_n := \sum_{j=0}^{2p-1} q^{\frac{(j+p+1)^2}{2}} q^{j(n-1)} = - q^{\frac{p^2+1}{2}} \sum_{j=0}^{2p-1} q^{\frac{j(j+2p+2n)}{2}}. \]
Le dernier terme est une somme de Gauss (cf. par exemple \cite[§ IV.3]{Lan94}), dont l'indice $j$ ne dépend pas du représentant choisi dans $\Z / 2p \Z$. Pour l'évaluer, on effectue le changement de variables $j \mapsto j-p-n$ :
\begin{equation} \tag{2}
\begin{aligned}
\sum_{j=0}^{2p-1} q^{\frac{j(j+2p+2n)}{2}} &= \sum_{j=0}^{2p-1} q^{\frac{(j-p-n)(j+p+n)}{2}} 
	= q^{-\frac{(p+n)^2}{2}} \sum_{j=0}^{2p-1} q^{\frac{j^2}{2}}
	= (-1)^n q^{-\frac{p^2+n^2}{2}} (1+i) \sqrt{p}.
\end{aligned}
\end{equation}
Il s'ensuit que :
\[ A_s = (-1)^{s-1} q^{-\frac{s^2-1}{2}} (1+i) \sqrt{p} = A_{2p-s}. \]
En injectant ce résultat dans les équations $(1)$, on obtient :
\[ a_s = (-1)^{\delta(s-1)} q^{-\frac{{s}^2-1}{2}}. \]

Soient $\alpha \in \{+,-\}$, $s \in \{1,...,p-1\}$, et $y_0^\alpha(s)$ le vecteur de poids $\alpha q^{s-1}$ qui engendre $\PIM^\alpha(s)$ sous $\Uq$ (cf. la proposition \ref{Prop: PIMs}). Pour tous $m \in \{0,...,p-1\}$ et $j \in \{0,...,2p-1\}$, on a :
\begin{align*}
F^m K^{j+(1-\delta)p} E^m y_0^\alpha(s) &= \begin{multlined}[t]
		\delta_{m,0} K^{j+(1-\delta)p} y_0^\alpha(s) \\
		+ \delta_{m \geq 1} F^m K^{j+(1-\delta)p} E^{m-1} a_{p-s-1}^{-\alpha}(p-s) 
		\end{multlined} \\
	&= \begin{multlined}[t]
		\delta_{m,0} \alpha^{j+(1-\delta)p} q^{(s-1)(j+p-\delta p)} y_0^\alpha(s) 
		+ \delta_{1 \leq m \leq p-s} (-\alpha)^{m-1} \\
		\times \prod_{i=1}^{m-1} [p-s-i] [i] F^m K^{j+(1-\delta)p} a_{p-s-m}^{-\alpha}(p-s) 
		\end{multlined} \\
	&= \begin{multlined}[t]
		\delta_{m,0} (-1)^{(1-\delta)(s-1)} \alpha^{j+(1-\delta)p} q^{j(s-1)} y_0^\alpha(s) \\
		+ \delta_{1 \leq m \leq p-s} (-\alpha)^{m-1+j+(1-\delta)p} q^{(-p+s-1+2m)(j+p-\delta p)} \\
		\times \prod_{i=1}^{m-1} [s+i] [i] F^m a_{p-s-m}^{-\alpha}(p-s) 
		\end{multlined} \\
	&= \begin{multlined}[t]
		\delta_{m,0} (-1)^{(1-\delta)(s-1)} \alpha^{j+(1-\delta)p} q^{j(s-1)} y_0^\alpha(s) \\
		+ \delta_{1 \leq m \leq p-s} (-1)^{m-1+(1-\delta)(s-1)} \alpha^{m-1+j+(1-\delta)p} q^{j(s-1+2m)} \\
		\times \frac{[m-1]![s+m-1]!}{[s]!} x_0^{\alpha}(s).
		\end{multlined}
\end{align*}
Donc :
\begin{gather*}
\mathbf{v}_\delta y_0^+(s) - a_s y_0^+(s) = \begin{multlined}[t]
	\frac{1-i}{2 \sqrt{p}} (-1)^{(1-\delta)(s-1)} \sum_{m=1}^{p-s} (-1)^{m-1} q^{-\frac{m}{2}} \frac{(q-q^{-1})^m [s+m-1]!}{[m] [s]!}  \\
	\times \underbrace{\sum_{j=0}^{2p-1} q^{-mj+\frac{(j+p+1)^2}{2}} q^{j(s-1+2m)}}_{B_{m,s}} x_0^+(s), 
	\end{multlined} \\
\mathbf{v}_\delta y_0^-(s) - a_{p-s} y_0^-(s) = \begin{multlined}[t]
	\frac{1-i}{2 \sqrt{p}} (-1)^{(1-\delta)(p+s-1)} \sum_{m=1}^{p-s} q^{-\frac{m}{2}} \frac{(q-q^{-1})^m [s+m-1]!}{[m] [s]!} \\
	\times \underbrace{\sum_{j=0}^{2p-1} q^{-mj+\frac{(j+p+1)^2}{2}} q^{j(p+s-1+2m)}}_{B_{m,p+s}} x_0^-(s).
	\end{multlined}
\end{gather*}
Les coordonnées de $\mathbf{v}_\delta$ suivant $w^+_s$ et $w^-_s$ sont respectivement données par :
\begin{equation} \tag{3}
\begin{aligned}
&b_s^+ = \frac{1-i}{2 \sqrt{p}} (-1)^{(1-\delta)(s-1)} \sum_{m=1}^{p-s} (-1)^{m-1} q^{-\frac{m}{2}} \frac{(q-q^{-1})^m [s+m-1]!}{[m] [s]!} B_{m,s}, \\
&b_s^- = \frac{1-i}{2 \sqrt{p}} (-1)^{(1-\delta)(s-1)} \sum_{m=1}^{s} q^{-\frac{m}{2}} \frac{(q-q^{-1})^m [p-s+m-1]!}{[m] [p-s]!} B_{m,2p-s}.
\end{aligned}
\end{equation}
Soient $n \in \N$ et $m \in \{1,...,p-1\}$. Le coefficient $B_{m,n}$ vérifie :
\[ B_{m,n} := \sum_{j=0}^{2p-1} q^{-mj+\frac{(j+p+1)^2}{2}} q^{j(n-1+2m)} = - q^{\frac{p^2+1}{2}} \sum_{j=0}^{2p-1} q^{\frac{j(j+2p+2n+2m)}{2}} . \]
En reprenant les calculs de l'équation $(1)$, la somme de Gauss s'évalue comme :
\[ \sum_{j=0}^{2p-1} q^{\frac{j(j+2p+2n+2m)}{2}} = (-1)^{n+m} q^{-\frac{p^2+(n+m)^2}{2}} (1+i) \sqrt{p} . \]
Il s'ensuit que :
\begin{align*}
B_{m,s} &= (-1)^{s+m-1} q^{-\frac{s^2+m^2-1}{2}-ms} (1+i) \sqrt{p}, \\
B_{m,2p-s} &= (-1)^{s+m-1} q^{-\frac{s^2+m^2-1}{2}+ms} (1+i) \sqrt{p}.
\end{align*}
En injectant ce résultat dans les équations $(3)$, on obtient :
\begin{align*}
b_s^+ &= (-1)^{\delta(s-1)+1} q^{-\frac{s^2-1}{2}} \sum_{m=1}^{p-s} q^{-\frac{m(m+2s+1)}{2}} \frac{(q-q^{-1})^m [s+m-1]!}{[m] [s]!} \\
&= (-1)^{\delta(s-1)+1} q^{-\frac{s^2-1}{2}} \sum_{m=1}^{p-s} q^{-\frac{m(m+2s+1)}{2}} \frac{(q-q^{-1})^2}{q^m-q^{-m}} \prod_{i=1}^{m-1} (q^{s+i}-q^{-s-i}) \\
&= (-1)^{\delta(s-1)} q^{-\frac{s^2-1}{2}} q^{-s} (q-q^{-1})^2 \sum_{m=1}^{p-s} \frac{1}{1-q^{2m}} \prod_{i=1}^{m-1} (1-q^{2p-2s-2i}), \\
b_s^- &= (-1)^{\delta(s-1)+m} q^{-\frac{s^2-1}{2}} \sum_{m=1}^{s} q^{-\frac{m(m-2s+1)}{2}} \frac{(q-q^{-1})^m [p-s+m-1]!}{[m] [p-s]!} \\
&= (-1)^{\delta(s-1)+m} q^{-\frac{s^2-1}{2}} \sum_{m=1}^{s} q^{-\frac{m(m-2s+1)}{2}} \frac{(q-q^{-1})^2}{q^m-q^{-m}} \prod_{i=1}^{m-1} (q^{p-s+i}-q^{p+s-i}) \\
&= (-1)^{\delta(s-1)} q^{-\frac{s^2-1}{2}} q^s (q-q^{-1})^2 \sum_{m=1}^{s} \frac{1}{1-q^{2m}} \prod_{i=1}^{m-1} (1-q^{2s-2i}),
\end{align*}
Le dernier terme est une identité de partition (cf. par exemple \cite{AE04}) :
\[ \forall d \in \N \quad \sum_{m=1}^{d} \frac{1}{1-q^{2m}} \prod_{i=1}^{m-1} (1-q^{2d-2i}) = \frac{d}{1-q^{2d}} . \]

Par conséquent :
\begin{align*}
b_s^+ &= (-1)^{\delta(s-1)} q^{-\frac{s^2-1}{2}} (q-q^{-1})^2 \frac{(p-s)}{q^s-q^{-s}} = (-1)^{\delta(s-1)} q^{-\frac{s^2-1}{2}} (q-q^{-1}) \frac{(p-s)}{[s]}, \\
b_s^- &= (-1)^{\delta(s-1)} q^{-\frac{s^2-1}{2}} (q-q^{-1})^2 \frac{s}{q^{-s}-q^s} = -(-1)^{\delta(s-1)} q^{-\frac{s^2-1}{2}} (q-q^{-1}) \frac{s}{[s]}.
\end{align*}
D'où le résultat.
\end{proof}

\subsection{Actions (co-)adjointes et caractères}
\label{subsection: caracteres}

On munit maintenant $\Uq$ et $\Uq^*$ d'une structure de $\Uq$-modules à gauche grâce aux actions adjointe et co-adjointe (cf. par exemple \cite[§ IX.3]{Kas95}). Le centre $\Zf$ et les espaces des $q$-caractères de $\Uq$ sont respectivement des stabilisateurs sous l'action adjointe et co-adjointe. On introduit enfin l'anneau de Grothendieck $\G$ de $\Uq$ (cf. par exemple \cite[Def. 16.5]{CR81}) pour construire des caractères remarquables à partir des éléments de balancement généralisés.

Pour toute algèbre de Hopf $A$, on utilisera les notations de Sweedler (cf. par exemple \cite[Not. III.1.6]{Kas95}) :
\[ \forall a \in A \quad \Delta(a) = \sum_{(a)} a' \otimes a'' . \]

\begin{Def}[{\cite[§ A.1]{FGST06b}}]
Soit $A$ une $\C$-algèbre de Hopf de dimension finie. 
\begin{enumerate}[(i)]
	\item On appelle \emph{action adjointe} l'action à gauche (resp. à droite) de $A$ sur $A$ définie par :
\[ \forall a, y \in A \quad a \cdot y := \sum_{(a)} a' y S(a'') \quad \left( \text{resp. } y \cdot a := \sum_{(a)} S(a') y a'' \right). \]
	\item On appelle \emph{action co-adjointe} l'action à gauche (resp. à droite) de $A$ sur $A^*$ définie par :
\[ \begin{cases} \forall a \in A \\ \forall \beta \in A^* \end{cases} \quad a \cdot \beta := \beta \left( \sum_{(a)} S(a') ? a'' \right) \quad \left( \text{resp. } \beta \cdot a := \beta \left( \sum_{(a)} a' ? S(a'') \right) \right), \]
où le symbol $?$ indique la place de la variable.
\end{enumerate}
\end{Def}

\begin{Rem}
Soient $A$ une $\C$-algèbre de Hopf de dimension finie et $\X$ un $A$-module à gauche (resp. à droite). Il est possible de munir $\X^*$ d'une structure de $A$-module à droite (resp. à gauche) ou de $A$-module à gauche (resp. à droite) avec :
\begin{align*}
& \forall a \in A \quad \forall \beta \in \X^* \quad \beta \overset{1}{\cdot} a  = \beta(a \cdot ?) \quad \left( \text{resp. } a \overset{1}{\cdot} \beta = \beta(? \cdot a) \right), \\
& \forall a \in A \quad \forall \beta \in \X^* \quad a \overset{2}{\cdot} \beta  = \beta(S^{-1}(a) \cdot ?) \quad \left( \text{resp. } \beta \overset{2}{\cdot} a = \beta(? \cdot S^{-1}(a)) \right),
\end{align*}
où le symbol $?$ indique la place de la variable. 

Dans notre cas, on travaille implicitement avec la première structure de sorte que les actions adjointe et co-adjointe se transposent. Autrement dit :
\[ \forall a, y \in A \quad \forall \beta \in A^* \quad (a \cdot \beta)(y) = \beta(y \cdot a) \quad \text{ et } \quad (\beta \cdot a)(y) = \beta(a \cdot y) .\]

Dans les ouvrages, on utilise davantage la seconde structure en vue de construire des produits bicroisés (cf. par exemple \cite[§ IX.2]{Kas95}). Dans ce cas, on définit l'action co-adjointe à gauche (resp. à droite) de $A$ sur $A^*$ par :
\[ \begin{cases} \forall a \in A \\ \forall \beta \in A^* \end{cases} \quad a \cdot \beta := \beta \left( \sum_{(a)} S^{-1}(a'') ? a' \right) \quad \left( \text{resp. } \beta \cdot a := \beta \left( \sum_{(a)} a'' ? S^{-1}(a') \right) \right), \]
où le symbol $?$ indique la place de la variable. La nomenclature se justifie encore du fait que les actions adjointe et co-adjointe se transposent :
\[ \forall a, y \in A \quad \forall \beta \in A^* \quad (a \cdot \beta)(y) = \beta( S^{-1}(a) \cdot y) \quad \text{ et } \quad (\beta \cdot a)(y) = \beta(y \cdot S^{-1}(a)) .\]
\end{Rem}

\begin{Def}[{\cite[§ A.1]{FGST06b}}]
Soit $A$ une $\C$-algèbre de Hopf de dimension finie.
On appelle \emph{$q$-caractère} à gauche (resp. à droite) de $A$ toute forme linéaire $\beta \in A^*$ telle que :
\[ \forall x,y \in A \quad \beta(xy) = \beta \left( S^2(y)x \right) \quad \left( \text{resp. } \beta(xy) = \beta \left( y S^2(x) \right) \right). \]
\end{Def}

Le terme de \emph{caractère}, introduit par Drinfeld dans \cite{Dri90}, prendra sens dans la proposition \ref{Prop: qch}.

\begin{Prop} 
\label{Prop: stabilisateurs}
Soit $A$ une $\C$-algèbre de Hopf de dimension finie. On note $\Zf(A)$ le centre de $A$ et $\Ch^l(A)$ (resp. $\Ch^r(A)$ l'espace des $q$-caractères à gauche (resp. à droite) de $A$. On a :
\begin{enumerate}[(a)]
	\item $\mathfrak{Z}(A) = \left\{ y \in A \; ; \; \forall a \in A \quad a \cdot y = \varepsilon(a) y \right\} = \left\{ y \in A \; ; \; \forall a \in A \quad y \cdot a = \varepsilon(a) y \right\}$,
	\item $\Ch^l(A) = \left\{ \beta \in A^* \; ; \; \forall a \in A \quad a \cdot \beta = \varepsilon(a) \beta \right\}$,
	\item $\Ch^r(A) = \left\{ \beta \in A^* \; ; \; \forall a \in A \quad \beta \cdot a = \varepsilon(a) \beta \right\}$.
\end{enumerate}
\end{Prop}

\begin{proof}
\begin{enumerate}[(a)]
	\item Il s'agit de montrer que :
\[ y \in \mathfrak{Z}(A) \quad \Longleftrightarrow \quad \forall a \in A \quad a \cdot y = \varepsilon(a) y \quad \Longleftrightarrow \quad \forall a \in A \quad y \cdot a = \varepsilon(a) y. \]
Les sens directs sont évidents car, par définition de l'antipode (cf. par exemple \cite[Def. III.3.2]{Kas95}), on a :
\[ \forall a \in A \quad \sum_{(a)} a' S(a'') = \varepsilon(a) 1 = \sum_{(a)} S(a') a''. \]
Réciproquement, soient $y, z \in A$ tels que, pour tout $a \in A$, $a \cdot y = \varepsilon(a) y$ et $z \cdot a = \varepsilon(a) z$. Alors :
\begin{align*}
\forall x \in A \quad xy &= \sum_{(x)} x' \varepsilon(x'') y 
= \sum_{(x)} x' y \varepsilon(x'')
= \sum_{(x)} x' y S(x'') x''' \\
&= \sum_{(x)} (x' \cdot y) x''
= \sum_{(x)} \varepsilon(x') y x''
= \sum_{(x)} y \varepsilon(x') x''
= yx, \\
\forall x \in A \quad zx &= \sum_{(x)} z \varepsilon(x') x''
= \sum_{(x)}  \varepsilon(x') z x''
= \sum_{(x)} x' S(x'') z x''' \\
&= \sum_{(x)} x' (z \cdot x'')
= \sum_{(x)} x' \varepsilon(x'') z
= xz.
\end{align*}
Donc $y,z \in \mathfrak{Z}(A)$.
	\item De même, il s'agit de montrer que :
\[ \beta \in \Ch^l(A) \quad \Longleftrightarrow \quad \forall a \in A \quad a \cdot \beta = \varepsilon(a) \beta. \]
Le sens direct est évident car :
\[ \forall a \in A \quad \sum_{(a)} S^2(a'') S(a') = S \left( \sum_{(a)} a' S(a'') \right) = S(\varepsilon(a) 1) = \varepsilon(a) S(1) = \varepsilon(a) 1. \]
Réciproquement, soit $\beta \in A^*$ tel que, pour tout $a \in A$, $a \cdot \beta = \varepsilon(a) \beta$. Alors :
\begin{align*}
\forall x,y \in A \quad \beta(S^2(y)x) &= \sum_{(y)} \beta \left( S^2(\varepsilon(y')y'') x \right)
= \sum_{(y)} \beta \left( S^2(y'') x  \varepsilon(y') \right) \\
&= \sum_{(y)} \beta \left( S^2(y''') x  y' S(y'') \right) \\
&= \sum_{(S(y))} \beta \left( S(S(y)') x  S^{-1}(S(y)''') S(y)'' \right) \\
&= \sum_{(S(y))} (S(y)' \cdot \beta) \left( x  S^{-1}(S(y)'') \right) \\
&= \sum_{(S(y))} \varepsilon(S(y)') \beta \left( x  S^{-1}(S(y)'') \right) \\
&= \sum_{(S(y))} \beta \left( x S^{-1} \left( \varepsilon(S(y)') S(y)'' \right) \right) \\
&= \beta \left( x  S^{-1}(S(y)) \right)
= \beta \left( x y \right).
\end{align*}
Donc $\beta \in \Ch^l(A)$.
	\item De même, il s'agit de montrer que :
\[ \beta \in \Ch^r(A) \quad \Longleftrightarrow \quad \forall a \in A \quad \beta \cdot a = \varepsilon(a) \beta. \]
Le sens direct est évident car :
\[ \forall a \in A \quad \sum_{(a)} S(a'') S^2(a') = S \left( \sum_{(a)} S(a') a'' \right) = S(\varepsilon(a) 1) = \varepsilon(a) S(1) = \varepsilon(a) 1. \]
Réciproquement, soit $\beta \in A^*$ tel que, pour tout $a \in A$, $\beta \cdot a= \varepsilon(a) \beta$. Alors :
\begin{align*}
\forall x,y \in A \quad \beta(y S^2(x)) &= \sum_{(x)} \beta \left( y S^2(x'\varepsilon(x''))  \right)
= \sum_{(x)} \beta \left( \varepsilon(x'') y S^2(x') \right) \\
&= \sum_{(x)} \beta \left( S(x'') x'''  y S^2(x') \right) \\
&= \sum_{(S(x))} \beta \left( S(x)'' S^{-1}(S(x)') y S(S(x)''') \right) \\
&= \sum_{(S(x))} (\beta \cdot S(x)'') \left( S^{-1}(S(x)') y \right) \\
&= \sum_{(S(x))} \varepsilon(S(x)'') \beta \left( S^{-1}(S(x)') y \right) \\
&= \sum_{(S(x))} \beta \left( S^{-1} \left( S(x)' \varepsilon(S(x)'') \right) y \right) \\
&= \beta \left( S^{-1}(S(x)) y \right)
= \beta \left( xy \right).
\end{align*}
Donc $\beta \in \Ch^l(A)$.
\end{enumerate}
\end{proof}

On introduit maintenant des $q$-caractères particuliers. Pour cela on rappelle que, pour toute algèbre de Hopf $A$, le produit tensoriel $\X \otimes \Y$ de deux $A$-modules $\X, \Y$ à gauche (resp. à droite) est muni d'une structure de $A$-module à gauche (resp. à droite) donnée par :
\[ a (x \otimes y) = \Delta(a) (x \otimes y) \quad \left( \text{resp. } (x \otimes y)a = (x \otimes y) \Delta(a) \right), \]
pour tous $a \in A$, $x \in \X$ et $y \in \Y$. De plus, si $A$ est de dimension finie, alors toute co-intégrale à gauche (resp. à droite) $\mathbf{c}$ de $A$ fournit une unité $A \mathbf{c} \simeq \C$ (resp. $\mathbf{c} A \simeq \C$) pour ce produit tensoriel. Ainsi, le groupe de Grothendieck à gauche (resp. à droite) $\G(A)$ de $A$ est muni d'une structure d'anneau (cf. par exemple \cite[§ 16.B]{CR81}). Les prochains résultats sont valables pour des modules à gauche ou des modules à droite. C'est pourquoi on ne précisera pas la latéralité des modules considérés. 

\begin{Lemme}
\label{Lemme: qch}
Soient $A$ une $\C$-algèbre de Hopf de dimension finie, et $s \in A$. Pour tout $A$-module $\X$, associé à une représentation $\rho : A \rightarrow \End(\X)$, on définit la forme linéaire :
\[ \qch^s_\X :
	\begin{cases}
	A \longrightarrow \C \\
	a \longmapsto \tr_\X \left( s a \right) := \tr \left( \rho(s a) \right)
	\end{cases} . \]
On note $\G(A)$ l'anneau de Grothendieck de $A$. Soient $\X$ et $\Y$ deux $A$-modules.
\begin{enumerate}[(i)]
	\item Si $\X$ et $\Y$ ont la même classe dans $\G(A)$, alors $\qch^s_\X = \qch^s_\Y$.
	\item Si $s$ est inversible et $\qch^s_\X = \qch^s_\Y$, alors $\X$ et $\Y$ ont la même classe dans $\G(A)$.
\end{enumerate} 
\end{Lemme}

\begin{proof}
On note $\X_1, \X_2, ..., \X_k$ et $\Y_1, \Y_2, ..., \Y_l$ les facteurs de compositions respectifs de $\X$ et $\Y$ (cf. par exemple \cite[§ 2.6]{Pie82}). On choisit une suite de composition de $\X$ (resp. de $\Y$) et une base de $\X$ (resp. de $\Y$) adaptée à cette suite de composition. Dans celle-ci, la matrice de représentation $M_\X(b)$ (resp. $M_\Y(b)$) de $b \in A$ dans $\X$ (resp. $\Y$) est triangulaire supérieure par bloc, de la forme :
\begin{gather*}
M_\X(b) = \begin{pmatrix}
	M_{\X_1}(b) &&& (*) \\
	& M_{\X_2}(b) && \\
	&& \ddots &  \\
	(0) &&&  M_{\X_k}(b)
	\end{pmatrix} \\
\left( \text{resp. }
	M_\Y(b) = \begin{pmatrix}
	M_{\Y_1}(b) &&& (*) \\
	& M_{\Y_2}(b) && \\
	&& \ddots &  \\
	(0) &&&  M_{\Y_l}(b)
	\end{pmatrix} \right), 
\end{gather*}
où, pour tout $i \in \{1,...,k\}$ (resp. $j \in \{1,...,l\}$), $M_{\X_i}(b)$ (resp. $M_{\Y_j}(b)$) est une matrice de représentation de $b$ dans $\X_i$ (resp. $\Y_j$). Il s'ensuit que :
\[ \forall b \in A \quad \tr_\X(b) = \sum_{i=1}^k \tr_{\X_i}(b), 
\quad \tr_\Y(b) = \sum_{j=1}^l \tr_{\X_j}(b). \]

\begin{enumerate}[(i)]
	\item Si $\X$ et $\Y$ ont la même classe dans $\G(A)$, alors ils ont les mêmes facteurs de composition à isomorphisme et ordre près (cf. par exemple \cite[Prop. 16.6]{CR81}). D'après ce qui précède, on a $k=l$ et :
	\[ \forall b \in A \quad \tr_\X(b) = \sum_{i=1}^k \tr_{\X_i}(b) = \sum_{i=1}^k \tr_{\Y_i}(b) = \tr_\Y(b) . \]
	Donc $\qch^s_{\X} = \qch^s_{\Y}$.
	\item Si $s$ est inversible, alors $a \mapsto sa$ est une bijection de $A$. Donc :
	\[ \qch^s_\X = \qch^s_\Y \quad \Longleftrightarrow \quad \qch^1_\X = \qch^1_\Y . \]
	Or, d'après le théorème de Frobenius-Schur (cf. par exemple \cite[Thm 27.8]{CR62}), la condition $\qch^1_\X = \qch^1_\Y$ implique que $\X$ et $\Y$ ont les mêmes facteurs de composition à isomorphisme et ordre près. D'où le résultat.
\end{enumerate}
\end{proof}

\begin{Prop} 
\label{Prop: qch}
Soit $A$ une $\C$-algèbre de Hopf de dimension finie, et $t \in A$ un élément inversible tel que :
\[ \Delta(t) = t \otimes t
\quad \text{ et } \quad 
\forall x \in A \quad S^2(x) = t x t^{-1} \]
On note $\G(A)$ l'anneau de Grothendieck de $A$ et $[\X]$ la classe du $A$-module $\X$ dans $\G(A)$. On a deux morphismes de $\C$-algèbres injectifs :
\[ \qch^{t^{-1}} :
	\begin{cases}
	 \C \otimes_\Z \G(A)  \longrightarrow \Ch^l(A) \\
	\left[ \X \right] \longmapsto \qch^{t^{-1}}_{[\X]}
	\end{cases}
	\quad \text{ et } \quad
	\qch^t :
	\begin{cases}
	 \C \otimes_\Z \G(A)  \longrightarrow \Ch^r(A) \\
	\left[ \X \right] \longmapsto \qch^t_{[\X]}
	\end{cases} \]
où $\Ch^l(A)$ (resp. $\Ch^r(A)$) désigne l'espace des $q$-caractères à gauche (resp. à droite) de $A$. 
\end{Prop}

\begin{proof}
D'après le lemme \ref{Lemme: qch}, on sait que les applications $\qch^{t^{\mp 1}}$ sont bien définies et injectives. De plus, pour tout $A$-module $\X$ et pour tous $a, b \in A$, on a :
\begin{align*}
&\qch^{t^{-1}}_\X(ab)
= \tr_\X \left( t^{-1} ab \right)
= \tr_\X \left( b t^{-1} a \right)
= \tr_\X \left( t^{-1} S^2(b) a  \right)
= \qch^{t^{-1}}_\X \left( S^2(b) a \right), \\
&\qch^t_\X(ab)
= \tr_\X \left( t ab \right)
= \tr_\X \left( S^2(a) t b \right)
= \tr_\X \left( t b S^2(a)  \right)
= \qch^t_\X \left( b S^2(a) \right),
\end{align*}
car $S^2(b) = t b t^{-1}$. D'où $\qch^{t^{-1}}_\X \in \Ch^l(A)$ et $\qch^t_\X \in \Ch^r(A)$. 

Soit $s \in \{t^{-1},t\}$. Il reste à montrer que $\qch^s$ réalise un morphisme de $\C$-algèbres. Soient $\lambda \in \C$, $\X$ et $\Y$ deux $A$-modules à gauche. Un représentant de $\lambda[\X]+[\Y]$ est $\lambda \X \oplus \Y$. Donc :
\[ \qch^s_{\lambda[\X]+[\Y]} = \qch^s_{\lambda \X \oplus \Y} = \lambda \qch^s_\X + \qch^s_\Y = \lambda \qch^s_{[\X]} + \qch^s_{[\Y]}. \]
Un représentant de $[\X] \cdot [\Y]$ est $\X \otimes \Y$, et pour tous $a \in \Uq$ :
\[ \qch^s_{[\X] \cdot [\Y]}(a) = \tr_{\X \otimes \Y}(\Delta(s a)) = \tr_{\X \otimes \Y}(s \otimes s \cdot \Delta(a)) = \qch^s_\X \otimes \tr_\Y^s (\Delta(a)), \]
car $\Delta( s a ) = \Delta( s ) \Delta(a) = s \otimes s \cdot \Delta(a)$. Or, le produit $\alpha \beta$ de deux formes linéaires $\alpha, \beta \in A^*$ est défini par :
\[ \forall a \in A \quad (\alpha \beta) (a) = \left( \alpha \otimes \beta \right) (\Delta(a)). \]
Donc :
\[ \qch^s_{[\X] \cdot [\Y]} = \qch^s_\X \cdot \qch_\Y^s = \qch^s_{[\X]} \cdot \qch_{[\Y]}^s. \]
\end{proof}

On utilise ce résultat sur $\Uq$ avec ses éléments de balancement généralisés (cf. le corollaire \ref{Cor: balancement}). D'après la remarque \ref{Rem: droite}, son anneau de Grothendieck à droite s'identifie à son anneau de Grothendieck à gauche. On le note $\G$\index{G@ $\G$}, et on note $\Ch^l$ \index{Ch@$\Ch^l$} (resp. $\Ch^r$\index{Ch@$\Ch^r$}) son espace des $q$-caractères à gauche (resp. à droite). On obtient ainsi quatre morphismes de $\C$-algèbres :
\begin{align}
\label{Eqn: qch}
& \qch^\delta := \qch^{K^{\delta p-1}} : \index{qch@ $\qch^\delta$}
	\begin{cases}
	 \C \otimes_\Z \G  \longrightarrow \Ch^l \\
	\left[ \X \right] \longmapsto \qch^{K^{\delta p-1}}_{[\X]}
	\end{cases} ;
	\quad \delta \in \{0,1\}, \\
\label{Eqn: qtr}	
& \qtr^\delta := \qch^{K^{\delta p+1}} : \index{qtr@ $\qtr^\delta$}
	\begin{cases}
	 \C \otimes_\Z \G  \longrightarrow \Ch^r \\
	\left[ \X \right] \longmapsto \qch^{K^{\delta p+1}}_{[\X]}
	\end{cases} ;
	\quad \delta \in \{0,1\}.
\end{align} 
Dans l'équation \eqref{Eqn: qtr}, on retrouve les traces quantiques des algèbres de Hopf tressées et enrubannées (cf. par exemple \cite[§ XIV.4, Prop. XIV.6.4]{Kas95}) pour lesquelles les éléments de balancement $\mathbf{k}$ et d'enrubannement $\mathbf{v}$ sont liés par $\mathbf{k} = \mathbf{v}^{-1} \mathbf{u}$ (cf. l'introduction du chapitre II).

\section{Description du centre à partir des morphismes de Drinfeld et de Radford}
\label{section: Drinfeld-Radford}

Avec des éléments mis en place dans la section \ref{section: structures}, on construit une nouvelle famille génératrice du centre de $\Uq$ à partir de son anneau de Grothendieck $\G$. Plus précisément, on considère les images des morphismes de Drinfeld et de Radford sur des $q$-caractères particuliers, associés aux classes des modules simples via le morphisme d'algèbres \eqref{Eqn: qch}.

\subsection{L'anneau de Grothendieck}
\label{subsection: Grothendieck}

On commence par détailler la structure de l'anneau de grothendieck $\G$ de $\Uq$ grâce à la donnée des produits tensoriels de modules simples, donnée dans la section \ref{section: reps prod}. Pour tout $\Uq$-module à gauche $\X$, on notera $[\X]$ \index{X @$[\X]$} sa classe d'équivalence dans $\G$.

\begin{Lemme} 
\label{Lemme: Grothendieck}
L'anneau de Grothendieck $\G$ de $\Uq$ est engendré par l'unité $[\X^+(1)]$ et la classe $[\X^+(2)]$.
\end{Lemme}

\begin{proof}
Par construction, le groupe de Grothendieck est librement engendré par les classes des modules simples (cf. par exemple \cite[Prop. 16.6]{CR81}). D'après la description des produits tensoriels de modules simples donnée dans le lemme \ref{Lemme: produits simples1}, on a :
\[ [\X^+(s)] [\X^+(2)] = [\X^+(s-1)] + [\X^+(s+1)], \quad 2 \leq s \leq p-1. \]
Une récurrence sur $s \in \{1,...,p\}$ montre alors que $[\X^+(s)]$ est engendré par $[\X^+(2)]$. De plus, d'après ce même lemme, on a aussi :
\begin{align*}
& [\X^+(p)] [\X^+(2)] = [\PIM^+(p-1)] = 2 [\X^+(p-1)] + 2 [\X^-(1)], \\
& [\X^-(1)] [\X^+(2)] = [\X^-(2)], \\
&[\X^-(s)] [\X^+(2)]
= [\X^-(s-1)] + [\X^-(s+1)], \quad 2 \leq s \leq p-1.
\end{align*}
Donc $[\X^-(1)]$, et par suite $[\X^-(2)]$, sont engendrés par $[\X^+(2)]$. Une récurrence sur $s \in \{1,...,p\}$ montre enfin que $[\X^-(s)]$ est engendré par $[\X^+(2)]$ .
\end{proof}

Par conséquent, on a un morphisme surjectif de $\C$-algèbres défini par :
\[ \begin{cases}
	\C[x] \longrightarrow \C \otimes_\Z \G \\
	x \longmapsto [\X^+(2)]
	\end{cases} . \]
Il s'ensuit que la $\C$-algèbre $\C \otimes_\Z \G$, obtenue par extension des scalaires sur $\G$, s'identifie au quotient de $\C[x]$ par l'idéal $I$ engendré par le polynôme minimal $\widehat{\psi}_{2p}(x)$ \index{psi@ $\widehat{\psi}_{2p}(x)$} de $[\X^+(2)]$. Il reste à déterminer $\widehat{\psi}_{2p}(x)$. A cette fin, on a besoin des \emph{polynômes de Chebychev}.

\begin{Def} 
On appelle polynômes de Chebychev (de seconde espèce) l'unique famille $(U_s(x))_{s \in \N}$ \index{U@ $U_s(x)$} solution du système de récurrence :
\begin{equation}
\label{Eqn: Chebychev}
\begin{cases}
	x U_s(x) = U_{s-1}(x) + U_{s+1}(x), & s \geq 1, \\
	U_0(x) = 0, \\
	U_1(x) = 1.
\end{cases}.
\end{equation}
\end{Def}

\begin{Rem} 
\label{Rem: Chebychev}
Pour tout $s \in \N$, on montre par récurrence que :
\[ \forall t \in \R \qquad U_s(t+t^{-1}) = \frac{t^s-t^{-s}}{t-t^{-1}} . \]
En particulier, on a :
\[ \forall t \in \R \qquad U_s(2 \cos t) = \frac{\sin(st)}{\sin(t)} . \]
Ce n'est pas la normalisation standard des polynômes de Chebychev de seconde espèce, laquelle est :
\[ \forall t \in \R \qquad U_s(\cos t) = \frac{\sin(st)}{\sin(t)}. \]
\end{Rem}

\begin{Prop} 
\label{Prop: Grothendieck}
On a un isomorphisme de $\C$-algèbres :
\[ \begin{cases}
	\C[x] / \widehat{\psi}_{2p}(x) \overset{\sim}{\longrightarrow} \C \otimes_\Z \G \\
	U_s(x) \longmapsto [\X^+(s)] \\
	\frac{1}{2} U_{p+s}(x) - \frac{1}{2} U_{p-s}(x) \longmapsto [\X^-(s)]
	\end{cases} \]
où $\widehat{\psi}_{2p}(x)$ est le polynôme minimal  de $[\X^+(2)]$. De plus, le polynôme $\widehat{\psi}_{2p}(x)$ vérifie :
\[ \widehat{ \psi }_{2p}(x) = U_{2p+1}(x) - U_{2p-1}(x) -2 . \]
\end{Prop}

\begin{proof}
L'existence d'un tel isomorphisme découle directement du lemme \ref{Lemme: Grothendieck}. Il s'agit de vérifier les relations polynomiales. Pour cela, on exploite les relations données par la décomposition des produits tensoriels $\X^\pm(s) \otimes \X^+(2)$, $1 \leq s \leq p$ (cf. le lemme \ref{Lemme: produits simples1}). 

Pour tout polynôme $P(x) \in \C[x]$, on note $\bar{P}(x)$ sa classe dans $\C[x] / I$. On considère la famille de polynômes $(\bar{P}_n(x))_{0 \leq n \leq 2p}$ telle que :
\begin{align*}
& [\{0\}] = \bar{P}_0 \left( [\X^+(2)] \right), \\
& [\X^+(s)] = \bar{P}_s \left( [\X^+(2)] \right) , && 1 \leq s \leq p, \\
& [\X^-(s)] = \bar{P}_{p+s} \left( [\X^+(2)] \right) , && 1 \leq s \leq p.
\end{align*}
On a $\bar{P}_0(x)=0$, $\bar{P}_1(x)=1$ et $\bar{P}_2(x)=x$. 

D'après la description des PIMs donnée dans la proposition \ref{Prop: PIMs}, on sait que :
\[ \forall \alpha \in \{+,-\} \quad \forall s \in \{1,...,s\} \quad [\PIM^\alpha(s)] = 2 [\X^\alpha(s)] + 2 [\X^{-\alpha}(p-s)]. \]
En utilisant la décomposition des produits $\X^\pm(s) \otimes \X^+(2)$, $1 \leq s \leq p$, donnée dans le lemme \ref{Lemme: produits simples1}, on obtient alors le système de récurrence :
\[ \left\{ \begin{aligned}
& x \bar{P}_s(x) = \bar{P}_{s-1}(x) + \bar{P}_{s+1}(x), && 1 \leq s  \leq p-1, \\
& x \bar{P}_p(x) = 2 \bar{P}_{p-1}(x) + 2 \bar{P}_{p+1}(x), \\
& x \bar{P}_{p+1}(x)  = \bar{P}_{p+2}(x), \\
& x \bar{P}_{p+s}(x) = \bar{P}_{p+s-1}(x) + \bar{P}_{p+s+1}(x), && 2 \leq s \leq p-1, \\
& x \bar{P}_{2p}(x) = 2 \bar{P}_{2p-1}(x) + 2 \bar{P}_1(x).
\end{aligned} \right. \]
D'après la définition des polynômes $(U_s)_{s \in \N}$ de Chebychev \eqref{Eqn: Chebychev}, pour tout $s \in \{0,...,p\}$, on a $\bar{P}_s(x) = \bar{U}_s(x)$. Pour le reste, on a alors :
\begin{align*}
& \begin{cases}
	x \bar{U}_p(x) = 2 \bar{U}_{p-1}(x) + 2 \bar{P}_{p+1}(x), \\
	x \bar{P}_{p+1}(x) ) = \bar{P}_{p+2}(x), \\
	x \bar{P}_{p+s}(x) = \bar{P}_{p+s-1}(x) + \bar{P}_{p+s+1}(x), & 2 \leq s \in \leq p-1, \\
	x \bar{P}_{2p}(x) = 2 \bar{P}_{2p-1}(x) + 2 \bar{U}_1(x).
	\end{cases} \\
\Longleftrightarrow \quad
& \begin{cases}
	\bar{P}_{p+1}(x) = \frac{1}{2} \bar{U}_{p+1}(x) - \frac{1}{2} \bar{U}_{p-1}(x), \\
	\bar{P}_{p+2}(x) ) = \frac{1}{2} \bar{U}_{p+2}(x) - \frac{1}{2} \bar{U}_{p-2}(x), \\
	\bar{P}_{p+s+1}(x) = x \bar{P}_{p+s}(x) - \bar{P}_{p+s-1}(x), & 2 \leq s \leq p-1, \\
	x \bar{P}_{2p}(x)= 2 \bar{P}_{2p-1}(x) + 2 \bar{U}_1(x).
	\end{cases} \\
\Longleftrightarrow \quad
& \begin{cases}
	\bar{P}_{p+s}(x) = \frac{1}{2} \bar{U}_{p+s}(x) - \frac{1}{2} \bar{U}_{p-s}(x), & 1 \leq s \leq p, \\
	\frac{1}{2} x \bar{U}_{2p}(x) - \frac{1}{2} x \bar{U}_0(x) = \bar{U}_{2p-1}(x) - \bar{U}_1(x) + 2 \bar{U}_1(x).
	\end{cases}
\end{align*}
Enfin, la dernière relation donne:
\[ \widehat{ \psi }_{2p}(x) = x U_{2p}(x) - 2 U_{2p-1}(x) - 2 U_1(x) = U_{2p+1}(x) - U_{2p-1}(x) - 2 . \]
D'où le résultat.
\end{proof}

\begin{Cor} 
\label{Cor: Grothendieck}
On a :
\[ \widehat{ \psi }_{2p}(x) 
= \left( x - \widehat{\beta}_0 \right) \left( \prod_{j=1}^{p-1} \left( x - \widehat{\beta}_j \right)^2 \right) \left( x - \widehat{\beta}_p \right), \]
où, pour tout $j \in \{0,...,p\}$, $\widehat{\beta_j} = q^j+q^{-j} = 2 \cos \frac{\pi j}{p}$.
\end{Cor}

\begin{proof}
D'après l'expression de $\widehat{\psi}_{2p}(x)$ donnée dans la proposition \ref{Prop: Grothendieck}, et la remarque \ref{Rem: Chebychev} sur les polynômes de Chebychev, on a :
\begin{align*}
\forall t \in \R \qquad \widehat{ \psi }_{2p}(2 \cos t)
&= U_{2p+1}(2\cos t) - U_{2p-1}(2\cos t) - 2 \\
&= \frac{\sin(2p+1)t -\sin(2p-1)t}{\sin t} - 2
= 2\cos(2pt) - 2 \\
&= 4({\cos}^2(pt) - 1)
= 4 (\cos(pt) - 1)(\cos(pt) + 1).
\end{align*}
L'ensemble des racines de $\widehat{ \psi }_{2p}(x)$ est donc $\left\{ \widehat{\beta_j} \; ; \; j \in \{0,...,p\} \right\}$. De plus, pour tout $j \in \{1,...,p-1\}$, la racine $\widehat{\beta_j}$ est d'ordre $2$ car :
\begin{align*}
-2 \sin \frac{\pi j}{p} \; \widehat{ \psi }_{2p}'( \widehat{\beta_j} )
&= -8p \sin(\pi j) \cos(\pi j)
= 0 .
\end{align*}
Par conséquent, on a au moins $2p$ racines, comptées avec leurs multiplicités. Or, le polynôme $\widehat{ \psi }_{2p}(x)$ est de degré $2p$ car, pour tout $s \in \N$, le $s$-ième polynôme de Chebychev $U_s(x)$ est de degré $s-1$ (récurrence immédiate). On a donc toutes les racines du polynôme $\widehat{\psi}_{2p}$.
\end{proof}

\begin{Rem}
En particulier, le polynôme $\widehat{\psi}_{2p}(x)$ est de degré $2p$. D'où la notation choisie.
\end{Rem}

\subsection{Le morphisme de Drinfeld}
\label{subsection: Drinfeld}

La donnée de l'élément $\mathbf{M} \in \Uq \otimes \Uq$ \eqref{Eqn: M-matrice} et d'un élément de balancement $\mathbf{k}$ \eqref{Cor: balancement} de $\Uq$ permet de définir un morphisme de $\C$-algèbres injectifs :
\[ \bchi^\mathbf{k} : \C \otimes_\Z \G \longrightarrow \Zf, \]
où $\G$ désigne l'anneau de Grothendieck de $\Uq$ et $\Zf$ son centre. On cherche à décrire son image $\Df_{2p}$ dans le centre.

\begin{Def}[{\cite[Prop. 3.3]{Dri90}}]
Soit $(A, \mathbf{R})$ une $\C$-algèbre de Hopf tressée de $R$-matrice $\mathbf{R}$. On définit le \emph{morphisme de Drinfeld} par :
\[ \bchi : \begin{cases}
	A^* \longrightarrow A \\
	\beta \longmapsto (\beta \otimes id) \mathbf{R}_{21} \mathbf{R}
	\end{cases} . \]
\end{Def}

Soit $\mathbf{M}$ l'élément défini par \eqref{Eqn: M-matrice}. On considère le morphisme de Drinfeld $\bchi$ associé à $\Uq$, vu comme une sous-algèbre de Hopf tressée $(\bar{D}, \bar{\mathbf{R}})$ via le morphisme \eqref{Eqn: inj double}. D'après l'article \cite[Prop. 3.3]{Dri90} (on pourra également consulter \cite[Lemme 2]{Ker95}), le morphisme $\bchi$ induit un isomorphisme de $\C$-algèbres : \index{chi1@ $\bchi$}
\begin{equation}
\label{Eqn: Drinfeld1}
\bchi : \Ch^l \overset{\sim}{\longrightarrow} \Zf,
\end{equation}
où $\Ch^l$ désigne l'espace des $q$-caractères à gauche de $\Uq$. On s'intéresse à la composée des morphismes $\mathrm{qch}^\delta$ \eqref{Eqn: qch} et $\bchi$ : \index{chi2@ $\bchi^\delta$}
\begin{equation} 
\label{Eqn: Drinfeld2}
\bchi^\delta : \begin{cases}
	\C \otimes_\Z \G \hookrightarrow \Zf \\
	[\X] \longmapsto (\qch^\delta_{[\X]} \otimes id) \mathbf{M}
	\end{cases} ; \quad \delta \in \{0,1\}.
\end{equation}

Soit $\delta \in \{0,1\}$. On note $\Df_{2p}$ \index{Drinfeld@ $\Df_{2p}$} l'image de $\bchi^\delta$ et :
\begin{equation}
\forall \alpha \in \{+,-\} \quad \forall s \in \{1,...,p\} \qquad \chi_\delta^\alpha(s) := \bchi^\delta \left( [\X^\alpha(s) ] \right) .
\end{equation}
Comme le groupe de Grothendieck $\G$ de $\Uq$ est librement engendré par les classes des modules simples $\X^\pm(s)$, $1 \leq s \leq p$, la famille $\{ \chi_\delta^\pm(s) \; ; \; 1 \leq s \leq p \}$ une $\C$-base de la sous-algèbre $\Df_{2p}$. On étudie cette famille de $\Zf$.

\begin{Lemme} 
\label{Lemme: Drinfeld}
Soient $\delta \in \{0,1\}$ et $s \in \{1,...,p\}$. L'image des classes $[\X^+(s)]$ et $[\X^-(s)]$ par le morphisme $\bchi^\delta$ \eqref{Eqn: Drinfeld2} sont respectivement :
\begin{align*}
&\chi_\delta^+(s) = \begin{multlined}[t]
	(-1)^{\delta(s-1)} \alpha^{\delta p-1} \sum_{k=0}^{s-1} \sum_{m=0}^{k} (q-q^{-1})^{2m} q^{(m-1)(s-1-2k+m)} \\
	\times {k \brack m} {s-k+m-1 \brack m} E^m K^{s-1-2k+m} F^m ,
	\end{multlined} \\
&\chi_\delta^-(s) = \begin{multlined}[t]
	(-1)^{\delta(s-1)} \alpha^{\delta p-1} \sum_{k=0}^{s-1} \sum_{m=0}^{k} (q-q^{-1})^{2m} q^{(m-1)(s-1-2k+m)} \\
	\times {k \brack m} {s-k+m-1 \brack m} E^m K^{p+s-1-2k+m} F^m .
	\end{multlined}
\end{align*}
\end{Lemme}

\begin{proof}
D'après l'équation \eqref{Eqn: M-matrice}, la $M$-matrice de $\Uq$ est donnée par :
\[ \mathbf{M} = \frac{1}{2p} \sum_{m,n=0}^{p-1} \sum_{i,j=0}^{2p-1} \frac{(q-q^{-1})^{m+n}}{[m]![n]!} q^{\frac{m(m-1)}{2}+\frac{n(n-1)}{2}-m^2-mj-ij+mi} F^m K^j E^n \otimes E^m K^i F^n . \]
Soit $\alpha \in \{+,-\}$. On a donc :
\[ \chi_\delta^\alpha(s) \overset{\eqref{Eqn: Drinfeld2}}{=} \begin{multlined}[t]
	\frac{1}{2p} \sum_{m,n=0}^{p-1} \sum_{i,j=0}^{2p-1} \frac{(q-q^{-1})^{m+n}}{[m]![n]!} q^{\frac{m(m-1)}{2}+\frac{n(n-1)}{2}-m^2-mj-ij+mi} \\
	\times \qch^\delta_{[\X^\alpha(s)]} \left( F^m K^j E^n \right) E^m K^i F^n .
	\end{multlined} \]

Soient $m,n \in \{0,...,p-1\}$ et $j \in \{0,...,2p-1\}$. Il faut calculer :
\begin{align*}
\qch^\delta_{[\X^\alpha(s)]} \left( F^m K^j E^n \right) & \;\overset{\ref{Lemme: qch}}{=} \qch^\delta_{\X^\alpha(s)} \left( F^m K^j E^n \right) \\
&\overset{\eqref{Eqn: qch}}{=} \tr_{\X^\alpha(s)} \left( K^{\delta p-1} F^m K^j E^n \right) \\
&\; \overset{\eqref{Eqn: comm1}}{=} q^{2m} \tr_{\X^\alpha(s)} \left( F^m K^{j-1+\delta p} E^n \right).
\end{align*}
On utilise la $\C$-base canonique $\{ x^\alpha_k(s) \; ; \; 0 \leq k \leq s-1\}$ de $\X^\alpha(s)$ (cf. la proposition \ref{Prop: simples}), et on fixe $k \in \{0,...,s-1\}$. Alors, pour tous $m,n \in \{0,...,p-1\}$ et $j \in \{0,...,2p-1\}$, on a :
\begin{align*}
F^m K^j E^n x^\alpha_k(s)
&= \delta_{k \geq n} \alpha^n \prod_{i=0}^{n-1} [k-i][s-k+i] F^m K^j x^\alpha_{k-n}(s) \\
&= \delta_{k \geq n} \alpha^n \frac{[k]! [s-k+n-1]!}{[k-n]! [s-k-1]!} F^m K^j x^\alpha_{k-n}(s) \\
&= \delta_{k \geq n} \alpha^{n+j} q^{j(s-1-2k+2n)} ([n]!)^2 {k \brack n} {s-k+n-1 \brack n} F^m x^\alpha_{k-n}(s) \\
&= \delta_{n \leq k \leq n-m+p-1} \alpha^{n+j} q^{j(s-1-2k+2n)} ([n]!)^2 {k \brack n} {s-k+n-1 \brack n} x^\alpha_{k-n+m}(s).
\end{align*}
En particulier, $x^\alpha_k(s)$ est vecteur propre sous l'action de $F^m K^j E^n$ si et seulement si $k \geq n=m$. On en déduit que :
\begin{align*}
\tr_{\X^\alpha(s)} \left( F^m K^{j-1+\delta p} E^n \right)
&= \begin{multlined}[t]
	\delta_{n,m} (-1)^{\delta(s-1)} \alpha^{m+j-1+\delta p} ([m]!)^2 \\
	\sum_{k=m}^{s-1} q^{(j-1)(s-1-2k+2m)} {k \brack m} {s-k+m-1 \brack m} .
	\end{multlined}
\end{align*}

Il s'ensuit que :
\begin{align*}
\chi_\delta^\alpha(s) &= \begin{multlined}[t]
	\frac{1}{2p} \sum_{m,n=0}^{p-1} \sum_{i,j=0}^{2p-1} \frac{(q-q^{-1})^{m+n}}{[m]![n]!} q^{\frac{m(m-1)}{2}+\frac{n(n-1)}{2}-m^2-mj-ij+mi+2m} \\
	\times \tr_{\X^\alpha(s)} \left( F^m K^{j-1+\delta p} E^n \right) E^m K^i F^n
	\end{multlined} \\
&= \begin{multlined}[t]
	\frac{(-1)^{\delta(s-1)}}{2p} \sum_{m=0}^{p-1} \sum_{i,j=0}^{2p-1} (q-q^{-1})^{2m} \alpha^{m+j-1+\delta p} q^{m-mj-ij+mi} \\
	\times \sum_{k=m}^{s-1} q^{(j-1)(s-1-2k+2m)} {k \brack m} {s-k+m-1 \brack m} E^m K^i F^m
	\end{multlined} \\
&= \begin{multlined}[t]
	\frac{(-1)^{\delta(s-1)}}{2p} \alpha^{\delta p-1} \sum_{k=0}^{s-1} \sum_{m=0}^{k} (q-q^{-1})^{2m} \alpha^m q^{-s+1+2k-m} {k \brack m} {s-k+m-1 \brack m} \\
	\times \sum_{i=0}^{2p-1} q^{mi} \sum_{j=0}^{2p-1} \alpha^j q^{j(s-1-2k+m-i)} E^m K^i F^m.
	\end{multlined}
\end{align*}
Pour évaluer le dernier terme, on introduit $A \in \{0,1\}$ tel que $q^{Ap}=\alpha 1$. Alors :
\begin{align*}
\chi_\delta^\alpha(s) &= \begin{multlined}[t]
	\frac{(-1)^{\delta(s-1)}}{2p} \alpha^{\delta p-1} \sum_{k=0}^{s-1} \sum_{m=0}^{k} (q-q^{-1})^{2m} \alpha^m q^{-s+1+2k-m} {k \brack m} {s-k+m-1 \brack m} \\
	\times \sum_{i=0}^{2p-1} q^{mi} \sum_{j=0}^{2p-1} q^{j(Ap+s-1-2k+m-i)} E^m K^i F^m
	\end{multlined} \\
&= \begin{multlined}[t]
	(-1)^{\delta(s-1)} \alpha^{\delta p-1} \sum_{k=0}^{s-1} \sum_{m=0}^{k} (q-q^{-1})^{2m} \alpha^m q^{-s+1+2k-m} {k \brack m} {s-k+m-1 \brack m} \\	\times q^{m(Ap+s-1-2k+m)} E^m K^{Ap+s-1-2k+m} F^m
	\end{multlined} \\
&= \begin{multlined}[t]
	(-1)^{\delta(s-1)} \alpha^{\delta p-1} \sum_{k=0}^{s-1} \sum_{m=0}^{k} (q-q^{-1})^{2m} q^{(m-1)(s-1-2k+m)} {k \brack m} {s-k+m-1 \brack m} \\
	\times E^m K^{Ap+s-1-2k+m} F^m.
	\end{multlined}
\end{align*}
D'où le résultat.
\end{proof}

On en déduit les expressions de $\chi^\pm_\delta(s)$, $1 \leq s \leq p$, dans la sous-algèbre $\langle C \rangle$ engendrée par l'élément de Casimir \eqref{Eqn: Casimir}, puis dans la base canonique du centre $\{ e_s \; ; \; 0 \leq s \leq p \} \cup \{ w^\pm_s \; ; \; 1 \leq s \leq p-1 \}$ (cf. la proposition \ref{Prop: centre}).

\begin{Prop}
\label{Prop: Drinfeld}
Soit $\delta \in \{0,1\}$. Pour tout $s \in \{1,...,p\}$, on a :
\begin{align*}
\chi_\delta^+(s) &= U_s((-1)^\delta\widehat{C}) \\
	&= \sum_{j=0}^p U_s((-1)^\delta\widehat{\beta}_j) e_j
		+ (-1)^\delta (q-q^{-1})^2 \sum_{j=1}^{p-1} U_s'((-1)^\delta\widehat{\beta}_j) \left( w^+_j + w^-_j \right), \\
\chi_\delta^-(s) &= \frac{1}{2} \left( U_{p+s}-U_{p-s} \right) ((-1)^\delta\widehat{C}) \\
&= \begin{multlined}[t]
	\frac{1}{2} \sum_{j=0}^p (U_{p+s}-U_{p-s})((-1)^\delta \widehat{\beta}_j) e_j \\
	+ (-1)^\delta \frac{(q-q^{-1})^2}{2} \sum_{j=1}^{p-1}(U_{p+s}'-U_{p-s}')((-1)^\delta \widehat{\beta}_j) \left( w^+_j + w^-_j \right), 
	\end{multlined}
\end{align*}
où $\widehat{C} := (q-q^{-1})^2 C$ et pour tout $s \in \N$, $U_s(x)$ est le $s$-ième polynôme de Chebychev et $\widehat{\beta}_s := q^s+q^{-s}$.
\end{Prop}

\begin{proof}
D'après la proposition \ref{Prop: Grothendieck}, la $\C$-algèbre $\C \otimes_\Z \G$ est engendrée par la classe $[\X^+(2)]$.  Plus précisément, on a :
\begin{align*}
& [\X^+(s)] = U_s \left( [\X^+(2)] \right) , && 1 \leq s \leq p, \\
& [\X^-(s)] = \frac{1}{2} \left( U_{p+s}-U_{p-s} \right) \left( [\X^+(2)] \right) , && 1 \leq s \leq p.
\end{align*}
Or, d'après la proposition \ref{Lemme: Drinfeld}, l'image de $[\X^+(2)]$ par le morphisme $\bchi^\delta$ est :
\[ \chi^+_\delta(2) = (-1)^\delta \left( q^{-1} K + q K^{-1} + (q-q^{-1})^2 EF \right) = (-1)^\delta \widehat{C}. \]
 Comme $\bchi^\delta$ est un morphisme d'algèbres, il s'ensuit que :
\begin{align*}
&\chi_\delta^+(s) = U_s((-1)^\delta\widehat{C}), && 1 \leq s \leq p, \\
&\chi_\delta^-(s) = \frac{1}{2} \left( U_{p+s}-U_{p-s} \right) ((-1)^\delta\widehat{C}), && 1 \leq s \leq p,
\end{align*}
Pour décomposer ces polynômes en $(-1)^\delta\widehat{C}$ dans la base canonique du centre, il suffit alors d'utiliser la remarque \ref{Rem: poly en C}.
\end{proof}

\begin{Cor} 
\label{Cor: Drinfeld}
Soit $\delta \in \{0,1\}$. L'image $\Df_{2p}$ du morphisme $\bchi^\delta$ \eqref{Eqn: Drinfeld2} est la sous-algèbre $\langle C \rangle$ du centre $\Zf$ de $\Uq$ engendrée par l'élément de Casimir $C$ \eqref{Eqn: Casimir}.
\end{Cor}

\begin{Rem}
\label{Rem: D2p}
D'après les précédents résultats et la remarque \ref{Rem: poly en C} sur les polynômes en $C$, $\Df_{2p}$ est un $\C$-espace vectoriel de dimension $2p$ dont des $\C$-bases sont $\{ \chi_\delta^\pm(s) \; ; \; 1 \leq s \leq p \}$ ou $\{ e_s \; ; \; 0 \leq s \leq p \} \cup \{ w^+_s+w^-_s \; ; \; 1 \leq s \leq p-1 \}$.
\end{Rem}

\subsection{Le morphisme de Radford}
\label{subsection: Radford}

La donnée d'une co-intégrale bilatère $\mathbf{c}$ \eqref{Prop: integrales} et d'un élément de balancement $\mathbf{k}$ \eqref{Cor: balancement} de $\Uq$ permet de définir un morphisme de $\C$-espaces vectoriels injectif :
\[ \hbphi^\mathbf{k} : \C \otimes_\Z \G \longrightarrow \Zf, \]
où $\G$ désigne l'anneau de Grothendieck de $\Uq$ et $\Zf$ son centre. On cherche à décrire son image $\Rf_{2p}$ dans le centre.

\begin{Def}[{\cite[Prop. 3]{Rad90}}]
Soient $A$ une $\C$-algèbre de Hopf de dimension finie et $\mathbf{c}$ une co-intégrale à gauche. On définit le \emph{morphisme de Radford} (associé à $\mathbf{c}$) par :
\[ \hbphi : \begin{cases}
	A^* \longrightarrow A \\
	\beta \longmapsto (\beta \otimes id) \Delta(\mathbf{c})
	\end{cases} . \]
\end{Def}

\begin{Rem}
\label{Rem: Radford}
Soient $A$ une $\C$-algèbre de Hopf de dimension finie et $\mathbf{c}$ une co-intégrale à gauche.
\begin{enumerate}[(i)]
	\item D'après le résultat \cite[Prop. 3]{Rad90} de Radford, le morphisme de Radford $\hbphi$ est bijectif et sa bijection réciproque est :
	\[ \hbphi^{-1} : \begin{cases}
	A \longrightarrow A^* \\
	x \longmapsto \bmu^r(S(a) ?)
	\end{cases} . \]
	où le symbole $?$ désigne la place de la variable et $\bmu^r$ est une intégrale à droite de $A$ telle que $\bmu^r(\mathbf{c})=1$. 
	\item Si $\mathbf{c}$ est une co-intégrale bilatère, alors $\hbphi$ est $A$-linéaire pour les actions adjointe et co-adjointe à gauche. En effet, pour tous $a \in A$ et $\beta \in A^*$, on a :
	\[ \hbphi ( a \cdot \beta ) 
	= \left( \beta \left( \sum_{(a)} S(a') ? a'' \right) \otimes id \right) \Delta(\mathbf{c}) 
	= (\beta \otimes id) \left( \sum_{(a)} S(a') \mathbf{c}' a'' \otimes \mathbf{c}'' \right) , \]
	où :
	\begin{align*}
	\sum_{(a)} S(a') \mathbf{c}' a'' \otimes \mathbf{c}''
	&= \sum_{(a)} S(a') \mathbf{c}' a'' \varepsilon(a''') \otimes \mathbf{c}''
	= \sum_{(a)} S(a') \mathbf{c}' a'' \otimes \mathbf{c}'' \varepsilon(a''') \\
	&= \sum_{(a)} S(a') \mathbf{c}' a'' \otimes \mathbf{c}'' a''' S(a^{(4)}).
	\end{align*}
	Or, pour tout $x \in A$, on sait que $\mathbf{c} x = x \mathbf{c}$ car $\mathbf{c}$ est une co-intégrale bilatère. Donc $\Delta(\mathbf{c} x) = \Delta(x \mathbf{c})$. Il s'ensuit que :
	\begin{align*}
	\sum_{(a)} S(a') \mathbf{c}' a'' \otimes \mathbf{c}''
	&= \sum_{(a)} S(a') a'' \mathbf{c}' \otimes a''' \mathbf{c}'' S(a^{(4)})
	= \sum_{(a)} \varepsilon(a') \mathbf{c}' \otimes a'' \mathbf{c}'' S(a''') \\
	&= \sum_{(a)} \mathbf{c}' \otimes \varepsilon(a') a'' \mathbf{c}'' S(a''')
	= \sum_{(a)} \mathbf{c}' \otimes a' \mathbf{c}'' S(a'').
	\end{align*}
	D'où $\hbphi ( a \cdot \beta ) = (\beta \otimes id) \left( \sum_{(a)} \mathbf{c}' \otimes a \cdot \mathbf{c}'' \right) = a \cdot \hbphi(\beta)$.
	\item Le morphisme de Radford permet de montrer que $\Uq$ est une algèbre de Frobenius comme mentionné dans la remarque \ref{Rem: droite}.
\end{enumerate}
\end{Rem}

Soit $\mathbf{c}$ une co-intégrale bilatère de $\Uq$ (cf. proposition \ref{Prop: integrales}). On considère le morphisme de Radford $\hbphi$ associé à $\Uq$ (et à la co-intégrale $\mathbf{c}$). D'après la remarque \ref{Rem: Radford} et les caractérisations données dans la proposition \ref{Prop: stabilisateurs}, il induit un isomorphisme de $\C$-espaces vectoriels : \index{phi1@ $\hbphi$}
\begin{equation}
\label{Eqn: Radford1}
\hbphi : \Ch^l \overset{\sim}{\longrightarrow} \Zf, 
\end{equation}
où $\Ch^l$ désigne l'espace des $q$-caractères à gauche de $\Uq$. On s'intéresse à la composée des morphismes $\mathrm{qch}^\delta$ \eqref{Prop: qch} et $\hbphi$ : \index{phi2@ $\hbphi^\delta$}
\begin{equation} 
\label{Eqn: Radford2}
\hbphi^\delta : \begin{cases}
	\C \otimes _\Z \G \hookrightarrow \Zf \\
	[\X] \longmapsto (\mathrm{qch}^\delta_{[\X]} \otimes id) \Delta(\mathbf{c})
	\end{cases} ; \quad \delta \in \{0,1\}.
\end{equation}

Soit $\delta \in \{0,1\}$. On note $\Rf_{2p}$ \index{Radford@ $\Rf_{2p}$} l'image de $\hbphi^\delta$ et :
\begin{equation}
\forall \alpha \in \{+,-\} \quad \forall s \in \{1,...,p\} \qquad \widehat{\phi}^\alpha_\delta(s) := \hbphi^\delta \left( [\X^\alpha(s) ] \right) .
\end{equation}
Comme le groupe de Grothendieck $\G$ de $\Uq$ est librement engendré par les classes des modules simples $\X^\pm(s)$, $1 \leq s \leq p$, la famille $\{ \widehat{\phi}^\pm_\delta(s) \; ; \; 1 \leq s \leq p \}$ est une $\C$-base du sous-espace vectoriel $\Rf_{2p}$. On étudie cette famille de $\Zf$.

\begin{Lemme} 
\label{Lemme: Radford}
Soient $\delta \in \{0,1\}$, $\alpha \in \{+,-\}$ et $s \in \{1,...,p\}$. L'image de la classe $[\X^\alpha(s)]$ par le morphisme $\hbphi^\delta$ \eqref{Eqn: Radford2} est :
\[ \widehat{\phi}_\delta^\alpha(s) = \begin{multlined}[t]
	\zeta (-1)^{(\delta-1)(s-1)} \alpha^{(\delta-1)p} \sum_{k=0}^{s-1} \sum_{r=0}^k ([r]!)^2 {k \brack r} {s-k+r-1 \brack r} \\
	 \times \sum_{j=0}^{2p-1} \alpha^{j+r} q^{j(s+1-2k+2r)} F^{p-1-r} K^j E^{p-1-r},
	\end{multlined} \]
où $\zeta \in \C$ dépend du choix de la co-intégrale bilatère $\mathbf{c}_\zeta$ \eqref{Prop: integrales}.
\end{Lemme}

\begin{proof}
D'après la proposition \ref{Prop: integrales}, les co-intégrales bilatères de $\Uq$ sont données par :
\[ \mathbf{c}_\zeta = \zeta \sum_{j=0}^{2p-1} q^{2j} F^{p-1} K^j E^{p-1}; \quad \zeta \in \C . \]
On fixe $\zeta \in \C$. Le coproduit de la co-intégrale bilatère $\mathbf{c}_\zeta$ associé est :
\begin{align*}
\Delta(\mathbf{c}_\zeta) & \; \; = \zeta \sum_{j=0}^{2p-1} q^{2j} \Delta(F)^{p-1} \Delta(K)^j \Delta(E)^{p-1} \\
& \overset{\eqref{Eqn: coprod2}}{=} \begin{multlined}[t]
	\zeta \sum_{j=0}^{2p-1} \sum_{r,t=0}^{p-1} q^{2j+r(p-1-r)-t(p-1-t)} {p-1 \brack r} {p-1 \brack t} \\
	\times F^r K^{r-p+1+j} E^{p-1-t} \otimes F^{p-1-r} K^{j+p-1-t} E^t 
	\end{multlined} \\
& \overset{\eqref{Eqn: qcoeff2}}{=} \begin{multlined}[t]
	\zeta \sum_{j=0}^{2p-1} \sum_{r,t=0}^{p-1} (-1)^{r-t} q^{2j-r(r+1)+t(t+1)} \\
	\times F^r K^{r-p+1+j} E^{p-1-t} \otimes F^{p-1-r} K^{j+p-1-t} E^t.
	\end{multlined}
\end{align*}
Donc :
\[ \widehat{\phi}_\delta^\alpha(s) \overset{\eqref{Eqn: Radford2}}{=} \begin{multlined}[t]
	\zeta \sum_{j=0}^{2p-1} \sum_{r,t=0}^{p-1} (-1)^{r-t} q^{2j-r(r+1)+t(t+1)} \qch^\delta_{[\X^\alpha(s)]} \left( F^r K^{r-p+1+j} E^{p-1-t} \right) \\
	\times F^{p-1-r} K^{j+p-1-t} E^t .
	\end{multlined} \]

Soient $r,t \in \{0,...,p-1\}$ et $j \in \{0,...,2p-1\}$. Il faut calculer :
\begin{align*}
\qch^\delta_{[\X^\alpha(s)]} \left( F^r K^{r-p+1+j} E^{p-1-t} \right) & \;\overset{\ref{Lemme: qch}}{=} \qch^\delta_{\X^\alpha(s)} \left( F^r K^{r-p+1+j} E^{p-1-t} \right) \\
&\overset{\eqref{Eqn: qch}}{=} \tr_{\X^\alpha(s)} \left( K^{\delta p-1} F^r K^{r-p+1+j} E^{p-1-t} \right) \\
&\; \overset{\eqref{Eqn: comm1}}{=} q^{2r} \tr_{\X^\alpha(s)} \left( F^r K^{r+j+(\delta-1)p} E^{p-1-t} \right).
\end{align*}
On utilise la $\C$-base canonique $\{ x^\alpha_k(s) \; ; \; 0 \leq k \leq s-1\}$ de $\X^\alpha(s)$ (cf. la proposition \ref{Prop: simples}), et on fixe $k \in \{0,...,s-1\}$. Alors, pour tous $m,n \in \{0,...,p-1\}$ et $j \in \{0,...,2p-1\}$, on a :
\begin{align*}
F^m K^j E^n x^\alpha_k(s)
&= \delta_{k \geq n} \alpha^n \prod_{i=0}^{n-1} [k-i][s-k+i] F^m K^j x^\alpha_{k-n}(s) \\
&= \delta_{k \geq n} \alpha^n \frac{[k]! [s-k+n-1]!}{[k-n]! [s-k-1]!} F^m K^j x^\alpha_{k-n}(s) \\
&= \delta_{k \geq n} \alpha^{n+j} q^{j(s-1-2k+2n)} ([n]!)^2 {k \brack n} {s-k+n-1 \brack n} F^m x^\alpha_{k-n}(s) \\
&= \delta_{n \leq k \leq n-m+p-1} \alpha^{n+j} q^{j(s-1-2k+2n)} ([n]!)^2 {k \brack n} {s-k+n-1 \brack n} x^\alpha_{k-n+m}(s).
\end{align*}
En particulier, $x^\alpha_k(s)$ est vecteur propre sous l'action de $F^m K^j E^n$ si et seulement si $k \geq n=m$. On en déduit que :
\begin{align*}
\tr_{\X^\alpha(s)} \left( F^r K^{r+j+(\delta-1)p} E^{p-1-t} \right)
&= \begin{multlined}[t]
	\delta_{r,p-1-t} (-1)^{(\delta-1)(s-1)} \alpha^{j+(\delta-1)p} ([r]!)^2 \\
	\sum_{k=r}^{s-1} q^{(r+j)(s-1-2k+2r)} {k \brack r} {s-k+r-1 \brack r} .
	\end{multlined}
\end{align*}

Il s'ensuit que :
\begin{align*}
\widehat{\phi}_\delta^\alpha(s) &= \begin{multlined}[t]
	\zeta \sum_{j=0}^{2p-1} \sum_{r,t=0}^{p-1} (-1)^{r-t} q^{2j-r(r-1)+t(t+1)} \tr_{\X^\alpha(s)} \left( F^r K^{r+j+(\delta-1)p} E^{p-1-t} \right) \\
	\times F^{p-1-r} K^{j+p-1-t} E^t
	\end{multlined} \\
&= \begin{multlined}[t]
	\zeta (-1)^{(\delta-1)(s-1)} \sum_{j=0}^{2p-1} \sum_{r=0}^{p-1} \alpha^{j+(\delta-1)p} ([r]!)^2 \\
	\times \sum_{k=r}^{s-1}  q^{2j+2r} q^{(r+j)(s-1-2k+2r)} {k \brack r} {s-k+r-1 \brack r} F^{p-1-r} K^{j+r} E^{p-1-r}
	\end{multlined} \\
&= \begin{multlined}[t]
	\zeta (-1)^{(\delta-1)(s-1)} \alpha^{(\delta-1)p} \sum_{k=0}^{s-1} \sum_{r=0}^k ([r]!)^2 {k \brack r} {s-k+r-1 \brack r} \\
	 \times \sum_{j=0}^{2p-1} \alpha^j q^{(r+j)(s+1-2k+2r)} F^{p-1-r} K^{j+r} E^{p-1-r} .
	\end{multlined}
\end{align*}
Le dernier terme est une somme dont l'indice $j$ ne dépend pas du représentant choisi dans $\Z / 2p \Z$. En effectuant le changement de variables $j \mapsto j+r$, on obtient :
\[ \widehat{\phi}_\delta^\alpha(s) = \begin{multlined}[t]
	\zeta (-1)^{(\delta-1)(s-1)} \alpha^{(\delta-1)p} \sum_{k=0}^{s-1} \sum_{r=0}^k ([r]!)^2 {k \brack r} {s-k+r-1 \brack r} \\
	\times \sum_{j=0}^{2p-1} \alpha^{j+r} q^{j(s+1-2k+2r)} F^{p-1-r} K^j E^{p-1-r} .
	\end{multlined} \]
\end{proof}

On en déduit les expressions de $\widehat{\phi}^\pm_\delta(s)$, $1 \leq s \leq p$, dans la base canonique du centre $\{ e_s \; ; \; 0 \leq s \leq p \} \cup \{ w^\pm_s \; ; \; 1 \leq s \leq p-1 \}$ (cf. la proposition \ref{Prop: centre}).

\begin{Prop} 
\label{Prop: Radford}
Soit $\delta \in \{0,1\}$. On a :
\begin{align*}
& \widehat{\phi}_\delta^+(s) = \zeta (-1)^{p+\delta(s-1)} 2p \frac{([p-1]!)^2}{[s]^2} w^+_s, && 1 \leq s \leq p-1, \\
& \widehat{\phi}_\delta^-(s) = \zeta (-1)^{\delta(p-s-1)} 2p \frac{([p-1]!)^2}{[s]^2} w^-_{p-s}, && 1 \leq s \leq p-1, \\
& \widehat{\phi}_\delta^+(p) = \zeta (-1)^{(\delta-1)(p-1)} 2p ([p-1]!)^2 e_p, \\
& \widehat{\phi}_\delta^-(p) = \zeta (-1)^{p-\delta} 2p ([p-1]!)^2 e_0,
\end{align*}
où $\zeta \in \C$ dépend du choix de la co-intégrale bilatère $\mathbf{c}_\zeta$ \eqref{Prop: integrales}.
\end{Prop}

\begin{proof}
Soient $\alpha \in \{+,-\}$ et $s \in \{1,...,p\}$. Conformément au corollaire \ref{Cor: centre}, on étudie l'action de $\widehat{\phi}^\alpha_\delta(s)$ (explicité dans le lemme \ref{Lemme: Radford}) sur $\X^\pm(s')$, $1 \leq s' \leq p$, puis sur $\PIM^\pm(s')$, $1 \leq s' \leq p-1$.

Soient $\alpha' \in \{+,-\}$, $s' \in \{1,...,p\}$, et $x_0^{\alpha'}(s')$ le vecteur de plus haut poids de $\X^{\alpha'}(s')$ (cf. la proposition \ref{Prop: simples}). Pour tous $r \in \{0,...,p-1\}$ et $j \in \{0,...,2p-1\}$, on a :
\[ F^{p-1-r} K^j E^{p-1-r} x_0^{\alpha'}(s) = \delta_{r,p-1} K^j x_0^{\alpha'}(s') = \delta_{r,p-1} {\alpha'}^j q^{(s'-1)j} x_0^{\alpha'}(s'). \]
Donc :
\begin{align*}
\widehat{\phi}^\alpha_\delta(s) x_0^{\alpha'}(s') &= \delta_{s,p} \zeta (-1)^{(\delta-1)(p-1)} \alpha^{(\delta-1)p} ([p-1]!)^2 \sum_{j=0}^{2p-1} \alpha^{j+p-1} {\alpha'}^j q^{j(p+1)+j(s'-1)} \\
&= \delta_{s,p} \zeta (-1)^{(\delta-1)(p-1)} \alpha^{\delta p-1} ([p-1]!)^2 \sum_{j=0}^{2p-1} (\alpha \alpha') ^j q^{j(p+s')} \\
&= \delta_{s,s',p} \delta_{\alpha,\alpha'} \zeta (-1)^{(\delta -1)(p-1)} \alpha^{\delta p-1} 2p ([p-1]!)^2.
\end{align*}
La coordonnée de $\widehat{\phi}^\alpha_\delta(s)$ suivant $e_{s'}$ est donnée par :
\begin{align*}
a_0 &= \delta_{s,p} \delta_{\alpha,-} \zeta (-1)^{p-\delta} 2p ([p-1]!)^2, && \text{ si } s'=0, \\
a_{s'} &= 0, && \text{ si } 1 \leq s' \leq p-1, \\
a_p &= \delta_{s,p} \delta_{\alpha,+} \zeta (-1)^{(\delta-1)(p-1)} 2p ([p-1]!)^2, && \text{ si } s' = p.
\end{align*}

Soient $\alpha' \in \{+,-\}$, $s' \in \{1,...,p-1\}$, et $y_0^{\alpha'}(s')$ le vecteur de poids $\alpha' q^{s'-1}$ qui engendre $\PIM^{\alpha'}(s')$ sous $\Uq$ (cf. la proposition \ref{Prop: PIMs}). Pour tous $r \in \{0,...,p-1\}$ et $j \in \{0,...,2p-1\}$, on a :
\begin{align*}
F^{p-1-r} K^j E^{p-1-r} y_0^{\alpha'}(s')
&= \delta_{r,p-1} K^j y_0^{\alpha'}(s') + \delta_{r \leq p-2} F^{p-1-r} K^j E^{p-2-r} a_{p-s'-1}^{-\alpha'}(p-s') \\
&= \begin{multlined}[t]
	\delta_{r,p-1} {\alpha'}^j q^{(s'-1)j} y_0^{\alpha'}(s') + \delta_{s'-1 \leq r \leq p-2} (-\alpha')^{p-r} \\
	\prod_{i=1}^{p-2-r} [p-s'-i] [i] F^{p-1-r} K^j a_{r-s'+1}^{-\alpha'}(p-s') 
	\end{multlined} \\
&= \begin{multlined}[t]
	\delta_{r,p-1} {\alpha'}^j q^{(s'-1)j} y_0^{\alpha'}(s') + \delta_{s'-1 \leq r \leq p-2} (-\alpha')^{p-r+j} q^{(p+s'-2r-3)j} \\
	\prod_{i=1}^{p-2-r} [s'+i] [i] F^{p-1-r} a_{r-s'+1}^{-\alpha'}(p-s') 
	\end{multlined} \\
&= \begin{multlined}[t] 
	\delta_{r,p-1} {\alpha'}^j q^{(s'-1)j} y_0^{\alpha'}(s') + \delta_{s'-1 \leq r \leq p-2} (-1)^j (-\alpha')^{p-r+j} q^{j(s'-2r-3)} \\
	\frac{[p-2-r]![s'+p-2-r]!}{[s']!} x_0^{\alpha'}(s').
	\end{multlined}
\end{align*}
Donc :
\begin{align*}
\widehat{\phi}^\alpha_\delta(s) y_0^{\alpha'}(s')
&= \zeta (-1)^{(\delta-1)(s-1)} \alpha^{(\delta-1)p} \sum_{k=0}^{s-1} \sum_{r=s'-1}^k ([r]!)^2 {k \brack r} {s-k+r-1 \brack r} \\
&\qquad \times \frac{[p-2-r]![s'+p-2-r]!}{[s']!} \sum_{j=0}^{2p-1} (-1)^j \alpha^{j+r} (-\alpha')^{p-r+j} q^{j(s+s'-2-2k)} x_0^{\alpha'}(s')  \\
&= \begin{multlined}[t]
	\zeta (-1)^{(\delta-1)(s-1)} \alpha^{\delta p} \sum_{k=0}^{s-1} \sum_{r=s'-1}^k \frac{[k]! [s-k+r-1]! [p-2-r]![s'+p-2-r]!}{[k-r]! [s-k-1]! [s']!} \\
	\times (-\alpha \alpha')^{p-r} \sum_{j=0}^{2p-1} (\alpha \alpha')^{j} q^{j(s+s'-2-2k)} x_0^{\alpha'}(s') 
	\end{multlined} \\
&= \begin{multlined}[t]
	\zeta \delta_{\alpha,\alpha'} \delta_{s-s' \in 2 \N} (-1)^{p+(\delta-1)(s-1)} \alpha^{\delta p} 2p \frac{[\frac{s+s'}{2}-1]!}{[\frac{s-s'}{2}]! [s']!} \\
	\times \underbrace{\sum_{r=s'-1}^{\frac{s+s'}{2}-1} (-1)^r \frac{[\frac{s-s'}{2}+r]! [p-2-r]![s'+p-2-r]!}{[\frac{s+s'}{2}-1-r]!}}_{B_{\frac{s-s'}{2}}}.
	\end{multlined}
\end{align*}
Les coordonnées de $\widehat{\phi}^\alpha_\delta(s)$ suivant $w^+_{s'}$ et $w^-_{s'}$ sont respectivement données par :
\begin{equation} \tag{1}
\begin{aligned}
&b_{s'}^+ = \zeta \delta_{\alpha,+} \delta_{s-s' \in 2 \N} (-1)^{p+(\delta-1)(s-1)} 2p \frac{[\frac{s+s'}{2}-1]!}{[\frac{s-s'}{2}]! [s']!} B_{\frac{s-s'}{2}}, \\
&b_{s'}^- = \zeta \delta_{\alpha,-} \delta_{s+s'-p \in 2 \N} (-1)^{(\delta-1)(p-s-1)} 2p \frac{[\frac{s+p-s'}{2}-1]!}{[\frac{s+s'-p}{2}]! [p-s']!} B_{\frac{s+s'-p}{2}}.
\end{aligned}
\end{equation}
Soit $n \in \N$. Le coefficient $B_{n}$ vérifie :
\begin{align*}
B_n & \; := \sum_{r=s'-1}^{s'-1+n} (-1)^r \frac{[n+r]! [p-2-r]![s'+p-2-r]!}{[s'-1+n-r]!} \\
& \; \; = \sum_{r=0}^{n} (-1)^{r+s'-1} \frac{[n+r+s'-1]! [p-1-r-s']![p-1-r]!}{[n-r]!} \\
&\overset{\eqref{Eqn: qcoeff2}}{=} (-1)^{s'-1} ([p-1]!)^2 \sum_{r=0}^{n} (-1)^{r} \frac{[n+r+s'-1]!}{[n-r]! [r+s']! [r]!} \\
& \; \; = \begin{multlined}[t]
	\delta_{n,0} (-1)^{s'-1} \frac{([p-1]!)^2}{[s']} + \delta_{n\geq1} (-1)^{s'-1} \frac{([p-1]!)^2}{[n]} 
	\times \sum_{r=0}^{n} (-1)^{r} {n \brack r} {n+r+s'-1 \brack n-1} .
	\end{multlined}
\end{align*}
Le dernier terme est une identité de partition (cf. par exemple \cite{AE04}) :
\[ \sum_{r=0}^{n} (-1)^{r} {n \brack r} {n+r+s'-1 \brack n-1} = 0 . \]
Il s'ensuit que :
\begin{gather*}
B_{\frac{s-s'}{2}} = \delta_{s,s'} (-1)^{s-1} \frac{([p-1]!)^2}{[s]}, \\
B_{\frac{s+s'-p}{2}} = \delta_{p-s,s'} (-1)^{p-s-1} \frac{([p-1]!)^2}{[p-s]} = \delta_{p-s,s'} (-1)^{p-s-1} \frac{([p-1]!)^2}{[s]}.
\end{gather*}
En injectant ce résultat dans les équations $(1)$, on obtient :
\begin{align*}
&b_{s'}^+ = \zeta \delta_{\alpha,+} \delta_{s,s'} (-1)^{p+\delta(s-1)} 2p \frac{([p-1])^2}{[s]^2}, \\
&b_{s'}^- = \zeta \delta_{\alpha,-} \delta_{p-s,s'} (-1)^{\delta(p-s-1)} 2p \frac{([p-1]!)^2}{[s]^2}.
\end{align*}
D'où le résultat.
\end{proof}

\begin{Cor} 
\label{Cor: Radford}
Soit $\delta \in \{0,1\}$. L'image $\Rf_{2p}$ du morphisme $\hbphi^\delta$ \eqref{Eqn: Radford2} est le socle du centre $\Zf$ de $\Uq$.
\end{Cor}

\begin{proof}
Par construction, le groupe de Grothendieck est librement engendré par les classes des modules simples $\X^\pm(s)$, $1 \leq s \leq p$ (cf. par exemple \cite[Prop. 16.6]{CR81}). D'après la proposition \ref{Prop: Radford}, l'image $\Rf_{2p}$ du morphisme $\hbphi^\delta$ est :
\[ \Vect_\C \left( \widehat{\phi}_\delta^\pm(s) \; ; \; 1 \leq s \leq p \right) = \Vect_\C \left( e_0 , e_p , w^\pm_s \; ; \; 1 \leq s \leq p-1 \right) . \]
Or, d'après la proposition \ref{Prop: centre}, la famille $\{ w^\pm_s \; ; \; 1 \leq s \leq p-1 \}$ est une $\C$-base du radical $\Rf(\Zf)$ de $\Zf$, et la famille $\{ e_0 , e_p \} \cup \{ w^\pm_s \; ; \; 1 \leq s \leq p-1 \}$ est une $\C$-base de l'ensemble des éléments de $\Zf$ qui annulent $\Rf(\Zf)$. Donc $\Rf_{2p} = \mathrm{Ann}_\Zf(\Rf(\Zf))$ est le socle de $\Zf$ (cf. par exemple \cite[Exercice 4.1.2]{Pie82}).
\end{proof}

\begin{Rem}
\label{Rem: R2p}
D'après les précédents résultats, $\Rf_{2p}$ est un $\C$-espace vectoriel de dimension $2p$, dont des $\C$-bases sont $\{ \widehat{\phi}_\delta^\pm(s) \; ; \; 1 \leq s \leq p \}$ ou $\{ e_0 , e_p \} \cup \{ w^\pm_s \; ; \; 1 \leq s \leq p-1 \}$.
\end{Rem}
\chapter{Idempotents de Jones-Wenzl aux racines de l'unité}

Les constructions explicites des invariants de 3-variété de type Reshetikhin-Turaev reposent principalement sur la théorie des groupes quantiques \emph{quotients} $\overline{U}_q \mathfrak{g}$, où $\mathfrak{g}$ est une algèbre de Lie et $q$ une racine de l'unité, et de leurs représentations de dimension finie. Pour les groupes quantiques quotients $\Uq$, associés à l'algèbre de lie $sl_2$ et aux racines $q$ de l'unité, Lickorish offre une construction topologique alternative de ces invariants dans \cite{Lic91}, \cite{Lic92} et \cite{Lic93}, tout en gardant le procédé de chirurgie et l'étude des mouvements de Kirby (cf. par exemple \cite[§ 16, § 19]{PS97}). Dans ce cadre, l'étude des algèbres de Hopf (tressées et enrubannées) est remplacée par l'étude des \emph{modules d'écheveaux} (cf. par exemple \cite[§ 26]{PS97}), et les représentations de dimension finie par les \emph{idempotents de Jones-Wenzl} (cf. par exemple \cite[§ 27]{PS97}). En effet, la catégorie des classes d'écheveaux coloriés par les idempotents de Jones-Wenzl possède une structure analogue à celle des représentations de dimension finie d'une algèbre de Hopf modulaire (cf. par exemple \cite[§ XI-XII]{Tur94}). Ainsi, pour chaque racine $q$ paire de l'unité, on associe à toute surface fermée orientée des espaces vectoriels de classes d'écheveaux coloriés par des idempotents de Jones-Wenzl.

Introduits par Jones en 1983, ces idempotents de Jones-Wenzl se définissent pour un paramètre $A$ \emph{formel} dans les modules d'écheveaux du disque fermé, qui s'identifient canoniquement à des sous $\Z[A,A^{-1}]$-algèbres des $\C(A)$-algèbres de Temperley-Lieb (cf. par exemple \cite[§ 26]{PS97}). Ils se calculent grâce à un système de récurrence établi dans \cite{Wen87}, et correspondent à des projecteurs sur les modules indécomposables de dimension finie du groupe quantique générique $\UA$ semi-simple (cf. par exemple \cite[§ 3.5]{CFS95}). Par contre, lorsque $A^2$ s'évalue en une racine $q$ de l'unité, une large partie de ces idempotents n'est plus définie correctement, et les idempotents restants correspondent uniquement aux projecteurs sur les $\Uq$-modules \emph{simples}.

Comme dans les chapitres précédents, on concentre notre travail
sur le groupe quantique restreint $\Uq$ associé à une racine $q$ primitive \emph{paire} de l'unité. Dans ce troisième chapitre, on propose une nouvelle construction des idempotents de Jones-Wenzl basée sur l'étude des idempotents orthogonaux primitifs (POIs, cf. par exemple \cite[§ 25-26]{CR62}) des algèbres de Temperley-Lieb. Pour cela, on commence par des rappels succincts sur les espaces d'écheveaux. On détaille ensuite la structure des algèbres de Temperley-Lieb \emph{génériques} $\TL_n(A^2)$, $n \in \N$, définies pour un paramètre formel $A^2$, et celle des algèbres de Temperley-Lieb \emph{évaluées} $\TL_n(q)$, $n \in \N$, obtenues après évaluation de $A^2$ en $q$. \`A cette occasion, on s'appropriera les outils de \cite{GW93}. Les POIs résultant de ces analyses permettent enfin de définir des idempotents de Jones-Wenzl \emph{évaluables} en $q$, qui étendent les propriétés des idempotents de Jones-Wenzl usuels, et correspondent à des projecteurs sur les $\Uq$-modules simples et les $\Uq$-PIMs. 

\minitoc

\section{Rappels sur les espaces d'écheveaux}
\label{section: echeveaux}

On note $I$ \index{I@ $I$} l'intervalle $[0,1]$ et $\mathbb{S}^1$ \index{S@ $\mathbb{S}^1$} le cercle. On fixe un paramètre formel $A$\index{A@ $A$}, un entier $n \in \N$ \index{n4@ $n$} et une 3-variété $M$ compacte orientée de bord $\partial M$. On distingue deux cas :
\begin{itemize}
	\item si $\partial M$ est non vide, on le marque par $2n$ segment(s) $r_1,...,r_{2n}$ ;
	\item sinon, on prend $n=0$. 
\end{itemize}	
On commence par définir l'espace d'écheveaux $K_A(M,2n)$ associé à $(M, \partial M)$, puis on en donne une description lorsque $M$ est homéomorphe à une surface épaissie $\Sigma \times I$, où $\Sigma$ est une 2-variété compacte orientée de bord $\partial \Sigma$.

\begin{Def}[{\cite[§ 1]{BHMV92}}]
\begin{enumerate}[(i)]
	\item On appelle \emph{enchevêtrement enrubanné} de $M$ l'image $L$ de toute union finie disjointe d'anneaux $\mathbb{S}^1 \times I$ et de bandes $I \times I$ par un plongement propre dans $M$, telle que $L \cap \partial M$ est l'union disjointe des segments $r_1,...,r_{2n}$.
	\item On dit que $L$ est un \emph{entrelacs enrubanné} si toutes ses composantes sont des anneaux. En particulier, on appelle \emph{ruban} tout entrelacs enrubanné connexe.
\end{enumerate}
\end{Def}

\begin{Rem}
Autrement dit, un enchevêtrement enrubanné de $M$ est une sous-variété à bord de $M$ orientée homéomorphe à une union finie disjointe d'anneaux $\mathbb{S}^1 \times I$ et de bandes $I \times I$.
\end{Rem}

Pour tout enchevêtrement enrubanné $L$ de $M$, on désigne par \emph{classe d'isotopie} de $L$ sa classe d'équivalence modulo les isotopies dans $M$ qui fixent $\partial M$.

\begin{Def}[{\cite[§ 1]{BHMV92}}]
\begin{enumerate}[(i)]
	\item On appelle \emph{espace d'écheveaux} de $M$ le $\Z[A,A^{-1}]$- module à gauche quotient $K_A(M,2n) = V_A(M,2n) / V_A^0(M,2n)$ \index{K1@ $K_A(M,2n)$} où :
	\begin{itemize}
		\item $V_A(M,2n)$ est le $\Z[A,A^{-1}]$-module à gauche librement engendré par les classes d'isotopie d'enchevêtrements enrubannés de $M$, y compris celle de l'enchevêtrement vide ;
		\item $V_A^0(M,2n)$ est le sous-module à gauche de $V_A(M,2n)$ engendré par les éléments de la forme :
		\begin{equation} 
		\label{Eqn: echeveaux1}
		L-AL_0-A^{-1}L_\infty, \quad L \cup \bigcirc +(A^2+A^{-2})L,
		\end{equation}
		où $L$, $L_0$ et $L_\infty$ coïncident hors d'une boule $\bar{B} \subseteq M$ et ont dans $\bar{B}$ les projections sur un plan équatorial représentées dans la figure \ref{Fig: echeveaux1}, et $L \cup \bigcirc$ désigne l'union disjointe de $L$ et d'un ruban trivial $\bigcirc$ contenu dans une boule $\bar{B}$ disjointe de $L$.
	\end{itemize} 
	\item On appelle \emph{classe d'écheveau} de $M$ l'image de toute classe d'isotopie d'enchevêtrement enrubanné de $M$ dans $K_A(M,2n)$. 
\end{enumerate}
\end{Def}

\begin{figure}[!ht]
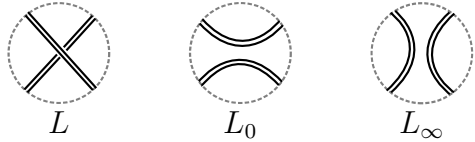

\caption{Les enchevêtrements enrubannés $L$, $L_0$ et $L_\infty$ dans $\bar{B}$.}
\label{Fig: echeveaux1}
\[ \begin{array}{ccccc}
	\KRrm && \KRhi && \KRid \\
	L && L_0 && L_\infty
	\end{array}  \]
\end{figure}

Dans la suite, pour tout enchevêtrement enrubanné $L$, on notera encore $L$ ses classes d'isotopie et d'écheveau.

\begin{Rem} 
\label{Rem: KA homeo}
Les relations d'écheveaux étant locales, l'espace d'écheveaux \\ $K_A(M,2n)$ est invariant sous les homéomorphismes de $M$ qui fixent $\partial M$. 
\end{Rem}

Il existe un produit tensoriel naturel sur les espaces d'écheveaux de deux 3-variétés dont les bords contiennent une composante connexe commune à orientation opposée. En effet, soit $\Sigma$ une composante connexe de $\partial M$ ($\Sigma = \emptyset$ si $\partial M = \emptyset$) marquée par $2k \leq 2n$ segments ($k \in \N$). On note $\overline{\Sigma}$ la variété correspondante munie de l'orientation opposée. Soient $n' \in \N$ et $M'$ une autre 3-variété compacte orientée dont le bord $\partial M'$ est marqué par $2n'$ segment(s) (avec $n'=0$ si $\partial M' = \emptyset$) et contient une composante connexe $\overline{\Sigma}$. En recollant $M$ et $M'$ le long de $\Sigma$ via un homéomorphisme $\varphi : \Sigma \rightarrow \Sigma$ préservant l'orientation, on obtient une 3-variété $M \cup_{\Sigma} M'$ compacte orientée de bord $\partial M \cup_{\Sigma} \partial M'$. Les relations d'écheveaux étant locales, cette opération de recollement induit un morphisme de $\Z[A,A^{-1}]$-modules :
\begin{equation}
\label{Eqn: KA produit1}
\begin{cases}
	K_A(M,2n) \otimes K_A(M',2n') \longrightarrow K_A(M \cup_{\Sigma} M',2n+2n'-4k) \\
	L \otimes L' \longmapsto L \cup_\Sigma L'
	\end{cases}.
\end{equation}

On suppose désormais que $M$ est homéomorphe à une \emph{surface épaissie} $\Sigma \times I$, où $\Sigma$ une 2-variété compacte orientée de bord $\partial \Sigma$. Comme précédemment, on distingue à deux cas :
\begin{itemize}
	\item si $\partial \Sigma$ est non vide, on le marque de $2n$ point(s) $R_1,...,R_{2n}$ ;
	\item sinon, on prend $n=0$.
\end{itemize}
D'après la remarque \ref{Rem: KA homeo}, on sait que $K_A(M,2n) \simeq K_A(\Sigma \times I, 2n)$. On redéfinit alors l'espace d'écheveaux $K_A(\Sigma \times I,2n)$ à l'aide de \emph{diagrammes sur $\Sigma$} comme suit.

Pour tout enchevêtrement enrubanné $L$ de $\Sigma \times I$, il existe un enchevêtrement enrubanné $L'$ de $\Sigma \times I$ isotope à $L$ dont la projection sur $\Sigma$ parallèlement à $I$ est régulière. Ainsi, on obtient un \emph{diagramme} $D_L$ sur $\Sigma$ (cf. par exemple \cite[§ 1]{PS97}) qui intersecte $\partial \Sigma$ en $\{R_1,...,R_{2n}\}$. Comme pour les enchevêtrements enrubannés, pour tout diagramme $D$ sur $\Sigma$, on désigne par \emph{classe d'isotopie} de $D$ sa classe d'équivalence modulo les isotopies dans $\Sigma$ qui fixent $\partial \Sigma$. Alors les classes d'isotopie d'enchevêtrements enrubannés de $\Sigma \times I$ sont en bijection avec les classes d'isotopie de diagrammes sur $\Sigma$ modulo les \emph{mouvements de Reidemeister} $\Omega_1', \Omega_2$, $\Omega_3$ représentés dans la figure \ref{Fig: Reidemeister} (cf. par exemple \cite[§ 19.6-7]{PS97}). 
\begin{figure}[!ht]
\caption{Mouvements de Reidemeister.}
\label{Fig: Reidemeister}
\[ \begin{array}{ccccc}
	\vcenter{\hbox{ 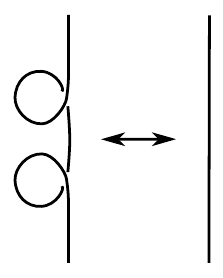 }} & \hspace*{1cm} & \vcenter{\hbox{ 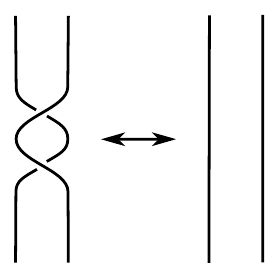 }} & \hspace*{1cm} & \vcenter{\hbox{ 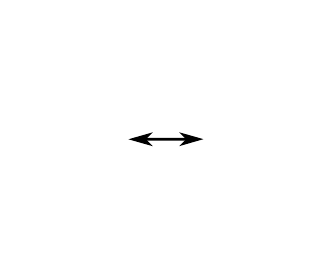 }} 
	\end{array} \]
\end{figure}
Autrement dit, on a un isomorphisme de $\Z[A,A^{-1}]$-modules à gauche :
\[ \begin{cases}
	V_A(\Sigma \times I, 2n) \overset{\sim}{\longrightarrow} V_A(\Sigma, 2n) / V_A^\Rf(\Sigma, 2n) \\
	L \longmapsto D_L \mod (\Omega_1', \Omega_2, \Omega_3) 
	\end{cases} \]
où :
\begin{itemize}
	\item $V_A(\Sigma \times I,2n)$ est le $\Z[A,A^{-1}]$-module à gauche librement engendré par les classes d'isotopie d'enchevêtrements enrubannés de $\Sigma \times I$, y compris celle de l'enchevêtrement vide ;
	\item $V_A(\Sigma,2n)$ est le $\Z[A,A^{-1}]$-module à gauche librement engendré par les classes d'isotopie de diagrammes sur $\Sigma$, y compris celle du diagramme vide ;
	\item $V_A^\Rf(\Sigma,2n)$ est le sous-module à gauche de $V_A(\Sigma,2n)$ engendré par les relations (locales) de Reidemeister.
\end{itemize}
Comme les mouvements de Reidemeister passent au quotient sous les relations d'écheveaux \eqref{Eqn: echeveaux1} (cf. par exemple \cite[Thm 26.4]{PS97}), il induit un isomorphisme d'espaces d'écheveaux :
\[ K_A(\Sigma \times I, 2n) \overset{\sim}{\longrightarrow} K_A(\Sigma, 2n), \]
où $K_A(\Sigma,2n)$ est défini ci-après. On représentera implicitement les classes d'écheveaux de $\Sigma \times I$ par leurs classes d'écheveaux de diagrammes sur $\Sigma$ via cette identification.

\begin{Def}[{\cite[§ 26.3]{PS97}}]
\begin{enumerate}[(i)]
	\item On appelle \emph{espace d'écheveaux} de $\Sigma$ le $\Z[A,A^{-1}]$-module à gauche quotient $K_A(\Sigma,2n) = V_A(\Sigma,2n) / V_A^0(\Sigma,2n)$ \index{K2@ $K_A(\Sigma,2n)$} où :
	\begin{itemize}
		\item $V_A(\Sigma,2n)$ est le $\Z[A,A^{-1}]$-module à gauche librement engendré par les classes d'isotopie de diagrammes sur $\Sigma$, y compris celle du diagramme vide ;
		\item $V_A^0(\Sigma,2n)$ est le sous-module à gauche de $V_A(\Sigma,2n)$ engendré par les éléments de la forme :
		\begin{equation} 
		\label{Eqn: echeveaux2}
		D-AD_0-A^{-1}D_\infty, \quad D \cup \bigcirc +(A^2+A^{-2})D,
		\end{equation}
		où $D$, $D_0$ et $D_\infty$ coïncident hors d'une boule $\bar{B} \subseteq \Sigma$ et sont dans $\bar{B}$ représentées dans la figure \ref{Fig: echeveaux2}, et $D \cup \bigcirc$ désigne l'union disjointe de $D$ et du diagramme trivial $\bigcirc$ contenu dans une boule $\bar{B}$ disjointe de $D$.
	\end{itemize} 
	\item On appelle \emph{classe d'écheveau} de $\Sigma$ l'image de toute classe d'isotopie de diagramme sur $\Sigma$ dans $K_A(\Sigma,2n)$. 
\end{enumerate}
\end{Def}

\begin{figure}[!ht]
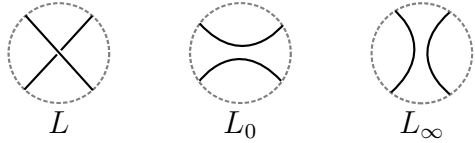

\caption{Les diagrammes $D$, $D_0$ et $D_\infty$ dans $\bar{B}$.}
\label{Fig: echeveaux2}
\[ \begin{array}{ccccc}
	\Krm && \Khi && \Kid \\
	L && L_0 && L_\infty
	\end{array}  \]
\end{figure}

Dans la suite, pour tout diagramme $D$, on notera encore $D$ ses classes d'isotopie et d'écheveau.

\begin{Prop}[{\cite[Thm 26.5]{PS97}}]
\label{Prop: multi-courbes}
L'espace d'écheveaux $K_A(\Sigma,2n)$ est un module à gauche de $\Z[A,A^{-1}]$ libre, dont une $\Z[A,A^{-1}]$-base est donnée par les classes d'écheveaux de diagrammes sur $\Sigma$ ne contenant ni croisement, ni courbe fermée contractile. On appelle ces diagrammes les \emph{multi-courbes} de $\Sigma$.
\end{Prop}

De même que précédemment, il existe un produit naturel sur les espaces d'écheveaux de deux surfaces épaissies dont les bords contiennent un composante connexe commune à orientation opposée. Dans ce cadre, le produit tensoriel \eqref{Eqn: KA produit1} devient :
\begin{equation}
\label{Eqn: KA produit2}
\begin{cases}
	K_A(\Sigma,2n) \otimes K_A(\Sigma',2n') \longrightarrow K_A(\Sigma \cup_{C} \Sigma',2n+2n'-4k) \\
	D \otimes D' \longmapsto D \cup_C D'
	\end{cases}
\end{equation}
où $C$ est une composante connexe de $\partial \Sigma$ marquée par $2k \leq 2n$ point(s) ($k \in \N$), $n' \in \N$, et $\Sigma'$ est une autre 2-variété  compacte orientée dont le bord $\partial \Sigma'$ est marqué par $2n'$ point(s) (avec $n=0$ si $\partial \Sigma' = \emptyset$) et contient une composante connexe $\overline{C}$.

\section{Algèbres de Temperley-Lieb et idempotents}
\label{section: TL et POIs}

Dans cette section, on s'intéresse à l'espace d'écheveaux $K_A(\bar{D},2n)$ du disque fermé $\bar{D}$\index{D@ $\bar{D}$}. On commence par établir le lien entre le $\Z[A,A^{-1}]$-module $K_A(\bar{D},2n)$ et la $\C(A)$-algèbre de Temperley-Lieb \emph{générique} $\TL_n(A^2)$, définie pour un paramètre $A$ formel. On donne ensuite des rappels sur la structure de cette algèbre, que l'on pourra trouver dans \cite[§ 5]{KT08} avec d'autres notations. Après évaluation de $A$ en un nombre complexe $\zeta$ qui est racine $4p$-ième de l'unité ($p \in \N^*$), on obtient une nouvelle algèbre de Temperley-Lieb \emph{évaluée} $\TL_n(\zeta^2)$ dont la structure n'est plus la même. On illustre ce cas en évaluant $A^2$ en une racine $2p$-ième de l'unité $q := e^{\frac{i \pi}{p}}$, et on explicite la structure de l'algèbre de Temperley-Lieb évaluée $\TL_n(q)$ correspondante. Pour cela, on reformulera une partie du travail de \cite{GW93} avec la normalisation des algèbres de Temperley-Lieb de \cite{CFS95}.

\subsection{Algèbres de Temperley-Lieb et évaluation}
\label{subsection: TL}

L'espace d'écheveaux $K_A(\bar{D},2n)$ est muni d'une structure de $\Z[A,A^{-1}]$-algèbre unitaire en identifiant $\bar{D}$ avec $I^2$, de sorte que $I \times \{0\}$ et $I \times \{1\}$ contiennent tout deux $n$ point(s) marqué(s). Une telle identification n'est évidemment pas unique : elle dépend du choix de la répartition des points sur les composantes de bord $I \times \{0\}$ et $I \times \{1\}$. Toutefois, d'après la remarque \ref{Rem: KA homeo}, les espaces d'écheveaux $K_A(\bar{D},2n)$ et $K_A(I^2, 2n)$ sont isomorphes. 

Le produit de deux classes d'écheveaux $L_1$ et $L_2$ consiste à identifier les $n$ point(s) $L_1 \cap (I \times \{1\})$ du bord de $L_1$ avec les $n$ point(s) $L_2 \cap (I \times \{0\})$ du bord de $L_2$ :
\begin{equation}
\label{Eqn: produit en pile}
\vcenter{\hbox{\begin{tikzpicture}
	\draw (0,0) node{$\begin{array}{c} |\cdots| \\  L_1 \\ \underbrace{| \cdots |}_{n} \end{array}$};
	\draw (-0.5,-0.48) rectangle (0.5,0.96);
	\draw (0,1.5) node{};
	\end{tikzpicture}}}
\cdot \vcenter{\hbox{\begin{tikzpicture}
	\draw (0,0) node{$\begin{array}{c} |\cdots| \\  L_2 \\ \underbrace{| \cdots |}_{n} \end{array}$};
	\draw (-0.5,-0.48) rectangle (0.5,0.96);
	\draw (0,1.5) node{};
	\end{tikzpicture}}}
= \vcenter{\hbox{\begin{tikzpicture}
	\draw (0,0) node{$\begin{array}{c} |\cdots| \\  L_2 \\ L_1 \\ \underbrace{|\cdots|}_{n} \end{array}$};
	\draw (-0.5,-0.73) rectangle (0.5,1.21);
	\draw (0,1.5) node{};
	\end{tikzpicture}}} .
\end{equation}
En tant que $\Z[A,A^{-1}]$-algèbre, $K_A(\bar{D},2n)$ est librement engendrée par les classes d'écheveaux de la forme :
\[ h_0:= \vcenter{\hbox{\begin{tikzpicture}
	\draw (0,0) node{$\underbrace{\Big| \cdots \Big|}_{n}$};
	\draw (-0.5,-0.17) rectangle (0.5,0.69);
	\draw (0,1) node{};
	\end{tikzpicture}}}, 
\quad h_i:= \vcenter{\hbox{\begin{tikzpicture}
	\draw (0,0) node{$\underbrace{\Big| \cdots \Big|}_{i-1}  \bigcupcap \underbrace{\Big| \cdots \Big|}_{n-i-1}$};
	\draw (-1.34,-0.18) rectangle (1.34,0.67);
	\draw (0,1) node{};
	\end{tikzpicture}}},
\quad 1 \leq i \leq n-1, \]
(cf. par exemple \cite[Thm 26.10]{PS97}). Ces générateurs fournissent un isomorphisme explicite entre $K_A(\bar{D},2n)$ et une sous $\Z[A,A^{-1}]$-algèbre de la $\C(A)$-algèbre générique $\TL_n(A^2)$ définie ci-dessous. Pour celle-ci, on choisit la normalisation quadratique (avec $A^2$) qui découle naturellement des relations d'écheveaux \eqref{Eqn: echeveaux2}. On pourra trouver une normalisation non quadratique dans \cite[§ 5.7]{KT08} ou \cite[§ 2.4]{CFS95} par exemple.

\begin{Def}
Pour $n \in \N^*$, l'\emph{algèbre de Temperley-Lieb} $\TL_n(A^2)$ \index{TLm@ $\TL_n(A^2)$} est la $\C(A)$-algèbre unitaire engendrée par $h_0:=1$, $h_1$,...,$h_{n-1}$ sous les relations :
\begin{equation}
\label{Eqn: TLA}
\begin{aligned}
&\forall i,j \in \{1,...,n-1\} \mbox{ tels que } |i-j|=1
	&\quad& h_i h_j h_i = h_i, \\
&\forall i,j \in \{1,...,n-1\} \text{ tels que } |i-j| \geq 2
	&& h_i h_j = h_j h_i, \\
&\forall i \in \{1,...,n-1\}
	&& h_i^2 = -(A^2+A^{-2}) h_i.
\end{aligned}
\end{equation}
Pour $n=0$, on pose $\TL_0(A^2)=\C(A)$ par convention.
\end{Def}

On utilise les notations standards pour les coefficients $A^2$-entiers : \index{n5@ $[n]_{A^2}$}
\begin{equation}
\label{Eqn: Acoeff1}
\forall n \in \N \qquad [n]_{A^2} := \frac{A^{2n}-A^{-2n}}{A^2-A^{-2}} .
\end{equation}
Ainsi, dans l'algèbre de Temperley-Lieb $\TL_n(A^2)$, on a :
\[ \forall i \in \{1,...,n-1\} \qquad h_i^2 = -[2]_{A^2} h_i . \]
On rappelle que les coefficients $A^2$-entiers vérifient :
\begin{equation}
\label{Eqn: Acoeff2}
\begin{aligned}
&\forall n \in \N && [n+1]_{A^2} + [n-1]_{A^2} = [2]_{A^2} [n]_{A^2}, \\
&\forall n,m \in \N && [n]_{A^2} [m+1]_{A^2} - [n+1]_{A^2} [m]_{A^2} = [n-m]_{A^2}.
\end{aligned}
\end{equation} 

D'autre part, pour tout $k \in \N$, on a une injection canonique de $\TL_n(A^2)$ dans $\TL_{n+k}(A^2)$ par identification des générateurs $h_0,...,h_n$.  Au niveau des classes d'écheveaux, cette injection se traduit par :
\begin{equation}
\label{Eqn: injections de TL}
\rect{L_1}{} \longmapsto
		\vcenter{\hbox{\begin{tikzpicture}
		\draw (0,0) node{$L_1 \underbrace{\Big| \cdots \Big|}_{k}$};
		\draw (-0.8,-0.15) rectangle (0.8,0.7);
		\draw (0,1.4) node{};
		\end{tikzpicture}}}.
\end{equation}
On identifiera tacitement les éléments de $\TL_n(A^2)$ avec ceux de $\TL_{n+k}(A^2)$ via cette injection.

L'algèbre de Temperley-Lieb évaluée $\TL_n(q)$ \index{TLn@ $\TL_n(q)$} est la $\C$-algèbre définie de manière analogue en remplaçant $A^2$ par $q = e^\frac{i \pi}{p}$. Elle est liée à l'algèbre de Temperley-Lieb générique $\TL_n(A^2)$ par un morphisme d'évaluation explicité ci-après.

\begin{Def}
On note $\C[A]_{(A^2-q)}$ le localisé de l'anneau polynomial $\C[A]$ en l'idéal engendré par $(A^2-q)$. Pour $n \in \N^*$, l'\emph{algèbre de Temperley-Lieb} $\TL_n(A^2)_q$ est la $\C[A]_{(A^2-q)}$-algèbre unitaire engendrée par $h_0:=1$, $h_1$,...,$h_{n-1}$ sous les relations :
\begin{align*}
&\forall i,j \in \{1,...,n-1\} \mbox{ tels que } |i-j|=1
	&& h_i h_j h_i = h_i, \\
&\forall i,j \in \{1,...,n-1\} \text{ tels que } |i-j| \geq 2
	&& h_i h_j = h_j h_i, \\
&\forall i \in \{1,...,n-1\}
	&& h_i^2 = -(A^2+A^{-2}) h_i.
\end{align*}
Pour $n=0$, on pose $\TL_0(A^2)_q=\C[A]_{(A^2-q)}$ par convention.
\end{Def}

\begin{Prop}[{\cite[Prop. 0.1]{GW93}}]
\label{Prop: evaluations}
Soit le morphisme d'évaluation défini par :
	\[ ev_n : \begin{cases} 
			\TL_n(A^2)_q \longrightarrow \TL_n(q) \\
			\sum_{i}^k P_i(A) w_i \longmapsto \sum_{i}^k P_i(q^{\frac{1}{2}}) w_i
			\end{cases} \]
où, pour tout $i \in \{1,...,k\}$, $w_i$ est un mot en $\{h_0,...,h_{n-1}\}$. Alors $ev_n$ est surjectif.
\end{Prop}

\begin{Def}
On dit que $u \in \TL_n(A^2)$ est \emph{évaluable} si $u \in \TL_n(A^2)_q$. Dans ce cas, on note $\bar{u}:=ev_n(u)$ l'évaluation de $u$ dans $\TL_n(q)$.
\end{Def}

\subsection{Tableaux standards et représentations}
\label{subsection: TLA}

On suppose désormais que $n \geq 2$ ; on écarte les cas $n=0$ et $n=1$ pour lesquels $\TL_n(A^2) \cong \C(A)$. La $\C(A)$-algèbre $\TL_n(A^2)$ est de dimension finie (cf. par exemple \cite[Prop. 5.28]{KT08}). D'après le théorème de Krull-Schmidt (cf. par exemple \cite[Thm 14.5]{CR62}), elle se décompose de manière unique, à isomorphisme et ordre des facteurs près, en somme directe d'idéaux indécomposables de $\TL_n(A^2)$, appelés \emph{facteurs directs} de $\TL_n(A^2)$. Afin de les expliciter, on utilise les diagrammes de Young à $n$ cases et, pour chacun de ces diagrammes $\lambda$, on associe une représentation $V_\lambda$ de $\TL_n(A^2)$. Enfin, ces facteurs directs correspondent à des idempotents (orthogonaux) centraux primitifs (PCIs\index{PCIs@ PCIs}) de $\TL_n(A^2)$ (cf. par exemple \cite[§ 25, § 54]{CR62}). On en donne une description à partir d'idempotents orthogonaux primitifs (POIs\index{POIs@ POIs}, non centraux) de $\TL_n(A^2)$, associés aux tableaux standards à $n$ cases.

\begin{Def}[{\cite[§ 5.1.2, § 5.2.2]{KT08}}]
Soient $\lambda_1 \geq \lambda_2 \geq ... \geq \lambda_k$ tels que $\lambda = (\lambda_1,..., \lambda_k)$ est une partition de $n$.
\begin{enumerate}[(i)]
	\item On appelle \emph{diagramme de Young} de $\lambda$ la collection $[\lambda_1,..., \lambda_k]$ de $n$ cases justifiées à gauche telle que, pour tout ligne $i \in \{1,...,k\}$, on a $\lambda_i \in \N$ cases :
	\[ [\lambda_1,..., \lambda_k] = {\scriptstyle \begin{array}{|c|c|c|c|c|c|}
		\hline 
		\textcolor{gray} 1 & \textcolor{gray} 2 & \multicolumn{3}{c|}{\cdots} & \textcolor{gray}{\lambda_1} \\
		\hline 
		\textcolor{gray} 1 & \textcolor{gray} 2 & \multicolumn{2}{c|}{\cdots} & \textcolor{gray}{\lambda_2} \\
		\cline{1-5} 
		\vdots & \vdots & \multicolumn{2}{c|}{\vdots} \\	
		\cline{1-4}
		\textcolor{gray} 1 & \textcolor{gray} 2 & \cdots & \textcolor{gray}{\lambda_k} \\
		\cline{1-4} 	
		\end{array}} \]
	On note $\Delta_n$ \index{Delta2@ $\Delta_n$} l'ensemble des diagrammes de Young à $n$ cases et $\Delta^{\leq 2}_n$ \index{Delta3@ $\Delta_n^{\leq2}$} le sous-ensemble de $\Delta_n$ des diagrammes de Young à $n$ cases avec au plus deux lignes.
	\item On appelle \emph{treillis de Young} le graphe orienté dont les sommets sont étiquetés par les diagrammes de Young, y compris les diagrammes $\emptyset$ et $\lyng{1}$, tel que tout sommet étiqueté par $\nu$ est relié au sommet étiqueté par $\mu$ si $\nu \subset \mu$ et que $\mu$ possède exactement une case de plus que $\nu$ (cf. la figure \ref{Fig: treillis de Young}).
\end{enumerate}
\end{Def}

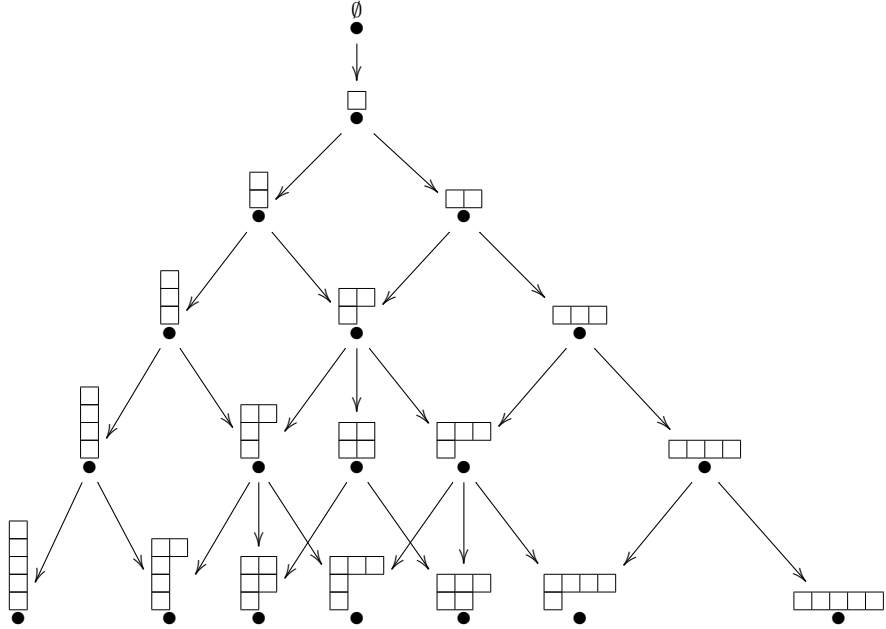
\begin{figure}[!ht]
\caption{Treillis de Young.}
\label{Fig: treillis de Young}
\[ \Yboxdim{7pt} 
\xymatrix @-2ex {
	&&&& \overset{\emptyset}{\bullet} \ar[d] \\
	&&&& \overset{\yng(1)}{\bullet} \ar[ld] \ar[rd] \\
	&&& \overset{\yng(1,1)}{\bullet} \ar[ld] \ar[rd] && \overset{\yng(2)}{\bullet} \ar[ld] \ar[rd] \\
	&& \overset{\yng(1,1,1)}{\bullet} \ar[ld] \ar[rd] && \overset{\yng(2,1)}{\bullet} \ar[ld] \ar[d] \ar[rd] && \overset{\yng(3)}{\bullet} \ar[ld] \ar[rd] \\
	& \overset{\yng(1,1,1,1)}{\bullet} \ar[ld] \ar[rd] && \overset{\yng(2,1,1)}{\bullet} \ar[ld] \ar[d] \ar[rd] & \overset{\yng(2,2)}{\bullet} \ar[ld] \ar[rd] & \overset{\yng(3,1)}{\bullet} \ar[ld] \ar[d] \ar[rd] && \overset{\yng(4)}{\bullet} \ar[ld] \ar[rd] \\
 	\overset{\yng(1,1,1,1,1)}{\bullet} && \overset{\yng(2,1,1,1)}{\bullet} & \overset{\yng(2,2,1)}{\bullet} & \overset{\yng(3,1,1)}{\bullet} & \overset{\yng(3,2)}{\bullet} & \overset{\yng(4,1)}{\bullet} && \overset{\yng(5)}{\bullet} } \]
\end{figure}

Il y a une bijection canonique entre l'ensemble des partitions de $n$ et l'ensemble $\Delta_n$ des diagrammes de Young à $n$ cases. Dans la suite, pour toute partition $\lambda$, on notera encore $\lambda$ son diagramme de Young associé.

\begin{Def}[{\cite[§ 5.1.4, § 5.1.5]{KT08}}]
Soit $\lambda \in \Delta_n$ et $\sigma$ une permutation de $\{1,...,n\}$.
\begin{enumerate}[(i)]
	\item On appelle \emph{tableau} de forme $\lambda$ tout diagramme de Young $\lambda$ muni d'une bijection $T : \{ \text{cases de } \lambda \} \rightarrow \{1,...,n\}$. La fonction $T$ est appelée l'\emph{étiquetage} de $\lambda$ et ses valeurs sont appelées les \emph{étiquettes} des cases correspondantes ; on les inscrits à l'intérieur de celles-ci.
	\item Pour tout tableau $t$ de forme $\lambda$ et d'étiquetage $T$, on note $\sigma(t)$ le tableau de forme $\lambda$ dont l'étiquetage est $\sigma \circ T$.
	\item On dit qu'un tableau est \emph{standard} si ses étiquettes sont croissantes de gauche à droite sur chaque ligne, et de haut en bas sur chaque colonne. On note $T_n$ \index{T1@ $T_n$} (resp. $T_\lambda$\index{T2@ $T_\lambda$}) l'ensemble des tableaux standards à $n$ cases (resp. de forme $\lambda$) et $T^{\leq 2}_n$ \index{T3@ $T_n^{\leq2}$} le sous-ensemble de $T_n$ des tableaux standards dont la forme est un diagramme de Young dans $\Delta_n^{\leq 2}$.
	\item Pour tout tableau standard $t$, on appelle \emph{graphe} de $t$ le sous-graphe orienté du treillis de Young : \index{gamma@ $\gamma(t)$}
	\[ \gamma(t):= \xymatrix @-2ex {
	\overset{\emptyset}{\bullet} \ar[r] &
	\overset{\lambda^{(1)}}{\bullet} \ar[r] &
		\overset{\lambda^{(2)}}{\bullet} \ar[r] &
		\cdots \ar[r] &
		\overset{\lambda^{(n)}}{\bullet} } \]
	où, pour tout $i \in \{1,...,n\}$, $\lambda^{(i)}$ est le diagramme de Young contenant les étiquettes $1,...,i$ de $t$ (cf. la figure colorée \ref{Fig: graphes de tableaux}).
\end{enumerate}
\end{Def}

\begin{figure}[!ht]
\caption{Exemples de graphes de tableaux.}
\label{Fig: graphes de tableaux}
\[ \Yboxdim{7pt}
\xymatrix @-2ex {
	&&&& \overset{\emptyset}{\bullet} \ar@[red]@<-0.3ex>[d] \ar@[blue]@<0.3ex>[d] \\
	&&&& \overset{\yng(1)}{\bullet} \ar@[red][ld] \ar@[blue][rd] \\
	&&& \overset{\yng(1,1)}{\bullet} \ar[ld] \ar@[red][rd] && \overset{\yng(2)}{\bullet} \ar[ld] \ar@[blue][rd] \\
	&& \overset{\yng(1,1,1)}{\bullet} \ar[ld] \ar[rd] && \overset{\yng(2,1)}{\bullet} \ar[ld] \ar@[red][d] \ar[rd] && \overset{\yng(3)}{\bullet} \ar@[blue][ld] \ar[rd] \\
	& \overset{\yng(1,1,1,1)}{\bullet} \ar[ld] \ar[rd] && \overset{\yng(2,1,1)}{\bullet} \ar[ld] \ar[d] \ar[rd] & \overset{\yng(2,2)}{\bullet} \ar@[red][ld] \ar[rd] & \overset{\yng(3,1)}{\bullet} \ar[ld] \ar[d] \ar@[blue][rd] && \overset{\yng(4)}{\bullet} \ar[ld] \ar[rd] \\
 	\overset{\yng(1,1,1,1,1)}{\bullet} && \overset{\yng(2,1,1,1)}{\bullet} & \underset{\color{red} \gamma \left( \vcenter{ \hbox{ \tiny \young(13,24,5) }} \right)}{\overset{\yng(2,2,1)}{\bullet}} & \overset{\yng(3,1,1)}{\bullet} & \overset{\yng(3,2)}{\bullet} & \underset{\color{blue} \gamma \left( \vcenter{ \hbox{ \tiny \young(1235,4) }} \right)}{\overset{\yng(4,1)}{\bullet}} && \overset{\yng(5)}{\bullet} } \]
\end{figure}
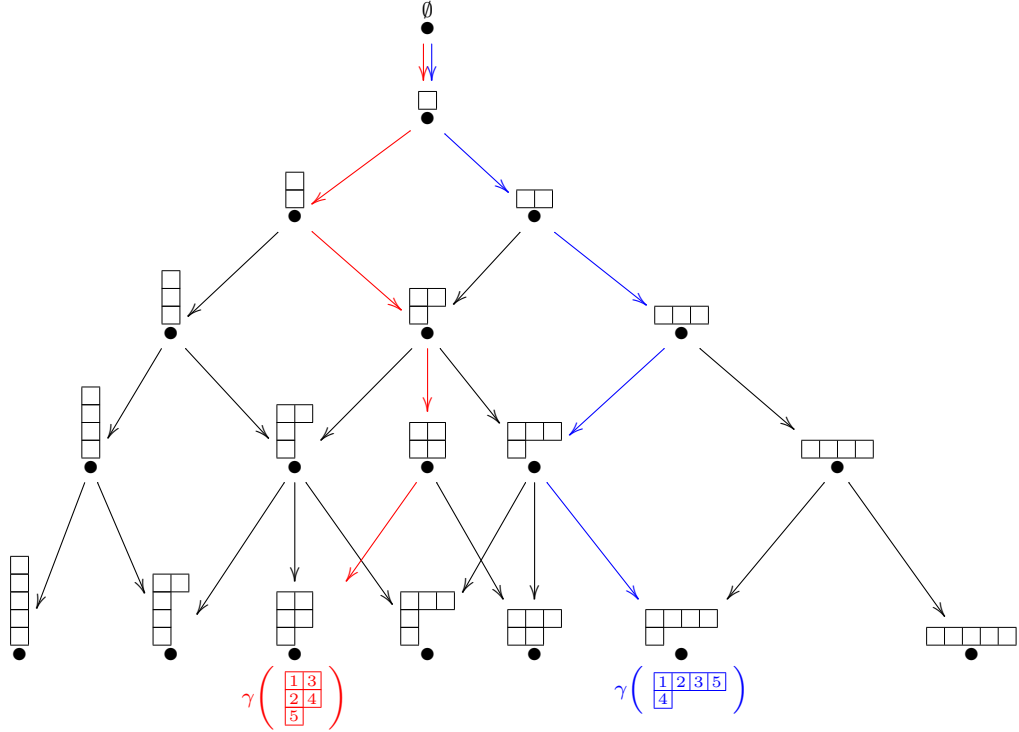

Il y a une bijection canonique entre l'ensemble $T_n$ des tableaux standards de taille $n$ et l'ensemble des sous-graphes orientés à $n-1$ arêtes du Treillis de Young. Cette représentation graphique sera utile dans la suite.

\begin{Def}[{\cite[§ 5.1.6]{KT08}}]
Soit $t \in T_n$. Pour tout $i \in \{1,...,n-1\}$, on définit la \emph{distance axiale} entre $i$ et $i+1$ dans $t$ par : \index{d @$d_t(i)$}
\[ d_t(i) := c(i) - c(i+1) + r(i+1) - r(i) \] 
où, pour tout $j \in \{1,...,n\}$, le couple $(r(j),c(j))$ désigne les coordonnées de la case de $t$ étiquetée par $j$.
\end{Def}

\begin{Prop} 
\label{Prop: modules de TLA}
Soit $\lambda \in \Delta_n^{\leq 2}$. On considère le $\C(A)$-espace vectoriel $V_\lambda$ \index{V @$V_\lambda$} librement engendré par $\left\{ v_t \; ; \; t \in T_\lambda \right\}$.
\begin{enumerate}[(i)]
	\item Le $\C(A)$-espace vectoriel $V_\lambda$ est muni d'une structure de $\TL_n(A^2)$-module à gauche par :
\[ h_i \cdot v_t = -\frac{[d_t(i)+1]_{A^2}}{[d_t(i)]_{A^2}} v_t - (1-\delta_{d_t(i),-1} ) \frac{[d_t(i)-1]_{A^2}}{[d_t(i)]_{A^2}} v_{\sigma_i(t)}, \]
où $i \in \{1,...,n-1\}$, $t \in T_\lambda$ et $\sigma_i$ désigne la transposition $(i,i+1)$. 
	\item Il existe un isomorphisme canonique de $\TL_{n-1}(A^2)$-modules à gauche :
\[ {V_\lambda}_{\vert \TL_{n-1}(A^2)} \cong \bigoplus_{\mu \in \Delta_{n-1} \; ; \; \mu \subset \lambda} V_\mu . \]
\end{enumerate}
Les modules $V_\lambda$, $\lambda \in \Delta_n^{\leq 2}$, sont appelés \emph{représentations semi-normales} de $\TL_n(A^2)$.
\end{Prop}

\begin{Rem}
\label{Rem: modules de TLA}
\begin{enumerate}[(i)]
	\item Soient $t \in T_n$ et $i \in \{1,...,n-1\}$. Le tableau $\sigma_i(t)$, obtenu en échangeant les étiquettes $i$ et $i+1$ dans le tableau $t$, est standard si et seulement si les étiquettes $i$ et $i+1$ ne sont pas sur la même ligne ou sur la même colonne. Or, comme $t$ est standard, les étiquettes $i$ et $i+1$ sont sur la même ligne (resp. sur la même colonne) si et seulement si $d_t(i)=-1$ (resp. $d_t(i)=1$). L'action de $\TL_n(A^2)$ sur le $\C(A)$-espace vectoriel $V_\lambda$ est donc bien définie.
	\item Soit $\lambda = [\lambda_1, \lambda_2] \in \Delta_n^{\leq 2}$. Il existe au plus deux diagrammes de Young $\mu$ à $n-1$ case(s) tels que $\mu \subset \lambda$ : un diagramme $\mu^l = [\lambda_1-1, \lambda_2]$ si $\lambda_1 > \lambda_2$, un diagramme $\mu^r = [\lambda_1, \lambda_2-1]$ si $\lambda_2 > 0$ (cf. la figure \ref{Fig: treillis de Young} ou \ref{Fig: treillis de TL}). De plus, ces sous-diagrammes de $\lambda$ ont au plus deux lignes et ils caractérisent $\lambda$ dès que $n \geq 3$.
\end{enumerate}
\end{Rem}

\begin{proof}

\begin{enumerate}[(i)]
	\item Il s'agit de vérifier que les endomorphismes de $V_\lambda$ définis par $v_t \mapsto h_i \cdot v_t$, $1 \leq i \leq n-1$, vérifient les relations \eqref{Eqn: TLA} de l'algèbre de Temperley-Lieb. Soient $t \in T_\lambda$ et $i,j \in \{1,...,n-1\}$. On pose $d:= d_t(i)$ et $e:= d_t(j)$. On observe que $d_{\sigma_i(t)}(i)=-d$.
	
	Pour la dernière relation, on obtient :
	\begin{align*}
	h_i^2 \cdot v_t &\; \; = h_i \cdot \left( h_i \cdot v_t \right) \\
	&\; \; = \begin{multlined}[t]
		\left( \frac{[d+1]_{A^2}^2}{[d]_{A^2}^2} + \frac{[d-1]_{A^2} [-d-1]_{A^2}}{[d]_{A^2} [-d]_{A^2}} \right) v_t \\
		+ \left( \frac{[d+1]_{A^2} [d-1]_{A^2}}{[d]_{A^2}^2} + (1-\delta_{d,-1}) \frac{[d-1]_{A^2} [-d+1]_{A^2}}{[d]_{A^2} [-d]_{A^2}} \right) v_{\sigma_i(t)} 
		\end{multlined} \\
	&\; \; = \begin{multlined}[t]
		\frac{[d+1]_{A^2}^2 + [d-1]_{A^2} [d+1]_{A^2}}{[d]_{A^2}^2} v_t \\
		+ (1-\delta_{d,-1}) \frac{[d+1]_{A^2} [d-1]_{A^2} + [d-1]_{A^2}^2}{[d]_{A^2}^2} v_{\sigma_i(t)} 
		\end{multlined} \\
	&\overset{\eqref{Eqn: Acoeff2}}{=} [2]_{A^2} \frac{[d+1]_{A^2}}{[d]_{A^2}} v_t + [2]_{A^2} (1-\delta_{d,-1}) \frac{[d-1]_{A^2}}{[d]_{A^2}} v_{\sigma_i(t)} .
	\end{align*}
	Donc $h_i^2 \cdot v_t = -[2]_{A^2} h_i \cdot v_t$.
	
	Pour la deuxième relation, on suppose que $|i-j|\geq 2$. Dans ce cas, on observe que $d_{\sigma_i(t)}(j)=d$. On obtient :
	\begin{align*}
	(h_i h_j) \cdot v_t &= h_i \cdot \left( h_j \cdot v_t \right) \\
	&= \begin{multlined}[t]
		\frac{[e+1]_{A^2} [d+1]_{A^2}}{[e]_{A^2}[d]_{A^2}} v_t + (1-\delta_{d,-1}) \frac{[e+1]_{A^2} [d-1]_{A^2}}{[e]_{A^2} [d]_{A^2}} v_{\sigma_i(t)} \\
		+  (1-\delta_{e,-1}) \frac{[e-1]_{A^2} [d+1]_{A^2}}{[e]_{A^2}[d]_{A^2}} v_{\sigma_j(t)} \\
		+  (1-\delta_{e,-1})  (1-\delta_{d,-1}) \frac{[e-1]_{A^2} [d-1]_{A^2}}{[e]_{A^2} [d]_{A^2}} v_{\sigma_i \sigma_j (t)} .
		\end{multlined}
	\end{align*}
	Donc $(h_i h_j) \cdot v_t = (h_j h_i) \cdot v_t$ par symétrie du résultat en $(i,j)$.
	
	Pour le première relation, on suppose que $|i-j|=1$. Dans ce cas, on observe que $d_{\sigma_j(t)}(i)=d+e=d_{\sigma_i(t)}(j)$ et $d_{\sigma_j \sigma_i(t)}(i)=e$. On obtient :
	\begin{align*}
	(h_i h_j h_i) \cdot v_t &= h_i \cdot \left( h_j \cdot \left( h_i \cdot v_t \right) \right) \notag \notag \\
	&= \begin{multlined}[t]
		- \left( \frac{[d+1]_{A^2}^2 [e+1]_{A^2}}{[d]_{A^2}^2 [e]_{A^2}} + \frac{[d-1]_{A^2} [d+e+1]_{A^2} [-d-1]_{A^2}}{[d]_{A^2} [d+e]_{A^2} [-d]_{A^2}} \right) v_t \tag{1} \\
		- \Bigg( \frac{[d+1]_{A^2} [e+1]_{A^2} [d-1]_{A^2}}{[d]_{A^2}^2 [e]_{A^2}} \notag \\
		+ (1-\delta_{d,-1}) \frac{[d-1]_{A^2} [d+e+1]_{A^2} [-d+1]_{A^2}}{[d]_{A^2} [d+e]_{A^2} [-d]_{A^2}} \Bigg) v_{\sigma_i(t)} \notag \\
		- (1-\delta_{e,-1}) \frac{[d+1]_{A^2} [e-1]_{A^2} [d+e+1]_{A^2}}{[d]_{A^2} [e]_{A^2} [d+e]_{A^2}} v_{\sigma_j(t)} \notag \\
		- (1-\delta_{e,-1})(1-\delta_{d+e,-1}) \frac{[d+1]_{A^2} [e-1]_{A^2} [d+e-1]_{A^2}}{[d]_{A^2} [e]_{A^2} [d+e]_{A^2}} v_{\sigma_i \sigma_j(t)} \notag \\
		- (1-\delta_{d,-1}) (1-\delta_{d+e,-1}) \frac{[d-1]_{A^2} [d+e-1]_{A^2} [e+1]_{A^2}}{[d]_{A^2} [d+e]_{A^2} [e]_{A^2}} v_{\sigma_j \sigma_i(t)} \notag \\
		- (1-\delta_{d,-1}) (1-\delta_{d+e,-1}) (1-\delta_{e,-1}) \notag \\
		\times \frac{[d-1]_{A^2} [d+e-1]_{A^2} [e-1]_{A^2}}{[d]_{A^2} [d+e]_{A^2} [e]_{A^2}} v_{\sigma_i \sigma_j \sigma_i(t)} .
		\end{multlined}
	\end{align*}
	On distingue deux cas.
	\begin{Cas}
	On suppose que les étiquettes $i$ et $j$ sont sur une même ligne dans la diagramme de Young $\lambda$. Alors $d=-1$ ou $e=-1$. Pour $d=-1$, l'égalité $(1)$ donne $(h_i h_j h_i) \cdot v_t = 0 = h_i \cdot v_t$. Pour $e=-1$, l'égalité $(1)$ donne :
	\[ (h_i h_j h_i) \cdot v_t = - \frac{[d+1]_{A^2}}{[d]_{A^2}} v_t + (1-\delta_{d,-1}) \frac{[d-1]_{A^2}}{[d]_{A^2}} v_{\sigma_i(t)} = h_i \cdot v_t . \]
	\end{Cas}
	\begin{Cas}
	On suppose que les étiquettes $i$ et $j$ ne sont pas sur une même ligne dans la diagramme de Young $\lambda$. Alors $e \in \{-1,-d-1\}$ ou $d \in \{-1, e-1\}$. Les cas $e=-1$ et $d=-1$ ont déjà été traités dans le cas 1. Pour $e=-d-1$ (équivalent à $d=-e-1$), l'égalité $(1)$ donne encore :
	\[ (h_i h_j h_i) \cdot v_t = - \frac{[d+1]_{A^2}}{[d]_{A^2}} v_t + (1-\delta_{d,-1}) \frac{[d-1]_{A^2}}{[d]_{A^2}} v_{\sigma_i(t)} = h_i \cdot v_t . \]
	\end{Cas}
	
	\item On a une partition des tableaux standards de forme $\lambda$ :
	\[ T_\lambda = \coprod_{\mu \in \Delta_{n-1} \; ; \; \mu \subset \lambda} T_\mu . \]
	Elle induit un isomorphisme de $\C(A)$-espaces vectoriels :
	\[ V_\lambda \cong \bigoplus_{\mu \in \Delta_{n-1} \; ; \; \mu \subset \lambda} V_\mu . \]
	Or, pour tout diagramme de Young $\mu \in \Delta_{n-1}$, le $\C(A)$-espace vectoriel $V_{\mu}$ est stable sous l'action des générateurs $h_0,..., h_{n-2}$ de $\TL_{n-1}(A^2)$ (mais pas sous l'action de $h_{n-1}$). En identifiant $TL_{n-1}(A^2)$ avec la sous-algèbre de $\TL_n(A^2)$ engendrée par les éléments $h_0,..., h_{n-2}$ (via l'injection canonique \ref{Eqn: injections de TL}), on a donc un isomorphisme de $\TL_{n-1}(A^2)$-modules à gauche.
\end{enumerate}
\end{proof}

L'ensemble $\{ V_\lambda \; ; \; \lambda \in \Delta^{\leq 2}_n \}$ des représentations semi-normales de $\TL_n(A^2)$ fournit une décomposition de $\TL_n(A^2)$.

\begin{Thm} 
\label{Thm: structure TLA}
\begin{enumerate}[(i)]
	\item Pour tout $\lambda \in \Delta_n^{\leq 2}$, le $\TL_n(A^2)$-module à gauche $V_\lambda$ est simple.
	\item Pour tous $\lambda, \mu \in \Delta_n^{\leq 2}$ tels que $\lambda \not = \mu$, les $\TL_n(A^2)$-modules à gauche $V_\lambda$ et $V_\mu$ ne sont pas isomorphes.
	\item On a un isomorphisme de $\TL_n(A^2)$-modules à gauche :
	\[ \TL_n(A^2) \cong \bigoplus_{\lambda \in \Delta_n^{\leq 2}} f_\lambda V_\lambda \]
	où, pour tout $\lambda \in \Delta_n^{\leq 2}$, $f_\lambda$ est le nombre de tableaux standards de forme $\lambda$.
\end{enumerate}
\end{Thm}

\begin{Rem}
\label{Rem: structure TLA}
Le théorème \ref{Thm: structure TLA} s'étend pour la valeur $n=1$. Dans ce cas, on a $\Delta_1^{\leq2} = \{ \lyng{1} \}$ et on considère le $\C(A)$-espace vectoriel $V_{\lyng{1}}$ librement engendré par $v_{\lyoung{1}}$. Alors $V_{\lyng{1}} \cong \C(A)$ et $\TL_1(A^2) \cong \C(A) \cong V_{\lyng{1}}$.
\end{Rem}

\begin{proof}
\begin{enumerate}[(i)]
	\item On procède par récurrence sur $n \geq 2$. 
	
	Pour $n=2$, on a $\Delta_2^{\leq2} = \{ \lyng{2}, \lyng{1,1} \}$. Les $\TL_2(A^2)$-modules à gauches $V_{\lyng{2}}$ et $V_{\lyng{1,1}}$, respectivement engendrés sous $\C(A)$ par $ v_{\lyoung{12}}$ et $v_{\lyoung{1,2}}$, sont de $\C(A)$-dimension $1$ donc simples.
	
	Soit $n \geq 2$. On suppose que, pour tout $\mu \in \Delta_n^{\leq 2}$, le $\TL_n(A^2)$-module à gauche $V_\mu$ est simple. Soit $\lambda \in \Delta_{n+1}^{\leq 2}$. Montrons que le $\TL_{n+1}(A^2)$-module à gauche $V_\lambda$ est simple. Soient $V$ un sous-module à gauche non nul de $V_\lambda$, et $V'$ un sous-module à gauche simple de $V_{\vert \TL_n(A^2)}$. D'après la proposition \ref{Prop: modules de TLA}, le $\TL_n(A^2)$-module à gauche induit ${V_\lambda}_{\vert \TL_n(A^2)}$ admet la décomposition en somme directe :
	\begin{equation} \tag{1}
	{V_\lambda}_{\vert \TL_n(A^2)} \cong \bigoplus_{\mu \in \Delta_n \; ; \; \mu \subset \lambda} V_\mu = \bigoplus_{\mu \in \Delta_n^{\leq 2} \; ; \; \mu \subset \lambda} V_\mu .
	\end{equation}
	Or, d'après l'hypothèse de récurrence, les $\TL_n(A^2)$-modules à gauche de cette décomposition sont simples. Il existe donc un diagramme de Young $\mu \in \Delta_n^{\leq 2}$ tel que $\mu \subset \lambda$ et $V' \cong V_\mu$. On distingue deux cas.
	
	\begin{Cas}
	On suppose que $\mu$ est le seul diagramme de Young à $n$ cases tel que $\mu \subset \lambda$. Alors on a :
	\[ {V_\lambda}_{\vert \TL_n(A^2)} \overset{(1)}{\cong} V_\mu \cong V' \subseteq V_{\vert \TL_n(A^2)}. \]
	Comme $V$ est un sous-module à gauche de $V_\lambda$, il s'ensuit que $V_{\vert \TL_n(A^2)} = {V_\lambda}_{\vert \TL_n(A^2)}$ donc $V = V_\lambda$.
	\end{Cas}
	
	\begin{Cas}
	On suppose qu'il existe un diagramme de Young $\mu' \not = \mu$ à $n$ cases tel que $\mu' \subset \lambda$. On note $(r,c)$ la coordonnée de la case $\lambda \setminus \mu$ et $(r',c')$ celle de $\lambda \setminus \mu'$. Alors $|r-r'| = 1$ et $|c-c'| \geq 1$ (cf. la remarque \ref{Rem: modules de TLA}). On considère un tableau standard $t \in T_\lambda$ dont les cases $(r,s)$ et $(r',c')$ sont respectivement étiquetées par $n+1$ et $n$. Par construction, les étiquettes $n$ et $n+1$ ne sont ni sur la même ligne, ni sur la même colonne. Donc $\sigma_n(t) \in T_\lambda$ et $d_t(n) \not = \pm 1$ (cf. la remarque \ref{Rem: modules de TLA}). De plus :
	\begin{equation} \tag{2}
	\begin{aligned}
	&h_n \cdot v_t = - \frac{[d_t(n)+1]_{A^2}}{[d_t(n)]_{A^2}} v_t  - (1-\delta_{d_t(n),-1} ) \frac{[d_t(n)-1]_{A^2}}{[d_t(n)]_{A^2}} v_{\sigma_n(t)}, \\
	\Longleftrightarrow \qquad& h_n \cdot v_t + \frac{[d_t(n)+1]_{A^2}}{[d_t(n)]_{A^2}} v_t = - \underbrace{(1-\delta_{d_t(n),-1} ) \frac{[d_t(n)-1]_{A^2}}{[d_t(n)]_{A^2}}}_{\neq 0} v_{\sigma_n(t)},
	\end{aligned}
	\end{equation}	
	 où $v_t \in V_\mu \cong V' \subseteq V$ et $v_{\sigma_n(t)} \in V_{\mu'}$ selon la décomposition $(1)$. Comme $V$ est un $\TL_{n+1}(A^2)$-module à gauche, on a aussi $h_n \cdot v_t \in V$, donc $v_{\sigma_n(t)} \in V$ d'après $(2)$. Il s'ensuit que $V_{\vert \TL_n(A^2)} \cap V_{\mu'}$ est un sous-module non nul de $V_{\mu'}$. Or, d'après l'hypothèse de récurrence, le $\TL_n(A^2)$-module à gauche $V_{\mu'}$ est simple. Donc $V_{\vert \TL_n(A^2)} \cap V_{\mu'} = V_{\mu'}$ i.e. $V_{\mu'} \subseteq V_{\vert \TL_n(A^2)}$. Ceci étant valable pour tout diagramme de Young $\mu' \not = \mu$ à $n$ cases tel que $\mu' \subset \lambda$ (en fait il n'y a que $\mu'$ : cf. la remarque \ref{Rem: modules de TLA}), on a :
	\[ {V_\lambda}_{\vert \TL_n(A^2)} \overset{(1)}{\cong} V_\mu \oplus \bigoplus_{\substack{\mu' \in \Delta_n^{\leq2} \; ; \; \mu' \subset \lambda \\ \mu' \neq \mu}} V_{\mu'} \subseteq V_{\vert \TL_n(A^2)} . \]
	Comme $V$ est un sous-module à gauche de $V_\lambda$, il s'ensuit que $V_{\vert \TL_n(A^2)} = {V_\lambda}_{\vert \TL_n(A^2)}$ donc $V = V_\lambda$.
	\end{Cas}
	
	Dans chacun des cas, on a $V = V_\lambda$. Par conséquent, le $\TL_{n+1}(A^2)$-module à gauche $V_\lambda$ est simple. Ceci étant valable pour tout diagramme de Young $\lambda \in \Delta_{n+1}^{\leq2}$, cela prouve l'hérédité et achève la récurrence.
	
	\item Pour $n=2$, on a $\Delta_2^{\leq2} = \{ \lyng{2}, \lyng{1,1} \}$. Les $\TL_2(A^2)$-modules à gauches $V_{\lyng{2}}$ et $V_{\lyng{1,1}}$, respectivement engendrés par $ v_{\lyoung{12}}$ et $v_{\lyoung{1,2}}$, ne sont pas isomorphes puisque : 
	\[ h_1 \cdot v_{\lyoung{12}} = 0, \quad h_1 \cdot v_{\lyoung{1,2}} = - [2]_{A^2} v_{\lyoung{1,2}} . \]
	On suppose que $n \geq 3$. Soient $\lambda, \mu \in \Delta_n^{\leq 2}$ tels que $V_\lambda \cong V\mu$. Montrons que $\lambda = \mu$. D'après la proposition \ref{Prop: modules de TLA}, l'isomorphisme $V_\lambda \cong V_\mu$ induit un isomorphisme de $\TL_{n-1}(A^2)$-modules à gauche :
	\[ \bigoplus_{\lambda' \in \Delta_{n-1} \; ; \; \lambda' \subset \lambda} V_{\lambda'} \cong \bigoplus_{\mu' \in \Delta_{n-1} \; ; \; \mu' \subset \mu} V_{\mu'} , \]
	où tous les $\TL_{n-1}(A^2)$-modules à gauche de ces décompositions sont simples d'après $(i)$. D'après le théorème de Krull-Schmidt (cf. par exemple \cite[Thm 14.5]{CR62}), on en déduit que :
	\[ \left\{ \lambda' \in \Delta_n^{\leq 2} \; ; \; \lambda' \subset \lambda \right\} = \left\{ \mu' \in \Delta_n^{\leq 2} \; ; \; \mu' \subset \mu \right\} . \]
	 Par conséquent, $\lambda = \mu$ (cf. l'assertion $(ii)$ de la remarque \ref{Rem: modules de TLA}).
	 
	 \item D'après ce qui précède, on a un ensemble $\{ V_\lambda \; ; \; \lambda \in \Delta_n^{\leq 2} \}$ de $\TL_n(A^2)$-modules à gauche simples, deux-à-deux non isomorphes. Montrons que, à isomorphisme près, on obtient ainsi tous les facteurs directs de la représentation régulière à gauche de $\TL_n(A^2)$.
	 
	 On sait que tout $\TL_n(A^2)$-module simple est isomorphe à un facteur direct du module quotient $\TL_n(A^2) / \Rad(\TL_n(A^2))$ (cf. par exemple \cite[Thm 54.5, Cor. 54.13]{CR62}), et que sa multiplicité est donnée par sa dimension (cf. par exemple \cite[Thm 54.11]{CR62}). On a donc un morphisme injectif de $\TL_n(A^2)$-modules à gauche :
	 \begin{equation} \tag{2}
	 \bigoplus_{\lambda \in \Delta_n^{\leq2}} \dim_{\C(A)}(V_\lambda) V_\lambda \hookrightarrow \TL_n(A^2) / \Rad(\TL_n(A^2))
	 \end{equation}
	 où $\dim_{\C(A)} (V_\lambda) = f_\lambda$. La $\C(A)$-algèbre de gauche a pour dimension :
	 \[ \sum_{\lambda \in \Delta_n^{\leq 2}} \dim_{\C(A)} (V_\lambda)^2 = \sum_{\lambda \in \Delta_n^{\leq 2}} (f_\lambda)^2 = \frac{1}{n+1} {2n \choose n} \] 
	 (cf. par exemple \cite[Exercice 5.2.3]{KT08}) ; il s'agit du $n$-ième \emph{nombre Catalan}. Il s'ensuit que $\dim_{\C(A)} \TL_n(A^2) \geq  \frac{1}{n+1} {2n \choose n}$. Par ailleurs, on sait que $\TL_n(A^2)$ est librement engendrée sous $\C(A)$ par l'ensemble des mots réduits, dont la cardinalité est aussi $\frac{1}{n+1} {2n \choose n}$ (cf. par exemple \cite[§ 5.7.1]{KT08}). Donc $\dim_{\C(A)} \TL_n(A^2) =  \frac{1}{n+1} {2n \choose n}$. Par conséquent, on a $\Rad(\TL_n(A^2)) = \{0\}$ et le morphisme $(2)$ fournit un isomorphisme de $\TL_n(A^2)$-modules à gauche :
	 \[ \bigoplus_{\lambda \in \Delta_n^{\leq2}} f_\lambda V_\lambda \cong \TL_n(A^2). \]
\end{enumerate}
\end{proof}

Combiné avec le théorème de Wedderburn (cf. par exemple \cite[§ 3.5]{Pie82}), le théorème \ref{Thm: structure TLA} donne la semi-simplicité de l'algèbre de Temperley-Lieb $\TL_n(A^2)$ et sa décomposition en somme directe d'idéaux indécomposables :
\begin{equation}
\label{Eqn: structure TLA1}
\TL_n(A^2) \cong \bigoplus_{\lambda \in \Delta^{\leq2}_n} \End_{\C(A)}( V_\lambda ) .
\end{equation}

Pour tout $\lambda \in \Delta_n^{\leq2}$, on note $z_\lambda$ \index{z@ $z_\lambda$} le PCI associé au facteur direct $\End_{\C(A)}(V_\lambda)$ (cf. par exemple \cite[§ 25]{CR62}). Il est défini comme l'unique élément de $\TL_n(A^2)$ tel que :
\[ z_\lambda \TL_n(A^2) = \TL_n(A^2) z_\lambda \cong \End_{\C(A)} (V_\lambda ) . \]
Les PCIs $z_\lambda$, $\lambda \in \Delta_n^{\leq 2}$, sont deux-à-deux orthogonaux et vérifient :
\[ 1 = \sum_{ \lambda \in \Delta_n^{\leq 2} } z_\lambda, \qquad \forall \lambda \in \Delta_n^{\leq2} \qquad z_\lambda = \sum_{t \in T_\lambda} p_t , \]
où, pour tous $\lambda \in \Delta_n^{\leq 2}$ et $t \in T_\lambda$, l'élément $p_t \in z_\lambda \TL_n(A^2)$ \index{p2@ $p_t$} agit sur $V_\lambda$ comme la projection orthogonale sur $\C(A) v_t$ (i.e sa matrice de représentation dans la $\C(A)$-base $\{v_t \; ; \; t \in T_\lambda \}$ est une matrice élémentaire de la forme $E_{ii}$ avec $i \in \{1,...,f_\lambda\}$). Par construction des représentations semi-normales, les ensembles $\{ p_t \; ; \; t \in T_\lambda \}$ et $\{ p_t \; ; \; t \in T_n^{\leq2} \}$ sont respectivement les POIs de $z_\lambda \TL_n(A^2)$ et de $\TL_n(A^2)$ (cf. par exemple \cite[§ 25-26]{CR62}). Ils déterminent la décomposition de $\TL_n(A^2)$ en somme directe de $\TL_n(A^2)$-modules à gauche (resp. à droite) indécomposables :
\begin{equation}
\label{Eqn: structure TLA2}
\begin{gathered}
\TL_n(A^2) = \bigoplus_{\lambda \in \Delta^{\leq2}_n} \TL_n(A^2) z_\lambda = \bigoplus_{\lambda \in \Delta^{\leq2}_n} \bigoplus_{t \in T_\lambda} \TL_n(A^2) p_t \\
\left( \text{resp. } \TL_n(A^2) = \bigoplus_{\lambda \in \Delta^{\leq2}_n} z_\lambda \TL_n(A^2) = \bigoplus_{\lambda \in \Delta^{\leq2}_n} \bigoplus_{t \in T_\lambda} p_t \TL_n(A^2) \right) .
\end{gathered}
\end{equation}
On retrouve la décomposition donnée dans le théorème \ref{Thm: structure TLA} puisque :
\[ \forall \lambda \in \Delta_n^{\leq 2} \quad \forall t \in T_\lambda \qquad \TL_n(A^2) p_t \cong V_\lambda
\qquad \left( \text{resp. } p_t \TL_n(A^2) \cong V_\lambda \right). \]

On cherche une expression de ces POIs de $\TL_n(A^2)$. Pour cela, on établit des relations de type récurrence à l'aide de nouvelles notations.

\begin{Def}
Soit $t \in T_n^{\leq 2}$.
\begin{enumerate}[(i)]
	\item On appelle \emph{treillis de Temperley-Lieb} le sous-graphe orienté du treillis de Young formé des sommets étiquetés par des diagrammes de Young avec au plus deux lignes (cf. la figure \ref{Fig: treillis de TL}).
	\item On note $t'$ \index{t4 @$t'$} le sous-tableau standard de $t$ formé des cases étiquetées par $\{1,...,n-1\}$. Pour $n \geq 3$, on définit $t^{(i)}$ par la relation de récurrence :
	\[ t^{(i)} = {t^{(i-1)}}', \quad 2 \leq i \leq n-1. \]
	\item S'il existe, on note $\hat{t}$ \index{t5 @$\hat{t}$} le tableau standards de $T_n^{\leq 2}$ tel que $\hat{t}'=t'$ et $\hat{t}\not =t$.
\end{enumerate}
\end{Def}

Ces définitions s'interprètent géométriquement sur le treillis de Temperley-Lieb et sont illustrées dans la figure colorée \ref{Fig: treillis de TL}.

\begin{figure}[!ht]
\caption{Treillis de Temperley-Lieb.}
\label{Fig: treillis de TL}
\[ \Yboxdim{7pt}
\xymatrix @-1ex {
	\overset{\emptyset}{\bullet} \ar@[red][rd]^ {\color{red} t^{(5)}} \\
	& \overset{\yng(1)}{\bullet} \ar[ld] \ar@[red][rd]^ {\color{red} t^{(4)}} \\
	\overset{\yng(1,1)}{\bullet} \ar[rd] && \overset{\yng(2)}{\bullet} \ar@[red][ld]_{\color{red} t'''} \ar[rd] \\
	& \overset{\yng(2,1)}{\bullet} \ar[ld] \ar@[red][rd]^{\color{red} t''} && \overset{\yng(3)}{\bullet} \ar[ld] \ar[rd] \\
	\overset{\yng(2,2)}{\bullet} \ar[rd] && \overset{\yng(3,1)}{\bullet} \ar[ld] \ar@[red][rd]^{\color{red} t'} && \overset{\yng(4)}{\bullet} \ar[ld] \ar[rd] \\
	& \overset{\yng(3,2)}{\bullet} \ar[ld] \ar[rd] && \overset{\yng(4,1)}{\bullet} \ar@[red][ld]^{\color{red} \hat{t}} \ar@[red][rd]_{\color{red} t} && \overset{\yng(5)}{\bullet} \ar[ld] \ar[rd] \\
	\overset{\yng(3,3)}{\bullet} && \overset{\yng(4,2)}{\bullet} && \overset{\yng(5,1)}{\bullet} && \overset{\yng(6)}{\bullet} } \]
\end{figure}
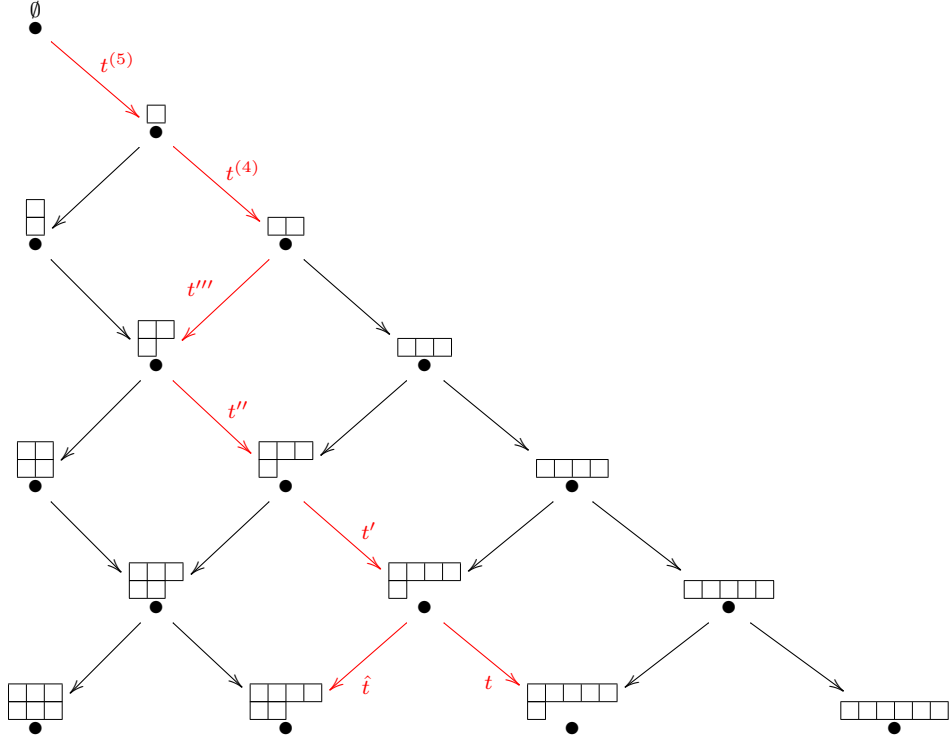

\begin{Rem}
\label{Rem: treillis de TL1}
Soit $t \in T_n^{\leq 2}$ de forme $\lambda = [\lambda_1, \lambda_2]$. 
\begin{enumerate}[(i)]
	\item Il existe au plus un tableau standard $s \in T_n^{\leq 2}$ tel que $s'=t'$ et $s \not =t$. En effet, pour tout $r \in T_{n-1}^{\leq 2}$, il existe au plus deux tableaux standards $s,t \in T_n^{\leq 2}$ tels que $t'=r=s'$. Plus précisément, si $r$ est de forme $\mu=[\mu_1,\mu_2]$ telle que $\mu_1=\mu_2$, alors il existe exactement un tableau standard $t \in T_n^{\leq 2}$ tel que $t'=r$. Sinon, il en existe exactement deux. Donc $\hat{t}$ est défini sans équivoque. 
	\item Si $\hat{t}$ n'existe pas, alors $d_{t}(n-1) = -2$. Sinon, les étiquettes $n-1$ et $n$ sont sur la même ligne dans $t$ ou exclusivement dans $\hat{t}$. Ces cas sont respectivement caractérisés par les égalités $d_t(n-1)=-1$ et $d_{\hat{t}}(n-1)=-1$. Plus précisément, si $t'$ est de forme $[\mu_1,\mu_2]$, alors on a quatre configurations possibles illustrées dans la figure \ref{Fig: conf sur TL1}.
\end{enumerate}
\end{Rem}

\begin{figure}[!ht]
\caption{Configurations possibles sur le treillis de Temperley-Lieb.}
\label{Fig: conf sur TL1}
\[ \begin{array}{|c|cc|cc|}
\hline 
& \multicolumn{2}{c|}{d_t(n-1)=-1} & \multicolumn{2}{c|}{d_t(n-1) \neq -1} \\  
\hline
\multirow{9}{*}{$\exists \hat{t}$} &
\vcenter{ \xymatrix @-2ex {
	\bullet \ar[rd]^{t'} \\
	& \bullet \ar[ld]_{\hat{t}} \ar[rd]^{t} \\
	\bullet && \bullet }} &
\begin{aligned}
	d_t(n-1) &= -1, \\ 
	d_{\hat{t}}(n-1) &= \mu_1-\mu_2 \\ &= \lambda_1-\lambda_2-1.
	\end{aligned} &
\vcenter{ \xymatrix @-2ex {
	\bullet \ar[rd]^{t'} \\
	& \bullet \ar[ld]_{t} \ar[rd]^{\hat{t}} \\
	\bullet && \bullet }} &
\begin{aligned} 
	d_t(n-1) &= \mu_1-\mu_2 \\ &= \lambda_1-\lambda_2+1, \\ 
	d_{\hat{t}}(n-1) &= -1.
	\end{aligned} \\
\cline{2-5}
& \vcenter{ \xymatrix @-2ex {
	&& \bullet \ar[ld]_{t'} \\
	& \bullet \ar[ld]_{t} \ar[rd]^{\hat{t}} \\
	\bullet && \bullet }} &
\begin{aligned} 
	d_t(n-1) &= -1, \\ 
	d_{\hat{t}}(n-1) &= -\mu_1+\mu_2-2 \\ &= -\lambda_1+\lambda_2-3.
	\end{aligned} &
\vcenter{ \xymatrix @-2ex {
	&& \bullet \ar[ld]_{t'} \\
	& \bullet \ar[ld]_{\hat{t}} \ar[rd]^{t} \\
	\bullet && \bullet }} &
\begin{aligned} 
	d_t(n-1) &= -\mu_1+\mu_2-2 \\ &= -\lambda_1+\lambda_2-1, \\ 
	d_{\hat{t}}(n-1) &= -1.
	\end{aligned} \\
\hline
\nexists \hat{t} && &
\vcenter{ \xymatrix @-2ex {
	& \ar@[gray]@{--}[dd] & \bullet \ar[ld]_{t'} \\
	& \bullet \ar[rd]^{t} \\
	&& \bullet }} &
\begin{aligned} 
	d_t(n-1) &= -2.
	\end{aligned} \\
\hline
\end{array} \]
\end{figure}

\begin{Prop} 
\label{Prop: POIs de TLA}
Les POIs $p_t$, $t \in T_n^{\leq 2}$, de $\TL_n(A^2)$ vérifient :
\begin{align*}
& p_{\lyoung{1}} = 1, \\
& p_t = \begin{multlined}[t]
	\begin{cases} 
	p_{t'} & \text{ si $\hat{t}$ n'existe pas }, \\
	- \frac{[d_t(n-1)]_{A^2}}{[d_t(n-1)+1]_{A^2}} p_{t'} h_{n-1} p_{t'} & \text{ si $\hat{t}$ existe et } d_t(n-1) \not = -1, \\
	p_{t'} + \frac{[d_{\hat{t}}(n-1)]_{A^2}}{[d_{\hat{t}}(n-1)+1]_{A^2}} p_{t'} h_{n-1} p_{t'} & \text{ si $\hat{t}$ existe et } d_t(n-1) = -1,
	\end{cases}
	 \quad t \in T_n^{\leq 2}, \; n \geq 2.
	\end{multlined}
\end{align*}
\end{Prop}

\begin{proof}
On rajoute la valeur $n=1$. Dans ce cas, on a $T_1^{\leq2} = \{ \lyoung{1} \}$. On considère l'élément $p_{\lyoung{1}} \in \TL_1(A^2)$ qui agit sur $V_{\lyng{1}} = \C(A) v_{\lyoung{1}}$ comme la projection orthogonale sur $\C(A) v_{\lyoung{1}}$. Or $\TL_1(A^2) \cong \C(A) \cong \End_{\C(A)}(V_{\lyng{1}})$. Donc le PCI de $\TL_1(A^2)$ est $z_ {\lyng{1}} = 1$ et $p_{\lyoung{1}} = z_ {\lyng{1}} = 1$. 

Soient $n \geq 2$ et $t, s \in T_n^{\leq2}$. On étudie les éléments de $\TL_n(A^2)$ via leurs actions sur les représentations semi-normales $\{ V_\lambda \; ; \; \lambda \in \Delta^{\leq2}_n \}$ grâce à l'isomorphisme \eqref{Eqn: structure TLA1}. D'après la description des $\TL_n(A^2)$-modules à gauche donnée dans la proposition \ref{Prop: modules de TLA}, on a :
\begin{equation} \tag{1}
p_{t'} \cdot v_s = \delta_{s',t'} v_s .
\end{equation}
On distingue deux cas.

\begin{Cas}
On suppose que $\hat{t}$ n'existe pas. Alors $t$ est le seul tableau standard $s \in T_n^{\leq 2}$ tel que $s'=t'$. L'équation $(1)$ donne donc $p_{t'} = p_t$.
\end{Cas}

\begin{Cas}
On suppose que $\hat{t}$ existe. Alors $t$ et $\hat{t}$ sont les tableaux standards $s \in T_n^{\leq 2}$ tels que $s'=t'$. L'équation $(1)$ donne donc $p_{t'} = p_t + p_{\hat{t}}$. On note $d:= d_s(n-1)$ et on réutilise la proposition \ref{Prop: modules de TLA} :
\begin{align*}
h_{n-1} p_{t'} \cdot v_s &= \delta_{s',t'} h_{n-1} \cdot v_s \\
&= - \delta_{s',t'} \frac{[d+1]_{A^2}}{[d]_{A^2}} v_s - \delta_{s',t'} (1-\delta_{d,-1}) \frac{[d-1]_{A^2}}{[d]_{A2}} v_{\sigma_{n-1}(s)}, \\
p_{t'} h_{n-1} p_{t'} \cdot v_s &= - \delta_{s',t'} \frac{[d+1]_{A^2}}{[d]_{A^2}} p_{t'} \cdot v_s - \delta_{s',t'} (1-\delta_{d,-1}) \frac{[d-1]_{A^2}}{[d]_{A2}} p_{t'} \cdot v_{\sigma_{n-1}(s)} \\
&= - \delta_{s',t'} \frac{[d+1]_{A^2}}{[d]_{A^2}} v_s.
\end{align*}
On en déduit que :
\[ p_{t'} h_{n-1} p_{t'} = - \frac{[d_t(n-1)+1]_{A^2}}{[d_t(n-1)]_{A^2}} p_t - \frac{[d_{\hat{t}}(n-1)+1]_{A^2}}{[d_{\hat{t}}(n-1)]_{A^2}} p_{\hat{t}}, \]
où $d_t(n-1)=-1$ ou exclusivement $d_{\hat{t}}(n-1)=-1$ (cf. la remarque \ref{Rem: treillis de TL1}). Autrement dit :
\[ p_{t'} h_{n-1} p_{t'} = \begin{cases} 
	- \frac{[d_t(n-1)+1]_{A^2}}{[d_t(n-1)]_{A^2}} p_t & \text{ si } d_t(n-1) \neq -1, \\
	- \frac{[d_{\hat{t}}(n-1)+1]_{A^2}}{[d_{\hat{t}}(n-1)]_{A^2}} p_{\hat{t}} & \text{ si } d_t(n-1) = -1.
	\end{cases} \]
Par conséquent, on a :
\[ p_t = \begin{cases} 
	- \frac{[d_t(n-1)]_{A^2}}{[d_t(n-1)+1]_{A^2}} p_{t'} h_{n-1} p_{t'} & \text{ si } d_t(n-1) \neq -1, \\
	p_{t'} - p_{\hat{t}} = p_{t'} + \frac{[d_{\hat{t}}(n-1)]_{A^2}}{[d_{\hat{t}}(n-1)+1]_{A^2}} p_{t'} h_{n-1} p_{t'} & \text{ si } d_t(n-1) = -1.
	\end{cases} \]
\end{Cas}
\end{proof}

\subsection{POIs évaluables aux racines de l'unité}
\label{subsection: TLq}

On explicite maintenant la structure de l'algèbre évaluée $\TL_n(q)$ à l'aide de celle de l'algèbre générique $\TL_n(A^2)$, grâce au morphisme d'évaluation \eqref{Prop: evaluations}. De même que pour $\TL_n(A^2)$, la $\C$-algèbre $\TL_n(q)$ se décompose de manière unique, à isomorphisme et ordre des facteurs près, en somme directe d'idéaux indécomposables de $\TL_n(q)$, appelés \emph{facteurs directs} de $\TL_n(q)$. Ces facteurs directs correspondent à des PCIs de $\TL_n(q)$. Afin de les expliciter, on construit des idempotents orthogonaux de $\TL_n(A^2)$ évaluables, à partir desquels s'expriment les POIs et les PCIs de $\TL_n(q)$.

On utilise les mêmes notions (diagrammes de Young, tableaux standards, distances axiales) que dans la sous-section \ref{subsection: TLA}. Les formules qui définissent la structure des $\TL_n(A^2)$-modules à gauche $V_\lambda$, $\lambda \in \Delta_n^{\leq2}$, ne sont pas toutes évaluables (cf. la proposition \ref{Prop: modules de TLA}). Plus précisément, on ne peut pas les évaluer dès que :
\[ \exists t \in T_n^{\leq2} \quad \exists i \in \{1,...,n-1\} \quad \text{ tels que } \quad d_t(i) = 0 \mod p. \]
Soit $t \in T_n^{\leq2}$ un tableau standard de graphe :
\[ \gamma(t) = \xymatrix @-2ex {
	\overset{\lambda^{(1)}}{\bullet} \ar[r] &
		\overset{\lambda^{(2)}}{\bullet} \ar[r] &
		\cdots \ar[r] &
		\overset{\lambda^{(n)}}{\bullet} }. \]
Ce cas de figure apparaît seulement s'il existe $i \in \{1,...,n-1\}$ tel que $\lambda_1^{(i)} - \lambda_2^{(i)} +1 = p$, où $\lambda^{(i)} = [\lambda_1^{(i)}, \lambda_2^{(i)}]$. Ce qui amène aux définitions suivantes.

\begin{Def}[{\cite[§ 1]{GW93}}]
Soit $\lambda=[\lambda_1,\lambda_2] \in \Delta_n^{\leq2}$.
\begin{enumerate}[(i)]
	\item On dit que $\lambda$ est \emph{critique} si $\omega(\lambda):=\lambda_1-\lambda_2+1$ \index{omega @$\omega(\lambda)$} est divisible par $p$.
	\item Pour tout $k \in \N^*$, on appelle $k$-ième \emph{ligne critique} du treillis de Temperley-Lieb l'ensemble des sommets étiquetés par des diagrammes de Young $\mu \in \Delta_n^{\leq2}$ tels que $\omega(\mu)=kp$.
\end{enumerate}
\end{Def}

\begin{Rem}
\label{Rem: treillis de TL2}
Soit $\lambda = [\lambda_1, \lambda_2] \in \Delta_n^{\leq2}$. Sur le diagramme de Temperley-Lieb, l'entier $\omega(\lambda) \geq 1$ s'interprète comme le numéro de la colonne du sommet étiqueté par $\lambda$. Il permet d'exprimer facilement les distances axiales. En effet, pour tout $t \in T_\lambda$, on a :
\[ d_t(n-1)= \begin{cases} 
	-1 & \text{ si $n-1$ et $n$ sont sur la même ligne, } \\
	\omega(\lambda) & \text{ si $n-1$ est sur la première ligne et $n$ sur la seconde, } \\
	-\omega(\lambda) & \text{ sinon. }
	\end{cases} \]
Pour plus de détails, on pourra se référer à la figure \ref{Fig: conf sur TL2}, qui reprend les configurations possibles explicitées dans la remarque \ref{Rem: treillis de TL1}.
\begin{figure}[!ht]
\caption{Configurations possibles sur le treillis de Temperley-Lieb.} \label{Fig: conf sur TL2}
\footnotesize
\[ \begin{array}{|c|cc|cc|}
\hline 
& \multicolumn{2}{c|}{d_t(n-1)=-1} & \multicolumn{2}{c|}{d_t(n-1) \neq -1} \\  
\hline
\multirow{9}{*}{$\exists \hat{t}$} &
\vcenter{ \xymatrix @-2.5ex {
	\bullet \ar[rd]^{t'} & \overset{\color{gray} \omega(\mu)}{} \ar@[gray]@{--}[dd] & \overset{\color{gray} \omega(\lambda)}{} \ar@[gray]@{--}[dd] \\
	& \bullet \ar[ld]_{\hat{t}} \ar[rd]^{t} \\
	\bullet && \bullet }} &
\begin{aligned}
	d_t(n-1) &= -1, \\ 
	d_{\hat{t}}(n-1) &= \omega(\mu) -1 \\ &= \omega(\lambda)-2.
	\end{aligned} &
\vcenter{ \xymatrix @-2.5ex {
	\overset{{\color{gray} \omega(\lambda)}}{\bullet} \ar[rd]^{t'} \ar@[gray]@{--}[dd] & \overset{\color{gray} \omega(\mu)}{} \ar@[gray]@{--}[dd] \\
	& \bullet \ar[ld]_{t} \ar[rd]^{\hat{t}} \\
	\bullet && \bullet }} &
\begin{aligned} 
	d_t(n-1) &= \omega(\mu) -1 \\ &= \omega(\lambda), \\ 
	d_{\hat{t}}(n-1) &= -1.
	\end{aligned} \\
\cline{2-5}
& \vcenter{ \xymatrix @-2.5ex {
	\overset{\color{gray} \omega(\lambda)}{} \ar@[gray]@{--}[dd] & \overset{\color{gray} \omega(\mu)}{} \ar@[gray]@{--}[dd] & \bullet \ar[ld]_{t'} \\
	& \bullet \ar[ld]_{t} \ar[rd]^{\hat{t}} \\
	\bullet && \bullet }} &
\begin{aligned} 
	d_t(n-1) &= -1, \\ 
	d_{\hat{t}}(n-1) &= -\omega(\mu)-1 \\ &= -\omega(\lambda)-2.
	\end{aligned} &
\vcenter{ \xymatrix @-2.5ex {
	& \overset{\color{gray} \omega(\mu)}{} \ar@[gray]@{--}[dd] & \overset{{\color{gray} \omega(\lambda)}}{\bullet} \ar[ld]_{t'} \ar@[gray]@{--}[dd] \\
	& \bullet \ar[ld]_{\hat{t}} \ar[rd]^{t} \\
	\bullet && \bullet }} &
\begin{aligned} 
	d_t(n-1) &= -\omega(\mu)-1 \\ &= -\omega(\lambda), \\ 
	d_{\hat{t}}(n-1) &= -1.
	\end{aligned} \\
\hline
\nexists \hat{t} && &
\vcenter{ \xymatrix @-2.5ex {
	& \overset{\color{gray} \omega(\mu)=1}{} \ar@{--}[dd] & \overset{{\color{gray} \omega(\lambda)}}{\bullet} \ar[ld]_{t'} \ar@[gray]@{--}[dd] \\
	& \bullet \ar[rd]^{t} \\
	&& \bullet }} &
\begin{aligned} 
	d_t(n-1) &= -2.
	\end{aligned} \\
\hline
\end{array} \]
\end{figure}
\end{Rem}

\begin{Def}[{\cite[§ 1]{GW93}}]
Soit $t \in T_n^{\leq2}$.
\begin{enumerate}[(i)]
	\item On dit que $t$ est \emph{critique} si sa forme $\lambda$ est critique.
	\item Si $t$ possède un sous-tableau critique propre, alors on appelle \emph{sous-tableau critique maximal} de $t$ le plus grand sous-tableau critique propre de $t$.
	\item Si $t$ possède un sous-tableau critique maximal $r \in T_\mu$ et que le sous-graphe $\gamma(t) \setminus \gamma(r)$ admet un graphe symétrique $\bar{\gamma}$ par rapport la $\frac{\omega(\mu)}{p}$-ième ligne critique, alors on appelle \emph{tableau conjugué} de $t$ le tableau standard $\bar{t}$ \index{t6@ $\bar{t}$} défini par le graphe $\gamma(r) \cup \bar{\gamma}$.
	\item On dit que $t$ est \emph{régulier} si son graphe $\gamma(t)$ ne passe pas deux fois consécutivement sur la même ligne critique.
\end{enumerate}
\end{Def}

Ces définitions s'interprètent géométriquement sur le treillis de Temperley-Lieb et sont illustrées dans la figure \ref{Fig: lignes critiques}.

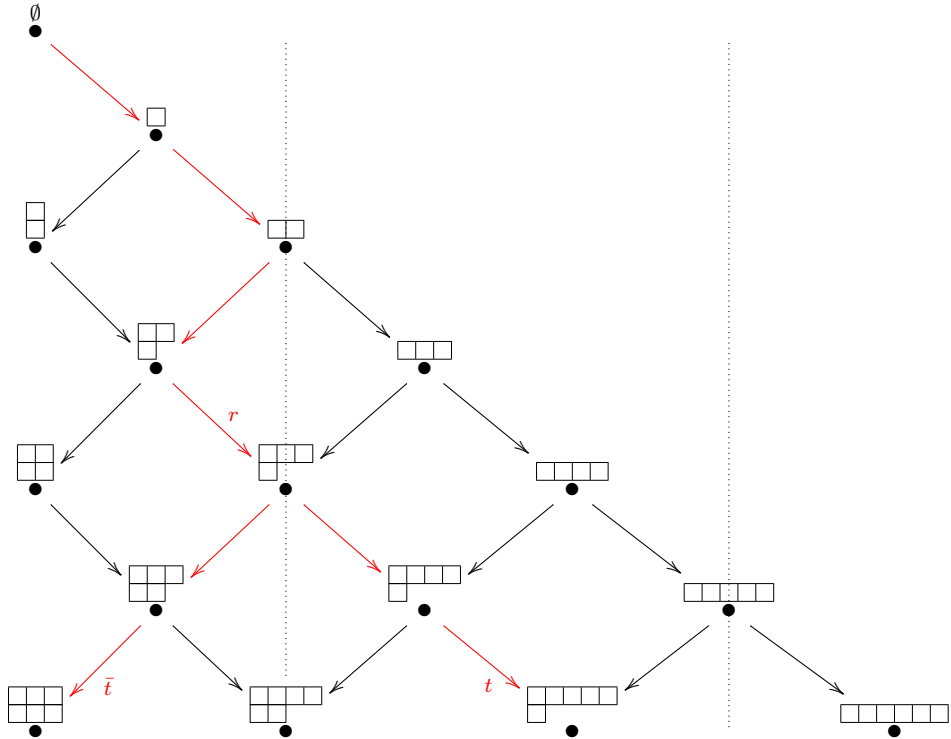
\begin{figure}[!ht]
\caption{Lignes critiques pour $p=3$.}
\label{Fig: lignes critiques}
\[ \Yboxdim{7pt}
\xymatrix @-1ex {
	\overset{\emptyset}{\bullet} \ar@[red][rd] && \ar@{.}[dddddd] &&& \ar@{.}[dddddd] \\
	& \overset{\yng(1)}{\bullet} \ar[ld] \ar@[red][rd] &&&& \\
	\overset{\yng(1,1)}{\bullet} \ar[rd] && \overset{\yng(2)}{\bullet} \ar@[red][ld] \ar[rd] \\
	& \overset{\yng(2,1)}{\bullet} \ar[ld] \ar@[red][rd]^{\color{red} r} && \overset{\yng(3)}{\bullet} \ar[ld] \ar[rd] \\
	\overset{\yng(2,2)}{\bullet} \ar[rd] && \overset{\yng(3,1)}{\bullet} \ar@[red][ld] \ar@[red][rd] && \overset{\yng(4)}{\bullet} \ar[ld] \ar[rd] \\
	& \overset{\yng(3,2)}{\bullet} \ar@[red][ld]^{\color{red} \bar{t}} \ar[rd] && \overset{\yng(4,1)}{\bullet} \ar[ld] \ar@[red][rd]_{\color{red} t} && \overset{\yng(5)}{\bullet} \ar[ld] \ar[rd] \\
	\overset{\yng(3,3)}{\bullet} && \overset{\yng(4,2)}{\bullet} && \overset{\yng(5,1)}{\bullet} && \overset{\yng(6)}{\bullet} } \]
\end{figure}

\begin{Rem}
\label{Rem: lignes critiques}
Soit $t \in T_n^{\leq2}$.
\begin{enumerate}[(i)]
	\item Le tableau standard $t$ est critique si et seulement si son graphe $\gamma(t)$ se termine sur une ligne critique.
	\item Le tableau standard $t$ possède un sous-tableau critique propre (et donc un sous-tableau critique maximal) si et seulement si $\gamma(t)$ traverse la première ligne critique.
	\item On suppose que $t$ admet un sous-tableau critique maximal $r$. Alors $t$ ne possède pas de tableau conjugué si et seulement si $\gamma(t)$ se termine sur la deuxième ligne critique et $\gamma(r)$ sur la première.
\end{enumerate}
\end{Rem}

On rappelle qu'il y a une correspondance bijective entre les POIs $p_t$, $t \in T_n^{\leq2}$, de $\TL_n(A^2)$ et les facteurs directs de sa représentation régulière à gauche :
\[ \TL_n(A^2) \overset{\eqref{Eqn: structure TLA2}}{=} \bigoplus_{\lambda \in \Delta_n^{\leq2}} \bigoplus_{t \in T_\lambda} \TL_n(A^2) p_t \overset{\eqref{Thm: structure TLA}}{\cong} \bigoplus_{\lambda \in \Delta_n^{\leq2}} f_\lambda V_\lambda \]
où, pour tout $\lambda \in \Delta_n^{\leq2}$, $f_\lambda$ est le nombre de tableaux standards de forme $\lambda$. On s'intéressent aux tableaux standards $t \in T_n^{\leq 2}$ pour lesquels $p_t$ est évaluable. En effet, dans ce cas, ils fournissent des $\TL_n(q)$-modules à gauche $\TL_n(q) \bar{p}_t$. Il s'agit de déterminer les modules ainsi obtenus et de recenser les autres.

Compte-tenu des relations de récurrence données dans la proposition \ref{Prop: POIs de TLA}, les POIs non évaluables apparaissent après passage des lignes critiques. En effet, soit $t \in T_n^{\leq 2}$. D'après la proposition \ref{Prop: POIs de TLA}, on sait que :
\[ p_t = \begin{cases} 
	p_{t'} & \text{ si $\hat{t}$ n'existe pas }, \\
	- \frac{[d_t(n-1)]_{A^2}}{[d_t(n-1)+1]_{A^2}} p_{t'} h_{n-1} p_{t'} & \text{ si $\hat{t}$ existe et } d_t(n-1) \not = -1, \\
	p_{t'} + \frac{[d_{\hat{t}}(n-1)]_{A^2}}{[d_{\hat{t}}(n-1)+1]_{A^2}} p_{t'} h_{n-1} p_{t'} & \text{ si $\hat{t}$ existe et } d_t(n-1) = -1.
	\end{cases} \]
On note $\mu$ la forme de $t'$. D'après la remarque \ref{Rem: treillis de TL2}, si $\hat{t}$ existe, alors on a :
\begin{align*} 
&\begin{cases} d_t(n-1)=-1, \\ d_{\hat{t}}(n-1) \in \{ \omega(\mu)-1, -\omega(\mu)-1 \}, \end{cases}
\text{ ou excl. } \; 
\begin{cases} d_t(n-1) \in \{ \omega(\mu)-1, -\omega(\mu)-1 \}, \\ d_{\hat{t}}(n-1)=-1. \end{cases}
\end{align*}
L'idempotent $p_t$ correspondant n'est pas évaluable si :
\begin{align*}
d_t(n-1) = d_{\hat{t}}(n-1) \mod p &\quad \Longleftrightarrow \quad \omega(\mu) = 0 \mod p
&\quad \Longleftrightarrow \quad t' \text{ est critique}.
\end{align*}
Dans ce cas, on considèrera : \index{p3 @$p_{[t]}$}
\begin{equation} 
\label{Eqn: POI sym}
p_{[t]}:= \begin{cases}
		p_t & \text{ si } \bar{t} \text{ n'existe pas, } \\
		p_t + p_{\bar{t}} & \text{ sinon.}
		\end{cases} 
\end{equation}

\begin{Lemme} 
\label{Lemme: idempotents eva1}
Soit $t \in T_n^{\leq 2}$. On suppose que le tableau standard $t$ possède un sous-tableau critique maximal $r \in T_k^{\leq2}$. Si $p_r$ est évaluable, alors $p_{[t]}$ est évaluable.
\end{Lemme}

\begin{proof}
On suppose que $p_r$ est évaluable. On procède par récurrence sur $n \geq k+1$. On se servira régulièrement de la proposition \ref{Prop: POIs de TLA}, et des remarques \ref{Rem: treillis de TL1}, \ref{Rem: treillis de TL2} et \ref{Rem: lignes critiques}.

Pour $n=k+1$, on a $t'=r$ et $\bar{t}= \hat{t}$ lorsqu'ils existent (ils peuvent ne pas exister pour $p=1$). Donc $p_{[t]} = p_r$ est évaluable. Ce cas de figure suffit à prouver le lemme pour $p=1$, valeur pour laquelle toutes les lignes sont critiques. On suppose désormais que $p\geq2$.

Soit $n \geq k+1$. On suppose que, pour tout $u \in T_n^{\leq2}$, si $u$ possède $r$ comme sous-tableau critique maximal, alors $p_{[u]}$ est évaluable. Soit $t \in T_{n+1}$ admettant $r$ comme sous-tableau critique maximal. On note $\lambda$ la forme de $t$, $\mu$ celle de $r$, et $l:= \frac{\omega(\mu)}{p}$. On a $d:= d_t(n) \in \{-1, -\omega(\lambda), \omega(\lambda)\}$. On distingue trois cas.

\begin{Cas}
On suppose que $d \neq -1$ et $\omega(\lambda)=\omega(\mu)=lp$. Alors $\bar{t}$ existe et on a quatre configurations possibles illustrées ci-après.
\begin{figure}[!ht]
\caption{Cas 1 : configurations possibles.}
\[ \begin{array}{|cc|cc|}
	\hline 
	\multicolumn{2}{|c|}{d_t(n)=-\omega(\lambda)} & \multicolumn{2}{c|}{d_t(n)=\omega(\lambda)} \\  
\hline
	\vcenter{ \xymatrix @-2ex {
		\bullet \ar[rd]^{r} & \vdots \ar@{.}[ddd] \\
		& \bullet \ar[ld] \ar[rd] \\
		\bullet \ar[rd]_{t} && \bullet \ar[ld]^{\bar{t}} \\
		& \bullet }} &
	\begin{aligned} 
		&d_t(n) = -\omega(\lambda), \\
		&d_t(n-1) = \omega(\lambda)-1, \\  
		&d_{\bar{t}}(n) = \omega(\lambda), \\
		&d_{\bar{t}}(n-1) = -1.
		\end{aligned} &
	\vcenter{ \xymatrix @-2ex {
		\bullet \ar[rd]^{r} & \vdots \ar@{.}[ddd] \\
		& \bullet \ar[ld] \ar[rd] \\
		\bullet \ar[rd]_{\bar{t}} && \bullet \ar[ld]^{t} \\
		& \bullet }} &
	\begin{aligned} 
		&d_t(n) = \omega(\lambda), \\
		&d_t(n-1) = -1, \\ 
		&d_{\bar{t}}(n) = -\omega(\lambda), \\
		&d_{\bar{t}}(n-1) = \omega(\lambda)-1. \\
		\end{aligned} \\
	\hline
	\vcenter{ \xymatrix @-2ex {
		& \vdots \ar@{.}[ddd] & \bullet \ar[ld]_{r} \\
		& \bullet \ar[ld] \ar[rd] \\
		\bullet \ar[rd]_{t} && \bullet \ar[ld]^{\bar{t}} \\
		& \bullet }} &
	\begin{aligned} 
		&d_t(n) = -\omega(\lambda), \\
		&d_t(n-1) = -1, \\  
		&d_{\bar{t}}(n) = \omega(\lambda), \\
		&d_{\bar{t}}(n-1) = -\omega(\lambda)-1.
		\end{aligned} &
	\vcenter{ \xymatrix @-2ex {
		& \vdots \ar@{.}[ddd] & \bullet \ar[ld]_{r} \\
		& \bullet \ar[ld] \ar[rd] \\
		\bullet \ar[rd]_{\bar{t}} && \bullet \ar[ld]^{t} \\
		& \bullet }} &
	\begin{aligned} 
		&d_t(n) = \omega(\lambda), \\
		&d_t(n-1) = -\omega(\lambda)-1, \\ 
		&d_{\bar{t}}(n) = -\omega(\lambda), \\
		&d_{\bar{t}}(n-1) = -1.
		\end{aligned} \\
	\hline
	\end{array} \]
\end{figure}
Dans chacun des cas, on a $d_{\bar{t}}(n) = -d \neq -1$. On en déduit que :
\begin{align*}
p_{[t]} &\quad= p_t + p_{\bar{t}} = - \frac{[d]_{A^2}}{[d+1]_{A^2}} p_{t'} h_n p_{t'} - \frac{[-d]_{A^2}}{[-d+1]_{A^2}} p_{\bar{t}'} h_n p_{\bar{t}'} \\
&\quad= - \frac{[d]_{A^2}}{[d+1]_{A^2}} p_{t'} h_n p_{t'} - \frac{[d]_{A^2}}{[d-1]_{A^2}} p_{\bar{t}'} h_n p_{\bar{t}'} \\
&\quad= - \frac{[d]_{A^2}}{[d+1]_{A^2}} p_{t'} h_n p_{t'} - \frac{[d]_{A^2}}{[d-1]_{A^2}} (p_{[t']}-p_{t'}) h_n (p_{[t']}-p_{t'}) \\
&\quad= - \frac{[d]_{A^2}}{[d+1]_{A^2}} p_{t'} h_n p_{t'} - \frac{[d]_{A^2}}{[d-1]_{A^2}} (p_{[t']} h_n p_{[t']} - p_{[t']} h_n p_{t'} - p_{t'} h_n p_{[t']} + p_{t'} h_n p_{t'} ) \\
&\quad= \begin{multlined}[t]
	- \frac{[d]_{A^2}}{[d+1]_{A^2} [d-1]_{A^2}} ([d+1]_{A^2} + [d-1]_{A^2}) p_{t'} h_n p_{t'} - \frac{[d]_{A^2}}{[d-1]_{A^2}} p_{[t']} h_n p_{[t']} \\
	+ \frac{[d]_{A^2}}{[d-1]_{A^2}} p_{[t']} h_n p_{t'} + \frac{[d]_{A^2}}{[d-1]_{A^2}} p_{t'} h_n p_{[t']} 
	\end{multlined} \\
&\overset{\eqref{Eqn: Acoeff2}}{=} \begin{multlined}[t]
	- \frac{[d]_{A^2}^2 [2]_{A^2}}{[d+1]_{A^2} [d-1]_{A^2}} p_{t'} h_n p_{t'} - \frac{[d]_{A^2}}{[d-1]_{A^2}} p_{[t']} h_n p_{[t']} \tag{1} \\
	+ \frac{[d]_{A^2}}{[d-1]_{A^2}} p_{[t']} h_n p_{t'} + \frac{[d]_{A^2}}{[d-1]_{A^2}} p_{t'} h_n p_{[t']} .
	\end{multlined}
\end{align*}
Quitte à échanger $t$ et $\bar{t}$, on peut supposer que $d_t(k) = d_t(n-1) = -d-1 \neq -1$. Comme $d_{t'}(n-1)=d_t(n-1)$, on obtient alors :
\[ p_{t'} = - \frac{[d_t(k)]_{A^2}}{[d_t(k)+1]_{A^2}} p_{r} h_{n-1} p_{r} = - \frac{[-d-1]_{A^2}}{[-d]_{A^2}} p_{r} h_{n-1} p_{r} = - \frac{[d+1]_{A^2}}{[d]_{A^2}} p_{r} h_{n-1} p_{r}. \]
En insérant ce résultat dans l'équation $(1)$, on obtient :
\begin{gather*}
p_{[t]} = \begin{multlined}[t]
	- [2]_{A^2} \frac{[d+1]_{A^2}}{[d-1]_{A^2}} p_r h_{n-1} p_r h_n p_r h_{n-1} p_r - \frac{[d]_{A^2}}{[d-1]_{A^2}} p_{[t']} h_n p_{[t']}  \\
	- \frac{[d+1]_{A^2}}{[d-1]_{A^2}} p_{[t']} h_n p_r h_{n-1} p_r - \frac{[d+1]_{A^2}}{[d-1]_{A^2}} p_r h_{n-1} p_r h_n p_{[t']},
	\end{multlined}
\end{gather*}
où $d = \pm \omega(\lambda) \neq 1 \mod p$, $p_r$ est évaluable, et $p_{[t']}$ l'est aussi d'après l'hypothèse de récurrence (en l'occurrence $p_{[t']} = p_r$).
\end{Cas}

\begin{Cas}
On suppose que $d \neq -1$ et $\omega(\lambda) \neq \omega(\mu) = lp$ (ce qui exclut $p=2$). Alors $|\omega(\lambda)-lp| \in \{1,...,p-1\}$, $\bar{t}$ existe, et on a quatre configurations possibles illustrées ci-après.
\begin{figure}[!ht] 
\caption{Cas 2 : configurations possibles.}
\[ \begin{array}{|cc|cc|}
	\hline 
	\multicolumn{2}{|c|}{d_t(n)=-\omega(\lambda)} & \multicolumn{2}{c|}{d_t(n)=\omega(\lambda)} \\  
\hline
	\vcenter{ \xymatrix @-3ex {
		& \bullet \ar[ld] & \vdots \ar@{.}[dd] & \bullet \ar[rd] \\
		\bullet \ar[rd]_{t} &&&& \bullet \ar[ld]^{\bar{t}} \\
		& \bullet & \vdots & \bullet }} &
	\begin{aligned} 
		&d_t(n) = -\omega(\lambda), \\
		&d_{\bar{t}}(n) = 2\omega(\mu)-\omega(\lambda), \\
		&\omega(\lambda) \geq \omega(\mu)-p+2, \\
		&\omega(\lambda)\leq \omega(\mu)-1.
		\end{aligned} &
	\vcenter{ \xymatrix @-3ex {
		& \bullet \ar[ld] & \vdots \ar@{.}[dd] & \bullet \ar[rd] \\
		\bullet \ar[rd]_{\bar{t}} &&&& \bullet \ar[ld]^{t} \\
		& \bullet & \vdots & \bullet }} &
	\begin{aligned} 
		&d_t(n) = \omega(\lambda), \\
		&d_{\bar{t}}(n) = \omega(\lambda)-2\omega(\mu), \\
		&\omega(\lambda) \geq \omega(\mu)+1, \\
		&\omega(\lambda)\leq \omega(\mu)+p-2.
		\end{aligned} \\
	\hline
	\vcenter{ \xymatrix @-3ex {
		\bullet \ar[rd] && \vdots \ar@{.}[dd] && \bullet \ar[ld] \\
		& \bullet \ar[ld]^{\bar{t}} && \bullet \ar[rd]_{t} & \\
		\bullet && \vdots && \bullet }} &
	\begin{aligned} 
		&d_t(n) = -\omega(\lambda), \\
		&d_{\bar{t}}(n) = 2\omega(\mu)-\omega(\lambda), \\
		&\omega(\lambda) \geq \omega(\mu)+2, \\
		&\omega(\lambda) \leq \omega(\mu)+p-1.
		\end{aligned} &
	\vcenter{ \xymatrix @-3ex {
		\bullet \ar[rd] && \vdots \ar@{.}[dd] && \bullet \ar[ld] \\
		& \bullet \ar[ld]^{t} && \bullet \ar[rd]_{\bar{t}} & \\
		\bullet && \vdots && \bullet }} &
	\begin{aligned} 
		&d_t(n) = \omega(\lambda), \\
		&d_{\bar{t}}(n) = \omega(\lambda)-2\omega(\mu), \\
		&\omega(\lambda) \geq \omega(\mu)-p+1, \\
		&\omega(\lambda)\leq \omega(\mu)-2.
		\end{aligned} \\
	\hline
	\end{array} \]
\end{figure}
Dans chacun des cas, on a $d_{\bar{t}}(n)=d \pm 2lp \neq -1$. En procédant comme précédemment, on en déduit que :
\begin{gather*}
p_{[t]} = \begin{multlined}[t]
	- \left( \frac{[d]_{A^2}}{[d+1]_{A^2}} + \frac{[d\pm2lp]_{A^2}}{[d\pm2lp+1]_{A^2}} \right) p_{t'} h_n p_{t'} - \frac{[d\pm2lp]_{A^2}}{[d\pm2lp+1]_{A^2}} p_{[t']} h_n p_{[t']} \\
	+ \frac{[d\pm2lp]_{A^2}}{[d\pm2lp+1]_{A^2}} p_{[t']} h_n p_{t'} + \frac{[d+2lp]_{A^2}}{[d\pm2lp+1]_{A^2}} p_{t'} h_n p_{[t']}.
	\end{multlined}
\end{gather*}
Or, les tableaux standards $t'$ et $\bar{t}'$ n'ont pas la même forme. Donc $p_{\bar{t}'} h_n p_{t'} = 0 = p_{t'} h_n p_{\bar{t}'}$. Il s'ensuit que $p_{[t']} h_n p_{t'} = p_{t'} h_n p_{t'} = p_{t'} h_n p_{[t']}$ et :
\begin{align*}
p_{[t]} &\quad= - \frac{[d]_{A^2} [d\pm2lp+1]_{A^2} - [d+1]_{A^2} [d\pm2lp]_{A^2}}{[d+1]_{A^2} [d\pm2lp+1]_{A^2}} p_{t'} h_n p_{[t']} 
	- \frac{[d\pm2lp]_{A^2}}{[d\pm2lp+1]_{A^2}} p_{[t']} h_n p_{[t']} \\
&\overset{\eqref{Eqn: Acoeff2}}{=} - \frac{[\mp2lp]_{A^2}}{[d+1]_{A^2} [d\pm2lp+1]_{A^2}} p_{t'} h_n p_{[t']} - \frac{[d\pm2lp]_{A^2}}{[d\pm2lp+1]_{A^2}} p_{[t']} h_n p_{[t']} \\
&\quad= \pm (A^{2lp}+A^{-2lp}) \frac{[lp]_{A^2}}{[d+1]_{A^2} [d\pm2lp+1]_{A^2}} p_{t'} h_n p_{[t']} \tag{2}
	- \frac{[d\pm2lp]_{A^2}}{[d\pm2lp+1]_{A^2}} p_{[t']} h_n p_{[t']}.
\end{align*}
Pour les mêmes raisons que dans le cas 1, quitte à échanger $t$ et $\bar{t}$, on peut supposer que $d_{t'}(k) = \alpha lp-1 \neq -1$ et $d_{\bar{t}'}(k) = -1$, avec $\alpha \in \{+,-\}$. En procédant comme dans la preuve de la proposition \ref{Prop: POIs de TLA}, on obtient alors :
\begin{align*}
&p_{[t']} h_k p_{[t']} = p_{t'} h_k p_{t'} = - \frac{[d_{t'}(k)+1]_{A^2}}{[d_{t'}(k)]_{A^2}} p_{t'} = - \frac{[\alpha lp]_{A^2}}{[\alpha lp-1]_{A^2}} p_{t'} \\
\Longleftrightarrow \qquad& p_{t'} = - \frac{[\alpha lp-1]_{A^2}}{[\alpha lp]_{A^2}}  p_{[t']} h_k p_{[t']} = -\alpha \frac{[\alpha lp-1]_{A^2}}{[lp]_{A^2}}  p_{[t']} h_k p_{[t']}. 
\end{align*}
En insérant ce résultat dans l'équation $(2)$, on obtient :
\begin{gather*}
p_{[t]} = \begin{multlined}[t]
	\mp \alpha (A^{2lp}+A^{-2lp}) \frac{[\alpha lp-1]_{A^2}}{[d+1]_{A^2} [d\pm2lp+1]_{A^2}} p_{[t']} h_k p_{[t']} h_n p_{[t']} \\
	- \frac{[d\pm2lp]_{A^2}}{[d\pm2lp+1]_{A^2}} p_{[t']} h_n p_{[t']}, 
	\end{multlined}
\end{gather*}
où $d = \mp \omega(\lambda) \neq -1 \mod p$ (cf. les configurations possibles) et $p_{[t']}$ est évaluable d'après l'hypothèse de récurrence.
\end{Cas}

\begin{Cas}
On suppose que $d=-1$. Alors $\hat{t}$ existe et $d_{\hat{t}}(n) \neq -1$. Donc $p_{[t]} = p_{[t']} - p_{[\hat{t}]}$ où $p_{[\hat{t}]}$ est évaluable d'après les cas précédents, et $p_{[t']}$ l'est aussi d'après l'hypothèse de récurrence.
\end{Cas}

Par conséquent, dans chacun des cas, $p_{[t]}$ est évaluable. Ce qui prouve l'hérédité et achève la récurrence.
\end{proof}

\begin{Lemme} 
\label{Lemme: idempotents eva2}
Soit $t \in T_n^{\leq2}$ un tableau standard régulier. On suppose que le tableau standard régulier $t$ est critique et possède un sous-tableau critique maximal $r \in T_k^{\leq2}$. Si $p_r$ est évaluable, alors $p_t$ est évaluable.
\end{Lemme}

\begin{proof}
On suppose que $p_r$ est évaluable. On note $\lambda$ la forme de $t$, $\mu$ celle de $r$, $l:= \frac{\omega(\mu)}{p}$ et $l':= \frac{\omega(\lambda)}{p}$. On a $d:= d_t(k) \in \{-1, -\omega(\mu)-1, \omega(\mu)-1\}$ (cf. la remarque \ref{Rem: treillis de TL2}). On distingue trois cas. 

\begin{Cas}
On suppose que $d \neq -1$. Alors $d = \mp \omega(\mu)-1$, $l'=l \pm 1$, $r$ admet un sous-tableau maximal $s \in T_j^{\leq2}$, et on a deux configurations possibles illustrées ci-après.
\begin{figure}[!ht]
\caption{Cas 1 : configurations possibles.}
\[ \begin{array}{|c|c|}
	\hline
	d_t(k) = - \omega(\mu) -1 & d_t(k) = \omega(\mu) -1 \\
	\hline
	\xymatrix @-2ex {
		\vdots \ar@{.}[dddddddddddd] &&&&&& \overset{s \quad}{\bullet} \ar@[red][ld] \ar@{.}[dddddddddddd] \\
		&&&&& \bullet \ar[ld] \ar@[red][rd] \\
		&&&& \bullet \ar@{-->}@[red][rd] && \bullet \ar@[red][ld] \\
		&&& \cdots && \color{red} \cdots \ar@[red][rd] \ar@{-->}@[red][ld] \\
		&& \bullet \ar[ld] && \bullet \ar@{-->}@[red][rd] && \bullet \ar@[red][ld] \\
		& \bullet \ar[ld]_{r} \ar@{-->}@[red][rd] &&&& \bullet \ar@[red][rd] \ar@{-->}@[red][ld] \\
		\bullet \ar[rd] && \bullet \ar@{-->}@[red][ld] & \color{red} \cdots & \bullet \ar@{-->}@[red][rd] && \bullet  \ar@[red][ld] \\
		& \bullet \ar[rd] &&&& \bullet \ar@[red][rd] \ar@{-->}@[red][ld] \\
		&& \bullet && \bullet \ar@{-->}@[red][rd] && \bullet \ar@[red][ld] \\ 
		&&& \cdots && \color{red} \cdots \ar@[red][rd] \ar@{-->}@[red][ld] \\
		&&&& \bullet \ar[rd] && \bullet \ar@[red][ld] \\
		&&&&& \bullet \ar@[red][rd]_t^{\color{red} z} \\
		\vdots &&&&&& \bullet } &
		\xymatrix @-2ex {
		\overset{s \quad}{\bullet} \ar@[red][rd] \ar@{.}[dddddddddddd] &&&&&& \vdots \ar@{.}[dddddddddddd] \\
		& \bullet \ar[rd] \ar@[red][ld] \\
		\bullet \ar@[red][rd]&& \bullet \ar@{-->}@[red][ld] \\
		& \color{red} \cdots \ar@[red][ld] \ar@{-->}@[red][rd] && \cdots \\
		\bullet \ar@[red][rd] && \bullet \ar@{-->}@[red][ld] && \bullet \ar[rd] && \\
		& \bullet \ar@[red][ld] \ar@{-->}@[red][rd] &&&& \bullet \ar[rd]^{r} \ar@{-->}@[red][ld] \\
		\bullet \ar@[red][rd] && \bullet \ar@{-->}@[red][ld] & \color{red} \cdots & \bullet \ar@{-->}@[red][rd] && \bullet \ar[ld] \\
		& \bullet \ar@[red][ld] \ar@{-->}@[red][rd] &&&& \bullet \ar[ld] \\
		\bullet \ar@[red][rd] && \bullet \ar@{-->}@[red][ld] && \bullet \\ 
		& \color{red} \cdots \ar@[red][ld] \ar@{-->}@[red][rd] && \cdots \\
		\bullet \ar@[red][rd] && \bullet \ar[ld] \\
		& \bullet \ar@[red][ld]_{\color{red} z}^t \\
		\bullet &&&&&& \vdots } \\
		\hline
		\end{array} \]
\end{figure}
On considère le tableau standard $z \in T_\lambda$ dont le graphe $\gamma(s)$ étend $\gamma(t)$ entre la $l$-ième et la $(l \pm 1)$-ième ligne critique en intersectant $\frac{n-j}{2}+1$ fois la $(l \pm 1)$-ième ligne critique (illustré en rouge dans la figure des configurations possibles). Alors il existe une permutation $\sigma$ de $\{j+1,...,n\}$ telle que $z = \sigma(t)$. \\
Pour $p=1$, on a $\sigma=id=\sigma^{-1}$ et on lui associe l'élément $w^\sigma := 1$. Pour $p \geq 2$ (donc $n \geq 4$), la permutation $\sigma$ se décompose en une composée $\sigma = \sigma(p-1) ... \sigma(2) \sigma(1)$ avec :
\[ \forall \kappa \in \{1,...,p-2\} \quad \sigma(\kappa) := \prod_{\substack{j+1 \leq i \leq n-1 \\ d_{\sigma(\kappa)(t)}(i) = \mp lp-\kappa}} \sigma_i. \]
En l'occurrence, on a $\sigma(1) = \sigma_k$ et $\sigma(p-1) = \sigma_{j+2} \sigma_{j+4} ... \sigma_{n-2}$. Pour tout $i \in \{j+1,...,n-1\}$, en identifiant $\sigma_i$ avec $h_i$, on obtient un élément de $\TL_n(A^2)$, noté $w^\sigma := h(p-1) ... h(2) h(1)$, avec :
\[ \forall \kappa \in \{1,...,p-2\} \quad h(\kappa+1) := \prod_{\substack{j+1 \leq i \leq n-1 \\ d_{\sigma(\kappa)(t)}(i) = \mp \omega - \kappa}} h_i. \]
De même, pour la permutation $\sigma^{-1}$ et sa décomposition $\sigma^{-1} = \sigma(1)^{-1} \sigma(2)^{-1} ... \sigma(p-1)^{-1}$, on obtient un mot $w^{\sigma^{-1}}$ en $\{h_{j+1},...,h_{n-1}\}$. \\
Avec ces éléments, pour tout $p \in \N^*$, on montre qu'il existe $P(A) \in \C[A]_{(A^2-q)}$ tel que :
\[ p_t = P(A) p_{[t]} w^{\sigma^{-1}} (h_{j+1} h_{j+3} ... h_{n-1}) w^{\sigma} p_{[t]} . \]
Pour cela, on calcule $p_{[t]} w^{\sigma^{-1}} (h_{j+1} h_{j+3} ... h_{n-1}) w^{\sigma} p_{[t]}$ à l'aide des résultats généraux suivants, qui découlent de la remarque \ref{Rem: treillis de TL2}.
\[ \begin{array}{|c|c|} 
\hline
d_u(i) = \omega(\lambda^{(i)})-1 & d_u(i) = -\omega(\lambda^{(i)})-1 \\
\hline
\vcenter{ \xymatrix{
	\overset{\qquad}{} & \bullet \ar@{-->}[ld] \ar[rd] & \overset{\color{gray} \omega(\lambda^{(i)})}{} \ar@[gray]@{--}[dd] \\
	\bullet \ar@{-->}[rd]_{\sigma_i(u)} && \overset{\lambda^{(i)}}{\bullet} \ar[ld]^u \\
	& \bullet &}} &
\vcenter{ \xymatrix{
	\overset{\color{gray} \omega(\lambda^{(i)})}{} \ar@[gray]@{--}[dd] & \bullet \ar[ld] \ar@{-->}[rd] & \overset{\qquad}{} \\
	\overset{\lambda^{(i)}}{\bullet} \ar[rd]_{u} && \bullet \ar@{-->}[ld]^{\sigma_i(u)} \\
	& \bullet & }} \\
p_{\sigma_i(u)} h_i p_u = - \frac{[\omega(\lambda^{(i)})-2]_{A^2}}{[\omega(\lambda^{(i)})-1]_{A^2}} p_{\sigma_i(u)} & 
p_{\sigma_i(u)} h_i p_u = - \frac{[\omega(\lambda^{(i)})+2]_{A^2}}{[\omega(\lambda^{(i)})+1]_{A^2}} p_{\sigma_i(u)} \\
p_{u} h_i p_u = - \frac{[\omega(\lambda^{(i)})]_{A^2}}{[\omega(\lambda^{(i)})-1]_{A^2}} p_{u} & 
p_{u} h_i p_u = - \frac{[\omega(\lambda^{(i)})]_{A^2}}{[\omega(\lambda^{(i)})+1]_{A^2}} p_{u} \\
\hline
\end{array} \]
On en déduit qu'il existe $P_{t,z}(A), P_{z,t}(A) \in \C[A]_{(A^2-q)}$, non divisibles par $(A^2-q)$, tels que :
\begin{gather*}
p_z w^\sigma p_{[t]} = p_z w^\sigma p_t = P_{t,z}(A) [(l \pm 1)p]_{A^2}^{\frac{n-j}{2}-1} p_z , \\
p_z (h_{j+1} h_{j+3} ... h_{n-1}) p_z = \left( -\frac{[(l \pm 1)p \mp 1]_{A^2}}{[(l \pm 1)p]_{A^2}} \right)^{\frac{n-j}{2}} p_z, \\
p_{[t]} w^{\sigma^{-1}} p_z = p_t w^{\sigma^{-1}} p_z = P_{z,t}(A) [lp]_{A^2} p_t.
\end{gather*}
Par conséquent, on a :
\begin{align*}
p_{[t]} w^{\sigma^{-1}} (h_{j+1} h_{j+3} ... h_{n-1}) w^{\sigma} p_{[t]} &= p_{[t]} w^{\sigma^{-1}} p_z (h_{j+1} h_{j+3} ... h_{n-1}) p_z w^{\sigma} p_{[t]} \\
&= (-1)^{\frac{n-j}{2}} P_{t,z}(A) P_{z,t}(A) \frac{[(l \pm 1)p \mp 1]_{A^2}^{\frac{n-j}{2}} [lp]_{A^2}}{[(l \pm 1)p]_{A^2}} p_t .
\end{align*}
Le facteur $(A^2-q)$ de $[lp]_{A^2}$ se compensent avec celui de $[(l \pm 1)p]_{A^2}$. D'où l'existence de $P \in \C(A)_{(A^2-q)}$ tel que :
\[ p_t = P(A) p_{[t]} w^{\sigma^{-1}} (h_{j+1} h_{j+3} ... h_{n-1}) w^{\sigma} p_{[t]}. \]
\end{Cas}

\begin{Cas}
On suppose que $d=-1$. Alors on a :
\[ p_{t} = \begin{cases} 
	p_{[t]} - p_{\bar{t}} & \text{ si $\bar{t}$ existe, } \\
	p_{[t]} & \text{ si $\hat{t}$ n'existe pas, }
	\end{cases}  \]
où $p_{\bar{t}}$ est évaluable d'après le cas précédent car $d_{\bar{t}}(k) \neq -1$.
\end{Cas}

Enfin, d'après le lemme \ref{Lemme: idempotents eva1}, on sait que $p_{[t]}$ est évaluable. Par conséquent, dans chacun des cas, $p_{t}$ est aussi évaluable.
\end{proof}

\begin{Thm} 
\label{Thm: idempotents eva}
Soit $t \in T_n^{\leq 2}$ un tableau standard régulier.
\begin{enumerate}[(a)]
	\item Si $t$ ne possède pas de sous-tableau critique propre, alors $p_{[t]}=p_t$ est évaluable.
	\item Si $t$ est critique, alors $p_t$ est évaluable.
	\item Si $t$ n'est pas critique et possède un sous-tableau critique maximal $r \in T_k^{\leq2}$, alors $p_{[t]}$ est évaluable. De plus, $p_t$ et $p_{\bar{t}}$ ne sont pas évaluables, et l'élément $n_{[t]} := p_{[t]} h_k p_{[t]}$ \index{n6 @$n_{[t]}$} s'évalue sur un élément nilpotent d'ordre 2.
\end{enumerate}
\end{Thm}

\begin{proof}
\begin{enumerate}[(a)]
	\item On suppose que $t$ ne possède pas de sous-tableau critique propre. On rajoute la valeur $n=1$ (valeur clé pour le cas $p=1$) et on procède par récurrence sur $n \geq 1$.
	
	Pour $n=1$, on a $T_1^{\leq2} = \{ \lyoung{1} \}$ et $p_{\lyoung{1}} = 1$ (cf. la proposition \ref{Prop: POIs de TLA}). Donc $p_{\lyoung{1}}$ est évaluable. Ce cas de figure suffit à prouver l'assertion pour $p=1$, valeur pour laquelle toutes les lignes sont critiques. On suppose désormais que $p \geq 2$.
	
	Soit $n \geq 1$. Suppose que, pour tout $u \in T_n^{\leq2}$, si $u$ ne possède pas de sous-tableau critique propre, alors $p_u$ est évaluable. Soit $t \in T_{n+1}$ n'admettant pas de sous-tableau critique propre. On note $\lambda$ sa forme. D'après la proposition \ref{Prop: POIs de TLA}, on a :
	\[ p_t = \begin{cases} 
	p_{t'} & \text{ si $\hat{t}$ n'existe pas }, \\
	- \frac{[d_t(n)]_{A^2}}{[d_t(n)+1]_{A^2}} p_{t'} h_n p_{t'} & \text{ si $\hat{t}$ existe et } d_t(n) \not = -1, \\
	p_{t'} + \frac{[d_{\hat{t}}(n)]_{A^2}}{[d_{\hat{t}}(n)+1]_{A^2}} p_{t'} h_n p_{t'} & \text{ si $\hat{t}$ existe et } d_t(n) = -1,
	\end{cases} \]
	où $p_{t'}$ est évaluable d'après l'hypothèse de récurrence. Par ailleurs, comme $t$ ne possède pas de sous-tableau critique propre, son graphe $\gamma(t)$ ne traverse pas la première ligne critique (cf. la remarque \ref{Rem: lignes critiques}). Avec la remarque \ref{Rem: treillis de TL2}, on en déduit que :
	\begin{align*}
	&d_t(n) \in \{-p,...,-3\} \cup \{1,...,p-2\} && \text{ si $\hat{t}$ existe et } d_t(n) \not = -1, \\
	&d_{\hat{t}}(n) \in \{-p,...,-3\} \cup \{1,...,p-2\} && \text{ si $\hat{t}$ existe et } d_t(n) = -1.
	\end{align*}
	Par conséquent, $p_t$ est évaluable. Ce qui prouve l'hérédité et achève la récurrence.
	\item On suppose que $t$ est critique. On procède par récurrence sur le nombre $N \geq 0$ de sous-tableau(x) critique(s) propre(s).
	
	Pour $N=0$, $p_t$ est évaluable d'après le cas 1.
	
	Soit $N \geq 0$. On suppose que, pour tout tableau standard $u$ régulier et critique, si $u$ possède $N$ sous-tableau(x) critique(s) propre(s), alors $p_u$ est évaluable. Soit $t$ un tableau standard régulier et critique, admettant $N+1$ sous-tableau(x) critique(s) propre(s). Il possède un sous-tableau critique maximal $r$ régulier, admettant $N$ sous-tableau(x) critique(s) propre(s). Donc $p_r$ est évaluable d'après l'hypothèse de récurrence. Il s'ensuit que $p_t$ l'est aussi d'après le lemme \ref{Lemme: idempotents eva2}. Ce qui prouve l'hérédité et achève la récurrence.
	\item On suppose que $t$ n'est pas critique et possède un sous-tableau critique maximal $r \in T_k^{\leq2}$. Comme $t$ est régulier, le sous-tableau $r$ l'est aussi. Donc $p_r$ est évaluable d'après $(b)$. Il s'ensuit que $p_{[t]}$ l'est aussi d'après le lemme \ref{Lemme: idempotents eva1}. 
	
	Montrons que $n_{[t]}:= p_{[t]} h_k p_{[t]}$ s'évalue sur un élément nilpotent d'ordre 2. On note $\mu$ la forme de $r$ et $l:= \frac{\omega(\mu)}{p}$. On a $d:= d_t(k) \in \{-1, -lp-1, lp-1\}$ (cf. la remarque \ref{Rem: treillis de TL2}). Quitte à échanger $t$ et $\bar{t}$, on peut supposer que $d_t(k) = \mp lp-1 \neq -1$ et $d_{\bar{t}}(k) = -1$. En procédant comme dans la preuve de la proposition \ref{Prop: POIs de TLA}, on obtient alors :
	\begin{equation} \tag{1}
	n_{[t]} = p_{[t]} h_k p_{[t]} = p_{t} h_k p_{t} = - \frac{[d_t(k)+1]_{A^2}}{[d_t(k)]_{A^2}} p_t = - \frac{[\mp lp]_{A^2}}{[\mp lp-1]_{A^2}} p_t.
	\end{equation}
	Donc :
	\[ n_{[t]}^2 = \left( - \frac{[\mp lp]_{A^2}}{[\mp lp-1]_{A^2}} \right)^2 p_t = - \frac{[\mp lp]_{A^2}}{[\mp lp-1]_{A^2}} n_{[t]} \]
	s'évalue sur $0$. Pour montrer que $n_{[t]}$ ne s'évalue pas sur $0$, on considère le tableau standard $T \in T_{2n-k}^{\leq2}$ (non régulier) dont le graphe $\gamma(T)$ étend $\gamma(t)$ de $n-k$ arrêtes par symétrie par rapport à la ligne $\Delta_n^{\leq2}$. De même que pour $p_{[t]}$, l'idempotent $p_{[T]}$ est évaluable. En procédant comme dans la preuve du lemme \ref{Lemme: idempotents eva2}, on montre qu'il existe un élément $h \in \TL_{2n-k}$ tel que $p_{[T]} h p_{[T]}$ est évaluable et dont la matrice de représentation $M(p_{[T]} h p_{[T]})$ dans la base $\{v_T, v_{\bar{T}}\}$ est de la forme :
	\[ M(p_{[T]} h p_{[T]}) = \frac{1}{[\mp lp]_{A^2}} \begin{pmatrix}
		P_{T,T}(A) & P_{\bar{T},T}(A) \\
		P_{T,\bar{T}}(A) & P_{\bar{T},\bar{T}}(A)
	\end{pmatrix}, \]
	où $P_{T,T}(A), P_{\bar{T},T}(A), P_{T,\bar{T}}(A), P_{\bar{T},\bar{T}}(A) \in \C[A]_{(A^2-q)}$ ne sont pas divisibles par $(A^2-q)$. On en déduit la matrice de représentation de $n_{[t]} p_{[T]} h p_{[T]}$ dans cette même base :
	\[ M(n_{[t]} p_{[T]} h p_{[T]}) = -\frac{1}{[\mp lp-1]_{A^2}} \begin{pmatrix}
		P_{T,T}(A) & P_{\bar{T},T}(A) \\
		0 & 0
	\end{pmatrix}. \]
	Donc $n_{[t]} p_{[T]} h p_{[T]}$, et par suite $n_{[t]}$, ne s'évaluent pas sur $0$.

	D'après l'équation $(1)$, cela implique que $p_t$ n'est pas évaluable, et donc $p_{\bar{t}} = p_{[t]}-p_t$ non plus.
\end{enumerate}
\end{proof}

\begin{figure}[!ht]
\caption{Idempotents évaluables pour $p=3$ (i).}
\label{Fig: idempotents eva}
\[ \Yboxdim{7pt} 
\xymatrix @-1ex { 
	\overset{\emptyset}{\bullet} \ar@[green]@<0.3ex>[rd] \ar@[blue]@<-0.3ex>[rd] \ar@[red]@<-1ex>[rd] \ar@<1ex>[rd] && \ar@{.}[dddddd] &&& \ar@{.}[dddddd] \\
	& \overset{\yng(1)}{\bullet} \ar@[green]@<-0.3ex>[ld] \ar@[blue]@<0.3ex>[ld] \ar@[green]@{..>}@<0.3ex>[rd] \ar@[blue]@{..>}@<-0.3ex>[rd] \ar@[red]@<-1ex>[rd] \ar@<1ex>[rd] &&&& \\
	\overset{\yng(1,1)}{\bullet} \ar@[green]@<-0.3ex>[rd] \ar@[blue]@<0.3ex>[rd] && \overset{\yng(2)}{\bullet} \ar@[green]@{..>}@<-0.6ex>[ld] \ar@[blue]@{..>}@<0ex>[ld] \ar@[red]@<0.6ex>[ld] \ar@[green]@{..>}@<0.3ex>[rd] \ar@[blue]@{..>}@<-0.3ex>[rd] \ar@[red]@{-->}@<-1ex>[rd] \ar@<1ex>[rd] \\
	& \overset{\yng(2,1)}{\bullet} \ar@[green]@<-0.3ex>[ld] \ar@[blue]@<0.3ex>[ld] \ar@[green]@{-->}@<-0.6ex>[rd] \ar@[blue]@{-->}@<0ex>[rd] \ar@[red]@<0.6ex>[rd] && \overset{\yng(3)}{\bullet} \ar@[red]@{-->}[ld] \ar@[green]@{..>}@<0ex>[rd] \ar@[blue]@{..>}@<-0.6ex>[rd] \ar@<0.6ex>[rd] \\
	\overset{\yng(2,2)}{\bullet} \ar@[green]@<-0.3ex>[rd] \ar@[blue]@<0.3ex>[rd] && \overset{\yng(3,1)}{\bullet} \ar@[green]@{-->}@<-0.6ex>[ld] \ar@[blue]@{-->}@<0ex>[ld] \ar@[red]@<0.6ex>[ld] \ar@[green]@{-->}@<0.6ex>[rd] \ar@[blue]@{-->}@<0ex>[rd] \ar@[red]@<-0.6ex>[rd] && \overset{\yng(4)}{\bullet} \ar@[green]@{..>}@<0.3ex>[ld] \ar@[blue]@{..>}@<-0.3ex>[ld] \ar[rd] \\
	& \overset{\yng(3,2)}{\bullet} \ar@[green]@<-0.3ex>[ld] \ar@[red]@<0.3ex>[ld] \ar@[blue]@<-0.3ex>[rd] \ar@[red]@<0.3ex>[rd] && \overset{\yng(4,1)}{\bullet} \ar@[blue]@{-->}@<0.3ex>[ld] \ar@[red]@<-0.3ex>[ld] \ar@[green]@{-->}@<0.3ex>[rd] \ar@[red]@<-0.3ex>[rd] && \overset{\yng(5)}{\bullet} \ar[ld] \ar[rd] \\
	\overset{\yng(3,3)}{\bullet} && \overset{\yng(4,2)}{\bullet} && \overset{\yng(5,1)}{\bullet} && \overset{\yng(6)}{\bullet} } \]
\end{figure}
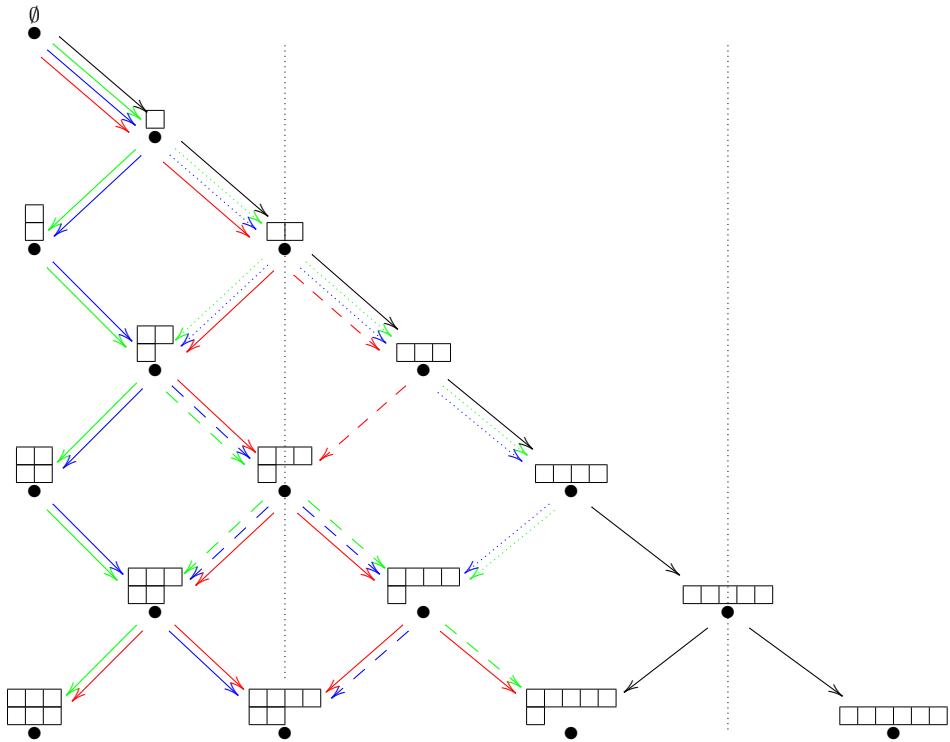

\begin{figure}[!ht]
\ContinuedFloat
\caption{Idempotents évaluables pour $p=3$ (ii).}
\footnotesize
\[ \begin{array}{|c|c|c|}
	\hline
	\text{ Graphes } &
		\text{POIs de } \TL_6(A^2) & 
		\text{Idempotents de } \TL_6(q) \\
	\hline
	\xymatrix@-2ex{\ar@[green][r]&} & 
		p_{\lyoung{135,246}} & 
		\bar{p}_{\lyoung{135,246}} \\
	\xymatrix@-2ex{\ar@{-->}@[green][r]&} & 
		p_{\lyoung{134,256}}, \;  p_{\lyoung{13456,2}} & 
		\bar{p}_{ \Big[ \lyoung{134,256} \Big] } \\
	\xymatrix@-2ex{\ar@{..>}@[green][r]&} & 
		p_{\lyoung{125,346}}, \;  p_{\lyoung{12346,5}} & 
		\bar{p}_{ \Big[ \lyoung{125,346} \Big] } \\
	\xymatrix@-2ex{\ar@[red]@<2ex>[r] \ar@{-->}@[red]@<-2ex>[r]&} & 
		\begin{cases}
			p_{\lyoung{124,356}}, \; p_{\lyoung{12356,4}} \\
	 		p_{\lyoung{123,456}}, \; p_{\lyoung{12456,3}}
		\end{cases} & 
		\begin{cases}
			\bar{p}_{ \Big[ \Big[ \lyoung{124,356} \Big] \Big]^1 } := - \bar{p}_{ \Big[ \lyoung{124,3} \Big] } \frac{h_3}{[2]_{A^2}} \left( 1 + \frac{h_5}{[2]_{A^2}} \right) \bar{p}_{ \Big[ \lyoung{124,3} \Big] } \\
			\bar{p}_{ \Big[ \Big[ \lyoung{124,356} \Big] \Big]^2 } := \bar{p}_{ \Big[ \lyoung{124,3} \Big] } \left( 1 + \frac{h_3}{[2]_{A^2}} \right) \left( 1 + \frac{h_5}{[2]_{A^2}} \right) \bar{p}_{ \Big[ \lyoung{124,3} \Big] }
		\end{cases} \\
	\xymatrix@-2ex{\ar@[blue][r]&} & 
		p_{\lyoung{1356,24}} & 
		\bar{p}_{{\lyoung{1356,24}}} \\
	\xymatrix@-2ex{\ar@{-->}@[blue][r]&} &
		p_{\lyoung{1346,25}}, \; p_{\lyoung{1345,26}} & 
		\begin{cases}
			\bar{p}_{ \Big[ \lyoung{1346,25} \Big]^1 } :=  -\bar{p}_{ \Big[ \lyoung{1346,25} \Big] } \frac{h_5}{[2]_{A^2}} \bar{p}_{ \Big[ \lyoung{1346,25} \Big] } \\
			\bar{p}_{ \Big[ \lyoung{1346,25} \Big]^2 } :=\bar{p}_{ \Big[ \lyoung{1346,25} \Big] } \left( 1 + \frac{h_5}{[2]_{A^2}} \right) \bar{p}_{ \Big[ \lyoung{1346,25} \Big] }
		\end{cases} \\ 
	\xymatrix@-2ex{\ar@{..>}@[blue][r]&} &
		p_{\lyoung{1256,34}}, \; p_{\lyoung{1234,56}} & 
		\begin{cases}
			\bar{p}_{ \Big[ \lyoung{1256,34} \Big]^1 } := \bar{p}_{ \Big[ \lyoung{1256,34} \Big] } \frac{h_4 h_3 h_5 h_4}{[2]_{A^2}^4} \bar{p}_{ \Big[ \lyoung{1256,34} \Big] } \\
			\bar{p}_{ \Big[ \lyoung{1256,34} \Big]^2 } := \bar{p}_{ \Big[ \lyoung{1256,34} \Big] } \left( 1 - \frac{h_4 h_3 h_5 h_4}{[2]_{A^2}^4} \right) \bar{p}_{ \Big[ \lyoung{1256,34} \Big] } 
		\end{cases} \\ 
	\xymatrix@-2ex{\ar@[red]@<2ex>[r] \ar@{-->}@[red]@<-2ex>[r]&} &
		\begin{cases}
			p_{\lyoung{1246,35}}, \; p_{\lyoung{1245,36}} \\
			p_{\lyoung{1236,45}}, \; p_{\lyoung{1235,46}} \\
		\end{cases} &  
		\begin{cases}
			\bar{p}_{ \Big[ \Big[ \lyoung{1246,35} \Big] \Big]^1 } := \bar{p}_{ \Big[ \Big[ \lyoung{1246,35} \Big] \Big] } \frac{h_3 h_5}{[2]_{A^2}^2} \bar{p}_{ \Big[ \Big[ \lyoung{1246,35} \Big] \Big] } {}^{(*)} \\
			\bar{p}_{ \Big[ \Big[ \lyoung{1246,35} \Big] \Big]^2 } := - \bar{p}_{ \Big[ \Big[ \lyoung{1246,35} \Big] \Big] } \left( 1 + \frac{h_3}{[2]_{A^2}} \right) \frac{h_5}{[2]_{A^2}} \bar{p}_{ \Big[ \Big[ \lyoung{1246,35} \Big] \Big] } \\
			\bar{p}_{ \Big[ \Big[ \lyoung{1246,35} \Big] \Big]^3 } := - \bar{p}_{ \Big[ \Big[ \lyoung{1246,35} \Big] \Big] } \frac{h_3}{[2]_{A^2}} \left( 1 + \frac{h_5}{[2]_{A^2}} \right) \bar{p}_{ \Big[ \Big[ \lyoung{1246,35} \Big] \Big] } \\
			\bar{p}_{ \Big[ \Big[ \lyoung{1246,35} \Big] \Big]^4 } := \bar{p}_{ \Big[ \Big[ \lyoung{1246,35} \Big] \Big] } \left( 1 + \frac{h_3}{[2]_{A^2}} \right) \left( 1 + \frac{h_5}{[2]_{A^2}} \right) \bar{p}_{ \Big[ \Big[ \lyoung{1246,35} \Big] \Big] } 	
		\end{cases}\\
	\xymatrix@-2ex{\ar[r]&} & 
		p_{\lyoung{12345,6}}, \; p_{\lyoung{123456}} & 
		\bar{p}_{ \Big[ \lyoung{12345,6} \Big]} \\
	\hline
	\end{array} \]
	
$ ^{(*)} \bar{p}_{ \Big[ \Big[ \lyoung{1246,35} \Big] \Big] } := \bar{p}_{ \Big[ \lyoung{1246,35} \Big] } + \bar{p}_{ \Big[ \lyoung{1236,45} \Big] } $
\end{figure}

\begin{Rem}
\label{Rem: idempotents eva}
Le lemme \ref{Lemme: idempotents eva2} et l'assertion $(b)$ du théorème \ref{Thm: idempotents eva} ne se généralisent pas aux tableaux standards \emph{non réguliers}. Toutefois, pour tout tableau standard $t \in T_n^{\leq2}$ non régulier de forme $\lambda$, il est possible de montrer les assertions suivantes (cf. les preuves de \cite[Lemme 1.2, Prop. 2.1]{GW93}, où les algèbres de Temperley-Lieb sont étudiées avec un autre système de générateurs).
\begin{enumerate}[(a)]
	\item Si $t$ est critique, alors il existe une unique famille $\mathbb{S}(t)$ de tableaux \emph{non réguliers} de forme $\lambda$, comprenant $t$, telle que $\sum_{s \in \mathbb{S}(t)} p_s$ est évaluable et se décompose en $\card(\mathbb{S}(t))$ idempotents orthogonaux évaluables.
	\item Si $t$ n'est pas critique, alors il existe une unique famille $\mathbb{S}(t)$ de tableaux \emph{non réguliers} de forme $\lambda$, comprenant $t$, telle que $\sum_{s \in \mathbb{S}(t)} p_{[s]}$ est évaluable et se décompose en $\card(\mathbb{S}(t))$ idempotents orthogonaux évaluables.
\end{enumerate}
Ces $\card(\mathbb{S}(t))$ idempotents orthogonaux évaluables se construisent par induction et en utilisant le même procédé que dans la preuve du lemme \ref{Lemme: idempotents eva2}. En général, il est difficile de les exprimer en fonction des POIs de $\TL_n(A^2)$. Des exemples simples sont explicités dans la figure \ref{Fig: idempotents eva}.
\end{Rem}

\begin{Def}
Soient $\lambda \in \Delta_n^{\leq2}$ et $\mathbb{A}_p$ le groupe de réflexion engendré par les symétries orthogonales par rapport aux lignes critiques.
\begin{enumerate}[(i)]
	\item Pour toute réflexion $s \in \mathbb{S}_p$, on note $s(\lambda)$ le diagramme de Young tel que les sommets du treillis de Temperley-Lieb étiquetés par $\lambda$ et $s(\lambda)$ sont reliés par $s$.
	\item On définit : \index{lambda @$[\lambda]$}
	\[ [\lambda] := \begin{cases}
		\{ \lambda \} & \text{ si $\lambda$ est critique, } \\
		\left\{ \mu \in \Delta_n^{\leq2} \; ; \; \exists s \in \mathbb{A}_p \text{ telle que } \mu = s(\lambda) \right\} & \text{ sinon. }
		\end{cases} \]
\end{enumerate}
\end{Def}

\begin{figure}[!ht]
\caption{Partitions de diagrammes de Young pour $p=3$.}
\label{Fig: orbites de Young}
\Yboxdim{7pt}
\[ \xymatrix @-1ex {
	\overset{\emptyset}{\bullet} \ar[rd] && \ar@{.}[dddddd] &&& \ar@{.}[dddddd] \\
	& \overset{\yng(1)}{\bullet} \ar[ld] \ar[rd] &&&& \\
	\overset{\color{blue} \yng(1,1)}{\bullet} \ar[rd] && \overset{\yng(2)}{\bullet} \ar[ld] \ar[rd] \\
	& \overset{\color{red} \yng(2,1)}{\bullet} \ar[ld] \ar[rd] && \overset{\color{red} \yng(3)}{\bullet} \ar[ld] \ar[rd] \\
	\overset{\color{orange} \yng(2,2)}{\bullet} \ar[rd] && \overset{\yng(3,1)}{\bullet} \ar[ld] \ar[rd] && \overset{\color{orange} \yng(4)}{\bullet} \ar[ld] \ar[rd] \\
	& \overset{\color{yellow} \yng(3,2)}{\bullet} \ar[ld] \ar[rd] && \overset{\color{yellow} \yng(4,1)}{\bullet} \ar[ld] \ar[rd] && \overset{\yng(5)}{\bullet} \ar[ld] \ar[rd] \\
	\overset{\color{green} \yng(3,3)}{\bullet} && \overset{\yng(4,2)}{\bullet} && \overset{\color{green} \yng(5,1)}{\bullet} && \overset{\color{green} \yng(6)}{\bullet} } \]
\end{figure}
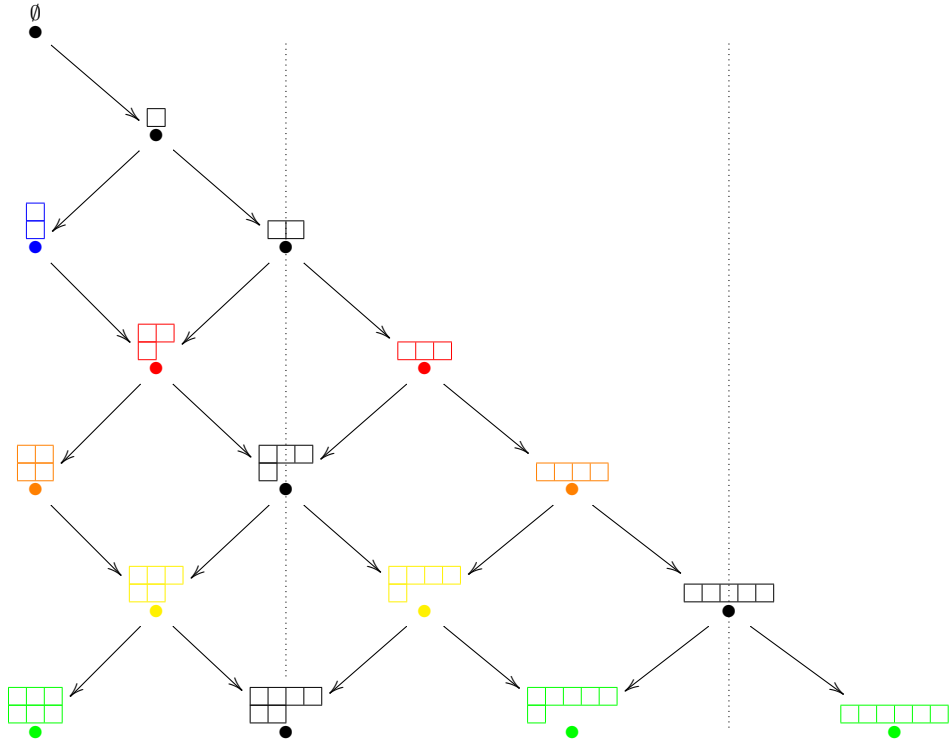

Les ensembles $\{ [\lambda] \; ; \; \lambda \in \Delta_n^{\leq2} \}$ donnent une partition de $\Delta_n^{\leq2}$, illustrée par un jeu de couleurs dans la figure \ref{Fig: orbites de Young}. De plus, l'inventaire des idempotents évaluables donné dans le théorème \ref{Thm: idempotents eva} et la remarque \ref{Rem: idempotents eva} incite à considérer les idempotents centraux : \index{z @$z_{[\lambda]}$}
\begin{equation}
\label{Eqn: PCI sym}
\forall \lambda \in \Delta_n^{\leq 2} \qquad z_{[\lambda]} := \sum_{\mu \in [\lambda]} z_{\mu}
= \sum_{\mu \in [\lambda]} \sum_{\substack{ t \in T_{\mu} \\ t \text{ régulier} }}  p_t + \sum_{\mu \in [\lambda]} \sum_{\substack{ t \in T_{\mu} \\ t \text{ non régulier} }} p_t . 
\end{equation}
Dans chaque somme, les termes non évaluables se regroupent (deux à deux pour les tableaux \emph{réguliers}) en des idempotents évaluables. Les éléments $z_{[\lambda]}$, $\lambda \in \Delta_n^{\leq2}$, sont donc évaluables. La structure des idéaux $\bar{z}_{[\lambda]} \TL_n(q)$ correspondants est étudiée dans \cite[§ 2]{GW93}. On synthétise et reformule leurs résultats dans le :

\begin{Thm}[{\cite[§ 2]{GW93}}]
\label{Thm: structure de TLq}
Soit $\lambda \in \Delta_n^{\leq2}$. 
\begin{enumerate}[(a)]
	\item Si $[\lambda] = \{\lambda\}$, alors $\bar{z}_\lambda \TL_n(q)$ est un idéal simple et vérifie :
	\[ \bar{z}_\lambda \TL_n(q) \cong \End(\C^{f_\lambda}), \]
	où $f_\lambda$ est le nombre de tableaux standards de forme $\lambda$.
	\item Si $[\lambda] = \{\mu_1,...,\mu_l\}$ avec $l \geq 2$ et $\omega(\mu_1) < ... < \omega(\mu_l)$, alors $\bar{z}_{[\lambda]} \TL_n(q)$ est un idéal indécomposable et admet une filtration d'idéaux $0 \subset R_{[\lambda]}^2 \subset R_{[\lambda]} \subset \bar{z}_{[\lambda]} \TL_n(q)$ tels que :
	\begin{align*}
	& \bar{z}_{[\lambda]} \TL_n(q) / R_{[\lambda]} \cong \bigoplus_{i=1}^l \End(\C^{f_{\mu_i}^L}), \\
	& R_{[\lambda]} / R_{[\lambda]}^2 \cong \bigoplus_{i=1}^{l-1} \Hom(\C^{f_{\mu_i}^L},\C^{f_{\mu_{i+1}}^L}) \oplus \Hom(\C^{f_{\mu_{i+1}}^L},\C^{f_{\mu_i}^L}), \\
	& R_{[\lambda]}^2 \cong \bigoplus_{i=2}^l \End(\C^{f_{\mu_i}^L}),
	\end{align*}
où, pour tout $\mu \in [\lambda]$, $f_\mu^L$ est le nombre de tableaux standards de forme $\mu$ qui ne possèdent pas de sous-tableau critique maximal ou en possèdent un dont la forme $\nu$ vérifie $\omega(\nu) \leq \omega(\mu)$.
\end{enumerate}
\end{Thm}

Pour davantage de lisibilité, la structure des idéaux $\bar{z}_{[\lambda]} \TL_n(q)$, $[\lambda] \neq \{\lambda\}$, est synthétisée dans la figure \ref{Fig: structure des ideaux de TLq}.

\begin{figure}[!ht] 
\caption{Structure des idéaux $\bar{z}_{[\lambda]} \TL_n(q)$, $[\lambda] = \{\mu_1,...,\mu_l\}$.}
\label{Fig: structure des ideaux de TLq}
Représentation diagrammatique de l'action à gauche de $\TL_n(q)$ :
\[ \scalebox{0.54}{\xymatrix{
	\bar{z}_{[\lambda]} \TL_n(q) / R_{[\lambda]} & \overset{\End( \C^{f_{\mu_1}^L} )}{\bullet} \ar[rrd] &&& \overset{\End( \C^{f_{\mu_2}^L} )}{\bullet} \ar[lld] \ar[rrd] \ar[dd] &&& \overset{\End( \C^{f_{\mu_3}^L} )}{\bullet} \ar[lld] \ar[rd] \ar[dd] &&&& \overset{\End( \C^{f_{\mu_l}^L} )}{\bullet} \ar[ld] \ar[dd] \\
	R_{[\lambda]} / R_{[\lambda]}^2 && \overset{\Hom( \C^{f_{\mu_1}^L}, \C^{f_{\mu_2}^L} )}{\bullet} \ar[rrd] & \overset{\Hom( \C^{f_{\mu_2}^L}, \C^{f_{\mu_1}^L} )}{\bullet} && \overset{\Hom( \C^{f_{\mu_2}^L}, \C^{f_{\mu_3}^L} )}{\bullet} \ar[rrd] & \overset{\Hom( \C^{f_{\mu_3}^L}, \C^{f_{\mu_2}^L} )}{\bullet} \ar[lld] && \ar[ld] & \cdots & \ar[rd] \\
	R_{[\lambda]}^2 &&&& \overset{\End( \C^{f_{\mu_2}^L} )}{\bullet} &&& \overset{\End( \C^{f_{\mu_3}^L} )}{\bullet} &&&& \overset{\End( \C^{f_{\mu_l}^L} )}{\bullet}
	}} \]
Représentation diagrammatique de l'action à droite de $\TL_n(q)$ :
\[ \scalebox{0.54}{\xymatrix{
	\bar{z}_{[\lambda]} \TL_n(q) / R_{[\lambda]} & \overset{\End( \C^{f_{\mu_1}^L} )}{\bullet} \ar[rd] &&& \overset{\End( \C^{f_{\mu_2}^L} )}{\bullet} \ar[ld] \ar[rd] \ar[dd] &&& \overset{\End( \C^{f_{\mu_3}^L} )}{\bullet} \ar[ld] \ar[rd] \ar[dd] &&&& \overset{\End( \C^{f_{\mu_l}^L} )}{\bullet} \ar[ld] \ar[dd] \\
	R_{[\lambda]} / R_{[\lambda]}^2 && \overset{\Hom( \C^{f_{\mu_1}^L}, \C^{f_{\mu_2}^L} )}{\bullet} & \overset{\Hom( \C^{f_{\mu_2}^L}, \C^{f_{\mu_1}^L} )}{\bullet} \ar[rd] && \overset{\Hom( \C^{f_{\mu_2}^L}, \C^{f_{\mu_3}^L} )}{\bullet} \ar[ld] & \overset{\Hom( \C^{f_{\mu_3}^L}, \C^{f_{\mu_2}^L} )}{\bullet} \ar[rd] && \ar[ld] & \cdots & \ar[rd] \\
	R_{[\lambda]}^2 &&&& \overset{\End( \C^{f_{\mu_2}^L} )}{\bullet} &&& \overset{\End( \C^{f_{\mu_3}^L} )}{\bullet} &&&& \overset{\End( \C^{f_{\mu_l}^L} )}{\bullet}
	}} \]
\end{figure}

\begin{Rem}
\label{Rem: structure TLq}
La théorème s'étend pour la valeur $n=1$. Dans ce cas, on a $\Delta_1^{\leq2} = \{ \lyng{1} \}$, $[ \lyng{1} \, ] = \{ \lyng{1} \}$, et on considère l'idempotent central $z_{\lyng{1}} := p_{\lyoung{1}} = 1$ (cf. la proposition \ref{Prop: POIs de TLA}). Alors $z_{\lyng{1}} \TL_1(q) \cong \C  \cong \End( \C )$.
\end{Rem}

\begin{Cor}[{\cite[§ 2]{GW93}}]
\label{Cor: PCIs de TLq}
Les PCIs de $\TL_n(q)$ sont les éléments $z_{[\lambda]}$, $\lambda \in \Delta_n^{\leq 2}$. 
\end{Cor}

Le corollaire \ref{Cor: PCIs de TLq} se traduit par la décomposition en somme directe d'idéaux indécomposables :
\begin{equation}
\label{Eqn: structure de TLq}
\TL_n(q) = \bar{z}_{[\lambda_1]} \TL_n(q) \oplus ... \oplus \bar{z}_{[\lambda_k]} \TL_n(q) \; ; \quad \Delta_n^{\leq2} = [\lambda_1] \sqcup ... \sqcup [\lambda_k].
\end{equation}

\begin{Rem}
\label{Rem: POIs de TLq}
Pour des raisons de cardinalité, les POIs de $\TL_n(q)$ sont donnés par les idempotents évaluables mentionnés dans le théorème \ref{Thm: idempotents eva} et la remarque \ref{Rem: idempotents eva}. Pour davantage de précision, on pourra consulter \cite[Prop. 0.1]{GW93}.
\end{Rem}

\`A titre d'exemple, pour $\TL_6(q)$, on a deux facteurs directs :
\[ \TL_6(q) = \bar{z}_{ \Big[ \lyng{3,3} \Big] } \TL_6(q) \oplus \bar{z}_{\lyng{4,2}} \TL_6(q) . \]
Conformément au théorème \ref{Thm: structure de TLq}, expliquons brièvement comment on obtient la structure de ces idéaux. Pour cela, on utilise les idempotents évaluables détaillés dans la figure \ref{Fig: idempotents eva}. On a :
\begin{align*}
\bar{z}_{ \Big[ \lyng{3,3} \Big] } &= \bar{p}_{\lyoung{135,246}} + \bar{p}_{\Big[ \lyoung{134,256} \Big]} + \bar{p}_{\Big[ \lyoung{125,346} \Big]} + \bar{p}_{\Big[ \Big[ \lyoung{124,356} \Big] \Big]^1} + \bar{p}_{\Big[ \Big[ \lyoung{124,356} \Big] \Big]^2} + \bar{p}_{\Big[ \lyoung{123456} \Big]}, \\
\bar{z}_{\lyng{4,2}} &= \begin{multlined}[t]
	\bar{p}_{\lyoung{1356,24}} + \bar{p}_{\Big[ \lyoung{1346,25} \Big]^1} + \bar{p}_{\Big[ \lyoung{1346,25} \Big]^2} + \bar{p}_{\Big[ \lyoung{1256,34} \Big]^1} + \bar{p}_{\Big[ \lyoung{1256,34} \Big]^2} \\
	+ \bar{p}_{\Big[ \Big[ \lyoung{1246,35} \Big] \Big]^1} + \bar{p}_{\Big[ \Big[ \lyoung{1246,35} \Big] \Big]^2} + \bar{p}_{\Big[ \Big[ \lyoung{1246,35} \Big] \Big]^3} + \bar{p}_{\Big[ \Big[ \lyoung{1246,35} \Big] \Big]^4}.
	\end{multlined}
\end{align*}
Un premier travail consiste à montrer que les 9 idempotents qui décomposent $\bar{z}_{\lyng{4,2}}$ sont deux-à-deux orthogonaux, primitifs, et qu'ils sont reliés entre eux sous l'action de $\TL_n(q)$ -- c'est la situation des POIs associés aux tableaux standards d'une même forme dans les algèbres de Temperley-Lieb génériques. Ainsi, on obtient $\bar{z}_{\lyng{4,2}} \TL_6(q) \cong \End(\C^9)$. \\
Un second travail consiste à identifier les idéaux indécomposables du quotient semi-simple maximal :
\[ \bar{z}_{ \Big[ \lyng{3,3} \Big] } \TL_6(q) / \Rad \left( \bar{z}_{ \Big[ \lyng{3,3} \Big] } \TL_6(q) \right) . \]
Pour cela, on considère :
\begin{align*}
&\bar{p}_{ \Big[ \lyng{3,3} \Big] } := \bar{p}_{\Big[ \lyoung{134,256} \Big]} + \bar{p}_{\Big[ \lyoung{125,346} \Big]} + \bar{p}_{\Big[ \Big[ \lyoung{124,356} \Big] \Big]^1} + \bar{p}_{\Big[ \Big[ \lyoung{124,356} \Big] \Big]^2}, \\ 
&\mathcal{P} := \left\{ \bar{p}_{\lyoung{135,246}}, \bar{p}_{\Big[ \lyng{3,3} \Big]}, \bar{p}_{\Big[ \lyoung{123456} \Big]} \right\}.\end{align*}
Pour tout $\bar{p} \in \mathcal{P}$, on obtient une sous-algèbre $\bar{p} \TL_6(q) \bar{p}$ dont les idempotents associés sont deux-à-deux orthogonaux, primitifs, et reliés sous l'action de $\TL_6(q)$. De plus, si $\bar{p} \neq \bar{p}_{\lyoung{135,246}}$, alors ils donnent lieu à autant de nilpotents, deux-à-deux orthogonaux, primitifs, et reliés sous l'action de $\TL_6(q)$. Leur somme $\bar{n}$ détermine le radical simple $\bar{n} \TL_6(q) \bar{n}$ de $\bar{p} \TL_6(q) \bar{p}$. Ainsi, les structures des sous-algèbres correspondantes sont données par :
\[ \scalebox{0.45}{\xymatrix{
	\bar{p}_{\lyoung{135,246}} \TL_6(q) \bar{p}_{\lyoung{135,246}} &&& \bar{p}_{\Big[ \lyng{3,3} \Big]} \TL_6(q) \bar{p}_{\Big[ \lyng{3,3} \Big]} &&& \bar{p}_{\Big[ \lyoung{123456} \Big]} \TL_6(q) \bar{p}_{\Big[ \lyoung{123456} \Big]} \\
	\overset{\End(\C)}{\bullet} &&& \overset{\End(\C^4)}{\bullet} \ar[d] &&& \overset{\End(\C)}{\bullet} \ar[d] \\
	&&& \overset{\End(\C^4)}{\bullet} &&& \overset{\End(\C)}{\bullet}
	}} \]
Il reste alors décrire le radical $\Rad\left( \bar{z}_{ \Big[ \lyng{3,3} \Big] } \TL_6(q) \right)$, dans lequel interviennent les $\C$-espaces vectoriels de la forme $\bar{p} \TL_6(q) \bar{p}'$ avec $\bar{p}, \bar{p}' \in \mathcal{P}$. Un dernier travail consiste à montrer que les idempotents qui décomposent $\bar{p}$ sont reliés sous l'action de $\TL_6(q)$ à ceux qui décomposent $\bar{p}'$. On obtient alors :
\[ \scalebox{0.45}{\xymatrix{
	\bar{p}_{\lyoung{135,246}} \TL_6(q) \bar{p}_{\lyoung{135,246}} & \bar{p}_{\lyoung{135,246}} \TL_6(q) \bar{p}_{ \Big[ \lyng{3,3} \Big]} & \bar{p}_{ \Big[ \lyng{3,3} \Big]} \TL_6(q)  \bar{p}_{\lyoung{135,246}} & \bar{p}_{\Big[ \lyng{3,3} \Big]} \TL_6(q) \bar{p}_{\Big[ \lyng{3,3} \Big]} & \bar{p}_{\Big[ \lyng{3,3} \Big]} \TL_6(q) \bar{p}_{ \Big[ \lyoung{123456} \Big] } & \bar{p}_{ \Big[ \lyoung{123456} \Big]} \TL_6(q) \bar{p}_{ \Big[ \lyng{3,3} \Big]} & \bar{p}_{\Big[ \lyoung{123456} \Big]} \TL_6(q) \bar{p}_{\Big[ \lyoung{123456} \Big]} \\
	\overset{\End(\C)}{\bullet} \ar[rd] \ar[rrd] &&& \overset{\End(\C^4)}{\bullet} \ar[lld] \ar[ld] \ar[rrd] \ar[rd] \ar[dd] &&& \overset{\End(\C)}{\bullet} \ar[lld] \ar[ld] \ar[dd] \\
	& \overset{\Hom(\C, \C^4)}{\bullet} \ar[rrd] & \overset{\Hom(\C^4, \C)}{\bullet} \ar[rd] && \overset{\Hom(\C^4, \C)}{\bullet} \ar[rrd] \ar[ld] & \overset{\Hom(\C, \C^4)}{\bullet} \ar[rd] \ar[lld] \\
	&&& \overset{\End(\C^4)}{\bullet} &&& \overset{\End(\C)}{\bullet}
	}} \]

\section{Idempotents de Jones-Wenzl évaluables aux racines de l'unité}
\label{section: JWG}

Lorsque $A$ est un paramètre formel, ou après évaluation de $A$ en un nombre complexe qui n'est pas racine $4p$-ième de l'unité ($p \in \N^*$), on peut définir une famille d'idempotents $\left\{ f_n \in \TL_n(A^2) \; ; \; n \in \N^* \right\}$ par le système de récurrence :
\begin{equation} \label{Eqn: recurrence JW}
\begin{cases}
f_1=1, \\
f_n = f_{n-1} + \frac{[n-1]_{A^2}}{[n]_{A^2}} f_{n-1} h_{n-1} f_{n-1}, 
	& n \geq 2,
\end{cases}
\end{equation}
(cf. par exemple \cite[Def. 3.5.1]{CFS95} et \cite{Wen87}). On les appelle \emph{idempotents de Jones-Wenzl}. En revanche, après évaluation de $A$ en une racine $4p$-ième de l'unité, ce système de récurrence n'est plus bien défini pour les entiers $n \geq p$ (un zéro apparaît dans le coefficient $[p]_{A^2}$). On propose une construction générale des idempotents de Jones-Wenzl, basée sur les POIs des algèbres de Temperley-Lieb, qui étend le système de récurrence \ref{Eqn: recurrence JW} aux racines $4p$-ième de l'unité. On fixe à nouveau $n \in \N$ et on illustre cette construction sur $\TL_n(q)$ en évaluant $A^2$ en une racine $2p$-ième de l'unité $q=e^{\frac{i \pi}{p}}$. Pour cela, on commence par expliciter la correspondances entre les POIs des algèbres de Temperley-Lieb évaluées et certaines représentations de $\Uq$. On met alors en évidence des idempotents remarquables, les \emph{idempotents de Jones-Wenzl évaluables}, qui correspondent à des projecteurs sur des multiples des $\Uq$-représentations $\X^+(s), \PIM^+(s), \PIM^-(s)$, $1 \leq s \leq p$ (cf. la section \ref{section: reps}). Enfin, on étudie l'espace d'écheveaux $K_A(\bar{D} \times \mathbb{S}^1, 0)$ du tore solide colorié par ces idempotents. On obtiendra une généralisation de résultats de Lickorish donnés dans \cite[§ 4]{Lic92}.

\subsection{Projections sur les représentations des groupes quantiques}
\label{subsection: POIs rep fonda}

Les algèbres de Temperley-Lieb évaluées jouent un rôle singulier pour les représentations de $\Uq$. Plus précisément, il existe une correspondance (pas tout à fait bijective) entre les POIs de $\TL_n(q)$ (cf. la remarque \ref{Rem: POIs de TLq}) et les facteurs directs de $\X^+(2)^{\otimes n}$. De manière analogue, dans le cas générique, il existe une correspondance bijective entre les POIs de $\TL_n(A^2)$ et les facteurs directs du produit $\X_{A^2}^+(2)^{\otimes n}$ de la représentation fondamentale $\X_{A^2}^+(2)$ du groupe quantique générique $\UA$. On commence par expliciter ces correspondances dans le cas générique, puis on établit celles après évaluation de $A^2$ en $q=e^{\frac{i \pi}{p}}$.

\begin{Def}[{\cite[Def. VI.1.1, Prop. VII.1.1]{Kas95}}]
Le \emph{groupe quantique} $\UA$ \index{UA@ $\UA$} est la $\C(A)$-algèbre engendrée par $E, F, K, K^{-1}$ sous les relations :
\begin{equation} 
\begin{gathered}
K E K^{-1} =A^4 E,
	\qquad K F K^{-1} =A^{-4} F ,
	\qquad [E,F] = \frac{K-K^{-1}}{A^2-A^{-2}}.
\end{gathered}
\end{equation}
Elle est munie d'une structure d'algèbre de Hopf, dont le coproduit $\Delta$, la co-unité $\varepsilon$, et l'antipode $S$ sont donnés par :
\begin{equation} 
\begin{gathered} 
\Delta(E) = 1 \otimes E + E \otimes K,
	\qquad \Delta(F) = K^{-1} \otimes F + F \otimes 1,
	\qquad \Delta(K) = K \otimes K, \\	
\varepsilon(E) = 0, 
	\qquad \varepsilon(F) = 0, 
	\qquad \varepsilon(K) = 1,  \\
S(E) = -EK^{-1},
	\qquad S(F) = -KF,
	\qquad S(K)=K^{-1}.
\end{gathered}
\end{equation}
\end{Def}

A isomorphisme près, il y a une infinité de $\UA$-modules simples à gauche de dimension finie $\X_{A^2}^\alpha(s)$, indexés par $\alpha \in \{+,-\}$ et $s \in \N^*$ (cf. par exemple \cite[§ VI.3]{Kas95}).

\begin{Prop}[{\cite[§ VI.3]{Kas95}}]
\label{Prop: simplesA}
Soient $\alpha \in \{-,+\}$ et $s \in \N^*$. Le module simple $\X_{A^2}^\alpha(s)$ \index{X@ $\X_{A^2}^\alpha(s)$} est engendré sous $\UA$ par un vecteur de plus haut poids $x_0^\alpha(s)$ (resp. de plus bas poids $x_{s-1}^\alpha(s)$), et il admet une $\C(A)$-base $\{ x_n^\alpha(s) \; ; \; 0 \leq n \leq s-1 \}$ telle que :
\begin{align*}
& K x_n^\alpha(s) = \alpha A^{2(s-1-2n)} x_n^\alpha(s), \quad 0 \leq n \leq s-1, \\
& E x_0^\alpha(s) = 0, 
	\qquad E x_n^\alpha(s) = \alpha [n]_{A^2} [s-n]_{A^2} x_{n-1}^\alpha(s), \quad 1 \leq n \leq s-1, \\
& F x_n^\alpha(s) = x_{n+1}^\alpha(s), \quad 0 \leq n \leq s-2,
	\qquad F x_{s-1}^\alpha(s) = 0.
\end{align*}
On l'appelle la \emph{base canonique de $\X_{A^2}^\alpha(s)$}.
\end{Prop}

A isomorphismes près, ces $\UA$-modules simple à gauche (il y en a d'autres de dimension \emph{non finie}) décomposent les $\UA$-modules à gauche de dimension finie (cf. par exemple \cite[Thm VII.2.2]{Kas95}). Aussi, contrairement aux $\Uq$-modules, on ne s'intéresse pas aux autres structures de $\UA$-modules (tels que les modules de Verma, cf. par exemple \cite[§ VI.3]{Kas95}). En particulier, les produits tensoriels $\X_{A^2}^\alpha(s) \otimes \X_{A^2}^{\alpha'}(s')$, $\alpha, \alpha' \in \{+,-\}$ et $s,s' \in \N^*$, dont la structure de $\UA$-module à gauche est donné par le coproduit $\Delta$ de $\UA$, sont des $\UA$-modules à gauche de dimension finie. Ils admettent les décompositions suivantes.

\begin{Lemme} 
\label{Lemme: produits simplesA}
Soient $\alpha \in \{+,-\}$ et $s \in \N^*$. On a :
\[ \X_{A^2}^\alpha(s) \otimes \X_{A^2}^+(2) \cong 
	\begin{cases}
		\X_{A^2}^\alpha(2) & \text{ si } s=1, \\
		\X_{A^2}^\alpha(s-1) \oplus \X_{A^2}^\alpha(s+1) & \text{ sinon}.
	\end{cases} \]
\end{Lemme}

\begin{proof}
On procède comme dans la preuve du lemme \ref{Lemme: produits simples1}.
\end{proof}

\begin{Thm} 
\label{Thm: produits simplesA}
Soient $\alpha, \alpha' \in \{+,-\}$ et $s,s' \in \N^*$. On a :
\[ \X_{A^2}^\alpha(s) \otimes \X_{A^2}^{\alpha'}(s') \cong 
	\bigoplus_{\substack{s''=s-s'+1 \\ \mathrm{pas}=2}}^{s+s'-1} \X_{A^2}^{\alpha \alpha'}(s''). \]
\end{Thm}

\begin{proof}
On procède comme dans la preuve du théorème \ref{Thm: produits simples}.
\end{proof}

On s'intéresse aux facteurs directs de la suite $\{ \X_{A^2}^+(2)^{\otimes n} \; ; \;  n \in \N \}$ de $\UA$-modules à gauche. D'après le lemme \ref{Lemme: produits simplesA}, les facteurs directs de cette suite font intervenir les classes d'isomorphisme des $\UA$-modules à gauche $\X_{A^2}^+(s)$, $s \in \N^*$. Précisons à quels rangs ils interviennent pour la première fois. Pour tout $\UA$-module à gauche $M$, on note $n(M)$ le plus petit entier $n$ pour lequel $M$ est isomorphe à un facteur de $\X_{A^2}^+(2)^{\otimes n}$, et $n(M)=\infty$ si un tel entier n'existe pas. Alors, pour tout $s \in \N^*$, on a $n(\X_{A^2}^+(s))= s-1$ et $n(\X_{A^2}^-(s)) = \infty$. Dans la suite, on établit une correspondance bijective entre les POIs de $\TL_n(A^2)$ et les facteurs directs de $\X_{A^2}^+(2)^{\otimes n}$.

\begin{Def}[{\cite{CFS95}[Def. 3.2.2]}]
On note $\{ x_0:= x^+_0(2), x_1:= x^+_1(2)\}$ la base canonique de $\X_{A^2}^+(2)$. On définit les morphismes linéaires $\cap_A$ et $\cup_A$ par :
\begin{align*}
& \cap_A : \left\{ \begin{array}{l}
	\X^+(2) \otimes \X^+(2) \longrightarrow \C(A) \\
	x_m \otimes x_n \longmapsto (n-m) i A^{n-m}
	\end{array} \right. ,
	\quad \forall m,n \in \{0,1\}, \\ 
& \cup_A : \left\{ \begin{array}{l}
	\C(A) \longrightarrow \X^+(2) \otimes \X^+(2) \\
	1 \longmapsto i A x_0 \otimes x_1 - i A^{-1} x_1 \otimes x_0
	\end{array} \right. .
\end{align*}
\end{Def}

Connaissant la structure de $\UA$-module de $\X_{A^2}^+(2)$ et $\C(A) \cong \X_{A^2}^+(1)$ (cf. la proposition \ref{Prop: simplesA}), on montre que les morphismes $\cap_A$ et $\cup_A$ sont $\UA$-équivariants et vérifient :
\[ \cap \circ \cup(1) = -[2]_{A^2},
\qquad \begin{cases}
\cup \circ \cap (x_0 \otimes x_0) = 0 = \cap \circ \cup (x_1 \otimes x_1), \\
\cup \circ \cap (x_0 \otimes x_1) = -A^2 x_0 \otimes x_1 + x_1 \otimes x_0, \\
\cup \circ \cap (x_1 \otimes x_0) = x_0 \otimes x_1 -A^{-2} x_1 \otimes x_0.
\end{cases} \]
Pour tout $n \in \N$, ils permettent de définir une action connue de $\TL_n(A^2)$ sur $\X_{A^2}^+(2)^{\otimes n}$ qui commute avec celle de $\UA$.

\begin{Thm}[{\cite[Thm 3.3.4]{CFS95}}]
\label{Thm: action de TLA}
On a un morphisme d'algèbres injectif défini par : \index{theta1 @$\theta_{A^2,n}$} 
\[ \theta_{A^2,n} : \begin{cases}
	\TL_n(A^2) \longrightarrow \End_{\UA} \left( \X_{A^2}^+(2)^{\otimes n} \right) \\
	h_0=1 \longmapsto id_n, \\
	h_i \longmapsto id_{i-1} \otimes \cup_A \circ \cap_A \otimes id_{n-i-1}, & i \in \{1,...,n-1\} . 
	\end{cases} \]
\end{Thm}

Ainsi, pour tout POI $p$ de $\TL_n(A^2)$, $\theta_{A^2,n}(p)$ est un projecteur non nul $\UA$-équivariant de $\X_{A^2}^+(2)^{\otimes n}$ et son image $p \cdot  \X_{A^2}^+(2)^{\otimes n}$ est un $\UA$-module dont la classe d'isomorphisme apparaît dans la décomposition de $\X_{A^2}^+(2)^{\otimes n}$. Pour $n=0$, l'algèbre de Temperley-Lieb $\TL_0(A^2)=\C(A)$ possède un unique POI $p=1$ et $\theta_{A^2,0}(1)$ est un projecteur sur $\C(A) \cong \X_{A^2}^+(1)$. Pour les autres valeurs de $n$, les projecteurs $\theta_{A^2n,}(p_t)$, $t \in T_n^{\leq2}$, sont détaillés dans le :

\begin{Thm}
\label{Thm: POIs rep fondaA}
On suppose que $n \geq 1$. Pour tout tableau standard $t \in T_n^{\leq 2}$ de forme $\lambda$, il existe un unique facteur direct $\X$ de $\X^+_{A^2}(2)^{\otimes n}$, isomorphe à $\X_{A^2}^+(\omega(\lambda))$, tel que $\theta_{A^2,n}(p_t)$ est un projecteur sur $\X$.
\end{Thm}

\begin{proof}
Il est clair que les morphismes $\theta_{A^2,n}(p_t)$, $t \in T_n^{\leq2}$, sont des projecteurs non nuls dont les images $p_t \cdot \X_{A^2}^+(2)^{\otimes n}$, $t \in T_n^{\leq2}$, sont des $\UA$-modules en somme directe. Comme $\theta_{A^2,n}$ est injectif, cela assure l'unicité de l'assertion. Il suffit donc de montrer que, pour tout $t \in T_n^{\leq2}$ de forme $\lambda$, $p_t \cdot \X_{A^2}^+(2)^{\otimes n} \cong \X_{A^2}^+(\omega(\lambda))$. On procède par récurrence double sur $n \geq 1$ et sur :
\begin{align*}
\omega(\lambda) &\in \{2,4,...,n+1\} && \text{ si $n$ est impair}, \\
&\in \{1,3,...,n+1\} && \text{ si $n$ est pair}.
\end{align*}

Pour $n=1$, on a $\Delta_1^{\leq2} = \{ \lyng{1} \}$ et $T_1^{\leq 2} = \{ \lyoung{1} \}$. Il y a un unique POI $p_{\lyoung{1}} = 1$ et $p_{\lyoung{1}} \cdot \X_{A^2}^+(2) = \X^+(2)$.

Soit $n \geq 1$. On suppose que, pour tout $r \in T_n^{\leq2}$ de forme $\nu$, $p_r \cdot \X_{A^2}^+(2)^{\otimes n} \cong \X_{A^2}^+(\omega(\nu))$. Soit $t \in T_{n+1}^{\leq2}$ de forme $\lambda=[\lambda_1, \lambda_2]$. On considère le tableau standard $t' \in T_n^{\leq 2}$ et, s'il existe, le tableau standard $\hat{t} \in T_{n+1}^{\leq 2}$ (cf. la remarque \ref{Rem: treillis de TL1}). On distingue deux cas.

\begin{Cas}
On suppose que $n$ est impair. On procède par récurrence sur :
\[ d:=\omega(\lambda) \in \{1,3,...,n+2\} . \]

Pour $d=1$, il existe un unique diagramme de Young $\nu = [\lambda_1, \lambda_2-1]$ tel que $\nu \subset \lambda$. Donc $t'$ est de forme $\nu$ et $\omega(\nu)=d+1=2$. D'après l'hypothèse de récurrence, on a $p_{t'} \cdot \X_{A^2}^+(2)^{\otimes n} \cong \X_{A^2}^+(2)$. On en déduit que :
\begin{align*}
p_{t'} \cdot \X_{A^2}^+(2)^{\otimes n+1} &\; \; = \left( p_{t'} \cdot \X_{A^2}^+(2)^{\otimes n} \right) \otimes \X_{A^2}^+(2)
\cong \X_{A^2}^+(2) \otimes \X_{A^2}^+(2) \\
&\overset{\eqref{Thm: produits simplesA}}{\cong} \X_{A^2}^+(1) \oplus \X_{A^2}^+(3).
\end{align*}
De plus, le tableau standard $\hat{t}$ existe et sa forme $\mu = [\lambda_1+1, \lambda_2-1]$ vérifie $\omega(\mu)=d+2=3$. En utilisant l'égalité $p_{t'} = p_t + p_{\hat{t}}$ dans $\TL_{n+1}(A^2)$ (cf. la preuve de la proposition \ref{Prop: POIs de TLA}) et leur orthogonalité, on obtient également :
\[ p_{t'} \cdot \X_{A^2}^+(2)^{\otimes n+1} = (p_t + p_{\hat{t}})  \cdot \X_{A^2}^+(2)^{\otimes n+1} = \left( p_t \cdot \X_{A^2}^+(2)^{\otimes n+1} \right) \oplus \left( p_{\hat{t}} \cdot \X_{A^2}^+(2)^{\otimes n+1} \right). \]
Supposons par l'absurde que $p_t \cdot \X_{A^2}^+(2)^{\otimes n+1} \cong \X_{A^2}^+(3)$. Alors, par identification des facteurs directs, on a $p_{\hat{t}} \cdot \X_{A^2}^+(2)^{\otimes n+1} \cong \X_{A^2}^+(1)$. Or, comme $\omega(\mu)>1$, il existe deux tableaux standards $s, \hat{s} \in T_{n+2}^{\leq2}$ tels que $s'=\hat{t}=\hat{s}'$. En reproduisant les raisonnements précédents sur $\hat{t}$, on obtient d'une part :
\begin{align*}
p_{\hat{t}} \cdot \X_{A^2}^+(2)^{\otimes n+2} = \left( p_{\hat{t}} \cdot \X_{A^2}^+(2)^{\otimes n+1} \right) \otimes \X_{A^2}^+(2) 
\cong \X_{A^2}^+(1) \otimes \X_{A^2}^+(2)
\overset{\eqref{Thm: produits simplesA}}{\cong} \X_{A^2}^+(2),
\end{align*}
et d'autre part :
\begin{align*}
p_{\hat{t}} \cdot \X_{A^2}^+(2)^{\otimes n+2} &= (p_s + p_{\hat{s}}) \cdot \X_{A^2}^+(2)^{\otimes n+2}
= \left( p_s \cdot \X_{A^2}^+(2)^{\otimes n+2} \right) \oplus \left( p_{\hat{s}} \cdot \X_{A^2}^+(2)^{\otimes n+2} \right).
\end{align*}
Ce qui est absurde car $\X_{A^2}^+(2)$ est simple. Donc $p_t \cdot \X_{A^2}^+(2)^{\otimes n+1} \cong \X_{A^2}^+(1)$.

Soit $d \in \{3,...,n+2\}$. On suppose que, pour $s \in T_{n+1}^{\leq2}$ de forme $\mu$ telle que $\omega(\mu) < d$, $p_s \cdot \X_{A^2}^+(2)^{\otimes n+1} \cong \X_{A^2}^+(\omega(\mu))$. Comme $d>1$, il existe deux diagrammes de Young $\nu \in \{ [\lambda_1, \lambda_2-1], [\lambda_1-1, \lambda_2] \}$ tels que $\nu \subset \lambda$. On a donc deux formes possibles pour $t'$.
\begin{enumerate}[(a)]
	\item On suppose que $t'$ est de forme $\nu = [\lambda_1-1, \lambda_2]$. Alors $\omega(\nu)=d-1$, le tableau standard $\hat{t}$ existe et sa forme $\mu = [\lambda_1-1, \lambda_2+1]$ vérifie $\omega(\mu)=d-2$. D'après la première hypothèse de récurrence, on a $p_{t'} \cdot \X_{A^2}^+(2)^{\otimes n} \cong \X_{A^2}^+(d-1)$. Comme précédemment, on obtient d'une part :
	\begin{align*}
	p_{t'} \cdot \X_{A^2}^+(2)^{\otimes n+1} &\; \; =	\left( p_{t'} \cdot \X_{A^2}^+(2)^{\otimes n} \right) \otimes \X_{A^2}^+(2)
	\cong \X_{A^2}^+(d-1) \otimes \X_{A^2}^+(d) \\
	&\overset{\eqref{Thm: produits simplesA}}{\cong} \X_{A^2}^+(d-2) \oplus \X_{A^2}^+(d),
	\end{align*}
	et d'autre part :
	\begin{align*}
	p_{t'} \cdot \X_{A^2}^+(2)^{\otimes n+1} &= (p_t + p_{\hat{t}}) \cdot \X_{A^2}^+(2)^{\otimes n+1}
	= \left( p_t \cdot \X_{A^2}^+(2)^{\otimes n+1} \right) \oplus \left( p_{\hat{t}} \cdot \X_{A^2}^+(2)^{\otimes n+1} \right).
	\end{align*}
	Or, d'après la seconde hypothèse de récurrence, on a aussi $p_{\hat{t}} \cdot \X_{A^2}^+(2)^{\otimes n+1} \cong \X_{A^2}^+(d-2)$. Donc, par identification des facteurs directs, $p_t \cdot \X_{A^2}^+(2)^{\otimes n+1} \cong \X_{A^2}^+(d)$.
	\item On suppose que $t'$ est de forme $\nu = [\lambda_1, \lambda_2-1]$. Alors $\omega(\nu)=d+1$, le tableau standard $\hat{t}$ existe et sa forme $\mu = [\lambda_1+1, \lambda_2-1]$ vérifie $\omega(\mu)=d+2$. Soit $s \in T_{n+1}^{\leq2}$ de forme $\lambda$ tel que $s'$ est de forme $[\lambda_1-1, \lambda_2]$ (il existe car $d>1$). D'après le théorème \ref{Thm: structure TLA}, on a $\TL_{n+1}(A^2) p_t \cong V_\lambda \cong \TL_{n+1}(A^2) p_s$. Il existe donc $u_{s,t}, u_{t,s} \in \TL_{n+1}(A^2)$ tels que $p_t = u_{t,s} p_s$ et $p_s = u_{s,t} p_t$. On en déduit les isomorphismes de $\UA$-modules réciproques :
	\begin{align*}
	&\begin{cases} 
			p_t \cdot \X_{A^2}^+(2)^{\otimes n+1} \longrightarrow p_s \cdot \X_{A^2}^+(2)^{\otimes n+1} \\
			p_t \cdot X \longmapsto u_{s,t} p_t \cdot X
			\end{cases}, \;
	\begin{cases} 
			p_s \cdot \X_{A^2}^+(2)^{\otimes n+1} \longrightarrow p_t \cdot \X_{A^2}^+(2)^{\otimes n+1} \\
			p_s \cdot X \longmapsto u_{t,s} p_s \cdot X
		\end{cases}.
	\end{align*}
		Or, $p_s \cdot \X_{A^2}^+(2)^{\otimes n+1} \cong \X_{A^2}^+(d)$ d'après $(a)$. Donc $p_t \cdot \X_{A^2}^+(2)^{\otimes n+1} \cong \X_{A^2}^+(d)$.
\end{enumerate}
Dans chacun des cas, pour $n$ impair, on a $p_t \cdot \X_{A^2}^+(2)^{\otimes n+1} \cong \X_{A^2}^+(d)$. Ce qui prouve l'hérédité et achève la seconde récurrence.
\end{Cas}

\begin{Cas}
On suppose que $n$ est pair. On procède par récurrence sur :
\[ d:= \omega(\lambda) \in \{2,4,...,n+2\} . \] 

Pour $d=2$, il existe deux diagrammes de Young $\nu \in [\lambda_1, \lambda_2-1], [\lambda_1-1, \lambda_2]$ tels que $\nu \subset \lambda$. Pour les mêmes raisons que dans le cas $1(b)$, on peut suppose que $t'$ est de forme $[\lambda_1-1, \lambda_2]$. Alors $\omega(\nu)=d-1=1$ et le tableau standard $\hat{t}$ n'existe pas. D'après l'hypothèse de récurrence, on a $p_{t'} \cdot \X_{A^2}^+(2)^{\otimes n} \cong \X_{A^2}^1(1)$. On en déduit que :
\[ p_{t'} \cdot \X_{A^2}^+(2)^{\otimes n+1} = \left( p_{t'} \cdot \X_{A^2}^+(2)^{\otimes n} \right) \otimes \X_{A^2}^+(2) \cong \X_{A^2}^+(1) \otimes \X_{A^2}^+(2) \overset{\eqref{Thm: produits simplesA}}{\cong} \X_{A^2}^+(2). \]
De plus, en utilisant l'égalité $p_{t'} = p_t$ dans $\TL_{n+1}(A^2)$ (cf. la preuve de la proposition \ref{Prop: POIs de TLA}), on obtient également :
\[ p_{t'} \cdot \X_{A^2}^+(2)^{\otimes n+1} = p_t \cdot \X_{A^2}^+(2)^{\otimes n+1}. \]
Donc $p_t \cdot \X_{A^2}^+(2)^{\otimes n+1} \cong \X_{A^2}^+(2)$.

Pour $d \in \{4,...,n+2\}$, on procède exactement comme dans le cas 1 avec une seconde récurrence et une distinction de cas.
\end{Cas}

Ainsi, quelque soit la parité de $n$, on a $p_t \cdot \X_{A^2}^+(2)^{\otimes n+1} \cong \X_{A^2}^+(d)$. Ce qui prouve l'hérédité et achève la première récurrence.
\end{proof}

\begin{Cor}
\label{Cor: POIs rep fondaA}
On a une bijection ensembliste :
\[ \begin{cases}
	\left\{ p_t \; ; \; t \in T_n^{\leq 2} \right\} \longrightarrow DF \left( \X_{A^2}^+(2)^{\otimes n} \right) \\
	p_t \longrightarrow p_t \cdot \X_{A^2}^+(2)^{\otimes n}
	\end{cases} \]
où $DF(\X^+(2)^{\otimes n})$ désigne l'ensemble des classes d'isomorphismes des facteurs directs de $\X^+(2)^{\otimes n}$ comptées avec multiplicité.
\end{Cor}

\begin{proof}
Le cas $n=0$ est évident : l'algèbre $\TL_0(A^2) = \C(A)$ possède un unique POI $p=1$ et $\X_{A^2}^+(1)^{\otimes 0} = \C(A) \cong \X_{A^2}^+(1)$. 

Soit $n \in \N^*$. Le théorème \ref{Thm: POIs rep fondaA} montre que cette application est bien définie et injective. Il suffit de montrer que les ensembles $\left\{ p_t \; ; \; t \in T_n^{\leq 2} \right\}$ et $DF(\X_{A^2}^+(2)^{\otimes n})$ ont le même cardinal. D'une part, on sait que :
\begin{equation} \tag{1}
\card \left\{ p_t \; ; \; t \in T_n^{\leq 2} \right\} = \sum_{\lambda \in \Delta_n^{\leq2}} f_\lambda
\end{equation}
où, pour tout $\lambda \in \Delta_n^{\leq2}$, $f_\lambda$ désigne le nombre de tableau standard de forme $\lambda$. D'autre part, pour tout $d \in \N^*$, on note $f_d^n \in \N$ la multiplicité de la classe d'isomorphisme de $\X_{A^2}^+(d)$ dans $DF(\X_{A^2}^+(2)^{\otimes n})$. Alors, d'après la décomposition des produits de $\UA$-modules simples donnée le lemme \ref{Lemme: produits simplesA}, on a :
\begin{equation} \tag{2}
\card DF(\X_{A^2}^+(2)^{\otimes n}) = \sum_{\substack{d=0 \\ \mathrm{pas}=2}}^n f_{n+1-d}^n .
\end{equation}
Les indices des sommes $(1)$ et $(2)$ s'identifient via la bijection ensembliste :
\[ \begin{cases}
	\Delta_n^{\leq2} \longrightarrow \left\{ n+1-d \; ; \; d \in \{0,...,n\} \text{ pair.} \right\}  \\
	\lambda \longmapsto \omega(\lambda).
	\end{cases} \]
Il reste à montrer que, pour tout $\lambda \in \Delta_n^{\leq2}$, on a $f_\lambda = f_{\omega(\lambda)}^n$. On procède par récurrence sur $n \geq 1$.

Pour $n=1$, on a $\Delta_1^{\leq2} = \{ \lyng{1} \}$ et $T_{\lyng{1}} = \{ \lyoung{1} \}$. Donc $f_{\lyng{1}} = 1 = f_2^1$.

Soit $n \geq 1$. On suppose que, pour tout $\mu \in \Delta_n^{\leq2}$, $f_\mu = f_{\omega(\mu)}^n$. Soit $\lambda = [\lambda_1, \lambda_2] \in \Delta_{n+1}^{\leq2}$. On distingue deux cas.

\begin{Cas}
On suppose que $\omega(\lambda)=1$. Alors il existe un unique diagramme de Young $\mu = [\lambda_1, \lambda_2-1] \in \Delta_n^{\leq2}$ tel que $\mu \subset \lambda$. Il s'ensuit que $f_\lambda = f_\mu$. Or, d'après l'hypothèse de récurrence, $f_\mu = f_{\omega(\lambda)+1}^n = f_2^n$. D'autre part, d'après le lemme \ref{Lemme: produits simplesA}, on a $f_1^{n+1} = f_2^n$. D'où $f_\lambda = f_1^{n+1}$. 
\end{Cas}

\begin{Cas}
On suppose que $\omega(\lambda) > 1$. Alors il existe exactement deux diagrammes de Young $\mu = [\lambda_1, \lambda_2-1] \in \Delta_n^{\leq2}$ et $\nu = [\lambda_1-1, \lambda_2] \in \Delta_n^{\leq2}$ tels que $\mu, \nu \subset \lambda$. Il s'ensuit que $f_\lambda = f_\mu + f_\nu$. Or, d'après l'hypothèse de récurrence, $f_\mu = f_{\omega(\lambda)+1}^n$ et $f_\mu = f_{\omega(\lambda)-1}^n$. D'autre part, d'après le lemme \ref{Lemme: produits simplesA}, on a $f_{\omega(\lambda)}^{n+1} = f_{\omega(\lambda)-1}^n + f_{\omega(\lambda)+1}^n$. D'où $f_\lambda = f_{\omega(\lambda),n+1}$. 
\end{Cas}

Dans chacun des cas, on a $f_\lambda = f_{\omega(\lambda)}^{n+1}$. Ce qui prouve l'hérédité et achève la récurrence. D'où le résultat.
\end{proof}

On reproduit la même raisonnement après évaluation de $A^2$ en $q = e^{\frac{i \pi}{p}}$. On s'intéresse alors aux facteurs directs de la suite de $\Uq$-modules $\{ \X^+(2)^{\otimes n} \; ; \;  n \in \N \}$. D'après le lemme \ref{Lemme: produits simples1} et le théorème \ref{Thm: produits simples-PIMs}, les facteurs directs de cette suite font intervenir les classes d'isomorphisme de $\X^+(s)$, $\PIM^\pm(s)$, $1 \leq s \leq p$. Précisons à quels rangs ils interviennent pour la première fois. Pour tout $\Uq$-module à gauche $M$, on note $n(M)$ le plus petit entier $n$ pour lequel $M$ est isomorphe à un facteur de $\X^+(2)^{\otimes n}$, et $n(M)=\infty$ si un tel entier n'existe pas. Alors, pour tout $s \in \{1,...,p\}$, on a $n(\X^+(s))= s-1$, $n(\X^-(s)) = \infty$, $n(\PIM^+(s))= 2p-s-1$ et $n(\PIM^-(s)) = n(2 \PIM^-(s)) = 3p-s-1 $. On remarque que, pour tout $s \in \{1,...,p\}$, la classe d'isomorphisme du module $\PIM^-(s)$ apparaît avec une multiplicité \emph{double}. Il reste à établir la correspondance entre les POIs de $\TL_n(q)$ et les facteurs directs de $\X^+(2)^{\otimes n}$ à ce phénomène de dédoublement près.

\begin{Def}[{\cite{CFS95}[Def. 3.2.2]}]
On note $\{ x_0:= x^+_0(2), x_1:= x^+_1(2)\}$ la base canonique de $\X^+(2)$. On définit les morphismes linéaires $\cap$ et $\cup$ par :
\begin{align*}
& \cap : \left\{ \begin{array}{l}
	\X^+(2) \otimes \X^+(2) \longrightarrow \C(A) \\
	x_m \otimes x_n \longmapsto (n-m) i q^{\frac{n-m}{2}}
	\end{array} \right. ,
	\quad \forall m,n \in \{0,1\}, \\ 
& \cup : \left\{ \begin{array}{l}
	\C(A) \longrightarrow \X^+(2) \otimes \X^+(2) \\
	1 \longmapsto i q^{\frac{1}{2}} x_0 \otimes x_1 - i q^{-\frac{1}{2}} x_1 \otimes x_0
	\end{array} \right. .
\end{align*}
\end{Def}

De même que précédemment, les morphismes $\cap$ et $\cup$ sont $\Uq$-équivariants et vérifient :
\[ \cap \circ \cup(1) = -[2],
\qquad \begin{cases}
\cup \circ \cap (x_0 \otimes x_0) = 0 = \cap \circ \cup (x_1 \otimes x_1), \\
\cup \circ \cap (x_0 \otimes x_1) = -q x_0 \otimes x_1 + x_1 \otimes x_0, \\
\cup \circ \cap (x_1 \otimes x_0) = x_0 \otimes x_1 -q^{-1} x_1 \otimes x_0.
\end{cases} \]
Le travail de \cite{GW93} (l'identification des POIs de $\TL_n(q)$ et un résultat de conservation des dimensions après évaluation) permet alors de généraliser les théorèmes \ref{Thm: action de TLA} et \ref{Thm: POIs rep fondaA} lorsque $A^2$ s'évalue en $q = e^{\frac{i \pi}{p}}$.

\begin{Thm}[{\cite[Thm 2.4]{GW93}}]
\label{Thm: action de TLq}
On a un morphisme d'algèbres injectif défini par : \index{theta2 @$\theta_n$}
\[ \theta_n : \begin{cases}
	\TL_n(q) \longrightarrow \End_{\Uq} \left( \X^+(2)^{\otimes n} \right) \\
	h_0=1 \longmapsto id_n, \\
	h_i \longmapsto id_{i-1} \otimes \cup \circ \cap \otimes id_{n-i-1}, & i \in \{1,...,n-1\} . 
	\end{cases} \]
\end{Thm}

\begin{Thm}
\label{Thm: POIs rep fonda}
On suppose que $n \geq 1$. Soit $t \in T_n^{\leq2}$ un tableau standard régulier de forme $\lambda$. On note $N \in \N$ le nombre d'intersection(s) entre le graphe $\gamma(t)$ et l'ensemble des lignes critiques privé de la première.
\begin{enumerate}[(a)]
	\item Si $t$ ne possède pas de sous-tableau critique propre, alors il existe un unique facteur direct $\X$ de $\X^+(2)^{\otimes n}$, isomorphe à $\X^+(\omega(\lambda))$, tel que $\theta_n(\bar{p}_t)$ est un projecteur sur $\X$.
	\item Si $t$ est critique, alors il existe un unique facteur $\X$ de $\X^+(2)^{\otimes n}$, isomorphe à \\ $2^N \PIM^{(-)^{\frac{\omega(\lambda)}{p}-1}}(p) = 2^N \X^{(-)^{\frac{\omega(\lambda)}{p}-1}}(p)$, tel que $\theta_n(\bar{p}_t)$ est un projecteur sur $\X$.
	\item Si $t$ n'est pas critique et possède un sous-tableau critique maximal $r$ de forme $\mu$, alors il existe un unique facteur $\X$ de $\X^+(2)^{\otimes n}$, isomorphe à \\ $2^N \PIM^{(-)^{\frac{\omega(\mu)}{p}-1}} \left( p-|\omega(\lambda)-\omega(\mu)| \right)$, tel que $\theta_n(\bar{p}_{[t]})$ est un projecteur sur $\X$. \\ 
	De plus, le morphisme $\theta_n(\bar{n}_{[t]})$ envoie le sous-module de $\X$ isomorphe à \\ $2^N \X^{(-)^{\frac{\omega(\mu)}{p}-1}} \left( p-|\omega(\lambda)-\omega(\mu)| \right)$ sur le module quotient isomorphe à \\ $2^N \X^{(-)^{\frac{\omega(\mu)}{p}-1}} \left( p-|\omega(\lambda)-\omega(\mu)| \right)$.
\end{enumerate}
\end{Thm}

\begin{proof}
On procède comme dans la preuve du théorème \ref{Thm: POIs rep fondaA} (tout sous-tableau d'un tableau standard régulier est régulier) avec les décompositions en somme directe données dans le lemme \ref{Lemme: produits simples1} et le théorème \ref{Thm: produits simples-PIMs}. Soit $\lambda = [\lambda_1, \lambda_2]$ tel que $\omega(\lambda)>1$. On peut encore supposer que $t'$ est de forme $[\lambda_1-1, \lambda_2]$ pour les raisons suivantes.
\begin{enumerate}[(a)]
	\item Si $t$ ne possède pas de sous-tableau critique propre, alors il existe $s \in T_\lambda$ sans sous-tableau critique propre tel que $s'$ est de forme $[\lambda_1-1, \lambda_2]$. D'après le théorème \ref{Thm: structure de TLq}, on a $\TL_{n+1}(q) \bar{p}_t \cong \C^{f_\lambda} \cong \TL_{n+1}(q) \bar{p}_s$. Il existe donc $\bar{u}_{s,t}, \bar{u}_{t,s} \in \TL_{n+1}(q)$ tels que $\bar{p}_t = \bar{u}_{t,s} \bar{p}_s$ et $\bar{p}_s = \bar{u}_{s,t} \bar{p}_t$. On en déduit les isomorphismes de $\Uq$-modules réciproques :
	\begin{align*}
	&\begin{cases} 
			\bar{p}_t \cdot \X^+(2)^{\otimes n+1} \longrightarrow \bar{p}_s \cdot \X^+(2)^{\otimes n+1} \\
			\bar{p}_t \cdot X \longmapsto \bar{u}_{s,t} \bar{p}_t \cdot X
			\end{cases}, \;
	\begin{cases} 
			\bar{p}_s \cdot \X^+(2)^{\otimes n+1} \longrightarrow \bar{p}_t \cdot \X^+(2)^{\otimes n+1} \\
			\bar{p}_s \cdot X \longmapsto \bar{u}_{t,s} \bar{p}_s \cdot X
		\end{cases}.
	\end{align*}
	\item Si $t$ est critique, alors il existe $s \in T_\lambda$ critique et régulier tel que $s'$ est de forme $[\lambda_1-1, \lambda_2]$. D'après le théorème \ref{Thm: structure de TLq}, on a encore $\TL_{n+1}(q) \bar{p}_t \cong \C^{f_\lambda} \cong \TL_{n+1}(q) \bar{p}_s$. Pour les mêmes raisons que précédemment, on en déduit que $\bar{p}_t \cdot \X^+(2)^{\otimes n+1} \cong \bar{p}_s \cdot \X^+(2)^{\otimes n+1}$.
	\item Si $t$ n'est pas critique et possède un sous-tableau maximal, alors $t$ ou $\bar{t}$ vérifie la configuration adéquate.
\end{enumerate}

On se place dans ce dernier cas. Il reste à discuter du morphisme $\theta_n(\bar{n}_{[t]})$. Pour cela, on considère le morphisme induit par $\theta_n$ sur la sous-algèbre $p_{[t]} \TL_n(q) p_{[t]}$ :
\[ \bar{p}_{[t]} \TL_n(q) \bar{p}_{[t]} \longrightarrow \End_{\Uq} \left( 2^N \PIM^{(-)^{\frac{\omega(\mu)}{p}-1}} \left( p-|\omega(\lambda)-\omega(\mu)| \right) \right). \]
D'après le théorème \ref{Thm: structure de TLq} et la proposition \ref{Prop: PIMs}, on connaît les structures respectives de $p_{[t]} \TL_n(q) p_{[t]}$ et de $2^N \PIM^{(-)^{\frac{\omega(\mu)}{p}-1}} \left( p-|\omega(\lambda)-\omega(\mu)| \right)$. Par passages successifs sur les radicaux, on obtient un morphisme d'algèbres :
\[ \bar{n}_{[t]} \TL_n(q) \bar{n}_{[t]} \longrightarrow \End_{\Uq} \left( 2^N \X^{(-)^{\frac{\omega(\mu)}{p}-1}} \left( p-|\omega(\lambda)-\omega(\mu)| \right) \right). \]
D'où le résultat.
\end{proof}

\begin{Rem}
\label{Rem: POIs rep fonda}
Le théorème \ref{Thm: POIs rep fonda} se généralise pour les POIs associés aux tableaux standards \emph{non réguliers}, dont les expressions sont bien plus complexes (cf. les remarques \ref{Rem: idempotents eva}). Avec ceux-ci, on pourrait donner un énoncé analogue au corollaire \ref{Cor: POIs rep fondaA} en considérant l'ensemble des classes d'isomorphisme de facteurs \emph{spécifiques} de $\X^+(2)^{\otimes n}$ comptées avec multiplicité (qui tiennent compte du phénomène de dédoublement).
\end{Rem}

\subsection{Construction générale des idempotents de Jones-Wenzl}
\label{subsection: JWG}

Lorsque $A$ est un paramètre formel, on peut définir les idempotents de Jones-Wenzl usuels $\{ f_n \in \TL_n(A^2) \; ; \; n \in \N^* \}$ en fonction des POIs des algèbres $\{ \TL_n(A^2) \; ; \; n \in \N^* \}$. Pour cela, pour tout $n \in \N^*$, on considère le POI $p_{t(n)}$ de $\TL_n(A^2)$ associé au tableau standard régulier :
\begin{equation}
\label{Eqn: t(n)}
t(n) := {\scriptstyle \begin{array}{|c|c|c|c|} 
		\hline 1 & 2 & \cdots & n \\
		\hline \end{array}} .
\end{equation} \index{t7 @$t(n)$}
Alors, d'après la proposition \ref{Prop: POIs de TLA} et la remarque \ref{Rem: treillis de TL2}, on a :
\[ \begin{cases}
	p_{t(1)} = 1, \\
	p_{t(n)} = p_{t(n-1)} + \frac{[n-1]_{A^2}}{[n]_{A^2}} p_{t(n-1)} h_{n-1} p_{t(n-1)},
	& n \geq 2.
	\end{cases} \]
On retrouve le système de récurrence \eqref{Eqn: recurrence JW} découvert par Wenzl (cf. \cite{Wen87}). Ainsi, pour tout $n \in \N^*$, on a $f_n = p_{t(n)}$. Conformément à la sous-section \ref{subsection: POIs rep fonda}, ce point de vue permet d'interpréter les idempotents de Jones-Wenzl comme les projecteurs sur les facteurs directs de la suite $\{ \X_{A^2}^+(2)^{\otimes n} ; \; ; n \in \N^* \}$ qui apparaissent pour la première fois. Plus précisément, pour tout $n \in \N^*$, il existe un unique facteur direct $\X$ isomorphe à $\X_{A^2}^+(n+1)$ tel que $\theta_{A^2,n}(f_n)$ est un projecteur sur $\X$ (cf. le théorème \ref{Thm: POIs rep fondaA}). On construit les idempotents de Jones-Wenzl \emph{évaluables} de manière analogue avec des POIs des algèbres $\{ \TL_n(q) \; ; \; n \in \N^* \}$. On donne ensuite quelques propriétés remarquables qui généralisent celles des idempotents de Jones-Wenzl usuels (cf. par exemple \cite[§ 3.5]{CFS95}).

\begin{Def}
Pour tout $n \in \N^*$, on définit le $n$-ième \emph{idempotent de Jones-Wenzl évaluable} $f_n \in \TL_n(A^2)$ et le $n$-ième \emph{nilpotent de Jones-Wenzl évaluable} $f_n' \in \TL_n(A^2)$ par : \index{f1 @$f_n$} \index{f1' @$f_n'$}
\begin{align*}
&f_n := \begin{cases}
	p_{t(n)} & \text{ si } n \leq p-1 \text{ ou } n = -1 \mod p, \\
	p_{[t(n)]} & \text{ sinon,} \\
	\end{cases} \\
&f_n' := \begin{cases}
	0 & \text{ si } n \in \{1,...,p-1\} \text{ ou } n = -1 \mod p, \\
	n_{[t(n)]} & \text{ sinon.} \\
	\end{cases}
\end{align*}
\end{Def}

Soient $n \in \N^*$ et $l \in \N$ le quotient de $n$ dans sa division euclidienne par $p$. Alors, d'après le théorème \ref{Thm: POIs rep fonda}, on a le lien suivant avec les facteurs directs de la suite $\{ \X^+(2)^{\otimes n} ; \; ; n \in \N^* \}$.
\begin{enumerate}[(a)]
	\item Si $n \leq p-1$, alors il existe un unique facteur direct $\X$ de $\X^+(2)^{\otimes n}$, isomorphe à $\X^+(n+1)$, tel que $\theta_n(\bar{f}_n)$ est un projecteur sur $\X$.
	\item Si $n = -1 \mod p$, alors il existe un unique facteur direct $\X$ de $\X^+(2)^{\otimes n}$, isomorphe à $2^{l-1} \PIM^{(-)^{l-1}}(p) = 2^{l-1} \X^{(-)^{l-1}}(p)$, tel que $\theta_n(\bar{f}_n)$ est un projecteur sur $\X$.
	\item Sinon, il existe un unique facteur direct $\X$ de $\X^+(2)^{\otimes n}$, isomorphe à $2^{l-1} \PIM^{(-)^{l-1}}((l+1)p-n-1)$, tel que $\theta_n(\bar{f}_n)$ est un projecteur sur $\X$. \\ De plus, le morphisme $\theta_n(\bar{f}_n')$ envoie le sous-module de $\X$ isomorphe à $2^{l-1} \X^{(-)^{l-1}}((l+1)p-n-1)$ sur le module quotient isomorphe à $2^{l-1} \X^{(-)^{l-1}}((l+1)p-n-1)$.
\end{enumerate}
Ces idempotents et nilpotents de Jones-Wenzl évaluables vérifient les systèmes de récurrences suivants.

\begin{Prop} 
\label{Prop: recurrence JWG}
\begin{enumerate}[(i)]
	\item Les idempotents de Jones-Wenzl évaluables vérifient le système de récurrence évaluable :
\[ \begin{cases}
f_1 = 1, \\
f_n = f_{n-1} + \frac{[n-1]_{A^2}}{[n]_{A^2}} f_{n-1} h_{n-1} f_{n-1}, \qquad 2 \leq n \leq p-1, \\
f_p = f_{p-1}, \\
f_{p+1} = f_p + \frac{[p]_{A^2}}{[p+1]_{A^2}} f_p h_p + \frac{[p-1]_{A^2}}{[p+1]_{A^2}} \left( h_p f_p' + f_p' h_p \right)	+ [2]_{A^2} \frac{[p-1]_{A^2}}{[p+1]_{A^2}} f_p' h_p f_p', \\
f_n = \begin{multlined}[t]
	f_{n-1} + \frac{[n-1]_{A^2}}{[n]_{A^2}} f_{n-1} h_{n-1} f_{n-1} \\[-3ex]
	+ \left( A^{2p}+A^{-2p} \right) \frac{[p-1]_{A^2}}{[n]_{A^2} [n-2p]_{A^2}} f_{n-1}' h_{n-1} f_{n-1}, \\
	p+2 \leq n \leq 2p-1, 
	\end{multlined} \\
f_{2p} = f_{2p-1}, \\
f_{2p+1} = \begin{multlined}[t]
	f_{2p} + \frac{[2p]_{A^2}}{[2p+1]_{A^2}} f_{2p} h_{2p} + \frac{[2p-1]_{A^2}}{[2p+1]_{A^2}} \left( h_{2p} f_{2p}' + f_{2p}' h_{2p} \right) \\[-3ex]
	+ [2]_{A^2} \frac{[2p-1]_{A^2}}{[2p+1]_{A^2}} f_{2p}' h_{2p} f_{2p}', 
	\end{multlined} \\
f_n = \begin{multlined}[t]
	f_{n-1} + \frac{[n-1]_{A^2}}{[n]_{A^2}} f_{n-1} h_{n-1} f_{n-1} \\[-3ex]
	+ \left( A^{4p}+A^{-4p} \right) \frac{[2p-1]_{A^2}}{[n]_{A^2} [n-4p]_{A^2}} f_{n-1}' h_{n-1} f_{n-1}, \\
	2p+2 \leq n \leq 3p-2.
	\end{multlined}
\end{cases} \]
	\item Les nilpotents de Jones-Wenzl évaluables vérifient le système de récurrence évaluable :
\[ \begin{cases}
f_p' = f_p h_{p-1} f_p, \\
f_{p+1}' = f_p' - f_p' h_p f_p', \\
f_n' = f_{n-1}' + \frac{[n-1-2p]_{A^2}}{[n-2p]_{A^2}} f_{n-1} h_{n-1} f_{n-1}', &\quad p+2 \leq p \leq 2p-2, \\
f_{2p}' = f_{2p} h_{2p-1} f_{2p}, \\
f_{2p+1}' = f_{2p}' - f_{2p}' h_{2p} f_{2p}', \\
f_n' = f_{n-1}' + \frac{[n-1-4p]_{A^2}}{[n- 4p]_{A^2}} f_{n-1} h_{n-1} f_{n-1}', &\quad 2p+2 \leq p \leq 3p-2. \\
\end{cases} \]
\end{enumerate}
\end{Prop}

\begin{proof}
\begin{enumerate}[(i)]
	\item Pour tout $n \in \{2,...,p-1\}$, on a $f_n = p_{t(n)}$ et $d_{t(n)} = n-1$ (cf. la remarque \ref{Rem: treillis de TL2}). Les premières relations de récurrence résultent donc de la proposition \ref{Prop: POIs de TLA} :
	\[ \begin{cases}
		f_1 = 1, \\
		f_n = f_{n-1} + \frac{[n-1]_{A^2}}{[n]_{A^2}} f_{n-1} h_{n-1} f_{n-1}, & 2 \leq n \leq p-1. 
		\end{cases} \]
	
	Les autres relations de récurrence sont des cas particuliers de relations de récurrence obtenues dans la preuve du lemme \ref{Lemme: idempotents eva1}. Soit $l \in \{1,2\}$. Pour $n=lp$, on a $f_{lp} = p_{[t(lp)]}$ et $f_{lp-1} = p_{t(lp-1)}$ car $t(lp)' = t(lp-1)$ est critique. D'après l'initialisation dans la preuve du lemme \ref{Lemme: idempotents eva1}, on a :
	\begin{equation} \tag{1}
	\begin{aligned}
	f_{lp} &= f_{lp-1}, \\
	f_{lp}' &= f_{lp} h_{lp-1} f_{lp} = f_{lp-1} h_{lp-1} f_{lp-1}.
	\end{aligned}
	\end{equation}
	
	Pour $n=lp+1$, on a $f_{lp+1} = p_{[t(lp+1)]}$ où $t(lp+1)$ possède un sous-tableau critique maximal $t(lp-1)$. On utilise les cas 1 et 3 :
	\begin{align*}
	f_{lp+1} &= \begin{multlined}[t]
		f_{lp} + [2]_{A^2} \frac{[-lp+1]_{A^2}}{[-lp-1]_{A^2}} f_{lp-1} h_{lp-1} f_{lp-1} h_{lp} f_{lp-1} h_{lp-1} f_{lp-1} \\
		+ \frac{[-lp]_{A^2}}{[-lp-1]_{A^2}} f_{lp} h_{lp} f_{lp} + \frac{[-lp+1]_{A^2}}{[-lp-1]_{A^2}} f_{lp} h_{lp} f_{lp-1} h_{lp-1} f_{lp-1} \\
		+ \frac{[-lp+1]_{A^2}}{[-lp-1]_{A^2}} f_{lp-1} h_{lp-1} f_{lp-1} h_{lp} f_{lp} 
		\end{multlined} \\
	&\overset{(1)}{=} \begin{multlined}[t]
		f_{lp} + [2]_{A^2} \frac{[lp-1]_{A^2}}{[lp+1]_{A^2}} f_{lp}' h_{lp} f_{lp}' + \frac{[lp]_{A^2}}{[lp+1]_{A^2}} f_{lp} h_{lp} f_{lp} \\
		+ \frac{[lp-1]_{A^2}}{[lp+1]_{A^2}} \left( f_{lp} h_{lp} f_{lp}' + f_{lp}' h_{lp} f_{lp} \right) .
		\end{multlined}
	\end{align*}
	Or, $f_{lp} = f_{lp-1}$ est une combinaison de mots en $\{h_0,h_1,...,h_{lp-2}\}$. D'où les simplifications :
	\begin{align*}
	&f_{lp} h_{lp} f_{lp} = f_{lp}^2 h_{lp} = f_{lp} h_{lp}, \\
	&f_{lp} h_{lp} f_{lp}' = h_{lp} f_{lp} f_{lp}' = h_{lp} f_{lp}', \\
	&f_{lp}' h_{lp} f_{lp} = f_{lp}' f_{lp} h_{lp} = f_{lp}' h_{lp}.
	\end{align*}
		
	Pour $n \in \{lp+2,...,(l+1)p-2\}$, on a $f_n = p_{[t(n)]}$ où $t(n)$ possède un sous-tableau critique maximal $t(lp-1)$. On utilise les cas 2 et 3 :
	\begin{gather*}
	f_n = \begin{multlined}[t]
		f_{n-1} + (A^{2lp}+A^{-2lp}) \frac{[lp-1]_{A^2}}{[n-2lp]_{A^2} [n]_{A^2}} f_{n-1} h_{lp-1} f_{n-1} h_{n-1} f_{n-1} \\
		+ \frac{[n-1]_{A^2}}{[n]_{A^2}} f_{n-1} h_{n-1} f_{n-1} .
		\end{multlined}
	\end{gather*}
	Or, par définition, on a $f_{n-1}' = f_{n-1} h_{lp-1} f_{n-1}$. D'où le résultat.
	
	Enfin, pour $n=2p-1$, le tableau standard $t(2p-1)$ est critique et ne possède pas de tableau conjugué (cf. la remarque \ref{Rem: lignes critiques}). Donc $f_{2p-1} = p_{t(2p-1)} = p_{[t(2p-1)]}$ ; les calculs ci-dessus s'appliquent encore.
	
	\item Soient $l \in \{1,2\}$ et $n \in \{lp+1,...,(l+1)p-2\}$. On a $f_n' = p_{[t(n)]} h_{lp-1} p_{[t(n)]}$ car $t(n)$ possède un sous-tableau critique maximal $t(lp-1)$. De plus, d'après la remarque \ref{Rem: treillis de TL2}, on a :
	\[ \begin{cases}
		d_{t(n)}(lp-1) = d_{t(lp)}(lp-1) = -1, \\
		d_{\overline{t(n)}}(lp-1) =  d_{\overline{t(lp)}}(lp-1) = d_{\widehat{t(lp)}}(lp-1) = lp-1.
		 \end{cases} \]
	En procédant comme dans la preuve de la proposition \ref{Prop: POIs de TLA}, on en déduit que :
	\begin{equation} \tag{2}
	\begin{aligned}
	&f_{lp}' = p_{[t(lp)]} h_{lp-1} p_{[t(lp)]} = p_{\overline{t(lp)}} h_{lp-1} p_{\overline{t(lp)}} = - \frac{[lp]_{A^2}}{[lp-1]_{A^2}} p_{\overline{t(lp)}} \\
	&f_{n}' = p_{[t(n)]} h_{lp-1} p_{[t(n)]} = p_{\overline{t(n)}} h_{lp-1} p_{\overline{t(n)}} = - \frac{[lp]_{A^2}}{[lp-1]_{A^2}} p_{\overline{t(n)}}. 
	\end{aligned}
	\end{equation}
	Or, d'après la proposition \ref{Prop: POIs de TLA}, on a aussi :
	\begin{equation} \tag{3}
	p_{\overline{t(n)}} = p_{\overline{t(n-1)}} + \frac{[n-1-2lp]_{A^2}}{[n-2lp]_{A^2}} p_{\overline{t(n-1)}} h_{n-1} p_{\overline{t(n-1)}} 
	\end{equation}	
	car $d_{\overline{t(n)}}(n-1) = -1$ et $d_{\widehat{\overline{t(n)}}}(n-1) = n-1-2lp$. \\
	Pour $n=lp+1$, en multipliant l'équation $(3)$ par $- \frac{[lp]_{A^2}}{[lp-1]_{A^2}}$, on obtient :
	\[ 	f_{lp+1}' = f_{lp}' + \frac{[-lp]_{A^2}}{[-lp+1]_{A^2}} f_{lp}' h_{lp} p_{\overline{t(lp)}} \overset{(2)}{=} f_{lp}' - f_{lp}' h_{lp} f_{lp}' . \]	
	Pour $n \geq lp+2$, les tableaux standards $t(n-1)$ et $\overline{t(n-1)}$ n'ont pas la même forme. Il s'ensuit que :
	\[ p_{\overline{t(n-1)}} h_{n-1} p_{\overline{t(n-1)}} = p_{ \left[ \overline{t(n-1)} \right] } h_{n-1} p_{\overline{t(n-1)}} = f_{n-1} h_{n-1} p_{\overline{t(n-1)}} . \]
	En multipliant l'équation $(3)$ par $-\frac{[lp]_{A^2}}{[lp-1]_{A^2}}$, on obtient :
	\[ 	f_n' = f_{n-1}' + \frac{[n-1-lp]_{A^2}}{[n-lp]_{A^2}} f_{n-1} h_{n-1} f_{n-1}'. \]
	D'où le résultat.
\end{enumerate}
\end{proof}

\begin{Rem}
\label{Rem: recurrence JWG}
\begin{enumerate}[(i)]
	\item Pour tout $n \in \{1,...,3p-2\}$, le système de récurrence donnée dans l'assertion (i) de la proposition \ref{Prop: recurrence JWG} est la seule manière de compléter le système de récurrence \eqref{Eqn: recurrence JW} pour obtenir des idempotents \emph{évaluables} (cf. la sous-section \ref{subsection: TLq}). Pour $n \geq 3p-1$, le passage des lignes critiques complexifie la dynamique de récurrence et requiert l'intervention d'idempotents associés aux lignes critiques antérieures (cf. la preuve du lemme \ref{Lemme: idempotents eva2}).
	\item A partir de toutes ces relations de récurrence, on peut montrer par récurrence que, pour tout $n \in \N^*$, $f_n-1$ est dans l'idéal de $\TL_n(A^2)$ engendré par $h_1,...,h_{n-1}$.
\end{enumerate}
\end{Rem}

\begin{Cor} 
\label{Cor: recurrence JWG}
\begin{enumerate}[(i)]
	\item Les idempotents de Jones-Wenzl évalués vérifient le système de récurrence :
\[ \begin{cases}
\bar{f}_1 = 1, \\
\bar{f}_n = \bar{f}_{n-1} + \frac{[n-1]}{[n]} \bar{f}_{n-1} h_{n-1} \bar{f}_{n-1}, & 2 \leq n \leq p-1, \\
\bar{f}_p = \bar{f}_{p-1}, \\
\bar{f}_{p+1} = \bar{f}_p - \left( h_p \bar{f}_p' + \bar{f}_p' h_p \right)	- [2] \bar{f}_p' h_p \bar{f}_p', \\
\bar{f}_n = \bar{f}_{n-1} + \frac{[n-1]}{[n]} \bar{f}_{n-1} h_{n-1} \bar{f}_{n-1} - \frac{2}{[n]^2} \bar{f}_{n-1}' h_{n-1} \bar{f}_{n-1}, &	p+2 \leq n \leq 2p-1,  \\
\bar{f}_{2p} = \bar{f}_{2p-1}, \\
\bar{f}_{2p+1} = \bar{f}_{2p} - \left( h_{2p} \bar{f}_{2p}' + \bar{f}_{2p}' h_{2p} \right) - [2] \bar{f}_{2p}' h_{2p} \bar{f}_{2p}', \\
\bar{f}_n = \bar{f}_{n-1} + \frac{[n-1]}{[n]} \bar{f}_{n-1} h_{n-1} \bar{f}_{n-1} - \frac{2}{[n]^2} \bar{f}_{n-1}' h_{n-1} \bar{f}_{n-1}, & 2p+2 \leq n \leq 3p-2.
\end{cases} \]
	\item Les nilpotents de Jones-Wenzl évalués vérifient le système de récurrence :
\[ \begin{cases}
\bar{f}_p' = \bar{f}_p h_{p-1} \bar{f}_p, \\
\bar{f}_{p+1}' = \bar{f}_p' - \bar{f}_p' h_p \bar{f}_p', \\
\bar{f}_n' = \bar{f}_{n-1}' + \frac{[n-1]}{[n]} \bar{f}_{n-1} h_{n-1} \bar{f}_{n-1}', &\quad p+2 \leq p \leq 2p-2, \\
\bar{f}_{2p}' = \bar{f}_{2p} h_{2p-1} \bar{f}_{2p}, \\
\bar{f}_{2p+1}' = \bar{f}_{2p}' - \bar{f}_{2p}' h_{2p} \bar{f}_{2p}', \\
\bar{f}_n' = \bar{f}_{n-1}' + \frac{[n-1]}{[n]} \bar{f}_{n-1} h_{n-1} \bar{f}_{n-1}', &\quad 2p+2 \leq p \leq 3p-2. \\
\end{cases} \]
\end{enumerate}
\end{Cor}

On retrouve les propriétés usuelles des idempotents de Jones-Wenzl (cf. par exemple \cite[Thm 5.3.2]{CFS95}) sur les premiers termes.

\begin{Prop}
\label{Prop: JWG}
Soient $n \in \N^*$ et $l \in \N$ le quotient de $n$ dans sa division euclidienne par $p$. On a :
\begin{enumerate}[(i)]
	\item $f_n^2 = f_n$ ;
	\item si $n \leq p-1$ ou $n = -1 \mod p$, alors :
	\[ \forall i \in \{1,...,n-1\} \qquad f_n h_i = 0 = h_i f_n \; ; \]
	\item sinon :
	\[ \forall i \in \{1,...,n-1\} \setminus \{lp-1\} \qquad f_n h_i = 0 = h_i f_n, \qquad f_n'^2 = - \frac{[lp]_{A^2}}{[lp-1]_{A^2}} f_n' . \]
\end{enumerate}
\end{Prop}

\begin{proof}
Il est clair que $f_n^2 = f_n$. Passons aux autres assertions. On étudie les éléments de $\TL_n(A^2)$ via leurs actions sur les représentations semi-normales $\{ V_\lambda \; ; \; \lambda \in \Delta^{\leq2}_n \}$ grâce à l'isomorphisme \eqref{Eqn: structure TLA1}.
\begin{enumerate}[(i)]
	\setcounter{enumi}{1}
	\item On suppose que $n \leq p-1$ ou $n = -1 \mod p$. On a $f_n = p_{t(n)}$. D'après la remarque \ref{Rem: treillis de TL2}, le tableau standard $t(n)$ vérifie :
	\[ \forall i \in \{1,...,n-1\} \qquad d_{t(n)}(i) = -1. \]
	Soit $s \in T_n^{\leq2}$. Pour tout $i \in \{1,...,n-1\}$, on a donc :
	\begin{align*}
	h_i f_n \cdot v_s &= h_i p_{t(n)} \cdot v_s = \delta_{s,t(n)} h_i \cdot v_{t(n)} = 0, \\
	f_n h_i \cdot v_{t(n)} &= p_{t(n)} h_i \cdot v_{t(n)} \\
	&= - \delta_{s,t(n)} \frac{[d_s(i)+1]_{A^2}}{[d_s(i)]_{A^2}} v_{t(n)} - \delta_{\sigma_i(s),t(n)} (1-\delta_{d_s(i),-1} ) \frac{[d_s(i)-1]_{A^2}}{[d_t(i)]_{A^2}} v_{t(n)} \\
	&= 0.
	\end{align*}
	Ceci étant valable pour tout $s \in T_n^{\leq2}$, on en déduit que $h_i f_n = 0 = f_n h_i$.
	\item On suppose que $n \geq p$ et $n \neq -1 \mod p$. On a $f_n = p_{[t(n)]}$. Le tableau standard $t(n)$ possède un sous-tableau critique maximal $t(lp-1)$ et un tableau conjugué $\overline{t(n)}$ (cf. la remarque \ref{Rem: lignes critiques}). D'après la remarque \ref{Rem: treillis de TL2}, les tableaux standards $t(n)$ et $\overline{t(n)}$ vérifient :
	\[	\forall i \in \{1,...,n-1\} \qquad d_{t(n)}(i) = -1, \qquad 
	d_{\overline{t(n)}}(i) = \begin{cases}
		-1 & \text{ si } i \neq lp-1, \\
		lp-1 & \text{ sinon.}
		\end{cases} \]
	En procédant comme dans $(a)$ avec $f_n = p_{[t(n)]} = p_{t(n)} + p_{\overline{t(n)}}$ et $i \in \{1,...,n-1\} \setminus \{lp-1\}$, on obtient $h_i f_n = 0 = f_n h_i$. Tandis que pour $i=lp-1$, on a :
\[ f_n' = p_{[t(n)]} h_{lp-1} p_{[t(n)]} = p_{\overline{t(n)}} h_{lp-1} p_{\overline{t(n)}} = - \frac{[lp]_{A^2}}{[lp-1]_{A^2}} p_{\overline{t(n)}} . \]
Il s'ensuit que :
\[ f_n'^2 = \left( - \frac{[lp]_{A^2}}{[lp-1]_{A^2}} p_{\overline{t(n)}} \right)^2 = \left( - \frac{[lp]_{A^2}}{[lp-1]_{A^2}} \right)^2 p_{\overline{t(n)}} = - \frac{[lp]_{A^2}}{[lp-1]_{A^2}} f_n' . \]
\end{enumerate}
\end{proof}

\subsection{Une base de l'espace d'écheveaux du tore solide}
\label{subsection: echeveaux JWG}

On considère dorénavant le $\Z[A,A^{-1}]$-module d'écheveaux $K_A(\bar{D} \times \mathbb{S}^1, 0)$ du tore solide et, pour tout $n \in \N$, l'application linéaire : \index{Theta3 @$\Theta_{A^2,n}$}
\begin{equation}
\label{Eqn: ThetaA}
\Theta_{A^2,n} : \left\{ \begin{array}{l}
	\TL_n(A^2) \longrightarrow \C(A) \otimes_{\Z[A,A^{-1}]} K_A(\bar{D} \times \mathbb{S}^1, 0) \\
	a \longmapsto \insertion{a}{n}{Tore}
	\end{array} \right. \hspace{-0.5cm} .
\end{equation}

On note $1$ la classe d'écheveau de l'enchevêtrement vide, et $\alpha$ la classe d'écheveau associée à l'âme $\{0\} \times \mathbb{S}^1$ de $\bar{D} \times \mathbb{S}^1$. On désigne par $\alpha^n$ la classe d'écheveau de $n \in \mathbb{N}^*$ copie(s) parallèle(s) de $\alpha$. Compte-tenu des relations d'écheveaux \eqref{Eqn: echeveaux2}, pour tout $n \in \N$, l'image de $\Theta_{A^2,n}$ est incluse dans l'ensemble des polynômes en $\alpha$ à coefficient dans $\C(A)$ de degré au plus $n$. On étudie l'ensemble des images $\{ \Theta_{A^2,n}(f_n), \Theta_{A^2,n}(f_n') \; ; \; n \in \N^* \}$ des idempotents et nilpotents de Jones-Wenzl évaluables. 

Pour cela, on commence par étudier les images des POIs des algèbres de Teperley-Lieb génériques. Pour $n=0$, l'algèbre $\TL_0(A^2)=\C(A)$ possède un unique POI $p=1$ et $\Theta{A^2,0}(1) = 1$. Pour les autres valeurs de $n$, les images $\Theta_{A^2,n}(p_t)$, $t \in T_n^{\leq2}$, font intervenir les polynômes de Chebychev $(U_s(x))_{s \in \N}$ \eqref{Eqn: Chebychev}. Afin d'épurer les représentations diagrammatiques, on note :
\begin{equation}
\label{Eqn: insertion des POIs}
\forall n \in \N^* \quad \forall t \in T_n^{\leq2} \qquad \rect{I}{t} := \rect{p_t}{} .
\end{equation} \index{p4 @$\rect{I}{t}$}

\begin{Prop}
\label{Prop: insertion des POIs}
On suppose que $n \geq 1$. Pour tout tableau standard $t \in T_n^{\leq2}$ de forme $\lambda$, on a $\Theta_{A^2,n}(p_t) = U_{\omega(\lambda)}(\alpha)$.
\end{Prop}

\begin{proof}
On procède par récurrence sur $n \in \N^*$.

Pour $n=1$, on a $\Delta_1^{\leq2} = \{ \lyng{1} \}$ et $T_1^{\leq 2} = \{ \lyoung{1} \}$. Il y a un unique POI $p_{\lyoung{1}} = 1$ et $\Theta_{A^2,1}(p_{\lyoung{1}}) = \alpha = U_2(\alpha)$.

Pour $n=2$, $\Delta_2^{\leq2} = \{ \lyng{1,1}, \lyng{2} \}$ et $T_1^{\leq 2} = \{ \lyoung{1,2}, \lyoung{12} \}$. On a exactement deux POIs qui, d'après la proposition \ref{Prop: POIs de TLA} et la remarque \ref{Rem: treillis de TL2}, vérifient :
\[ p_{\lyoung{1,2}} = - \frac{1}{[2]_{A^2}} p_{\lyoung{1}} h_1 p_{\lyoung{1}}, \qquad p_{\lyoung{12}} = p_{\lyoung{1}} + \frac{1}{[2]_{A^2}} p_{\lyoung{1}} h_1 p_{\lyoung{1}}. \]
On en déduit que :
\begin{align*}
\Theta_{A^2,2}(p_{\lyoung{1,2}}) &= - \frac{1}{[2]_{A^2}} \insertion{IH0}{\lyoung{1}}{Tore} = - \frac{-[2]_{A^2}}{[2]_{A^2}} \Tore
	= 1 = U_1(\alpha), \\
\Theta_{A^2,2}(p_{\lyoung{12}}) &= \insertion{I1}{\lyoung{1}}{Tore} + \frac{1}{[2]_{A^2}} \insertion{IH0}{\lyoung{1}}{Tore}
	= \alpha^2 - 1 = U_3(\alpha).
\end{align*}

Soit $n \geq 2$. On suppose que, pour tout $k \in \{1,...n\}$ et pour tout $r \in T_k^{\leq2}$ de forme $\mu$, $\Theta_{A^2,k}(p_r) = U_{\omega(\mu)}(\alpha)$. Soit $t \in T_{n+1}^{\leq2}$ de forme $\lambda$. On considère les tableaux standards $t' \in T_n^{\leq 2}$ et $t'' \in T_{n-1}^{\leq 2}$. On note $\mu$ la forme de $t'$ et $\nu$ celle de $t''$. On utilisera régulièrement la proposition \ref{Prop: POIs de TLA} et la remarque \ref{Rem: treillis de TL2}. On distingue 4 cas.

\begin{Cas}
On suppose que $d_t(n) = 1 = d_t(n-1)$. Alors $\omega(\nu) = \omega(\lambda) \mp 2$, $\omega(\mu) = \omega(\lambda) \mp 1$, et on a deux configurations possibles :
\[ \begin{array}{|c|c|}
\hline
\omega(\lambda) \geq 4 & \omega(\lambda) \geq 1 \\
\hline
\vcenter{ \xymatrix @-2ex {
	\bullet \ar[rd]^{t''} \\
	& \bullet \ar[rd]^{t'} \ar[ld]_{\widehat{t'}} \\
	\bullet && \bullet \ar[rd]^{t} \ar[ld]_{\hat{t}} \\
	& \bullet && \bullet }} &
\vcenter{ \xymatrix @-2ex {
	&&& \bullet \ar[ld]_{t''} \\
	&& \bullet \ar[ld]_{t'} \ar[rd]^{\widehat{t'}} \\
	& \bullet \ar[ld]_{t} \ar[rd]^{\hat{t}} && \bullet \\
	\bullet && \bullet }} \\
p_t = p_{t'} + \frac{[\omega(\lambda)-2]_{A^2}}{[\omega(\lambda)-1]_{A^2}} p_{t'} h_n p_{t'} & 
	p_t = p_{t'} + \frac{[\omega(\lambda)+2]_{A^2}}{[\omega(\lambda)+1]_{A^2}} p_{t'} h_n p_{t'} \\ 
p_{t'} = p_{t''} + \frac{[\omega(\lambda)-3]_{A^2}}{[\omega(\lambda)-2]_{A^2}} p_{t''} h_{n-1} p_{t''} & 
	p_{t'} = p_{t''} + \frac{[\omega(\lambda)+3]_{A^2}}{[\omega(\lambda)+2]_{A^2}} p_{t''} h_{n-1} p_{t''} \\
\hline
\end{array} \]
On en déduit que :
\begin{align*}
\Theta_{A^2,n+1}(p_t)
	&\; \; = \insertion{I1}{t'}{Tore} + \frac{[\omega(\lambda) \mp 2]_{A^2}}{[\omega(\lambda) \mp 1]_{A^2}} \insertion{IH1}{t'}{Tore} \\
	&\; \; = \alpha \Theta_{A^2,n}(p_{t'}) + \frac{[\omega(\lambda) \mp 2]_{A^2}}{[\omega(\lambda) \mp 1]_{A^2}} \insertion{IB1}{t'}{Tore} \\
	&\; \; = \alpha \Theta_{A^2,n}(p_{t'}) \\
	&\; \; \quad + \frac{[\omega(\lambda) \mp 2]_{A^2}}{[\omega(\lambda) \mp 1]_{A^2}} \insertion{IB2}{t''}{Tore}
		+ \frac{[\omega(\lambda) \mp 3]_{A^2}}{[\omega(\lambda) \mp 1]_{A^2}} \insertion{IH2}{t''}{Tore} \\
	&\; \; = \alpha \Theta_{A^2,n}(p_{t'}) + \left( -[2]_{A^2} \frac{[\omega(\lambda) \mp 2]_{A^2}}{[\omega(\lambda) \mp 1]_{A^2}} + \frac{[\omega(\lambda) \mp 3]_{A^2}}{[\omega(\lambda) \mp 1]_{A^2}} \right) \Theta_{A^2,n}(p_{t''}) \\
	&\overset{\eqref{Eqn: Acoeff2}}{=} \alpha \Theta_{A^2,n}(p_{t'}) - \Theta_{A^2,n-1}(p_{t''}).
\end{align*}
Or, d'après l'hypothèse de récurrence, on a :
\[ \Theta_{A^2,n}(p_{t'}) = U_{\omega(\lambda)\mp1}(\alpha), \qquad \Theta_{A^2,n-1}(p_{t''}) = U_{\omega(\lambda)\mp2}(\alpha) . \]
Par définition des polynômes de Chebychev \eqref{Eqn: Chebychev}, on a donc :
\[ \Theta_{A^2,n+1}(p_t) = \alpha U_{\omega(\lambda)\mp1}(\alpha) - U_{\omega(\lambda)\mp2}(\alpha) = U_{\omega(\lambda)}(\alpha). \]
\end{Cas}

\begin{Cas}
On suppose que $d_t(n) = 1$ et $d_t(n-1) \neq -1$. Alors $\omega(\nu) = \omega(\lambda) \mp 2$, $\omega(\mu) = \omega(\lambda) \mp 1$, et on a deux configurations possibles :
\[ \begin{array}{|c|c|}
\hline
\omega(\lambda) \geq 3 & \omega(\lambda) \geq 1 \\
\hline
\vcenter{ \xymatrix @-2ex {
	& \bullet \ar[ld]_{t''} \\
	\bullet \ar[rd]^{t'} \\
	& \bullet \ar[rd]^{t} \ar[ld]_{\hat{t}} \\
	\bullet && \bullet }} &
\vcenter{ \xymatrix @-2ex {
	& \bullet \ar[rd]^{t''}  \\
	&& \bullet \ar[ld]_{t'} \\
	& \bullet \ar[ld]_{t} \ar[rd]^{\hat{t}} \\
	\bullet && \bullet }} \\
p_t = p_{t'} + \frac{[\omega(\lambda)-2]_{A^2}}{[\omega(\lambda)-1]_{A^2}} p_{t'} h_n p_{t'} & 
	p_t = p_{t'} + \frac{[\omega(\lambda)+2]_{A^2}}{[\omega(\lambda)+1]_{A^2}} p_{t'} h_n p_{t'} \\ 
p_{t'} = - \frac{[\omega(\lambda)-1]_{A^2}}{[\omega(\lambda)-2]_{A^2}} p_{t''} h_{n-1} p_{t''} & 
	p_{t'} = - \frac{[\omega(\lambda)+1]_{A^2}}{[\omega(\lambda)+2]_{A^2}} p_{t''} h_{n-1} p_{t''} \\
\hline
\end{array} \]
En procédant comme dans le cas 1, on obtient :
\begin{align*}
\Theta_{A^2,n+1}(p_t) &= \alpha \Theta_{A^2,n}(p_{t'}) + \frac{[\omega(\lambda) \mp 2]_{A^2}}{[\omega(\lambda) \mp 1]_{A^2}} \insertion{IB1}{t'}{Tore} \\
	&= \alpha \Theta_{A^2,n}(p_{t'}) - \insertion{IH2}{t''}{Tore} 
	= \alpha \Theta_{A^2,n}(p_{t'}) - \Theta_{A^2,n-1}(p_{t''}).
\end{align*}
On conclut comme dans le cas 1.
\end{Cas}

\begin{Cas}
On suppose que $d_t(n) \neq 1$ et $d_t(n-1) = -1$. Alors $\omega(\nu) = \omega(\lambda)$, $\omega(\mu) = \omega(\lambda) \pm 1$, et on a trois configurations possibles :
\[ \begin{array}{|c|c|c|}
\hline
\omega(\lambda) \geq 2 & \omega(\lambda) \geq 3 & \omega(\lambda) = 2 \\
\hline
\vcenter{ \xymatrix @-2ex {
	\bullet \ar[rd]^{t''} \\
	& \bullet \ar[rd]^{t'} \ar[ld]_{\widehat{t'}} \\
	\bullet && \bullet \ar[ld]_{t} \\
	& \bullet }} &
\vcenter{ \xymatrix @-2ex {
	&& \bullet \ar[ld]_{t''} \\
	& \bullet \ar[ld]_{t'} \ar[rd]^{\widehat{t'}} \\
	\bullet \ar[rd]^{t} && \bullet \\
	& \bullet }} &
\vcenter{ \xymatrix @-2ex {
	\ar@{--}[ddd] && \bullet \ar[ld]_{t''} \\
	& \bullet \ar[ld]_{t'} \ar[rd]^{\widehat{t'}} \\
	\bullet \ar[rd]^{t} && \bullet \\
	& \bullet }} 	\\
p_t = - \frac{[\omega(\lambda)]_{A^2}}{[\omega(\lambda)+1]_{A^2}} p_{t'} h_n p_{t'} & 
	p_t = - \frac{[\omega(\lambda)]_{A^2}}{[\omega(\lambda)-1]_{A^2}} p_{t'} h_n p_{t'} &
	p_t = p_{t'} \\ 
p_{t'} = p_{t''} + \frac{[\omega(\lambda)-1]_{A^2}}{[\omega(\lambda)]_{A^2}} p_{t''} h_{n-1} p_{t''} & 
	p_{t'} = p_{t''} + \frac{[\omega(\lambda)+1]_{A^2}}{[\omega(\lambda)]_{A^2}} p_{t''} h_{n-1} p_{t''} & \\
\hline
\end{array} \]
Pour la dernière configuration, il est clair que $\Theta_{A^2,n+1}(p_t) = \alpha \Theta_{A^2,n} (p_{t'})$. Or, d'après l'hypothèse de récurrence, on a $\Theta_{A^2,n} (p_{t'}) = U_1(\alpha) = 1$. Donc $\Theta_{A^2, n+1}(p_t) = \alpha = U_2(\alpha)$. Pour les deux autres configurations, on a :
\begin{align*}
\Theta_{A^2,n+1}(p_t) 
	& = - \frac{[\omega(\lambda)]_{A^2}}{[\omega(\lambda) \pm 1]_{A^2}} \insertion{IH1}{t'}{Tore}
	= - \frac{[\omega(\lambda)]_{A^2}}{[\omega(\lambda) \pm 1]_{A^2}} \insertion{IB1}{t'}{Tore} \\
	& = - \frac{[\omega(\lambda)]_{A^2}}{[\omega(\lambda) \pm 1]_{A^2}} \insertion{IB2}{t''}{Tore}
		- \frac{[\omega(\lambda) \mp 1]_{A^2}}{[\omega(\lambda) \pm 1]_{A^2}} \insertion{IH2}{t''}{Tore} \\
	& = \left( [2]_{A^2} \frac{[\omega(\lambda)]_{A^2}}{[\omega(\lambda) \pm 1]_{A^2}} + \frac{[\omega(\lambda) \mp 1]_{A^2}}{[\omega(\lambda) \pm 1]_{A^2}} \right) \Theta_{A^2,n-1}(p_{t''})
	\overset{\eqref{Eqn: Acoeff2}}{=} \Theta_{A^2,n-1}(p_{t''}).
\end{align*}
Or, d'après l'hypothèse de récurrence, on a $\Theta_{A^2,n-1}(p_{t''}) = U_{\omega(\lambda)}(\alpha)$. Donc :
\[ \Theta_{A^2,n+1}(p_t) = U_{\omega(\lambda)}(\alpha) . \]
\end{Cas}

\begin{Cas}
On suppose que $d_t(n) \neq 1$ et $d_t(n-1) \neq -1$. Alors $\omega(\nu) = \omega(\lambda)$, $\omega(\mu) = \omega(\lambda) \pm 1$, et on a trois configurations possibles :
\[ \begin{array}{|c|c|c|}
\hline
\omega(\lambda) \geq 1 & \omega(\lambda) \geq 3 & \omega(\lambda) = 2 \\
\hline
\vcenter{ \xymatrix @-2ex {
	& \bullet \ar[ld]_{t''} \\
	\bullet \ar[rd]^{t'} \\
	& \bullet \ar[ld]_{t} \\
	\bullet }} &
\vcenter{ \xymatrix @-2ex {
	\bullet \ar[rd]^{t''} \\
	& \bullet \ar[ld]_{t'} \\
	\bullet \ar[rd]^{t} \\
	& \bullet }} &
\vcenter{ \xymatrix @-2ex {
	\ar@{--}[ddd] \bullet \ar[rd]^{t''} \\
	& \bullet \ar[ld]_{t'} \\
	\bullet \ar[rd]^{t} \\
	& \bullet }} 	\\
p_t = - \frac{[\omega(\lambda)]_{A^2}}{[\omega(\lambda)+1]_{A^2}} p_{t'} h_n p_{t'} & 
	p_t = - \frac{[\omega(\lambda)]_{A^2}}{[\omega(\lambda)-1]_{A^2}} p_{t'} h_n p_{t'} &
	p_t = p_{t'} \\ 
p_{t'} = - \frac{[\omega(\lambda)+1]_{A^2}}{[\omega(\lambda)]_{A^2}} p_{t''} h_{n-1} p_{t''} & 
	p_{t'} = - \frac{[\omega(\lambda)-1]_{A^2}}{[\omega(\lambda)]_{A^2}} p_{t''} h_{n-1} p_{t''} & \\
\hline
\end{array} \]
Pour la dernière configuration, on procède comme dans le cas 3. Pour les deux autres configurations, on a :
\begin{align*}
\Theta_{A^2,n+1}(p_t) 
	&= - \frac{[\omega(\lambda)]_{A^2}}{[\omega(\lambda) \pm 1]_{A^2}} \insertion{IH1}{t'}{Tore}
	= - \frac{[\omega(\lambda)]_{A^2}}{[\omega(\lambda) \pm 1]_{A^2}} \insertion{IB1}{t'}{Tore} \\
	&= \insertion{IH2}{t''}{Tore} = \Theta_{A^2,n-1}(p_{t''}).
\end{align*}
On conclut comme dans le cas 3.
\end{Cas}

Par conséquent, dans chacun des cas, on a $\Theta_{A^2,n+1}(p_t) = U_{\omega(\lambda)}(\alpha)$. Ce qui prouve l'hérédité et achève la récurrence.
\end{proof}

\begin{Rem}
Pour $n \geq p$, on sait qu'il existe des POIs $p_t$, avec $t \in T_n^{\leq2}$, \emph{non évaluables} (cf. le théorème \ref{Thm: idempotents eva}). Cependant, ils le deviennent sous l'image de $\Theta_{A^2,n}$ puisque les polynômes de Chebychev sont à coefficients relatifs (récurrence facile).
\end{Rem}

On en déduit la généralisation suivante du théorème \cite[Thm 2]{Lic92}.

\begin{Prop} 
\label{Prop: insertion des JWG}
Soient $n \in \N^*$ et $l \in \N$ le quotient de $n$ dans sa division euclidienne par $p$. On a :
\begin{align*}
\Theta_{A^2,n}(f_n) &= \begin{cases} 
	U_{n+1}(\alpha) & \text{ si } n \leq p-1 \text{ ou } n=-1 \mod p, \\
	U_{n+1}(\alpha)+U_{2lp-n-1}(\alpha) & \text{ sinon, } \end{cases} \\
\Theta_{A^2,n}(f_n') &= - \frac{[lp]_{A^2}}{[lp-1]_{A^2}} U_{2lp-n-1}(\alpha) \quad \text{ si } n \geq p \text{ et } n \neq -1 \mod p.
\end{align*}
\end{Prop}

\begin{proof}
Il suffit d'utiliser la proposition \ref{Prop: insertion des POIs}. On note $\lambda(n)$ la forme de $t(n)$, et $\overline{\lambda(n)}$ celle de $\overline{t(n)}$ lorsqu'il existe. 
\begin{enumerate}[(a)]
	\item Si $n \leq p-1$ ou $n = -1 \mod p$, alors $f_n = p_{t(n)}$ et $\omega(\lambda(n)) = n+1$. D'où $\Theta_{A^2,n}(f_n) = U_{n+1}(\alpha)$.
	\item Sinon, $f_n = p_{[t(n)]} = p_{t(n)} + p_{\overline{t(n)}}$ et $\omega(\lambda(n)) = n+1$, $\omega(\overline{\lambda(n)}) = 2lp-n-1$. D'où $\Theta_{A^2,n}(f_n) = U_{n+1}(\alpha)+U_{2lp-n-1}(\alpha)$. Enfin, on a vu que $f_n' = - \frac{[lp]_{A^2}}{[lp-1]_{A^2}} p_{\overline{t(n)}}$ dans la preuve de la proposition \ref{Prop: JWG}. Donc $\Theta_{A^2,n}(f_n') = - \frac{[lp]_{A^2}}{[lp-1]_{A^2}} U_{2lp-n-1}(\alpha)$.
\end{enumerate}
\end{proof}

\begin{Cor}
\label{Cor: insertion des JWG}
La famille $\{ \Theta_{A^2,n}(f_n) \; ; \; n \in \N^* \}$ est une $\Z[A,A^{-1}]$-base de l'espace d'écheveaux $K_A( \bar{D} \times \mathbb{S}^1, 0)$.
\end{Cor}

\begin{proof}
D'après la proposition \ref{Prop: multi-courbes}, la famille $\{ \alpha^n \: ; \; n \in \N^* \}$ est une $\Z[A,A^{-1}]$-base de $K_A( \bar{D} \times \mathbb{S}^1, 0)$. Or, la famille $\{ \Theta_{A^2,n}(f_n) \; ; \; n \in \N^* \}$ est formée de polynômes de Chebychev en $\alpha$ avec des indices croissants (cf. la proposition \ref{Prop: insertion des JWG}). Par construction, les polynômes de Chebychev sont échelonnés en degré et à coefficients relatifs (récurrences faciles). D'où le résultat.
\end{proof}

Enfin, on a une trace l'espace d'écheveaux $K_A(\bar{D} \times \mathbb{S}^1, 0)$ du tore solide définie par : \index{tr @$\tr$}
\begin{equation}
\label{Eqn: tr}
\tr : \left\{ \begin{array}{l}
	K_A(\bar{D} \times \mathbb{S}^1, 0) \longrightarrow \Z[A,A^{-1}] \\
	\insertion{a}{}{Tore} \longmapsto \insertion{a}{}{Empty}
	\end{array} \right. \hspace{-0.5cm} .
\end{equation}

\begin{Prop}
\label{Prop: traces des JWG}
Soient $n \in \N^*$ et $l \in \N$ le quotient de $n$ dans sa division euclidienne par $p$. On a :
\begin{align*}
\tr \left( \Theta_{A^2,n}(f_n) \right) &= \begin{cases}
	(-1)^n [n+1]_{A^2} \quad \text{ si } n \leq p-1 \text{ ou } n = -1 \mod p, \\
	(-1)^n [n+1]_{A^2} + (-1)^n [2lp-n-1]_{A^2} \quad \text{ sinon,}
	\end{cases} \\
\tr \left( \Theta_{A^2,n}(f_n') \right) &= (-1)^{n+1} \frac{[lp]_{A^2} [2lp-n-1]_{A^2}}{[lp-1]_{A^2}} \quad \text{ si } n \geq p \text{ et } n \neq -1 \mod p.
\end{align*}
\end{Prop}

\begin{proof}
D'après les relations d'écheveaux \eqref{Eqn: echeveaux2}, pour tout polynôme $R(x) \in \Z[\alpha]$, on a :
\[ \tr( R(\alpha)) = R(-A^2-A^{-2}) . \]
Il suffit alors d'utiliser la proposition \ref{Prop: insertion des JWG}) et la remarque \ref{Rem: Chebychev}.
\end{proof}

\begin{Rem}
Lorsque $A^2$ s'évalue en $q= e^{\frac{i \pi}{p}}$, les classes d'écheveaux :
\[\Theta_{A^2,n}(f_n), \; \Theta_{A^2,n}(f_n') \; ; \; \quad n \geq p, \]
coloriées par les idempotents et nilpotents de Jones-Wenzl évaluables d'indice $n \geq p$ sont toutes de trace nulle. Ce constat n'est pas surprenant puisque, dans notre construction, ces idempotents et nilpotents correspondent à des $\Uq$-modules de \emph{trace quantique} nulle (cf. par exemple \cite[§ 3.2, § A.1]{FGST06b}). Il s'agit d'une obstruction majeure pour étendre la construction des invariants de 3-variété de \cite{RT91} à tous les objets de $\Rep^{fd}_s$ (cf. l'introduction du chapitre I). Elle a été contournée dans \cite{CGPM14} grâce à l'introduction d'une \emph{trace modifiée} (voir aussi \cite{GPMT09}).
\end{Rem}
\chapter{Représentations projectives du groupe spécial linéaire}

La construction des invariants de Reshetikhin et Turaev s'accompagne de représentations linéaires projectives des groupes de difféotopie des surfaces (cf. \cite[§ 4.6]{RT91}). Dans le cas du \emph{tore}, le groupe de difféotopie s'identifie canoniquement au groupe spécial linéaire $\SL_2(\Z)$ (cf. par exemple [FM12, § 2]). On a alors une représentation projective de $\SL_2(\Z)$.

En parallèle, dans l'article \cite{LM94}, Lyubashenko et Majid mettent en évidence une représentation linéaire du groupe de difféotopie du tore \emph{épointé} sur toute algèbre de Hopf factorisable, tressée et enrubannée (cf. par exemple \cite[§ A.4]{FGST06b}). Pour les groupes quantiques quotients $\Uq$, associés à l'algèbre de lie $sl_2$ et aux racines $q$ de l'unité, Kerler détaille cette représentation dans \cite{Ker95} et s'intéresse plus particulièrement à la représentation induite sur le \emph{centre}. Il obtient ainsi une représentation de $\SL_2(\Z)$ sur le centre des groupes quantiques quotients de $sl_2$, qu'il conjecture être une extension non triviale de la représentation de $\SL_2(\Z)$ obtenue par la théorie des champs quantique topologique de \cite{RT91}. En 2005, Feigin, Gainutdinov, Semikhanov et Tipunin démontrent cette conjecture dans \cite{FGST06b} grâce à des outils introduits dans \cite{Lac03a}. En réalité, ces différents auteurs n'étudient pas exactement le même groupe quantique quotient de $sl_2$ (distinction nécessaire suivant la parité des racines de l'unité). Toutefois, les outils qu'ils mettent en œuvre se transposent sans difficultés majeures d'un groupe quantique quotient à l'autre.

Dans ce dernier chapitre, on propose un analogue topologique partiel de cette action de $\SL_2(\Z)$ due à \cite{LM94}. Comme dans les chapitres précédents, on concentre notre travail sur le groupe quantique restreint $\Uq$ associé à une racine $q$ primitive \emph{paire} de l'unité. On commence par des rappels sur la représentation de \cite{LM94} et sa représentation induite de $\SL_2(\Z)$ sur le centre de $\Uq$ (cf. \cite{FGST06b}). On détaille ensuite des actions remarquables sur l'espace d'écheveaux du tore solide colorié par les \emph{idempotents de Jones-Wenzl évaluables}. Ces calculs d'écheveaux permettent enfin de construire une analogie entre les éléments du centre de $\Uq$ et les classes d'écheveaux coloriés, pour laquelle une partie de l'action de $\SL_2(\Z)$ de \cite{LM94} s'interprète avec les actions de la vrille négative et du bouclage. \vspace{-0.1cm}

\minitoc

\section{Représentation modulaire sur le centre du groupe quantique restreint}
\label{section: SL2-rep}

Soit $(A,\mu,\nu,\Delta,\varepsilon,S,\mathbf{R}, \mathbf{v})$ une algèbre de Hopf  factorisable, tressée et enrubannée. La représentation linéaire projective à gauche de \cite{LM94} sur $A$ est définie par deux endomorphismes $\cS, \cT : A \rightarrow A$ \index{T8 @$\cT$}\index{S @$\cS$} tels que :
\begin{equation}
\label{Eqn: SL2-rep}
\forall x \in A \qquad \cS(x) =( id \otimes \bmu) \left( \mathbf{R}^{-1} (1 \otimes x) \mathbf{R}_{21}^{-1} \right), \quad
	\cT(x) = \mathbf{v} x ,
\end{equation}
où $\bmu$ est une intégrale à gauche de $A$ dont l'existence et l'unicité sont discutées dans \cite[§ 1]{Lyu95b}. Le lien avec le groupe de difféotopie du tore épointé se fait grâce aux \emph{relations modulaires} :
\[ (\cS \cT)^3 = \nu \cS^2, \qquad \cS^2 = S^{-1}, \]
où $\nu$ est une constante complexe (cf. \cite[Thm 1.1]{LM94} et \cite[§ 6.3]{Lyu95a}). 

On s'intéresse à cette représentation dans le cas où $A$ est le groupe quantique restreint $\Uq$ où $q = e^{\frac{i \pi}{p}}$ ($p \in \N^*$) ; pour lequel on a explicité les intégrales à gauche dans la proposition \eqref{Prop: integrales}, la $M$-matrice factorisable \eqref{Eqn: M-matrice}, et les éléments d'enrubannement possibles \eqref{Eqn: enrubannement}. D'après une étude générale des doubles quantiques dans \cite[§ 2.5]{Ker95}, il est connu que $\cT$ et $\cS$ induisent une représentation à gauche de $\SL_2(\Z)$ sur le centre $\Zf$ de $\Uq$.

\begin{Rem}
\label{Rem: SL2-rep1}
On peut montrer que $\cT$ et $\cS$ \eqref{Eqn: SL2-rep} induisent une représentation de $\SL_2(\Z)$ sur $\Zf$ avec des arguments simples, sans évoquer la construction catégorique sur le $\mathrm{coend}$ de \cite{LM94} comme dans \cite{Ker95}. Il est clair que $\cT$ est stable sur $\Zf$ puisque l'élément d'enrubannement est central. Pour l'endomorphisme $\cS$, on considère une des deux algèbres doubles tressées et enrubannées $(\bar{D}, \bar{\mathbf{R}}, \bar{\mathbf{v}}_\delta)$, $\delta \in \{0,1\}$, contenant une sous-algèbre de Hopf isomorphe à $\Uq$ (cf. la sous-section \ref{subsection: enrubannement}). On fixe une écriture de sa $R$-matrice :
\[ \bar{\mathbf{R}} = \sum_{i=1}^k a_i \otimes b_i \quad ; \quad a_i, b_i \in \bar{D}. \]
Alors, son inverse s'écrit :
\[ \bar{\mathbf{R}}^{-1} = \sum_{i=1}^k S(a_i) \otimes b_i = \sum_{i=1}^k a_i \otimes S^{-1}(b_i) \]
(cf. par exemple \cite[Thm VIII.2.4]{Kas95}). Soient $\zeta \in \C$ et $\bmu_\zeta^l$ l'intégrale à gauche correspondante de $\Uq$ (cf. la proposition \ref{Prop: integrales}). Pour tout $z \in \Zf$, on a donc :
\begin{align*}
\cS(z) &= ( id \otimes \bmu^l_\zeta) \left( \bar{\mathbf{R}}^{-1} (1 \otimes z) \bar{\mathbf{R}}_{21}^{-1} \right)
	= ( id \otimes \bmu^l_\zeta) \left( a_i b_j \otimes S^{-1}(b_i) z S(a_j) \right) \\
	&= ( id \otimes \bmu^l_\zeta) \left( a_i b_j \otimes S^{-1}(b_i) S(a_j) z \right) 
	= \bmu^l_\zeta \circ S^{-1} \left( S(z) S^2(a_j) b_i \right) a_i b_j.
\end{align*}
D'après la remarque \ref{Rem: integrales} et la proposition \ref{Prop: integrales} sur les intégrales, il existe un unique complexe $\zeta'$ tel que $\bmu^l_\zeta \circ S^{-1} = \bmu^r_{\zeta'}$. On considère le morphisme de Radford $\hbphi : \Ch^l \rightarrow \Zf$ \eqref{Eqn: Radford1} dont la bijection réciproque est :
\[ \hbphi^{-1} : \begin{cases}
	\Zf \longrightarrow \Ch^l \\
	z \longmapsto \bmu^r_{\zeta'}(S(z) ?)
	\end{cases} , \]
où le symbole $?$ désigne la place de la variable (cf. la remarque \ref{Rem: Radford}). On obtient alors :
\begin{align*}
\cS(z) &= \bmu^r_{\zeta'} \left( S(z) S^2(a_j) b_i \right) a_i b_j
	= \hbphi^{-1}(z)(S^2(a_j) b_i) a_i b_j 
	\overset{\ref{Prop: stabilisateurs}}{=} \hbphi^{-1}(z) \left( b_i a_j \right) a_i b_j \\
	&= \left( \hbphi^{-1}(z) \otimes id \right) \left( b_i a_j \otimes a_i b_j \right) 
	= \left( \hbphi^{-1}(z) \otimes id \right) \bar{\mathbf{R}}_{21} \bar{\mathbf{R}}
	=  \bchi \circ \hbphi^{-1} (z),
\end{align*}
où $\bchi : \Ch^l \rightarrow \Zf$ \eqref{Eqn: Drinfeld1} est le morphisme de Drinfeld. On en déduit que $\cS$ est stable sur $\Zf$. Enfin, on sait que $\Uq$ possède un élément de balancement généralisé (cf. le corollaire \ref{Cor: balancement}), donc $S_{\vert \Zf} = id$. Par conséquent, les deux endomorphismes $\cS_{\vert \Zf}, \cT_{\vert \Zf}$ vérifient les relations modulaires :
\[ (\cS_{\vert \Zf} \cT_{\vert \Zf})^3 = \nu \cS_{\vert \Zf}^2, \qquad \cS_{\vert \Zf}^2 = id. \]
D'où le représentation de $\SL_2(\Z)$ sur $\Zf$.
\end{Rem}

On commence par décrire une partie de l'action de $\cT$ et $\cS$ sur la base canonique du centre $\{ e_s \; ; \; 0 \leq s \leq p \} \cup \{ w^\pm_s \; ; \; 1  \leq s \leq p-1 \}$ (cf. la proposition \ref{Prop: centre}).

\begin{Prop} 
\label{Prop: SL2-rep}
Soit $\delta \in \{0,1\}$. On munit $\Uq$ de l'élément d'enrubannement $\mathbf{v}_\delta$ \eqref{Eqn: enrubannement} et de la représentation de \cite{LM94} \eqref{Eqn: SL2-rep}.
\begin{enumerate}[(i)]
	\item L'action de $\cT$ sur le centre $\Zf$ de $\Uq$ est donnée par :
	\begin{align*}
		\cT(e_s) &= \begin{multlined}[t]
			(-1)^{\delta(s-1)} q^{-\frac{s^2-1}{2}} e_s \\
			+ \delta_{1 \leq s \leq p-1} (-1)^{\delta(s-1)}  (q-q^{-1}) q^{-\frac{s^2-1}{2}} \left( \frac{p-s}{[s]} w_s^+ - \frac{s}{[s]} w_s^- \right), \quad
			1 \leq s \leq p,  
			\end{multlined} \\
		\cT(w_s^\pm) &= (-1)^{\delta(s-1)} q^{-\frac{s^2-1}{2}} w_s^\pm, \quad 1 \leq s \leq p-1.
	\end{align*}
	\item L'action de $\cS$ sur le centre $\Zf$ de $\Uq$ vérifie :
	\begin{align*}
		\cS(e_0) &= \frac{(-1)^{p-\delta}}{\zeta 4p ([p-1]!)^2} \sum_{j=0}^p U_{2p}((-1)^\delta \widehat{\beta}_j) e_j \\
	& \qquad + \frac{(-1)^{p}}{\zeta 4p ([p-1]!)^2} (q-q^{-1})^2 \sum_{s=1}^{p-1} U_{2p}'((-1)^\delta \widehat{\beta}_j) \left( w^+_j + w^-_j \right), \\
		\cS(e_p) &= \frac{(-1)^{(\delta-1)(p-1)}}{\zeta 2p ([p-1]!)^2} \sum_{j=0}^p U_p((-1)^\delta\widehat{\beta}_j) e_j \\
		& \qquad + \frac{(-1)^{(\delta-1)p+1}}{\zeta 2p ([p-1]!)^2} (q-q^{-1})^2 \sum_{j=1}^{p-1} U_p'((-1)^\delta\widehat{\beta}_s) \left( w^+_s + w^-_s \right) \\
		\cS(w_s^+) &= \begin{multlined}[t]
			\frac{(-1)^{p+\delta(s-1)}}{\zeta 2p ([p-1]!)^2} [s]^2 \sum_{j=0}^p U_s((-1)^\delta\widehat{\beta}_j) e_j \\
			+ \frac{(-1)^{p+\delta s}}{\zeta 2p ([p-1]!)^2} [s]^2 (q-q^{-1})^2 \sum_{j=1}^{p-1} U_s'((-1)^\delta\widehat{\beta}_j) \left( w^+_j + w^-_j \right), \\
			1 \leq s \leq p-1, 
			\end{multlined} \\
		\cS(w_s^-) &= \begin{multlined}[t]
			\frac{(-1)^{\delta(p-s-1)}}{\zeta 4p ([p-1]!)^2}[s]^2 \sum_{j=0}^p (U_{2p-s}-U_{s})((-1)^\delta \widehat{\beta}_j) e_j \\
			+ \frac{(-1)^{\delta(p-s)}}{\zeta 4p ([p-1]!)^2}[s]^2 (q-q^{-1})^2 \sum_{s=1}^{p-1}(U_{2p-s}'-U_{s}')((-1)^\delta \widehat{\beta}_j) \left( w^+_j + w^-_j \right), \\
			1 \leq s \leq p-1,
			\end{multlined}
	\end{align*}
	où $\zeta \in \C$, $(U_s(x))_{s \in \mathbb{N}}$ désigne les polynômes de Chebychev et, pour tout $s \in \mathbb{N}$, $\widehat{\beta}_s := q^s+q^{-s}$.
\end{enumerate}
\end{Prop}

\begin{proof}
\begin{enumerate}[(i)]
	\item Par construction, l'action de $\cT$ est donnée par la multiplication par l'élément d'enrubannement $\mathbf{v}_\delta$ \eqref{Eqn: SL2-rep}. Or, d'après la proposition \ref{Prop: enrubannement}, on a :
	\begin{gather*}
	\mathbf{v}_{\delta} =
		\sum_{s=0}^p (-1)^{\delta(s-1)} q^{-\frac{s^2-1}{2}} e_s 
		+ (q-q^{-1}) \sum_{s=1}^{p-1} (-1)^{\delta(s-1)} q^{-\frac{s^2-1}{2}} \left( \frac{p-s}{[s]} w^+_s - \frac{s}{[s]} w^-_s \right) .
	\end{gather*}
	D'où les expressions de $\cT(e_s)$, $0 \leq s \leq p$, et $\cT(w^\pm_s)$, $1  \leq s \leq p-1$.
	\item D'après la remarque \ref{Rem: SL2-rep1}, l'action de $\cS$ sur $\Zf$ est donnée par :
	\begin{equation} \tag{1}
	\forall z \in \Zf \qquad \cS(z) = \bchi \circ \hbphi^{-1}(z),
	\end{equation}
	où $\hbphi : \Ch^l \rightarrow \Zf$ \eqref{Eqn: Radford1} est le morphisme de Radford et $\bchi : \Ch^l \rightarrow \Zf$ \eqref{Eqn: Drinfeld1} celui de Drinfeld. On considère les morphismes $\hbphi^\delta$ \eqref{Eqn: Radford2} et $\bchi^\delta$ \eqref{Eqn: Drinfeld2} construits à partir de $\hbphi$ et $\bchi$ respectivement. Comme dans les sous-sections \ref{subsection: Radford} et \ref{subsection: Drinfeld}, on note :
	\begin{gather*}
	\forall \alpha \in \{+,-\} \quad \forall s \in \{1,...,p\} \qquad 
	\widehat{\phi}^\alpha_\delta(s) := \hbphi^\delta \left( [\X^\alpha(s)] \right), \quad
	\chi^\alpha_\delta(s) := \bchi^\delta \left( [\X^\alpha(s)] \right),
	\end{gather*}
	où, pour tout $\Uq$-module à gauche $\X$, $[\X]$ désigne la classe d'équivalence de $\X$ dans le groupe de Grothendieck $\G$ de $\Uq$. Alors, d'après la proposition \ref{Prop: Radford}, on sait que :
	\begin{align*}
	& w^+_s = \frac{(-1)^{p+\delta(s-1)}}{\zeta 2p ([p-1]!)^2} [s]^2 \widehat{\phi}_\delta^+(s), && 1 \leq s \leq p-1, \\
	& w^-_{p-s} = \frac{(-1)^{\delta(p-s-1)}}{\zeta 2p ([p-1]!)^2}[s]^2 \widehat{\phi}_\delta^-(s), && 1 \leq s \leq p-1, \\
	& e_p = \frac{(-1)^{(\delta-1)(p-1)}}{\zeta 2p ([p-1]!)^2} \widehat{\phi}_\delta^+(p), \\
	& e_0 = \frac{(-1)^{p-\delta}}{\zeta 2p ([p-1]!)^2} \widehat{\phi}_\delta^-(p),
	\end{align*}
	où $\zeta \in \C$ dépend du choix de la co-intégrale bilatère $\mathbf{c}_\zeta$ \eqref{Prop: integrales} dans la construction du morphisme de Radford $\hbphi$. On en déduit les expressions :
	\begin{align*}
	\cS(w^+_s) &= \frac{(-1)^{p+\delta(s-1)}}{\zeta 2p ([p-1]!)^2} [s]^2 \cS(\widehat{\phi}_\delta^+(s))
		\overset{(1)}{=} \frac{(-1)^{p+\delta(s-1)}}{\zeta 2p ([p-1]!)^2} [s]^2 \chi^+_\delta(s), 
		\quad 1 \leq s \leq p-1, \\
	\cS(w^-_{p-s}) &= \frac{(-1)^{\delta(p-s-1)}}{\zeta 2p ([p-1]!)^2}[s]^2 \cS(\widehat{\phi}_\delta^-(s))
		\overset{(1)}{=} \frac{(-1)^{\delta(p-s-1)}}{\zeta 2p ([p-1]!)^2}[s]^2 \chi_\delta^-(s)), 
		\quad 1 \leq s \leq p-1, \\
	\cS(e_p) &= \frac{(-1)^{(\delta-1)(p-1)}}{\zeta 2p ([p-1]!)^2} \cS(\widehat{\phi}_\delta^+(p)) \overset{(1)}{=} \frac{(-1)^{(\delta-1)(p-1)}}{\zeta 2p ([p-1]!)^2} \chi_\delta^+(p), \\
	\cS(e_0) &= \frac{(-1)^{p-\delta}}{\zeta 2p ([p-1]!)^2} \cS(\widehat{\phi}_\delta^-(p)) \overset{(1)}{=} \frac{(-1)^{p-\delta}}{\zeta 2p ([p-1]!)^2} \chi_\delta^-(p).
	\end{align*}
	Or, d'après la proposition \ref{Prop: Drinfeld}, on sait aussi que pour tout $s \in \{1,...,p\}$ :
	\begin{align*}
	\chi_\delta^+(s) &=
		\sum_{j=0}^p U_s((-1)^\delta\widehat{\beta}_j) e_j
		+ (-1)^\delta (q-q^{-1})^2 \sum_{j=1}^{p-1} U_s'((-1)^\delta\widehat{\beta}_j) \left( w^+_j + w^-_j \right), \\
	\chi_\delta^-(s) &= \begin{multlined}[t]
		\frac{1}{2} \sum_{j=0}^p (U_{p+s}-U_{p-s})((-1)^\delta \widehat{\beta}_j) e_j \\
		+ (-1)^\delta \frac{(q-q^{-1})^2}{2} \sum_{s=1}^{p-1}(U_{p+s}'-U_{p-s}')((-1)^\delta \widehat{\beta}_j) \left( w^+_j + w^-_j \right),
		\end{multlined}
	\end{align*}
	D'où les expressions de $\cS(e_0)$, $\cS(e_p)$, et $\cS(w^\pm_s)$, $1  \leq s \leq p-1$.
\end{enumerate}
\end{proof}

\begin{Rem}
\label{Rem: SL2-rep2}
Dans la proposition \ref{Prop: SL2-rep}, les expressions de $\cS(e_s)$, $1 \leq s \leq p-1$, ne sont pas données car elles sont très compliquées et peu exploitables. Pour les obtenir, on exprime les vecteurs $e_s$, $1 \leq s \leq p-1$, comme des polynômes en $(-1)^\delta (q-q^{-1})^2 C$ conformément à la remarque \ref{Rem: poly en C}. On décompose ensuite ces polynômes dans la famille libre $\left\{ \chi^\pm_\delta(s') \; ; \; 1 \leq s' \leq p \right\}$ (cf. la proposition \ref{Prop: Drinfeld}). D'après la remarque \ref{Rem: SL2-rep1}, on a :
\[ \forall s \in \{1,...,p\} \qquad \cS( \widehat{\phi}_\delta^\pm(s) ) = \chi_\delta^\pm(s) \quad \Longrightarrow \quad \cS( \chi_\delta^\pm(s) ) = \widehat{\phi}_\delta^\pm(s) . \]
On obtient ainsi les expressions de $\cS(e_s)$, $1 \leq s \leq p-1$, dans $\left\{ \widehat{\phi}^\pm_\delta(s') \; ; \; 1 \leq s' \leq p \right\}$ dont on connaît le lien avec la base canonique du centre (cf. la proposition \ref{Prop: Radford}).
\end{Rem}

En outre, en choisissant une autre base du centre construite à partir des sous-familles libres $\left\{ \widehat{\phi}^\pm_\delta(s) \; ; \; 1 \leq s \leq p \right\}$ et $\left\{ \chi^\pm_\delta(s) \; ; \; 1 \leq s \leq p \right\}$ (cf. les sous-sections \ref{subsection: Radford} et \ref{subsection: Drinfeld}), on dégage des réductions de Jordan remarquables pour les actions de $\cT$ et $\cS$. Plus précisément, ces réductions fournissent une décomposition de la $\SL_2(\Z)$-représentation à gauche sur $\Zf$ comme une extension non triviale de la $\SL_2(\Z)$-représentation obtenue par la théorie des champs quantique topologique (TQFT) \index{TQFT@ TQFT} de \cite{RT91} (cf. par exemple \cite[§ II.3.9, § IV.5.4]{Tur94}).

\begin{Thm}[{\cite[Thm 5.2]{FGST06b}}]
\label{Thm: SL2-rep}
La $SL_2(\mathbb{Z})$-représentation à gauche sur le centre $\Zf$ de $\Uq$ de \cite{LM94} se décompose en la somme directe $\mathfrak{Z} \cong \PIM_{p+1} \oplus \left( \C^2 \otimes \V_{p-1} \right)$ où :
\begin{enumerate}[(i)]
	\item le $\SL_2(\Z)$-module $\PIM_{p+1}$ est simple de $\C$-dimension $p+1$,
	\item le $\SL_2(\Z)$-module $\V_{p-1}$ est semi-simple de $\C$-dimension $p-1$ et isomorphe au $\SL_2(\Z)$-module obtenu par la TQFT de \cite{RT91},
	\item et $\C^2$ désigne la $\SL_2(\Z)$-représentation standard définie par les matrices :
	\[ T = \begin{pmatrix} 1 & 1 \\ 0 & 1 \end{pmatrix}, \qquad S = \begin{pmatrix} 0 & -1 \\ 0 & 1 \end{pmatrix} . \]
\end{enumerate}
\end{Thm}

\section{Actions remarquables sur l'espace d'écheveaux du tore solide}
\label{section: actions topo}

Compte-tenu du théorème \ref{Thm: SL2-rep} de décomposition de la $\SL_2(\Z)$-représentation de \cite{LM94} sur le centre $\Zf$ de $\Uq$, on cherche un analogue topologique.
Pour cela, on considère l'espace d'écheveaux $K_A(\bar{D} \times \mathbb{S}^1, 0)$ du tore solide et on étudie les actions de la vrille et du bouclage. Des résultats préliminaires sur les actions de l'enlacement seront nécessaires. On obtiendra des généralisations de résultats de Lickorish donnés dans \cite[§ 5]{Lic92}.

\subsection{Actions de l'enlacement}
\label{subsection: enlacement}

On commence par étudier les actions de l'\emph{enlacement positif} sur l'espace d'écheveaux $K_A(\bar{D} \times \mathbb{S}^1, 0)$ du tore solide. Elles sont données par les applications linéaires :
\[ \forall i \in \{1,...,n-1\} \qquad \mathbf{E}_i :
	\begin{cases}
	K_A(\bar{D} \times \mathbb{S}^1) \longrightarrow K_A(\bar{D} \times \mathbb{S}^1) \\
	\insertion{a}{n}{Tore} \longmapsto \enlac{a}{n-i \;}{Tore} .
	\end{cases} \] 
L'enlacement négatif étant l'image miroir du positive, ses actions se déduisent de celle de l'enlacement positif en remplaçant $A$ par $A^{-1}$ (cf. les relations d'écheveaux \eqref{Eqn: echeveaux2}). 

Comme $\left\{ \Theta_{A^2,n}(f_n) \; ; \; n \in \N^* \right\}$ est une $\Z[A,A^{-1}]$-base de $K_A(\bar{D} \times \mathbb{S}^1, 0)$ (cf. le corollaire \ref{Cor: insertion des JWG}), pour tout $i \in \{1,...,n-1\}$, on calcule l'image de $\mathbf{E}_i$ sur ces classes d'écheveaux coloriés par les idempotents de Jones-Wenzl évaluables. Pour cela, il suffit de calculer les classes d'écheveaux des coupons : \index{E@ $\mathbf{E}_{i,n-i}$}
\begin{equation}
\label{Eqn: enlacement}
\mathbf{E}_{i,n-i} := \begin{tikzbox}
	\draw (0.5,0) -- (1,0) node[above] {$\scriptstyle i$} -- (2,0) node[above] {$\scriptstyle n-i$} -- (2.5,0);
	\braid s_1^{-1} s_1^{-1} ;
	\draw (0.5,-2.5) -- (1,-2.5) node[below] {$\scriptstyle i$} -- (2,-2.5) node[below] {$\scriptstyle n-i$} -- (2.5,-2.5);
\end{tikzbox} ; \quad n \in \N^*, \quad i \in \{1,...,n-1\},
\end{equation}
où la barre horizontale supérieure (resp. inférieure) désigne l'insertion de l'idempotent de Jones-Wenzl évaluable dont l'indice correspond au nombre total de brin(s) sortant(s) (resp. entrant(s)). On commence par donner trois lemmes techniques.

\begin{Lemme}
\label{Lemme: enlacement1}
Soient $n \in \N^*$, $l \in \N$ le quotient de $n$ dans sa division euclidienne par $p$, et $i \in \{1,...,n-1\}$.
\begin{enumerate}[(i)]
	\item Si $n \leq p-1$ ou $n = -1 \mod p$, alors :
	\[ \mathbf{E}_{i,n-i} = \begin{tikzbox}
		\draw (0.5,0) -- (1,0) node[above] {$\scriptstyle i$} -- (2,0) node[above] {$\scriptstyle n-i$} -- (2.5,0);
		\braid s_1^{-1} s_1^{-1} ;
		\draw (0.5,-2.5) -- (1,-2.5) node[below] {$\scriptstyle i$} -- (2,-2.5) node[below] {$\scriptstyle n-i$} -- (2.5,-2.5);
	\end{tikzbox}
		= A^{2(n-i)i} \begin{tikzbox}
	\draw (0.5,0) -- (1,0) node[above] {$\scriptstyle n$} -- (1.5,0) ;
	\draw (1,0) -- (1,-2.5) ;
	\draw (0.5,-2.5) -- (1,-2.5) node[below] {$\scriptstyle n$} -- (1.5,-2.5) ;
\end{tikzbox} . \]
	\item Si $n \geq p$ et $i \leq lp-1 < n$, alors :
	\[ \mathbf{E}_{i,n-i} = \begin{tikzbox}
		\draw (0.5,0) -- (1,0) node[above] {$\scriptstyle i$} -- (2,0) node[above] {$\scriptstyle n-i$} -- (2.5,0);
		\braid s_1^{-1} s_1^{-1} ;
		\draw (0.5,-2.5) -- (1,-2.5) node[below] {$\scriptstyle i$} -- (2,-2.5) node[below] {$\scriptstyle n-i$} -- (2.5,-2.5);
	\end{tikzbox}
		= A^{2(lp-1-i)i} \begin{tikzbox}
		\draw (0.5,0) -- (1,0) node[above] {$\scriptstyle lp-1-i \qquad$} -- (1.5,0) node[above] {$\scriptstyle i$} -- (2.5,0) node[above] {$\scriptstyle n-lp+1$} -- (3,0);
		\draw (1,0) -- (1,-2.5);
		\braid at (1.5,0) s_1^{-1} s_1^{-1} ;
		\draw (0.5,-2.5) -- (1,-2.5) node[below] {$\scriptstyle lp-1-i \qquad$} -- (1.5,-2.5) node[below] {$\scriptstyle i$} -- (2.5,-2.5) node[below] {$\scriptstyle n-lp+1$} -- (3,-2.5);
	\end{tikzbox} . \]
	\item Si $n \geq p$ et $lp-1 \leq i < n$, alors :
	\[ \mathbf{E}_{i,n-i} = \begin{tikzbox}
		\draw (0.5,0) -- (1,0) node[above] {$\scriptstyle i$} -- (2,0) node[above] {$\scriptstyle n-i$} -- (2.5,0);
		\braid s_1^{-1} s_1^{-1} ;
		\draw (0.5,-2.5) -- (1,-2.5) node[below] {$\scriptstyle i$} -- (2,-2.5) node[below] {$\scriptstyle n-i$} -- (2.5,-2.5);
	\end{tikzbox}
		= A^{2(n-i)(i-lp+1)} \begin{tikzbox}
		\draw (0.5,0) -- (1,0) node[above] {$\scriptstyle lp-1$} -- (2,0) node[above] {$\scriptstyle n-i$} -- (2.5,0) node[above] {$\scriptstyle \qquad i-lp+1$} -- (3,0);
		\draw (2.5,0) -- (2.5,-2.5);
		\braid s_1^{-1} s_1^{-1} ;
		\draw (0.5,-2.5) -- (1,-2.5) node[below] {$\scriptstyle lp-1$} -- (2,-2.5) node[below] {$\scriptstyle n-i$} -- (2.5,-2.5) node[below] {$\scriptstyle \qquad i-lp+1$} -- (3,-2.5);
	\end{tikzbox} . \]
\end{enumerate}
\end{Lemme}

\begin{proof}
\begin{enumerate}[(i)]
	\item D'après la proposition \ref{Prop: JWG}, pour tout $i \in \{1,...,n-1\}$, on a $h_i f_n = 0 = f_n h_i$. Aussi, on dénoue chacun des $i(n-i)$ croisements supérieurs et inférieurs avec la composante $D_0$ des relations d'écheveaux \eqref{Eqn: echeveaux2} :
	\[ \mathbf{E}_{i,n-i} = \begin{tikzbox}
		\draw (0.5,0) -- (1,0) node[above] {$\scriptstyle i$} -- (2,0) node[above] {$\scriptstyle n-i$} -- (2.5,0) ;
		\braid s_1^{-1} s_1^{-1} ;
		\draw[green,->] (1.9,-0.6) arc (20:-20:0.5) ;
		\draw[green,->] (1.1,-0.9) arc (200:160:0.5) ;
		\draw[green,->] (1.9,-1.6) arc (20:-20:0.5) ;
		\draw[green,->] (1.1,-1.9) arc (200:160:0.5) ;
		\draw (0.5,-2.5) -- (1,-2.5) node[below] {$\scriptstyle i$} -- (2,-2.5) node[below] {$\scriptstyle n-i$} -- (2.5,-2.5);
	\end{tikzbox}
		= A^{2(n-i)i} \begin{tikzbox}
		\draw (0.5,0) -- (1,0) node[above] {$\scriptstyle n-i$} -- (2,0) node[above] {$\scriptstyle i$} -- (2.5,0) ;
		\draw (1,0) -- (1,-2.5) ;
		\draw (2,0) -- (2,-2.5) ;
		\draw (0.5,-2.5) -- (1,-2.5) node[below] {$\scriptstyle n-i$} -- (2,-2.5) node[below] {$\scriptstyle i$} -- (2.5,-2.5) ;
	\end{tikzbox}
			= A^{2(n-i)i} \begin{tikzbox}
		\draw (0.5,0) -- (1,0) node[above] {$\scriptstyle n$} -- (1.5,0) ;
		\draw (1,0) -- (1,-2.5) ;
		\draw (0.5,-2.5) -- (1,-2.5) node[below] {$\scriptstyle n$} -- (1.5,-2.5) ;
	\end{tikzbox} . \]
	\item D'après la proposition \ref{Prop: JWG}, pour tout $i \in \{1,...,lp-2\}$, on a $h_i f_n = 0 = f_n h_i$. On procède comme en $(i)$ sur les $i(lp-1-i)$ croisements de gauches dans les $i(n-i)$ croisements supérieurs et inférieurs :
	\begin{align*}
	\mathbf{E}_{i,n-i} &= \begin{tikzbox}
		\draw (0.5,0) -- (1,0) node[above] {$\scriptstyle i$} -- (2,0) node[above] {$\scriptstyle n-i$} -- (2.5,0);
		\braid s_1^{-1} s_1^{-1} ;
		\draw (0.5,-2.5) -- (1,-2.5) node[below] {$\scriptstyle i$} -- (2,-2.5) node[below] {$\scriptstyle n-i$} -- (2.5,-2.5);
	\end{tikzbox}
		= \begin{tikzbox}
		\draw (0.5,0) -- (1,0) node[above] {$\scriptstyle i$} -- (2,0) node[above] {$\scriptstyle lp-1-i$} -- (3,0) node[above] {$\scriptstyle \qquad n-lp+1$} -- (3.5,0);
		\braid[height=0.5cm] s_1^{-1} s_2^{-1} s_2^{-1} s_1^{-1} ;
		\draw[green,->] (1.9,-0.4) arc (10:-10:0.5) ;
		\draw[green,->] (1.1,-0.6) arc (190:170:0.5) ;
		\draw[green,->] (1.9,-1.9) arc (10:-10:0.5) ;
		\draw[green,->] (1.1,-2.1) arc (190:170:0.5) ;
		\draw (0.5,-2.5) -- (1,-2.5) node[below] {$\scriptstyle i$} -- (2,-2.5) node[below] {$\scriptstyle lp-1-i$} -- (3,-2.5) node[below] {$\scriptstyle \qquad n-lp+1$} -- (3.5,-2.5);
	\end{tikzbox}
		= A^{2(lp-1-i)i} \begin{tikzbox}
		\draw (0.5,0) -- (1,0) node[above] {$\scriptstyle lp-1-i$} -- (2,0) node[above] {$\scriptstyle i$} -- (3,0) node[above] {$\scriptstyle n-lp+1$} -- (3.5,0);
		\draw (1,0) -- (1,-2.5);
		\braid s_2^{-1} s_2^{-1} ;
		\draw (0.5,-2.5) -- (1,-2.5) node[below] {$\scriptstyle lp-1-i$} -- (2,-2.5) node[below] {$\scriptstyle i$} -- (3,-2.5) node[below] {$\scriptstyle n-lp+1$} -- (3.5,-2.5);
	\end{tikzbox} .
	\end{align*}
	\item D'après la proposition \ref{Prop: JWG}, pour tout $i \in \{lp,...,n-1\}$, on a $h_i f_n = 0 = f_n h_i$. On procède comme en $(i)$ sur les $(i-lp+1)(n-i)$ croisements de droite dans les $i(n-i)$ croisements supérieurs et inférieurs :
	\begin{align*}
	\mathbf{E}_{i,n-i} &= \begin{tikzbox}
		\draw (0.5,0) -- (1,0) node[above] {$\scriptstyle i$} -- (2,0) node[above] {$\scriptstyle n-i$} -- (2.5,0);
		\braid s_1^{-1} s_1^{-1} ;
		\draw (0.5,-2.5) -- (1,-2.5) node[below] {$\scriptstyle i$} -- (2,-2.5) node[below] {$\scriptstyle n-i$} -- (2.5,-2.5);
	\end{tikzbox}
		= \begin{tikzbox}
		\draw (0.5,0) -- (1,0) node[above] {$\scriptstyle lp-1$} -- (2,0) node[above] {$\scriptstyle i-lp+1$} -- (3,0) node[above] {$\scriptstyle n-i$} -- (3.5,0);
		\braid[height=0.5cm] s_2^{-1} s_1^{-1} s_1^{-1} s_2^{-1} ;
		\draw[green,->] (2.9,-0.4) arc (10:-10:0.5) ;
		\draw[green,->] (2.1,-0.6) arc (190:170:0.5) ;
		\draw[green,->] (2.9,-1.9) arc (10:-10:0.5) ;
		\draw[green,->] (2.1,-2.1) arc (190:170:0.5) ;
		\draw (0.5,-2.5) -- (1,-2.5) node[below] {$\scriptstyle lp-1$} -- (2,-2.5) node[below] {$\scriptstyle i-lp+1$} -- (3,-2.5) node[below] {$\scriptstyle n-i$} -- (3.5,-2.5);
	\end{tikzbox}
		= A^{2(n-i)(i-lp+1)} \begin{tikzbox}
		\draw (0.5,0) -- (1,0) node[above] {$\scriptstyle lp-1$} -- (2,0) node[above] {$\scriptstyle n-i$} -- (3,0) node[above] {$\scriptstyle i-lp+1$} -- (3.5,0);
		\draw (3,0) -- (3,-2.5);
		\braid s_1^{-1} s_1^{-1} ;
		\draw (0.5,-2.5) -- (1,-2.5) node[below] {$\scriptstyle lp-1$} -- (2,-2.5) node[below] {$\scriptstyle n-i$} -- (3,-2.5) node[below] {$\scriptstyle i-lp+1$} -- (3.5,-2.5);
	\end{tikzbox} .
	\end{align*}
\end{enumerate}
\end{proof}

\begin{Lemme}
\label{Lemme: enlacement2}
Soient $n \geq p$ tel que $n \neq -1 \mod p$, et $l \in \N^*$ le quotient de $n$ dans sa division euclidienne par $p$. On se donne $g,c,r \in \N$ et $d,e \in \N^*$ tels que $g+e+2c+d+r = n$ et $g+e = lp-1$. Alors :
\begin{gather*}
\mathbf{S}_{g,e,c,d,r} := \begin{tikzbox}
	\draw (0.5,0) -- (1,0) node[above] {$\scriptstyle g$} -- (1.5,0) node[above] {$\scriptstyle e$} -- (2,0) node[above] {$\scriptstyle c$} -- (3.5,0) node[above] {$\scriptstyle d$} -- (4,0) node[above] {$\scriptstyle r$} -- (4.5,0);
	\draw (1,0) -- (1,-2.5);
	\draw (4,0) -- (4,-2.5);
	\draw (2,0) to[bend right=90] (3,0);
	\draw (2,-2.5) to[bend left=90] (3,-2.5);
	\braid[width=2cm] at (1.5,0) s_1^{-1} s_1^{-1} ;
	\draw (0.5,-2.5) -- (1,-2.5) node[below] {$\scriptstyle g$} -- (1.5,-2.5) node[below] {$\scriptstyle e$} -- (2,-2.5) node[below] {$\scriptstyle c$} -- (3.5,-2.5) node[below] {$\scriptstyle d$} -- (4,-2.5) node[below] {$\scriptstyle r$} -- (4.5,-2.5);
\end{tikzbox}
	= A^{2d} \mathbf{S}_{g+1,e-1,c,d,r} + A^{-2e} \left( A^{2d} - A^{-2d} \right) \mathbf{S}_{g,e-1,c+1,d-1,r} .
\end{gather*}
\end{Lemme}

\begin{proof}
On commence par traiter le cas où $g=c=r=0$. On isole deux brins centraux sur lesquels on utilise les relations d'isotopie et d'écheveaux \eqref{Eqn: echeveaux2}:
\begin{align*}
\begin{tikzbox} 
	\draw (0.5,0) -- (1,0) node[above] {$\scriptstyle e$} -- (2,0) node[above] {$\scriptstyle d$} -- (2.5,0);
	\draw[red] (1,0) -- (1,-1.5);
	\draw[blue] (2,0) -- (2,-1.5);
	\braid[style strands={1}{red}, style strands={2}{blue}] at (1,-1.5) s_1^{-1} s_1^{-1} ;
	\draw[red] (1,-4) -- (1,-5.5);
	\draw[blue] (2,-4) -- (2,-5.5);
	\draw (0.5,-5.5) -- (1,-5.5) node[below] {$\scriptstyle e$} -- (2,-5.5) node[below] {$\scriptstyle d$} -- (2.5,-5.5);
\end{tikzbox}
	&= \begin{tikzbox} 
	\draw (0.5,0) -- (1,0) node[above] {$\scriptstyle e-1$} -- (2,0) node[above] {$\scriptstyle 1$} -- (3,0) node[above] {$\scriptstyle 1$}  -- (4,0) node[above] {$\scriptstyle d-1$} -- (4.5,0);
	\braid[height=0.833cm, style strands={1}{red}, style strands={4}{blue}] s_2^{-1} s_1^{-1}-s_3^{-1} s_2^{-1} s_2^{-1} s_1^{-1}-s_3^{-1} s_2^{-1} ;
	\draw (0.5,-5.5) -- (1,-5.5) node[below] {$\scriptstyle e-1$} -- (2,-5.5) node[below] {$\scriptstyle 1$} -- (3,-5.5) node[below] {$\scriptstyle 1$} -- (4,-5.5) node[below] {$\scriptstyle d-1$} -- (4.5,-5.5);
\end{tikzbox}
	= \begin{tikzbox} 
	\draw (0.5,0) -- (1,0) node[above] {$\scriptstyle e-1$} -- (2,0) node[above] {$\scriptstyle 1$} -- (3,0) node[above] {$\scriptstyle 1$}  -- (4,0) node[above] {$\scriptstyle d-1$} -- (4.5,0);
	\braid[height=0.5cm, style strands={1}{red}, style strands={4}{blue}] s_2^{-1} s_3^{-1} s_3^{-1} s_2^{-1} s_1^{-1} s_2^{-1} s_3^{-1} s_3^{-1} s_2^{-1} s_1 ;
	\draw[green] (2.5,-0.5) circle (0.3) ;
	\draw (0.5,-5.5) -- (1,-5.5) node[below] {$\scriptstyle e-1$} -- (2,-5.5) node[below] {$\scriptstyle 1$} -- (3,-5.5) node[below] {$\scriptstyle 1$} -- (4,-5.5) node[below] {$\scriptstyle d-1$} -- (4.5,-5.5);
\end{tikzbox} \\
	&= A \begin{tikzbox} 
	\draw (0.5,0) -- (1,0) node[above] {$\scriptstyle e-1$} -- (2,0) node[above] {$\scriptstyle 1$} -- (3,0) node[above] {$\scriptstyle 1$}  -- (4,0) node[above] {$\scriptstyle d-1$} -- (4.5,0);
	\draw[red] (1,0) -- (1,-0.5) ;
	\draw (2,0) -- (2,-0.5) ;
	\draw (3,0) -- (3,-0.5) ;
	\draw[blue] (4,0) -- (4,-0.5) ;
	\braid[height=0.5cm, style strands={1}{red}, style strands={4}{blue}] at (1,-0.5) s_3^{-1} s_3^{-1} s_2^{-1} s_1^{-1} s_2^{-1} s_3^{-1} s_3^{-1} s_2^{-1} s_1 ;
	\draw[green] (2.5,-2) circle (0.3) ;
	\draw (0.5,-5.5) -- (1,-5.5) node[below] {$\scriptstyle e-1$} -- (2,-5.5) node[below] {$\scriptstyle 1$} -- (3,-5.5) node[below] {$\scriptstyle 1$} -- (4,-5.5) node[below] {$\scriptstyle d-1$} -- (4.5,-5.5);
\end{tikzbox}
	+ A^{-1} \begin{tikzbox}
	\draw (0.5,0) -- (1,0) node[above] {$\scriptstyle e-1$} -- (2,0) node[above] {$\scriptstyle 1$} -- (3,0) node[above] {$\scriptstyle 1$}  -- (4,0) node[above] {$\scriptstyle d-1$} -- (4.5,0);
	\draw[red] (1,0) -- (1,-0.7);
	\draw[blue] (4,0) -- (4,-0.7);
	\draw (2,0) to[bend right=90] (3,0);
	\draw (2,-0.7) to[bend left=90] (3,-0.7);
	\braid[border height=0.15cm, height=0.5cm, style strands={1}{red}, style strands={4}{blue}] at (1,-0.7) s_3^{-1} s_3^{-1} s_2^{-1} s_1^{-1} s_2^{-1} s_3^{-1} s_3^{-1} s_2^{-1} s_1 ;
	\draw (0.5,-5.5) -- (1,-5.5) node[below] {$\scriptstyle e-1$} -- (2,-5.5) node[below] {$\scriptstyle 1$} -- (3,-5.5) node[below] {$\scriptstyle 1$} -- (4,-5.5) node[below] {$\scriptstyle d-1$} -- (4.5,-5.5);
\end{tikzbox} \\
	&= A^2 \begin{tikzbox} 
	\draw (0.5,0) -- (1,0) node[above] {$\scriptstyle e-1$} -- (2,0) node[above] {$\scriptstyle 1$} -- (3,0) node[above] {$\scriptstyle 1$}  -- (4,0) node[above] {$\scriptstyle d-1$} -- (4.5,0);
	\draw[red] (1,-1.3) -- (1,-2) ;
	\draw (2,-1.3) -- (2,-2);
	\draw (3,-1.3) -- (3,-2);
	\draw[blue] (4,-1.3) -- (4,-2) ;
	\braid[border height=0.15cm, height=0.5cm, style strands={1}{red}, style strands={4}{blue}] at (1,0) s_3^{-1} s_3^{-1} ;
	\braid[height=0.5cm, style strands={1}{red}, style strands={4}{blue}] at (1,-2) s_1^{-1} s_2^{-1} s_3^{-1} s_3^{-1} s_2^{-1} s_1 ;
	\draw (0.5,-5.5) -- (1,-5.5) node[below] {$\scriptstyle e-1$} -- (2,-5.5) node[below] {$\scriptstyle 1$} -- (3,-5.5) node[below] {$\scriptstyle 1$} -- (4,-5.5) node[below] {$\scriptstyle d-1$} -- (4.5,-5.5);
\end{tikzbox}
	+ \begin{tikzbox}
	\draw (0.5,0) -- (1,0) node[above] {$\scriptstyle e-1$} -- (2,0) node[above] {$\scriptstyle 1$} -- (3,0) node[above] {$\scriptstyle 1$}  -- (4,0) node[above] {$\scriptstyle d-1$} -- (4.5,0);
	\draw[red] (1,-1.3) -- (1,-2) ;
	\draw[blue] (4,-1.3) -- (4,-2) ;
	\draw (2,-1.3) to[bend right=90] (3,-1.3);
	\draw (2,-2) to[bend left=90] (3,-2);
	\braid[border height=0.15cm, height=0.5cm, style strands={1}{red}, style strands={4}{blue}] at (1,0) s_3^{-1} s_3^{-1} ;
	\braid[height=0.5cm, style strands={1}{red}, style strands={4}{blue}] at (1,-2) s_1^{-1} s_2^{-1} s_3^{-1} s_3^{-1} s_2^{-1} s_1 ;
	\draw (0.5,-5.5) -- (1,-5.5) node[below] {$\scriptstyle e-1$} -- (2,-5.5) node[below] {$\scriptstyle 1$} -- (3,-5.5) node[below] {$\scriptstyle 1$} -- (4,-5.5) node[below] {$\scriptstyle d-1$} -- (4.5,-5.5);
\end{tikzbox} \\
	&\qquad + A^{-1} \begin{tikzbox}
	\draw (0.5,0) -- (1,0) node[above] {$\scriptstyle e-1$} -- (2,0) node[above] {$\scriptstyle 1$} -- (3,0) node[above] {$\scriptstyle 1$}  -- (4,0) node[above] {$\scriptstyle d-1$} -- (5.5,0);
	\draw[red] (1,0) -- (1,-0.7);
	\draw[blue] (4,0) -- (4,-0.7);
	\draw (2,0) to[bend right=90] (3,0);
	\draw (2,-0.7) to[bend left=90] (3,-0.7);
	\braid[border height=0.1cm, height=0.5cm, style strands={1}{red}, style strands={4}{blue}] at (1,-0.7) s_3 s_4 s_3 ;
	\draw[green] (4.5,-1.5) circle (0.3) ;
	\draw (5,-0.7) to[bend left=150] (5,-2.4) ;
	\braid[border height=0.1cm, height=0.5cm, style strands={1}{red}, style strands={4}{blue}] at (1,-2.3) s_1^{-1} s_2^{-1} s_3^{-1} s_3^{-1} s_2^{-1} s_1 ;
	\draw (0.5,-5.5) -- (1,-5.5) node[below] {$\scriptstyle e-1$} -- (2,-5.5) node[below] {$\scriptstyle 1$} -- (3,-5.5) node[below] {$\scriptstyle 1$} -- (4,-5.5) node[below] {$\scriptstyle d-1$} -- (5.5,-5.5);
\end{tikzbox} \\
	&= A^2 \begin{tikzbox} 
	\draw (0.5,0) -- (1,0) node[above] {$\scriptstyle e-1$} -- (2,0) node[above] {$\scriptstyle 1$} -- (3,0) node[above] {$\scriptstyle 1$}  -- (4,0) node[above] {$\scriptstyle d-1$} -- (4.5,0);
	\draw[red] (1,-1.3) -- (1,-2) ;
	\draw (2,-1.3) -- (2,-2);
	\draw (3,-1.3) -- (3,-2);
	\draw[blue] (4,-1.3) -- (4,-2) ;
	\braid[border height=0.15cm, height=0.5cm, style strands={1}{red}, style strands={4}{blue}] at (1,0) s_3^{-1} s_3^{-1} ;
	\braid[height=0.5cm, style strands={1}{red}, style strands={4}{blue}] at (1,-2) s_1^{-1} s_2^{-1} s_3^{-1} s_3^{-1} s_2^{-1} s_1 ;
	\draw[green,->] (1.9,-2.4) arc (10:-10:0.5) ;
	\draw[green,->] (1.1,-2.6) arc (190:170:0.5) ;
	\draw[green,->] (1.9,-5.1) arc (-10:10:0.5) ;
	\draw[green,->] (1.1,-4.9) arc (170:190:0.5) ;
	\draw[green,->] (3.9,-0.8) arc (10:-10:0.5) ;
	\draw[green,->] (3.1,-1) arc (190:170:0.5) ;
	\draw[green,->] (3.9,-0.3) arc (10:-10:0.5) ;
	\draw[green,->] (3.1,-0.5) arc (190:170:0.5) ;
	\draw (0.5,-5.5) -- (1,-5.5) node[below] {$\scriptstyle e-1$} -- (2,-5.5) node[below] {$\scriptstyle 1$} -- (3,-5.5) node[below] {$\scriptstyle 1$} -- (4,-5.5) node[below] {$\scriptstyle d-1$} -- (4.5,-5.5);
\end{tikzbox}
	+ \begin{tikzbox}
	\draw (0.5,0) -- (1,0) node[above] {$\scriptstyle e-1$} -- (2,0) node[above] {$\scriptstyle 1$} -- (3,0) node[above] {$\scriptstyle 1$}  -- (4,0) node[above] {$\scriptstyle d-1$} -- (4.5,0);
	\draw[red] (1,-1.3) -- (1,-2) ;
	\draw[blue] (4,-1.3) -- (4,-2) ;
	\draw (2,-1.3) to[bend right=90] (3,-1.3);
	\draw (2,-2) to[bend left=90] (3,-2);
	\braid[border height=0.15cm, height=0.5cm, style strands={1}{red}, style strands={4}{blue}] at (1,0) s_3^{-1} s_3^{-1} ;
	\braid[height=0.5cm, style strands={1}{red}, style strands={4}{blue}] at (1,-2) s_1^{-1} s_2^{-1} s_3^{-1} s_3^{-1} s_2^{-1} s_1 ;
	\draw[green,->] (1.9,-5.1) arc (-10:10:0.5) ;
	\draw[green,->] (1.1,-4.9) arc (170:190:0.5) ;
	\draw[green,->] (3.9,-0.8) arc (10:-10:0.5) ;
	\draw[green,->] (3.1,-1) arc (190:170:0.5) ;
	\draw[green,->] (3.9,-0.3) arc (10:-10:0.5) ;
	\draw[green,->] (3.1,-0.5) arc (190:170:0.5) ;
	\draw (0.5,-5.5) -- (1,-5.5) node[below] {$\scriptstyle e-1$} -- (2,-5.5) node[below] {$\scriptstyle 1$} -- (3,-5.5) node[below] {$\scriptstyle 1$} -- (4,-5.5) node[below] {$\scriptstyle d-1$} -- (4.5,-5.5);
\end{tikzbox} \\
	&\qquad - A^{-4} \begin{tikzbox}
	\draw (0.5,0) -- (1,0) node[above] {$\scriptstyle e-1$} -- (2,0) node[above] {$\scriptstyle 1$} -- (3,0) node[above] {$\scriptstyle 1$}  -- (4,0) node[above] {$\scriptstyle d-1$} -- (4.5,0);
	\draw[red] (1,0) -- (1,-0.7);
	\draw[blue] (4,0) -- (4,-0.7);
	\draw (2,0) to[bend right=90] (3,0);
	\draw (2,-0.7) to[bend left=90] (3,-0.7);
	\braid[border height=0.15cm, height=0.5cm, style strands={1}{red}, style strands={4}{blue}] at (1,-0.7) s_3 s_3 ;
	\braid[height=0.5cm, style strands={1}{red}, style strands={4}{blue}] at (1,-2) s_1^{-1} s_2^{-1} s_3^{-1} s_3^{-1} s_2^{-1} s_1 ;
	\draw[green,->] (1.9,-5.1) arc (-10:10:0.5) ;
	\draw[green,->] (1.1,-4.9) arc (170:190:0.5) ;
	\draw[green,->] (3.9,-1.7) arc (-10:10:0.5) ;
	\draw[green,->] (3.1,-1.5) arc (170:190:0.5) ;
	\draw[green,->] (3.9,-1.2) arc (-10:10:0.5) ;
	\draw[green,->] (3.1,-1) arc (170:190:0.5) ;
	\draw (0.5,-5.5) -- (1,-5.5) node[below] {$\scriptstyle e-1$} -- (2,-5.5) node[below] {$\scriptstyle 1$} -- (3,-5.5) node[below] {$\scriptstyle 1$} -- (4,-5.5) node[below] {$\scriptstyle d-1$} -- (4.5,-5.5);
\end{tikzbox} .
\end{align*}

Or, d'après la proposition \ref{Prop: JWG}, pour tout $i \in \{1,...,n-1\} \setminus \{lp-1\}$, on a $h_i f_n = 0 = f_n h_i$. Aussi, on dénoue chacun des $e-1$ (resp. $d-1$) croisements extrêmes avec la composante qui ne s'annule pas. On obtient donc :
\begin{align*}
\mathbf{S}_{0,e,0,d,0} &= A^{2+2(d-1)+(e-1)-(e-1)} \begin{tikzbox} 
	\draw (0.5,0) -- (1,0) node[above] {$\scriptstyle 1$} -- (2,0) node[above] {$\scriptstyle e-1$} -- (3,0) node[above] {$\scriptstyle 1$}  -- (4,0) node[above] {$\scriptstyle d-1$} -- (4.5,0);
	\draw (1,0) -- (1,-1) ;
	\draw[red] (2,0) -- (2,-1);
	\draw (3,0) -- (3,-1);
	\draw[blue] (4,0) -- (4,-1) ;
	\braid[height=0.5cm, style strands={2}{red}, style strands={4}{blue}] at (1,-1) s_2^{-1} s_3^{-1} s_3^{-1} s_2^{-1} ;
	\draw (0.5,-3.5) -- (1,-3.5) node[below] {$\scriptstyle 1$} -- (2,-3.5) node[below] {$\scriptstyle e-1$} -- (3,-3.5) node[below] {$\scriptstyle 1$} -- (4,-3.5) node[below] {$\scriptstyle d-1$} -- (4.5,-3.5);
\end{tikzbox} \\
	&\qquad + A^{-(e-1)} \left( A^{2(d-1)} - A^{-2(d-1)-4} \right) \begin{tikzbox}
	\draw (0.5,0) -- (1,0) node[above] {$\scriptstyle e-1$} -- (2,0) node[above] {$\scriptstyle 1$} -- (3,0) node[above] {$\scriptstyle 1$}  -- (4,0) node[above] {$\scriptstyle d-1$} -- (4.5,0);
	\draw[red] (1,0) -- (1,-0.7) ;
	\draw[blue] (4,-0) -- (4,-0.7) ;
	\draw (2,0) to[bend right=90] (3,0);
	\draw (2,-0.7) to[bend left=90] (3,-0.7);
	\braid[border height=0.15cm, height=0.5cm, style strands={1}{red}, style strands={4}{blue}] at (1,-0.7) s_1^{-1} s_2^{-1} s_3^{-1} s_3^{-1} s_2^{-1} ;
	\draw[green,->] (2.3,-3.15) arc (-150:-30:0.2) ;
	\draw[green,->] (2.7,-3.05) arc (30:150:0.2) ;
	\draw (0.5,-3.5) -- (1,-3.5) node[below] {$\scriptstyle 1$} -- (2,-3.5) node[below] {$\scriptstyle e-1$} -- (3,-3.5) node[below] {$\scriptstyle 1$} -- (4,-3.5) node[below] {$\scriptstyle d-1$} -- (4.5,-3.5);
\end{tikzbox} \\
	&= A^{2d} \begin{tikzbox}
	\draw (0.5,0) -- (1,0) node[above] {$\scriptstyle 1$} -- (2,0) node[above] {$\scriptstyle e-1$} -- (3,0) node[above] {$\scriptstyle d-1$} -- (3.5,0);
	\draw (1,0) -- (1,-2.5);
	\braid[style strands={2}{red}, style strands={3}{blue}] s_2^{-1} s_2^{-1} ;
	\draw (0.5,-2.5) -- (1,-2.5) node[below] {$\scriptstyle 1$} -- (2,-2.5) node[below] {$\scriptstyle e-1$} -- (3,-2.5) node[below] {$\scriptstyle d$} -- (3.5,-2.5);
\end{tikzbox}
	+ A^{-2(e-1)-2} \left( A^{2d} - A^{-2d} \right) \begin{tikzbox}
	\draw (0.5,0) -- (1,0) node[above] {$\scriptstyle e-1$} -- (2,0) node[above] {$\scriptstyle 1$} -- (3,0) node[above] {$\scriptstyle 1$}  -- (4,0) node[above] {$\scriptstyle d-1$} -- (4.5,0);
	\draw (2,0) to[bend right=90] (3,0);
	\braid[width=3cm, style strands={1}{red}, style strands={2}{blue}] at (1,0) s_1^{-1} s_1^{-1} ;
	\draw (2,-2.5) to[bend left=90] (3,-2.5);
	\draw (0.5,-2.5) -- (1,-2.5) node[below] {$\scriptstyle e-1$} -- (2,-2.5) node[below] {$\scriptstyle 1$} -- (3,-2.5) node[below] {$\scriptstyle 1$} -- (4,-2.5) node[below] {$\scriptstyle d-1$} -- (4.5,-2.5);
\end{tikzbox} \\
	&= A^{2d} \begin{tikzbox}
	\draw (0.5,0) -- (1,0) node[above] {$\scriptstyle 1$} -- (2,0) node[above] {$\scriptstyle e-1$} -- (3,0) node[above] {$\scriptstyle d-1$} -- (3.5,0);
	\draw (1,0) -- (1,-2.5);
	\braid[style strands={2}{red}, style strands={3}{blue}] s_2^{-1} s_2^{-1} ;
	\draw (0.5,-2.5) -- (1,-2.5) node[below] {$\scriptstyle 1$} -- (2,-2.5) node[below] {$\scriptstyle e-1$} -- (3,-2.5) node[below] {$\scriptstyle d$} -- (3.5,-2.5);
\end{tikzbox}
	+ A^{-2e} \left( A^{2d} - A^{-2d} \right) \begin{tikzbox}
	\draw (0.5,0) -- (1,0) node[above] {$\scriptstyle e-1$} -- (2,0) node[above] {$\scriptstyle 1$} -- (3,0) node[above] {$\scriptstyle 1$}  -- (4,0) node[above] {$\scriptstyle d-1$} -- (4.5,0);
	\draw (2,0) to[bend right=90] (3,0);
	\braid[width=3cm, style strands={1}{red}, style strands={2}{blue}] at (1,0) s_1^{-1} s_1^{-1} ;
	\draw (2,-2.5) to[bend left=90] (3,-2.5);
	\draw (0.5,-2.5) -- (1,-2.5) node[below] {$\scriptstyle e-1$} -- (2,-2.5) node[below] {$\scriptstyle 1$} -- (3,-2.5) node[below] {$\scriptstyle 1$} -- (4,-2.5) node[below] {$\scriptstyle d-1$} -- (4.5,-2.5);
\end{tikzbox}.
\end{align*}

Enfin, l'ajout de brins à gauche, à droite, ou en demi-cercle ne change pas les calculs ci-dessus. D'où le résultat pour le coupon $\mathbf{S}_{g,e,c,d,r}$ avec $g, c, r \in \N$.
\end{proof}

\begin{Lemme}
\label{Lemme: enlacement3}
Soient $n \geq p$ tel que $n \neq -1 \mod p$, et $l \in \N^*$ le quotient de $n$ dans sa division euclidienne par $p$. On se donne un entier $c$ tel que $1 \leq c \leq n-lp+1$. Alors :
\[ \begin{tikzbox}
	\draw (0.5,0) -- (1,0) node[above] {$\scriptstyle lp-1-c \qquad$} -- (1.5,0) node[above] {$\scriptstyle c$} -- (3,0) node[above] {$\scriptstyle \qquad n-lp-c$} -- (3.5,0) ;
	\draw (1,0) -- (1,-2.5) ;
	\draw (3,0) -- (3,-2.5) ;
	\draw (2.5,0) arc (0:-180:0.5) ;
	\draw (2.5,-2.5) arc (0:180:0.5) ;
	\draw (0.5,-2.5) -- (1,-2.5) node[below] {$\scriptstyle lp-1-c \qquad$} -- (1.5,-2.5) node[below] {$\scriptstyle c$} -- (3,-2.5) node[below] {$\scriptstyle \qquad n-lp-c$} -- (3.5,-2.5) ;
\end{tikzbox} 
	= (-1)^{c-1} \frac{[lp-1]_{A^2}}{[lp-c]_{A^2}} \begin{tikzbox}
	\draw (0.5,0) -- (1,0) node[above] {$\scriptstyle lp-2 \quad$} -- (1.5,0) node[above] {$\scriptstyle 1$} -- (3,0) node[above] {$\scriptstyle \quad n-lp$} -- (3.5,0) ;
	\draw (1,0) -- (1,-2.5) ;
	\draw (3,0) -- (3,-2.5) ;
	\draw (2.5,0) arc (0:-180:0.5) ;
	\draw (2.5,-2.5) arc (0:180:0.5) ;
	\draw (0.5,-2.5) -- (1,-2.5) node[below] {$\scriptstyle lp-2 \quad$} -- (1.5,-2.5) node[below] {$\scriptstyle 1$} -- (3,-2.5) node[below] {$\scriptstyle \quad n-lp$} -- (3.5,-2.5) ;
\end{tikzbox} . \]
\end{Lemme}

\begin{proof}
Les barres horizontales désignent l'insertion de l'idempotent de Jones-Wenzl évaluable $f_n = p_{[t(n)]}$, où $t(n)$ \eqref{Eqn: t(n)} possède un sous-tableau critique maximal $t(lp-1)$ et un tableau conjugué $\overline{t(n)}$ (cf. la remarque \ref{Rem: lignes critiques}). Le coupon de droite correspond donc à $f_n h_{lp-1} f_n = p_{[t(n)]} h_{lp-1} p_{[t(n)]} = f_n'$. Or, dans la preuve de la proposition \ref{Prop: JWG}, on a vu que :
\begin{equation} \tag{1}
f_n' = p_{[t(n)]} h_{lp-1} p_{[t(n)]} = p_{\overline{t(n)}} h_{lp-1}  p_{\overline{t(n)}} = - \frac{[lp]_{A^2}}{[lp-1]_{A^2}} p_{\overline{t(n)}} .
\end{equation}
On construit l'élément $f_n^{(c)} \in \TL_n(A^2)$ correspondant au coupon de gauche. On a :
\[ \begin{tikzbox}
	\draw (0.5,0) -- (1,0) node[above] {$\scriptstyle lp-1-c \qquad$} -- (1.5,0) node[above] {$\scriptstyle c$} -- (3,0) node[above] {$\scriptstyle \qquad n-lp-c$} -- (3.5,0) ;
	\draw (1,0) -- (1,-3.5) ;
	\draw (3,0) -- (3,-3.5) ;
	\draw[red] (2.5,0) arc (0:-180:0.5) ;
	\draw[red] (2.5,-3.5) arc (0:180:0.5) ;
	\draw (0.5,-3.5) -- (1,-3.5) node[below] {$\scriptstyle lp-1-c \qquad$} -- (1.5,-3.5) node[below] {$\scriptstyle c$} -- (3,-3.5) node[below] {$\scriptstyle \qquad n-lp-c$} -- (3.5,-3.5) ;
\end{tikzbox} 
	= \begin{tikzbox}
	\draw (-0.5,0) -- (0,0) node[above] {$\scriptstyle lp-1-c \qquad$} -- (4,0) node[above] {$\scriptstyle \qquad n-lp-c$} -- (4.5,0) ;
	\draw (0,0) -- (0,-3.5) ;
	\draw (4,0) -- (4,-3.5) ;
	\draw[red] (2.25,0) arc (0:-180:0.25) ;
	\draw[red] (1.25,0) -- (1.25,-0.7);
	\draw[red] (2.75,0) -- (2.75,-0.7);
	\draw[red] (1.75,-0.7) arc (0:-180:0.25) ;
	\draw[red] (2.25,-0.7) arc (0:180:0.25) ;
	\draw[red] (2.75,-0.7) arc (0:-180:0.25) ;
	\node at (0.95,-0.5) {$\cdots$} ;
	\node at (3.15,-0.5) {$\cdots$} ;
	\draw[red] (0.5,0) -- (0.5,-1.4);
	\draw[red] (3.5,0) -- (3.5,-1.4);
	\draw[red] (1,-1.4) arc (0:-180:0.25) ;
	\draw[red] (1.5,-1.4) arc (0:180:0.25) ;
	\node at (2.05,-1.5) {$\cdots$};
	\draw[red] (3,-1.4) arc (0:180:0.25) ;
	\draw[red] (3.5,-1.4) arc (0:-180:0.25) ;
	\draw[red] (1,-2.1) arc (0:180:0.25) ;
	\draw[red] (1.5,-2.1) arc (0:-180:0.25) ;
	\node at (2.05,-2) {$\cdots$};
	\draw[red] (3,-2.1) arc (0:-180:0.25) ;
	\draw[red] (3.5,-2.1) arc (0:180:0.25) ;
	\draw[red] (0.5,-2.1) -- (0.5,-3.5);
	\draw[red] (3.5,-2.1) -- (3.5,-3.5);
	\node at (0.95,-3) {$\cdots$};
	\node at (3.15,-3) {$\cdots$};
	\draw[red] (1.75,-2.8) arc (0:180:0.25) ;
	\draw[red] (2.25,-2.8) arc (0:-180:0.25) ;
	\draw[red] (2.75,-2.8) arc (0:180:0.25) ;
	\draw[red] (1.25,-2.8) -- (1.25,-3.5) ;
	\draw[red] (2.75,-2.8) -- (2.75,-3.5) ;
	\draw[red] (2.25,-3.5) arc (0:180:0.25) ;
	\draw (-0.5,-3.5) -- (0,-3.5) node[below] {$\scriptstyle lp-1-c \qquad$} -- (4,-3.5) node[below] {$\scriptstyle \qquad n-lp-c$} -- (4.5,-3.5) ;
\end{tikzbox} . \]
Donc :
\begin{align*}
f_n^{(c)} &= f_n h_{lp-1} \left( h_{lp-2} h_{lp} \right) ... \left( h_{lp-c} h_{lp-c+2} ... h_{lp+c-2} \right) ... \left( h_{lp-2} h_{lp} \right) h_{lp-1} f_n \\
	&= p_{[t(n)]} h_{lp-1} \left( h_{lp-2} h_{lp} \right) ... \left( h_{lp-c} h_{lp-c+2} ... h_{lp+c-2} \right) ... \left( h_{lp-2} h_{lp} \right) h_{lp-1} p_{[t(n)]} .
\end{align*}
On se retrouve dans la configuration de la preuve du lemme \ref{Lemme: idempotents eva2}. En utilisant les mêmes arguments de calculs, on obtient :
\begin{align*}
f_n^{(c)} &= \begin{multlined}[t]
	\left( - \frac{[lp-2]_{A^2}}{[lp-1]_{A^2}} \right) \left( - \frac{[lp-3]_{A^2}}{[lp-2]_{A^2}} \right)^2 ... \left( - \frac{[lp-c]_{A^2}}{[lp-c+1]_{A^2}} \right)^{c-1}
		\left( - \frac{[lp-c+1]_{A^2}}{[lp-c]_{A^2}} \right)^{c} \\
		\times \left( - \frac{[lp-c+2]_{A^2}}{[lp-c+1]_{A^2}} \right)^{c-1} ... \left( - \frac{[lp-1]_{A^2}}{[lp-2]_{A^2}} \right)^2 \left( - \frac{[lp]_{A^2}}{[lp-1]_{A^2}} \right) p_{\overline{t(n)}} 
		\end{multlined} \\
	&= (-1)^c \prod_{j=1}^{c-1} \frac{[lp-j-1]_{A^2}^j [lp-j+1]_{A^2}^j}{[lp-j]_{A^2}^{2j}} \times \frac{[lp-c+1]_{A^2}^c}{[lp-c]_{A^2}^c} p_{\overline{t(n)}} \\
	&= (-1)^c \prod_{j=1}^{c-1} \frac{1}{[lp-j]_{A^2}^{2j}} \prod_{j=2}^{c} [lp-j]_{A^2}^{j-1} \prod_{j=0}^{c-2} [lp-j]_{A^2}^{j+1} \times \frac{[lp-c+1]_{A^2}^c}{[lp-c]_{A^2}^c} p_{\overline{t(n)}} \\
	&= (-1)^c \frac{[lp-c+1]_{A^2}^{c-2} [lp-c]_{A^2}^{c-1} [lp]_{A^2} [lp-1]_{A^2}^2 [lp-c+1]_{A^2}^c}{[lp-1]_{A^2}^2 [lp-c+1]_{A^2}^{2c-2} [lp-c]_{A^2}^c} p_{\overline{t(n)}} \\
	&= (-1)^c \frac{[lp]_{A^2}}{[lp-c]_{A^2}} p_{\overline{t(n)}} . \tag{2}
\end{align*}
En comparant les équations $(1)$ et $(2)$, on obtient :
\[ f_n^{(c)} = (-1)^{c-1} \frac{[lp-1]_{A^2}}{[lp-c]_{A^2}} f_n'. \]
D'où l'égalité des coupons correspondants.
\end{proof}

On peut maintenant exprimer certains coupons $\mathbf{E}_{i,n-i}$ \eqref{Eqn: enlacement}, $n \in \N^*$, $i \in \{1,...,n-1\}$, en fonction des coupons associés aux idempotents et nilpotents de Jones-Wenzl évaluables : \index{f2 @$\rect{I}{n} $} \index{f2' @$\rect{N}{n}$}
\begin{equation}
\label{Eqn: coupons JWGA}
\begin{aligned}
&\rect{I}{n} := \rect{f_n}{}
	= \begin{tikzbox}
		\draw (0.5,0) -- (1,0) node[above] {$\scriptstyle n$} -- (1.5,0) ;
		\draw (1,0) -- (1,-2.5) ;
		\draw (0.5,-2.5) -- (1,-2.5) node[below] {$\scriptstyle n$} -- (1.5,-2.5) ;
	\end{tikzbox} , && n \in \N^*, \\
&\rect{N}{n} := \rect{f_n'}{}
	= \begin{tikzbox}
		\draw (0.5,0) -- (1,0) node[above] {$\scriptstyle lp-2 \quad$} -- (1.5,0) node[above] {$\scriptstyle 1$} -- (3,0) node[above] {$\scriptstyle n-lp$} -- (3.5,0) ;
		\draw (1,0) -- (1,-2.5) ;
		\draw (3,0) -- (3,-2.5) ;
		\draw (2.5,0) arc (0:-180:0.5) ;
		\draw (2.5,-2.5) arc (0:180:0.5) ;
		\draw (0.5,-2.5) -- (1,-2.5) node[below] {$\scriptstyle lp-2 \quad$} -- (1.5,-2.5) node[below] {$\scriptstyle 1$} -- (3,-2.5) node[below] {$\scriptstyle n-lp$} -- (3.5,-2.5) ;
	\end{tikzbox} , && n \geq p,
\end{aligned}
\end{equation}
où, pour tout $n \in \N^*$, l'entier $l$ désigne le quotient de  $n$ dans sa division euclidienne par $p$.

\begin{Thm}[---, Andrews]
\label{Thm: enlacement}
Soient $n \in \N^*$, $l \in \N$ le quotient de $n$ dans sa division euclidienne par $p$, et $i \in \{1,...,n-1\}$.
\begin{enumerate}[(i)]
	\item Si $n \leq p-1$ ou $n = -1 \mod p$, alors $\mathbf{E}_{i,n-i}^{\pm1} = A^{\pm 2(n-i)i} \; \rect{I}{n}$.
	\item[(ii')] Si $n \geq p$ et $i = 1 < n$, alors :
	\begin{gather*}
	\mathbf{E}_{1,n-1}^{\pm1} = A^{\pm 2(n-1)} \; \rect{I}{n} \pm A^{\pm 2(lp-3)} \left( A^2 - A^{-2} \right) [n-lp+1]_{A^2} \; \rect{N}{n} , \\
	\mathbf{E}_{1,n-1}^{\pm1} \cdot \rect{N}{n} = \begin{multlined}[t]
		A^{\pm 2(n-1)} \; \rect{N}{n}
		\mp A^{\pm 2(lp-3)} \left( A^2 - A^{-2} \right) \frac{ [lp]_{A^2} [n-lp+1]_{A^2} }{ [lp-1]_{A^2} } \; \rect{N}{n}.
		\end{multlined}
	\end{gather*}
	\item[(iii)] Si $n \geq p$ et $lp-1 \leq i < n$, alors :
	\begin{gather*}
	\mathbf{E}_{i,n-i}^{\pm1} = \begin{multlined}[t]
		A^{\pm 2(n-i)i} \; \rect{I}{n}
		\pm A^{\pm 2(n-i)(i-lp)} \left( A^2 - A^{-2} \right) \frac{ [lp-1]_{A^2} [(n-i)lp]_{A^2} }{ [lp]_{A^2} } \; \rect{N}{n} ,
		\end{multlined} \\
	\mathbf{E}_{i,n-i}^{\pm1} \cdot \rect{N}{n} = \begin{multlined}[t]
		A^{\pm 2(n-i)i} \; \rect{N}{n}
		\mp A^{\pm 2(n-i)(i-lp)} \left( A^2 - A^{-2} \right) [(n-i)lp]_{A^2} \; \rect{N}{n} .
		\end{multlined}
	\end{gather*}
\end{enumerate}
\end{Thm}

\begin{proof}
L'enlacement négatif étant l'image miroir du positif, les formules avec les coupons $\mathbf{E}_{i,n-i}^{-1}$ se déduisent de celles avec les coupons $\mathbf{E}_{i,n-i}$ en remplaçant $A$ par $A^{-1}$ (cf. les relations d'écheveaux \eqref{Eqn: echeveaux2}). 

De même, pour les assertions $(ii')$ et $(iii)$, les formules des coupons $\mathbf{E}_{i,n-i} \cdot \rect{N}{n}$ (i.e. le produit en pile \eqref{Eqn: produit en pile} de $\mathbf{E}_{i,n-i}$ et $\rect{N}{n}$) se déduisent de celles des coupons $\mathbf{E}_{i,n-i}$ grâce aux relations :
\[ \rect{I}{n} \cdot \rect{N}{n} = \rect{N}{n} = \rect{N}{n} \cdot \rect{I}{n}, \qquad \rect{N}{n} \cdot \rect{N}{n} \overset{\ref{Prop: JWG}}{=} - \frac{ [lp]_{A^2} }{ [lp-1]_{A^2} } \rect{N}{n} . \]

Il suffit donc d'étudier le coupon $\mathbf{E}_{i,n-i}$. L'assertion $(i)$ découle directement du lemme \ref{Lemme: enlacement1}. On suppose désormais que $n \geq p$. D'après les lemmes \ref{Lemme: enlacement1} et \ref{Lemme: enlacement2}, on sait que :
\begin{equation} \tag{1}
\mathbf{E}_{i,n-i} = \begin{cases}
		A^{2(lp-1-i)i} \mathbf{S}_{lp-1-i,i,0,n-lp+1,0} & \text{ si } i \leq lp-1 <n.\\
		A^{2(n-i)(i-lp+1)} \mathbf{S}_{0,lp-1,0,n-i,i-lp+1} & \text{ si } lp-1 \leq i < n. 
	\end{cases}
\end{equation}
On se donne $g,r \in \N$ et $d,e \in \N^*$ tels que $g+e+d+r=n$ et $g+e=lp-1$. Alors, d'après le lemme \ref{Lemme: enlacement2}, on a :
\[ S_{g,e,0,d,r} = A^{2d} S_{g+1,e-1,0,d,r} + A^{-2e} \left( A^{2d} - A^{-2d} \right) S_{g,e-1,1,d-1,r} . \]
On en déduit que :
\[ \mathbf{S}_{g,e,0,d,r} = A^{2de} \mathbf{S}_{g+e,0,0,0,d+r} + \sum_{k=1}^{\min(e,d)} \sum_{0=i_0 < i_1 < ... < i_k < i_{k+1}=e+1} C_{i_1,...,i_k} \mathbf{S}_{g+e-k,0,k,0,d-k+r} , \]
où :
\begin{align*}
C_{i_1,...,i_k} &= \prod_{j=0}^{k} A^{2(d-j)(i_{j+1}-1-i_j)}  \prod_{j=0}^{k-1} A^{-2(e-i_{j+1}+1)} \left( A^{2(d-j)} - A^{-2(d-j)} \right) \\
	&= A^{2\sum_{j=0}^k (d-j)(i_{j+1}-i_j-1) - 2\sum_{j=0}^{k-1} (e-i_{j+1}+1) } \prod_{j=0}^{k-1} \left( A^{2(d-j)} - A^{-2(d-j)} \right) \\
		&= \begin{multlined}[t]
			A^{2d(e+1) - 2d(k+1) - 2k(e+1) + 2\sum_{j=1}^{k} i_j + k(k+1) -2ek + 2\sum_{j=1}^k i_j -2k }
			\prod_{j=0}^{k-1} \left( A^{2(d-j)} - A^{-2(d-j)} \right) 
			\end{multlined} \\
		&= A^{2de + k(k-3-2d-4e) + 4\sum_{j=1}^k i_j} \prod_{j=0}^{k-1} \left( A^{2(d-j)} - A^{-2(d-j)} \right) .
\end{align*}
De plus, d'après le lemme \ref{Lemme: enlacement3}, on sait que :
\[ \mathbf{S}_{g+e-k,0,k,0,d-k+r} = (-1)^{k-1} \frac{A^{2(lp-1)}-A^{-2(lp-1)}}{A^{2(lp-k)}-A^{-2(lp-k)}} \mathbf{S}_{lp-2,0,1,0,n-lp} . \]
Par conséquent :
\begin{equation} \tag{2}
\mathbf{S}_{g,e,0,d,r} = A^{2de} \; \rect{I}{n} + A^{2de} \left( A^{2(lp-1)}-A^{-2(lp-1)} \right) \Sigma_{e,d} \; \rect{N}{n} ,
\end{equation}
où :
\begin{align*}
\Sigma_{e,d} &:= \begin{multlined}[t]
	\sum_{k=1}^{\min(e,d)} (-1)^{k-1} \frac{A^{k(k-3-2d-4e)}}{A^{2(lp-k)}-A^{-2(lp-k)}}
	\sum_{1\leq i_1 <...<i_k \leq e} A^{4 \sum_{j=1}^k i_j} \; \prod_{j=0}^{k-1} \left( A^{2(d-j)} - A^{-2(d-j)} \right) .
	\end{multlined} 
\end{align*}
Il reste à simplifier ce terme $\Sigma_{e,d}$. Pour cela, on utilise une preuve de G. E. Andrews. On pose :
\[ \forall X,Y \in \C(A) \quad \forall k \in \N^* \qquad (Y;X)_k := \prod_{j=0}^{k-1} (1-YX^j) . \]
Alors, pour tout $k \in \N^*$, on a :
\[ \prod_{j=0}^{k-1} \left( A^{2(d-j)} - A^{-2(d-j)} \right) = A^{2dk - k(k-1)} (A^{-4d}; A^4)_k = A^{k(2d-k+1)} (A^{-4d}; A^4)_k . \]
Et d'après le théorème $q$-binomial (cf. par exemple \cite[Thm8]{AE04}), pour tout $k \in \N^*$ :
\begin{gather*}
\sum_{1\leq i_1 <...<i_k \leq e} A^{4 \sum_{j=1}^k i_j} 
	= A^{2k(k+1)} \frac{ (A^{4e}; A^{-4})_k }{ (A^4; A^4)_k }
	= (-1)^k A^{2k(k+1)} \underbrace{A^{4ke-2k(k-1)}}_{ A^{k(4e+4)} } \frac{ (A^{-4e}; A^{4})_k }{ (A^4; A^4)_k } .
\end{gather*}
Donc :
\begin{align*}
\Sigma_{e,d} &= - \sum_{k=1}^{\min(e,d)}  \frac{ A^{2k} (A^{-4d}; A^4)_k (A^{-4e}; A^{4})_k }{ A^{2(lp-k)} \left( 1-A^{-4(lp-k)} \right) (A^4; A^4)_k } \\
	&= - \sum_{k=1}^{\min(e,d)}  \frac{ A^{4k} (A^{-4d}; A^4)_k (A^{-4e}; A^{4})_k }{ A^{2lp} \left( 1-A^{-4(lp-k)} \right) (A^4; A^4)_k } \\
	&= - \frac{1}{A^{2lp}} \sum_{k=1}^{\min(e,d)}  \frac{ (A^{-4(lp-1)}; A^4)_{k-1} (A^{-4d}; A^4)_k (A^{-4e}; A^{4})_k }{ (A^{-4(lp-1)}; A^4)_k (A^4; A^4)_k } A^{4k} \\
	&= - \frac{1}{A^{2lp} \left( 1 - A^{-4lp} \right) } \sum_{k=1}^{\min(e,d)}  \frac{ (A^{-4lp}; A^4)_k (A^{-4d}; A^4)_k (A^{-4e}; A^{4})_k }{ (A^{-4(lp-1)}; A^4)_k (A^4; A^4)_k } A^{4k} \\
	&= - \frac{1}{A^{2lp} - A^{-2lp} } \sum_{k \geq 1}  \frac{ (A^{-4lp}; A^4)_k (A^{-4d}; A^4)_k (A^{-4e}; A^{4})_k }{ (A^{-4(lp-1)}; A^4)_k (A^4; A^4)_k } A^{4k} . \tag{3}
\end{align*}

On revient sur la distinction des cas $(ii')$ et $(iii)$.
\begin{enumerate}
	\item[(ii')] On suppose que $i = 1 < n$. Alors $e=1$ et $d=n-lp+1$. D'après les égalités $(1)$ et $(2)$, on a :
	\begin{align*}
	\mathbf{E}_{1,n-1} &= \begin{multlined}[t]
		A^{2(lp-2)} A^{2(n-lp+1)} \; \rect{I}{n} \\
		+ A^{2(lp-2)} A^{2(n-lp+1)} \left( A^{2(lp-1)} - A^{-2(lp-1)} \right) \Sigma_{1,n-lp+1} \; \rect{N}{n} 
		\end{multlined} \\
	&= A^{2(n-1)} \; \rect{I}{n} + \underbrace{A^{2(n-1)} (A^2-A^{-2}) [lp-1]_{A^2} \Sigma_{i,n-lp+1}}_{ B_{1, n-lp+1} } \: \rect{N}{n} .
	\end{align*}
	Or, d'après l'équation $(3)$, la somme $\Sigma_{1,n-lp+1}$ se réduit à :
	\begin{align*}
	\Sigma_{1,n-lp+1} &= - \frac{1}{A^{2lp} - A^{-2lp} } \frac{ (A^{-4lp}; A^4)_1 (A^{-4(n-lp+1)}; A^4)_1 (A^{-4}; A^{4})_1 }{ (A^{-4(lp-1)}; A^4)_1 (A^4; A^4)_1 } A^{4} \\
	&=  \frac{1}{A^{2lp} - A^{-2lp} } \frac{ (1-A^{-4lp}) (1-A^{-4(n-lp+1)}) }{ (1-A^{-4(lp-1)}) } \\
	&= A^{-2lp-2(n-lp+1)+2(lp-1)} \frac{ [n-lp+1]_{A^2} }{ [lp-1]_{A^2} }
	= A^{-2(n-lp+2)} \frac{ [n-lp+1]_{A^2} }{ [lp-1]_{A^2} } .
	\end{align*}
	Il s'ensuit que :
	\[ B_{1, n-lp+1} = A^{2(lp-3)} (A^2-A^{-2}) [n-lp+1]_{A^2}. \]
	\item[(iii)] On suppose que $lp-1 \leq i < n$. Alors $e=lp-1$ et $d=n-i$. D'après les égalités $(1)$ et $(2)$, on a :
	\begin{align*}
	\mathbf{E}_{i,n-i} &= \begin{multlined}[t]
		A^{2(n-i)(i-lp+1)} A^{2(n-i)(lp-1)} \; \rect{I}{n} \\
		+ A^{2(n-i)(i-lp+1)} A^{2(n-i)(lp-1)} \left( A^{2(lp-1)} - A^{-2(lp-1)} \right) \Sigma_{lp-1,n-i} \; \rect{N}{n} 
		\end{multlined} \\
	&= A^{2(n-i)i} \; \rect{I}{n} + \underbrace{A^{2(n-i)i} (A^2-A^{-2}) [lp-1]_{A^2} \Sigma_{lp-1,n-i}}_{ B_{lp-1, n-i} } \: \rect{N}{n} .
	\end{align*}
	Or, d'après l'équation $(3)$, la somme $\Sigma_{lp-1,n-i}$ est une série hypergéométrique basique (cf. par exemple \cite[Eqn 1.2.14]{GR04}) :
	\begin{align*}
	\Sigma_{lp-1,n-i} &= - \frac{1}{A^{2lp} - A^{-2lp} } \sum_{k \geq 1}  \frac{ (A^{-4lp},A^4)_k (A^{-4(n-i)},A^4)_k }{ (A^4, A^4)_k } A^{4k} \\
		&= - \frac{1}{A^{2lp} - A^{-2lp} } \left( {}_2\phi_1 \left( A^{-4lp}, A^{-4(n-i)}; 0; A^4, A^4 \right) -1 \right).
	\end{align*}
	D'après la formule de transformation de Heine (cf. par exemple \cite[§ 1.5]{GR04}), on a :
	\begin{align*}
	\Sigma_{lp-1,n-i} &= - \frac{1}{A^{2lp} - A^{-2lp} } \left( \frac{(0; A^4)}{(0; A^4)} A^{-4lp(n-i)} -1 \right) \\
		&= \frac{1 - A^{-4lp(n-i)}}{A^{2lp} - A^{-2lp} } 
		= A^{-2lp(n-i)} \frac{ [lp(n-i)]_{A^2} }{ [lp]_{A^2} }. 
	\end{align*}
	Il s'ensuit que :
	\[ B_{lp-1, n-i} = A^{2(n-i)(i-lp)} (A^2-A^{-2}) \frac{ [lp-1]_{A^2}  [lp(n-i)]_{A^2} }{ [lp]_{A^2} } . \]
\end{enumerate}
D'où le résultat.
\end{proof}

\subsection{Actions de la vrille}
\label{subsection: vrille}

A l'aide des résultats de la sous-section \ref{subsection: enlacement}, on étudie l'action de la \emph{vrille positive} sur l'espace d'écheveaux $K_A(\bar{D} \times \mathbb{S}^1, 0)$ du tore solide. Elle est donnée par l'application linéaire :
\[ \mathbf{T} : \begin{cases}
		K_A(\bar{D} \times \mathbb{S}^1) \longrightarrow K_A(\bar{D} \times \mathbb{S}^1) \\
		\insertion{a}{}{Tore} \longmapsto \vrille{a}{}{Tore}
	\end{cases} \hspace{-0.3cm} . \]
La vrille négative étant l'image miroir de celle positive, son action se déduit de celle de la vrille positive en remplaçant $A$ par $A^{-1}$ (cf. les relations d'écheveaux \eqref{Eqn: echeveaux2}).

Comme $\left\{ \Theta_{A^2,n}(f_n) \; ; \; n \in \N^* \right\}$ est une $\Z[A,A^{-1}]$-base de $K_A(\bar{D} \times \mathbb{S}^1, 0)$ (cf. le corollaire \ref{Cor: insertion des JWG}), on calcule l'image de $\mathbf{T}$ sur ces classes d'écheveaux coloriés par les idempotents de Jones-Wenzl évaluables. Pour cela, il suffit de calculer les classes d'écheveaux des coupons : \index{T9 @$\mathbf{T}_n$}
\begin{equation}
\label{Eqn: vrille}
\mathbf{T}_n := \begin{tikzbox}
	\draw (0.5,0) -- (2,0) node[above] {$\scriptstyle n$} -- (2.5,0);
	\draw (2,0) -- (2,-0.5) ;
	\braid at (1,-0.5) s_1^{-1} ;
	\draw (1,-0.5) to[bend right=150] (1,-2) ;
	\draw (2,-2) -- (2,-2.5) ;
	\draw (0.5,-2.5) -- (2,-2.5) node[below] {$\scriptstyle n$} -- (2.5,-2.5);
\end{tikzbox} 
	= \begin{tikzbox}
	\draw (0.5,0) -- (1,0) node[above] {$\scriptstyle n$} -- (2.5,0);
	\draw (1,0) -- (1,-0.5) ;
	\braid at (1,-0.5) s_1^{-1} ;
	\draw (2,-0.5) to[bend left=150] (2,-2) ;
	\draw (1,-2) -- (1,-2.5) ;
	\draw (0.5,-2.5) -- (1,-2.5) node[below] {$\scriptstyle n$} -- (2.5,-2.5);
\end{tikzbox} ; \quad n \in \N^*,
\end{equation}
où la barre horizontale supérieure (resp. inférieure) désigne l'insertion de l'idempotent de Jones-Wenzl évaluable dont l'indice correspond au nombre total de brin(s) sortant(s) (resp. entrant(s)). On utilisera encore les notations \eqref{Eqn: coupons JWGA} pour les coupons associés aux idempotents et nilpotents de Jones-Wenzl évaluables.

\begin{Lemme}
\label{Lemme: vrille}
Soient $n \in \N^*$ et $l \in \N$ le quotient de $n$ dans sa division euclidienne par $p$.
\begin{enumerate}[(i)]
	\item Si $n \leq p-1$ ou $n = -1 \mod p$, alors :
	\[ \mathbf{T}_n = \begin{tikzbox}
		\draw (0.5,0) -- (1,0) node[above] {$\scriptstyle n$} -- (2.5,0);
		\draw (1,0) -- (1,-0.5) ;
		\braid at (1,-0.5) s_1^{-1} ;
		\draw (2,-0.5) to[bend left=150] (2,-2) ;
		\draw (1,-2) -- (1,-2.5) ;
		\draw (0.5,-2.5) -- (1,-2.5) node[below] {$\scriptstyle n$} -- (2.5,-2.5);
	\end{tikzbox}
		= (-1)^n A^{n(n+2)} \begin{tikzbox}
	\draw (0.5,0) -- (1,0) node[above] {$\scriptstyle n$} -- (1.5,0) ;
	\draw (1,0) -- (1,-2.5) ;
	\draw (0.5,-2.5) -- (1,-2.5) node[below] {$\scriptstyle n$} -- (1.5,-2.5) ;
\end{tikzbox} . \]
	\item Si $n \geq p$, alors :
	\[ \mathbf{T}_n = \begin{tikzbox}
		\draw (0.5,0) -- (1,0) node[above] {$\scriptstyle n$} -- (2.5,0);
		\draw (1,0) -- (1,-0.5) ;
		\braid at (1,-0.5) s_1^{-1} ;
		\draw (2,-0.5) to[bend left=150] (2,-2) ;
		\draw (1,-2) -- (1,-2.5) ;
		\draw (0.5,-2.5) -- (1,-2.5) node[below] {$\scriptstyle n$} -- (2.5,-2.5);
	\end{tikzbox}
		= (-1)^n A^{(lp-1)(lp+1) + (n-lp+1)(n-lp+3)} \begin{tikzbox}
	\draw (0.5,0) -- (1,0) node[above] {$\scriptstyle lp-1$} -- (2,0) node[above] {$\scriptstyle n-lp+1$} -- (2.5,0);
	\braid s_1^{-1} s_1^{-1} ;
	\draw (0.5,-2.5) -- (1,-2.5) node[below] {$\scriptstyle lp-1$} -- (2,-2.5) node[below] {$\scriptstyle n-lp+1$} -- (2.5,-2.5);
\end{tikzbox} . \]
\end{enumerate}
\end{Lemme}

\begin{proof}
\begin{enumerate}[(i)]
	\item On isole un brin à gauche sur lequel on utilise les relations d'isotopie et d'écheveaux \eqref{Eqn: echeveaux2} :
	\begin{align*}
	\mathbf{T}_n &= \begin{tikzbox}
		\draw (0.5,0) -- (1,0) node[above] {$\scriptstyle n$} -- (2.5,0);
		\draw[red] (1,0) -- (1,-1.5) ;
		\braid[red] at (1,-1.5) s_1^{-1} ;
		\draw[red] (2,-1.5) to[bend left=150] (2,-3) ;
		\draw[red] (1,-3) -- (1,-4.5) ;
		\draw (0.5,-4.5) -- (1,-4.5) node[below] {$\scriptstyle n$} -- (2.5,-4.5);
	\end{tikzbox} 
	= \begin{tikzbox}
		\draw (0.5,0) -- (1,0) node[above] {$\scriptstyle 1$} -- (2,0) node[above] {$\scriptstyle n-1$} -- (5.5,0);
		\draw (1,0) -- (1,-1.25) ;
		\draw[red] (2,0) -- (2,-1.25) ;
		\braid[height=0.5cm, number of strands=5, style strands={2,4}{red}] at (1,-1.25) s_2^{-1} s_1^{-1}-s_3^{-1} s_2^{-1} ;
		\draw[red] (4,-1.25) to[bend left=150] (4,-3.25) ;
		\draw (3,-1.25) to[bend left=90] (5,-1.25) ;
		\draw (3,-3.25) to[bend right=90] (5,-3.25) ;
		\draw[red] (2,-3.25) -- (2,-4.5) ;
		\draw (1,-3.25) -- (1,-4.5) ;
		\draw (0.5,-4.5) -- (1,-4.5) node[below] {$\scriptstyle 1$} -- (2,-4.5) node[below] {$\scriptstyle n-1$} -- (5.5,-4.5);
	\end{tikzbox}
	= \begin{tikzbox}
		\draw (0.5,0) -- (1,0) node[above] {$\scriptstyle 1$} -- (2,0) node[above] {$\scriptstyle n-1$} -- (3.5,0);
		\draw (1,0) -- (1,-0.5) ;
		\draw[red] (2,0) -- (2,-0.5) ;
		\braid[height=0.5cm, style strands={2}{red}] at (1,-0.5) s_2^{-1} s_1^{-1} s_2^{-1} ;
		\draw (3,-0.5) to[bend left=150] (3,-2.5) ;
		\draw (1,-2.5) -- (1,-3) ;
		\draw[red] (2,-2.5) -- (2,-3) ;
		\braid[height=0.5cm, style strands={2,3}{red}] at (1,-3) s_2^{-1} ;
		\draw[red] (3,-3) to[bend left=150] (3,-4) ;
		\draw (1,-4) -- (1,-4.5) ;
		\draw[red] (2,-4) -- (2,-4.5) ;
		\draw (0.5,-4.5) -- (1,-4.5) node[below] {$\scriptstyle 1$} -- (2,-4.5) node[below] {$\scriptstyle n-1$} -- (3.5,-4.5);
	\end{tikzbox} \\
	&= \begin{tikzbox} 
		\draw (0.5,0) -- (1,0) node[above] {$\scriptstyle 1$} -- (2,0) node[above] {$\scriptstyle n-1$} -- (3.5,0);
		\draw (1,0) -- (1,-0.5) ;
		\draw[red] (2,0) -- (2,-0.5) ;
		\braid[height=0.5cm, style strands={2}{red}] at (1,-0.5) s_1^{-1} s_2^{-1} s_1^{-1} ;
		\draw[green] (2.5,-1.5) circle (0.3) ;
		\draw (3,-0.5) to[bend left=150] (3,-2.5) ;
		\draw (1,-2.5) -- (1,-3) ;
		\draw[red] (2,-2.5) -- (2,-3) ;
		\braid[height=0.5cm, style strands={2,3}{red}] at (1,-3) s_2^{-1} ;
		\draw[red] (3,-3) to[bend left=150] (3,-4) ;
		\draw (1,-4) -- (1,-4.5) ;
		\draw[red] (2,-4) -- (2,-4.5) ;
		\draw (0.5,-4.5) -- (1,-4.5) node[below] {$\scriptstyle 1$} -- (2,-4.5) node[below] {$\scriptstyle n-1$} -- (3.5,-4.5);
	\end{tikzbox} 
	= - A^3 \begin{tikzbox}
		\draw (0.5,0) -- (1,0) node[above] {$\scriptstyle 1$} -- (2,0) node[above] {$\scriptstyle n-1$} -- (3.5,0);
		\draw (1,0) -- (1,-0.5) ;
		\draw[red] (2,0) -- (2,-0.5) ;
		\braid[style strands={2}{red}] at (1,-0.5) s_1^{-1} s_1^{-1} ;
		\braid[height=0.5cm, style strands={2,3}{red}] at (1,-3) s_2^{-1} ;
		\draw[green,->] (1.9,-1.1) arc (20:-20:0.5) ;
		\draw[green,->] (1.1,-1.4) arc (200:160:0.5) ;
		\draw[green,->] (1.9,-2.1) arc (20:-20:0.5) ;
		\draw[green,->] (1.1,-2.4) arc (200:160:0.5) ;
		\draw[red] (3,-3) to[bend left=150] (3,-4) ;
		\draw (1,-4) -- (1,-4.5) ;
		\draw[red] (2,-4) -- (2,-4.5) ;
		\draw (0.5,-4.5) -- (1,-4.5) node[below] {$\scriptstyle 1$} -- (2,-4.5) node[below] {$\scriptstyle n-1$} -- (3.5,-4.5);
	\end{tikzbox} .
	\end{align*}
	
	Or, d'après la proposition \ref{Prop: JWG}, pour tout $i \in \{1,...,n-1\}$, on a $h_i f_n = 0 = f_n h_i$. Aussi, on dénoue chacun des $2(n-1)$ croisements supérieurs avec la composante $D_0$ des relations d'écheveaux \eqref{Eqn: echeveaux2} :
	\[ \begin{tikzbox}
		\draw (0.5,0) -- (1,0) node[above] {$\scriptstyle n$} -- (2.5,0);
		\draw[red] (1,0) -- (1,-0.5) ;
		\braid[red] at (1,-0.5) s_1^{-1} ;
		\draw[red] (2,-0.5) to[bend left=150] (2,-2) ;
		\draw[red] (1,-2) -- (1,-2.5) ;
		\draw (0.5,-2.5) -- (1,-2.5) node[below] {$\scriptstyle n$} -- (2.5,-2.5);
	\end{tikzbox} 
		= - A^{3+2(n-1)} \begin{tikzbox}
		\draw (0.5,0) -- (1,0) node[above] {$\scriptstyle 1$} -- (2,0) node[above] {$\scriptstyle n-1$} -- (3.5,0);
		\draw (1,0) -- (1,-2.5) ;
		\draw[red] (2,0) -- (2,-0.5) ;
		\braid[red] at (2,-0.5) s_1^{-1} ;
		\draw[red] (3,-0.5) to[bend left=150] (3,-2) ;
		\draw[red] (2,-2) -- (2,-2.5) ;
		\draw (0.5,-2.5) -- (2,-2.5) node[below] {$\scriptstyle n-1$} -- (3.5,-2.5);
	\end{tikzbox} 
	= - A^{2n+1} \begin{tikzbox}
		\draw (0.5,0) -- (1,0) node[above] {$\scriptstyle 1$} -- (2,0) node[above] {$\scriptstyle n-1$} -- (3.5,0);
		\draw (1,0) -- (1,-2.5) ;
		\draw[red] (2,0) -- (2,-0.5) ;
		\braid[red] at (2,-0.5) s_1^{-1} ;
		\draw[red] (3,-0.5) to[bend left=150] (3,-2) ;
		\draw[red] (2,-2) -- (2,-2.5) ;
		\draw (0.5,-2.5) -- (2,-2.5) node[below] {$\scriptstyle n-1$} -- (3.5,-2.5);
	\end{tikzbox} . \]
	Par une récurrence immédiate, on en déduit que :
	\begin{align*}
	\begin{tikzbox}
		\draw (0.5,0) -- (1,0) node[above] {$\scriptstyle n$} -- (2.5,0);
		\draw (1,0) -- (1,-0.5) ;
		\braid at (1,-0.5) s_1^{-1} ;
		\draw (2,-0.5) to[bend left=150] (2,-2) ;
		\draw (1,-2) -- (1,-2.5) ;
		\draw (0.5,-2.5) -- (1,-2.5) node[below] {$\scriptstyle n$} -- (2.5,-2.5);
	\end{tikzbox} 
	&= \prod_{i=1}^n -A^{2i+1} \begin{tikzbox}
		\draw (0.5,0) -- (1,0) node[above] {$\scriptstyle n$} -- (1.5,0) ;
		\draw (1,0) -- (1,-2.5) ;
		\draw (0.5,-2.5) -- (1,-2.5) node[below] {$\scriptstyle n$} -- (1.5,-2.5) ;
	\end{tikzbox}
	= (-1)^n A^{n(n+1)+n} \begin{tikzbox}
		\draw (0.5,0) -- (1,0) node[above] {$\scriptstyle n$} -- (1.5,0) ;
		\draw (1,0) -- (1,-2.5) ;
		\draw (0.5,-2.5) -- (1,-2.5) node[below] {$\scriptstyle n$} -- (1.5,-2.5) ;
	\end{tikzbox}
	= (-1)^n A^{n(n+2)} \begin{tikzbox}
		\draw (0.5,0) -- (1,0) node[above] {$\scriptstyle n$} -- (1.5,0) ;
		\draw (1,0) -- (1,-2.5) ;
		\draw (0.5,-2.5) -- (1,-2.5) node[below] {$\scriptstyle n$} -- (1.5,-2.5) ;
	\end{tikzbox} .
	\end{align*}
	\item On isole $lp-1$ brins à gauche sur lesquels on utilise les relations d'isotopie :
	\begin{align*}
	\mathbf{T}_n &= \begin{tikzbox}
		\draw (0.5,0) -- (1,0) node[above] {$\scriptstyle n$} -- (2.5,0);
		\draw[red] (1,0) -- (1,-1.5) ;
		\braid[red] at (1,-1.5) s_1^{-1} ;
		\draw[red] (2,-1.5) to[bend left=150] (2,-3) ;
		\draw[red] (1,-3) -- (1,-4.5) ;
		\draw (0.5,-4.5) -- (1,-4.5) node[below] {$\scriptstyle n$} -- (2.5,-4.5);
	\end{tikzbox} 
	= \begin{tikzbox} 
		\draw (0.5,0) -- (1,0) node[above] {$\scriptstyle lp-1$} -- (2,0) node[above] {$\scriptstyle n-lp+1$} -- (5.5,0);
		\draw[blue] (1,0) -- (1,-1.25) ;
		\draw[red] (2,0) -- (2,-1.25) ;
		\braid[height=0.5cm, number of strands=5, style strands={2,4}{red}, style strands={1,3,5}{blue}] at (1,-1.25) s_2^{-1} s_1^{-1}-s_3^{-1} s_2^{-1} ;
		\draw[red] (4,-1.25) to[bend left=150] (4,-3.25) ;
		\draw[blue] (3,-1.25) to[bend left=90] (5,-1.25) ;
		\draw[blue] (3,-3.25) to[bend right=90] (5,-3.25) ;
		\draw[red] (2,-3.25) -- (2,-4.5) ;
		\draw[blue] (1,-3.25) -- (1,-4.5) ;
		\draw (0.5,-4.5) -- (1,-4.5) node[below] {$\scriptstyle lp-1$} -- (2,-4.5) node[below] {$\scriptstyle n-lp+1$} -- (5.5,-4.5);
	\end{tikzbox}
	= \begin{tikzbox}
		\draw (0.5,0) -- (1,0) node[above] {$\scriptstyle lp-1$} -- (2,0) node[above] {$\scriptstyle \quad n-lp+1$} -- (3.5,0);
		\draw[blue] (1,0) -- (1,-0.5) ;
		\draw[red] (2,0) -- (2,-0.5) ;
		\braid[height=0.5cm, style strands={2}{red}, style strands={1,3}{blue}] at (1,-0.5) s_2^{-1} s_1^{-1} s_2^{-1} ;
		\draw[blue] (3,-0.5) to[bend left=150] (3,-2.5) ;
		\draw[blue] (1,-2.5) -- (1,-3) ;
		\draw[red] (2,-2.5) -- (2,-3) ;
		\braid[height=0.5cm, style strands={2,3}{red}, style strands={1}{blue}] at (1,-3) s_2^{-1} ;
		\draw[red] (3,-3) to[bend left=150] (3,-4) ;
		\draw[blue] (1,-4) -- (1,-4.5) ;
		\draw[red] (2,-4) -- (2,-4.5) ;
		\draw (0.5,-4.5) -- (1,-4.5) node[below] {$\scriptstyle lp-1$} -- (2,-4.5) node[below] {$\scriptstyle \quad n-lp+1$} -- (3.5,-4.5);
	\end{tikzbox} \\
	&= \begin{tikzbox} 
		\draw (0.5,0) -- (1,0) node[above] {$\scriptstyle lp-1$} -- (3,0) node[above] {$\scriptstyle n-lp+1$} -- (4.5,0);
		\draw[blue] (1,0) -- (1,-0.5) ;
		\draw[red] (3,0) -- (3,-0.5) ;
		\braid[width=2cm, style strands={2}{red}, style strands={1}{blue}] at (1,-0.5) s_1^{-1} s_1^{-1} ;
		\braid[height=0.5cm, style strands={3,4}{red}, style strands={1,2}{blue}, style floors={1}{green} ] at (1,-3) | s_1^{-1} - s_3^{-1} ;
		\draw[blue] (2,-3) to[bend left=150] (2,-4) ;
		\draw[red] (4,-3) to[bend left=150] (4,-4) ;
		\draw[blue] (1,-4) -- (1,-4.5) ;
		\draw[red] (3,-4) -- (3,-4.5) ;
		\draw (0.5,-4.5) -- (1,-4.5) node[below] {$\scriptstyle lp-1$} -- (2,-4.5) node[below] {$\scriptstyle n-lp+1$} -- (4.5,-4.5);
	\end{tikzbox} .
	\end{align*}
	
	Or, d'après la proposition \ref{Prop: JWG}, pour tout $i \in \{1,...,n-1\} \setminus \{lp-1\}$, on a $h_i f_n = 0 = f_n h_i$. On procède comme en $(i)$ sur les $lp-1$ premiers brins et les $n-lp+1$ derniers brins. On obtient ainsi :
	\begin{align*}
	\mathbf{T}_n &= \begin{tikzbox}
		\draw (0.5,0) -- (1,0) node[above] {$\scriptstyle n$} -- (2.5,0);
		\draw (1,0) -- (1,-0.5) ;
		\braid at (1,-0.5) s_1^{-1} ;
		\draw (2,-0.5) to[bend left=150] (2,-2) ;
		\draw (1,-2) -- (1,-2.5) ;
		\draw (0.5,-2.5) -- (1,-2.5) node[below] {$\scriptstyle n$} -- (2.5,-2.5);
	\end{tikzbox}
		= \underbrace{ (-1)^{(lp-1)+(n-lp+1)} }_{ (-1)^n } A^{(lp-1)(lp+1) + (n-lp+1)(n-lp+3)} \begin{tikzbox}
		\draw (0.5,0) -- (1,0) node[above] {$\scriptstyle lp-1$} -- (2,0) node[above] {$\scriptstyle n-lp+1$} -- (2.5,0);
		\braid[style strands={1}{blue}, style strands={2}{red}] s_1^{-1} s_1^{-1} ;
		\draw (0.5,-2.5) -- (1,-2.5) node[below] {$\scriptstyle lp-1$} -- (2,-2.5) node[below] {$\scriptstyle n-lp+1$} -- (2.5,-2.5);
	\end{tikzbox} \\
			&= (-1)^n A^{(lp-1)(lp+1) + (n-lp+1)(n-lp+3)} \begin{tikzbox}
		\draw (0.5,0) -- (1,0) node[above] {$\scriptstyle lp-1$} -- (2,0) node[above] {$\scriptstyle n-lp+1$} -- (2.5,0);
		\braid[style strands={1}{blue}, style strands={2}{red}] s_1^{-1} s_1^{-1} ;
		\draw (0.5,-2.5) -- (1,-2.5) node[below] {$\scriptstyle lp-1$} -- (2,-2.5) node[below] {$\scriptstyle n-lp+1$} -- (2.5,-2.5);
	\end{tikzbox} .
	\end{align*}
\end{enumerate}
\end{proof}

\begin{Thm}
\label{Thm: vrille}
Soient $n \in \N^*$ et $l \in \N$ le quotient de $n$ dans sa division euclidienne par $p$.
\begin{enumerate}[(i)]
	\item Si $n \leq p-1$ ou $n = -1 \mod p$, alors $\mathbf{T}_n^{\pm1} = (-1)^n A^{\pm n(n+2)} \; \rect{I}{n}$.
	\item Si $n \geq p$, alors :
	\begin{gather*}
	\mathbf{T}_n^{\pm1} = \begin{multlined}[t]
		(-1)^n A^{\pm n(n+2)} \; \rect{I}{n}
		\pm (-1)^n A^{\pm n(n+2) \mp 2lp(n-lp+1)} \\
		\times \left( A^2 - A^{-2} \right) \frac{ [lp-1]_{A^2} [(n-lp+1)lp]_{A^2} }{ [lp]_{A^2} } \; \rect{N}{n} , 
		\end{multlined} \\
	\mathbf{T}_n^{\pm1} \cdot \rect{N}{n} = \begin{multlined}[t]
		(-1)^n A^{\pm n(n+2)} \; \rect{N}{n}
		\mp (-1)^n A^{\pm n(n+2) \mp 2lp(n-lp+1)} \\
		\times \left( A - A^{-2} \right) [(n-lp+1)lp]_{A^2} \; \rect{N}{n} .
		\end{multlined}
	\end{gather*}
\end{enumerate}
\end{Thm}

\begin{proof}
La vrille négative étant l'image miroir de la positive, les formules avec les coupons $\mathbf{T}_n^{-1}$ se déduisent de celles avec les coupons $\mathbf{T}_n$ en remplaçant $A$ par $A^{-1}$ (cf. les relations d'écheveaux \eqref{Eqn: echeveaux2}). 

De même, pour les assertions $(ii)$ et $(iii)$, les formules des coupons $\mathbf{T}_n \cdot \rect{N}{n}$ (i.e. le produit en pile \eqref{Eqn: produit en pile} de $\mathbf{T}_n$ et $\rect{N}{n}$) se déduisent de celles des coupons $\mathbf{T}_n$ grâce aux relations :
\[ \rect{I}{n} \cdot \rect{N}{n} = \rect{N}{n} = \rect{N}{n} \cdot \rect{I}{n}, \qquad \rect{N}{n} \cdot \rect{N}{n} \overset{\ref{Prop: JWG}}{=} - \frac{ [lp]_{A^2} }{ [lp-1]_{A^2} } \rect{N}{n} . \]

Il suffit donc d'étudier le coupon $\mathbf{T}_n$. L'assertion $(i)$ découle directement du lemme \ref{Lemme: vrille}. On suppose désormais que $n \geq p$. D'après le lemme \ref{Lemme: vrille} et l'assertion $(iii)$ du théorème \ref{Thm: enlacement}, on a :
\begin{align*}
\mathbf{T}_n &= \begin{multlined}[t]
		(-1)^n A^{(lp-1)(lp+1)+(n-lp+1)(n-lp+3)+2(n-lp+1)(lp-1)} \rect{I}{n} \\
		+ (-1)^n A^{(lp-1)(lp+1)+(n-lp+1)(n-lp+3)-2(n-lp+1)} \\
		\times \left( A^2 - A^{-2} \right) \frac{ [lp-1]_{A^2} [(n-lp+1)lp]_{A^2} }{ [lp]_{A^2} } \rect{N}{n} 
		\end{multlined} \\
	&= \begin{multlined}[t]
		(-1)^n A^{(lp-1)(n+2)+(n-lp+1)(n+2)} \rect{I}{n}
		+ (-1)^n A^{(lp-1)(lp+1)+(n-lp+1)(n-lp+1)} \\
		\times \left( A^2 - A^{-2} \right) \frac{ [lp-1]_{A^2} [(n-lp+1)lp]_{A^2} }{ [lp]_{A^2} } \rect{N}{n} 
		\end{multlined} \\
	&= \begin{multlined}[t]
		(-1)^n A^{n(n+2)} \rect{I}{n}
		+ (-1)^n A^{n(n+2)-2lp(n-lp+1)} \\
		\left( A^2 - A^{-2} \right) \frac{ [lp-1]_{A^2} [(n-lp+1)lp]_{A^2} }{ [lp]_{A^2} } \rect{N}{n} .
		\end{multlined}
\end{align*}
D'où le résultat.
\end{proof}

\subsection{Actions du bouclage}
\label{subsection: bouclage}

A l'aide des résultats de la sous-section \ref{subsection: enlacement}, on étudie enfin l'action du \emph{bouclage} sur l'espace d'écheveaux $K_A(\bar{D} \times \mathbb{S}^1, 0)$ du tore solide. Elle est donnée par combinaisons linéaires des applications linéaires :
\[ \forall i \in \N^* \qquad \mathbf{B}_i : \begin{cases}
		K_A(\bar{D} \times \mathbb{S}^1) \longrightarrow K_A(\bar{D} \times \mathbb{S}^1) \\
		\insertion{a}{}{Tore} \longmapsto \boucle{a}{}{Tore}
	\end{cases} \hspace{-0.3cm} . \]
	
Comme $\left\{ \Theta_{A^2,n}(f_n) \; ; \; n \in \N^* \right\}$ est une $\Z[A,A^{-1}]$-base de $K_A(\bar{D} \times \mathbb{S}^1, 0)$ (cf. le corollaire \ref{Cor: insertion des JWG}), on calcule l'image de $\mathbf{B}_i$, $i \in \N^*$, sur ces classes d'écheveaux coloriés par les idempotents de Jones-Wenzl évaluables. Pour cela, il suffit de calculer les classes d'écheveaux des coupons : \index{B @$\mathbf{B}_{i,n}$}
\begin{equation}
\label{Eqn: bouclage}
\mathbf{B}_{i,n} := \begin{tikzbox}
	\draw (0.5,0) -- (2,0) node[above] {$\scriptstyle n$} -- (2.5,0);
	\draw (2,0) -- (2,-0.5) ;
	\braid[height=0.5cm] at (1,-0.5) s_1^{-1} s_1^{-1} ;
	\draw (1,-0.5) to[bend right=150] (1,-2) node[below left] {$\scriptstyle i$};
	\draw (2,-2) -- (2,-2.5) ;
	\draw (0.5,-2.5) -- (2,-2.5) node[below] {$\scriptstyle n$} -- (2.5,-2.5); 
\end{tikzbox}
	= \begin{tikzbox}
	\draw (0.5,0) -- (1,0) node[above] {$\scriptstyle n$} -- (2.5,0);
	\draw (1,0) -- (1,-0.5) ;
	\braid[height=0.5cm] at (1,-0.5) s_1 s_1 ;
	\draw (2,-0.5) to[bend left=150] (2,-2) node[below right] {$\scriptstyle i$};
	\draw (1,-2) -- (1,-2.5) ;
	\draw (0.5,-2.5) -- (1,-2.5) node[below] {$\scriptstyle n$} -- (2.5,-2.5); 
\end{tikzbox} ; \quad n \in \N^*, \quad i \in \N^*,
\end{equation}
où la barre horizontale supérieure (resp. inférieure) désigne l'insertion de l'idempotent de Jones-Wenzl évaluable dont l'indice correspond au nombre total de brin(s) sortant(s) (resp. entrant(s)). On utilisera encore les notations \eqref{Eqn: coupons JWGA} pour les coupons associés aux idempotents et nilpotents de Jones-Wenzl évaluables.

\begin{Lemme}
\label{Lemme: bouclage}
Soient $n \in \N^*$, $l \in \N$ le quotient de $n$ dans sa division euclidienne par $p$, et $i \in \N^*$.
\begin{enumerate}[(i)]
	\item Si $n \leq p-1$ ou $n = -1 \mod p$, alors :
	\[ \mathbf{B}_{i,n} := \begin{tikzbox}
		\draw (0.5,0) -- (1,0) node[above] {$\scriptstyle n$} -- (2.5,0);
		\draw (1,0) -- (1,-0.5) ;
		\braid[height=0.5cm] at (1,-0.5) s_1 s_1 ;
		\draw (2,-0.5) to[bend left=150] (2,-2) node[below right] {$\scriptstyle i$};
		\draw (1,-2) -- (1,-2.5) ;
		\draw (0.5,-2.5) -- (1,-2.5) node[below] {$\scriptstyle n$} -- (2.5,-2.5); 
	\end{tikzbox}
	= - \left( A^{2(n+1)} + A^{-2(n+1)} \right) \begin{tikzbox}
		\draw (0.5,0) -- (1,0) node[above] {$\scriptstyle n$} -- (2.5,0);
		\draw (1,0) -- (1,-0.5) ;
		\braid[height=0.5cm] at (1,-0.5) s_1 s_1 ;
		\draw (2,-0.5) to[bend left=150] (2,-2) node[below right] {$\scriptstyle i-1$};
		\draw (1,-2) -- (1,-2.5) ;
		\draw (0.5,-2.5) -- (1,-2.5) node[below] {$\scriptstyle n$} -- (2.5,-2.5); 
	\end{tikzbox} . \]
	\item Si $n \geq p$, alors :
	\[	\mathbf{B}_{i,n} := \begin{multlined}[t]
	\begin{tikzbox}
		\draw (0.5,0) -- (1,0) node[above] {$\scriptstyle n$} -- (2.5,0);
		\draw (1,0) -- (1,-0.5) ;
		\braid[height=0.5cm] at (1,-0.5) s_1 s_1 ;
		\draw (2,-0.5) to[bend left=150] (2,-2) node[below right] {$\scriptstyle i$};
		\draw (1,-2) -- (1,-2.5) ;
		\draw (0.5,-2.5) -- (1,-2.5) node[below] {$\scriptstyle n$} -- (2.5,-2.5); 
	\end{tikzbox}
	= - \left( A^{2(n+1)} + A^{-2(n+1)} \right) \begin{tikzbox}
		\draw (0.5,0) -- (1,0) node[above] {$\scriptstyle n$} -- (2.5,0);
		\draw (1,0) -- (1,-0.5) ;
		\braid[height=0.5cm] at (1,-0.5) s_1 s_1 ;
		\draw (2,-0.5) to[bend left=150] (2,-2) node[below right] {$\scriptstyle i-1$};
		\draw (1,-2) -- (1,-2.5) ;
		\draw (0.5,-2.5) -- (1,-2.5) node[below] {$\scriptstyle n$} -- (2.5,-2.5); 
	\end{tikzbox} \\
	- \left( A^2 - A^{-2} \right)^2 [lp-1]_{A^2} [n-lp+1]_{A^2} \begin{tikzbox}
		\draw (0.5,0) -- (1,0) node[above] {$\scriptstyle lp-2$}  -- (2,0) node[above] {$\scriptstyle 1$} -- (4,0) node[above] {$\scriptstyle n-lp$} -- (4.5,0);
		\draw (4,0) -- (4,-0.65) ;
		\draw (1,0) -- (1,-0.65) ;
		\draw (2,0) to[bend right=90] (3,0) ;
		\draw (2,-0.65) to[bend left=90] (3,-0.65) ;
		\braid[border height=0.1cm, height=0.5cm] at (1,-0.65) s_1-s_3^{-1} s_1-s_3^{-1} ;
		\draw (2,-1.85) node[below left] {$\scriptstyle i-1$} to[bend right=90] (3,-1.85) ;
		\draw (4,-1.85) -- (4,-2.5) ;
		\draw (1,-1.85) -- (1,-2.5) ;
		\draw (2,-2.5) to[bend left=90] (3,-2.5) ;
		\draw (0.5,-2.5) -- (1,-2.5) node[below] {$\scriptstyle lp-2$} -- (2,-2.5) node[below] {$\scriptstyle 1$} -- (4,-2.5) node[below] {$\scriptstyle n$-lp} -- (4.5,-2.5); 
	\end{tikzbox}.
	\end{multlined} \]
\end{enumerate}
\end{Lemme}

\begin{proof}
On commence par traiter le cas où $i=1$. On isole un brin à gauche sur lequel on utilise les relations d'isotopie et d'écheveaux \eqref{Eqn: echeveaux2} :
\begin{align*}
\mathbf{B}_{1,n} &= \begin{tikzbox}
	\draw (0.5,0) -- (1,0) node[above] {$\scriptstyle n$} -- (2.5,0);
	\draw[red] (1,0) -- (1,-1) ;
	\braid[height=0.5cm, style strands={1}{red}] at (1,-1) s_1 s_1 ;
	\draw (2,-1) to[bend left=150] (2,-2.5) node[below right] {$\scriptstyle 1$} ;
	\draw[red] (1,-2.5) -- (1,-3.5) ;
	\draw (0.5,-3.5) -- (1,-3.5) node[below] {$\scriptstyle n$} -- (2.5,-3.5);
\end{tikzbox}
= \begin{tikzbox} 
	\draw (0.5,0) -- (1,0) node[above] {$\scriptstyle 1$} -- (2,0) node[above] {$\scriptstyle n-1$} -- (3.5,0);
	\draw (1,0) -- (1,-0.5) ;
	\draw[red] (2,0) -- (2,-0.5) ;
	\braid[height=0.5cm, style strands={2}{red}] at (1,-0.5) s_2 s_1 s_1 s_2 ;
	\draw[green] (1.5,-1.5) circle (0.3) ;
	\draw (3,-0.5) to[bend left=150] (3,-3) node[below left] {$\scriptstyle 1$} ;
	\draw (1,-3) -- (1,-3.5) ;
	\draw[red] (2,-3) -- (2,-3.5) ;
	\draw (0.5,-3.5) -- (1,-3.5) node[below] {$\scriptstyle 1$} -- (2,-3.5) node[below] {$\scriptstyle n-1$} -- (3.5,-3.5);
\end{tikzbox} \\
&= A \begin{tikzbox} 
	\draw (0.5,0) -- (1,0) node[above] {$\scriptstyle 1$} -- (2,0) node[above] {$\scriptstyle n-1$} -- (3.5,0);
	\draw (1,0) -- (1,-0.5) ;
	\draw[red] (2,0) -- (2,-0.5) ;
	\braid[height=0.5cm, style strands={2}{red}] at (1,-0.5) s_2 ;
	\draw (1,-1.5) to[bend right=90] (2,-1.5) ;
	\draw (1,-2.5) to[bend left=90] (2,-2.5) ;
	\draw[red] (3,-1.5) -- (3,-2.5) ;
	\braid[height=0.5cm, style strands={3}{red}] at (1,-2.5) s_1 s_2 ;
	\draw[green] (1.5,-3) circle (0.3) ;
	\draw (3,-0.5) to[bend left=150] (3,-4) node[below left] {$\scriptstyle 1$} ;
	\draw (1,-4) -- (1,-4.5) ;
	\draw[red] (2,-4) -- (2,-4.5) ;
	\draw (0.5,-4.5) -- (1,-4.5) node[below] {$\scriptstyle 1$} -- (2,-4.5) node[below] {$\scriptstyle n-1$} -- (3.5,-4.5);
\end{tikzbox}  
+ A^{-1} \begin{tikzbox} 
	\draw (0.5,0) -- (1,0) node[above] {$\scriptstyle 1$} -- (2,0) node[above] {$\scriptstyle n-1$} -- (3.5,0);
	\draw (1,0) -- (1,-0.5) ;
	\draw[red] (2,0) -- (2,-0.5) ;
	\braid[height=0.5cm, style strands={2}{red}] at (1,-0.5) s_2 ;
	\draw (1,-1.5) -- (1,-2.5) ;
	\draw (2,-1.5)-- (2,-2.5) ;
	\draw[red] (3,-1.5) -- (3,-2.5) ;
	\braid[height=0.5cm, style strands={3}{red}] at (1,-2.5) s_1 s_2 ;
	\draw (3,-0.5) to[bend left=150] (3,-4) node[below left] {$\scriptstyle 1$} ;
	\draw (1,-4) -- (1,-4.5) ;
	\draw[red] (2,-4) -- (2,-4.5) ;
	\draw (0.5,-4.5) -- (1,-4.5) node[below] {$\scriptstyle 1$} -- (2,-4.5) node[below] {$\scriptstyle n-1$} -- (3.5,-4.5);
\end{tikzbox} \\
&= -A^4 \begin{tikzbox} 
	\draw (0.5,0) -- (1,0) node[above] {$\scriptstyle 1$} -- (2,0) node[above] {$\scriptstyle n-1$} -- (2.5,0);
	\draw (1,0) -- (1,-1) ;
	\draw[red] (2,0) -- (2,-1) ;
	\braid[style strands={2}{red}] at (1,-1) s_1^{-1} s_1^{-1} ;
	\draw (1,-3.5) -- (1,-4.5) ;
	\draw[red] (2,-3.5) -- (2,-4.5) ;
	\draw (0.5,-4.5) -- (1,-4.5) node[below] {$\scriptstyle 1$} -- (2,-4.5) node[below] {$\scriptstyle n-1$} -- (2.5,-4.5);
\end{tikzbox}  
+ A^{-1} \begin{tikzbox} 
	\draw (0.5,0) -- (1,0) node[above] {$\scriptstyle 1$} -- (2,0) node[above] {$\scriptstyle n-1$} -- (3.5,0);
	\draw (1,0) -- (1,-0.5) ;
	\draw[red] (2,0) -- (2,-0.5) ;
	\braid[style strands={2}{red}] at (1,-0.5) s_1 s_2 s_1 ;
	\draw (3,-0.5) to[bend left=150] (3,-4) node[below left] {$\scriptstyle 1$} ;
	\draw[green] (2.5,-2.2) circle (0.3) ;
	\draw (1,-4) -- (1,-4.5) ;
	\draw[red] (2,-4) -- (2,-4.5) ;
	\draw (0.5,-4.5) -- (1,-4.5) node[below] {$\scriptstyle 1$} -- (2,-4.5) node[below] {$\scriptstyle n-1$} -- (3.5,-4.5);
\end{tikzbox} \\
&= -A^4 \begin{tikzbox} 
	\draw (0.5,0) -- (1,0) node[above] {$\scriptstyle 1$} -- (2,0) node[above] {$\scriptstyle n-1$} -- (2.5,0);
	\braid[style strands={2}{red}] s_1^{-1} s_1^{-1} ;
	\draw (0.5,-2.5) -- (1,-2.5) node[below] {$\scriptstyle 1$} -- (2,-2.5) node[below] {$\scriptstyle n-1$} -- (2.5,-2.5);
\end{tikzbox}  
- A^{-4} \begin{tikzbox} 
	\draw (0.5,0) -- (1,0) node[above] {$\scriptstyle 1$} -- (2,0) node[above] {$\scriptstyle n-1$} -- (2.5,0);
	\braid[style strands={2}{red}] s_1 s_1 ;
	\draw (0.5,-2.5) -- (1,-2.5) node[below] {$\scriptstyle 1$} -- (2,-2.5) node[below] {$\scriptstyle n-1$} -- (2.5,-2.5);
\end{tikzbox} 
= -A^4 \mathbf{E}_{1,n} - A^{-4} \mathbf{E}_{1,n}^{-1}.
\end{align*}

Or, d'après le théorème \ref{Thm: enlacement}, on sait que :
\[ \mathbf{E}_{1,n}^{\pm1} = A^{\pm 2(n-1)} \; \rect{I}{n} \pm \delta_{n \geq p} A^{\pm 2(lp-3)} \left( A^2 - A^{-2} \right) [n-lp+1]_{A^2} \; \rect{N}{n} . \]
Donc :
\[ \mathbf{B}_{1,n} = -\left( A^{2(n+1)} + A^{-2(n+1)} \right) \rect{I}{n} - \delta_{n \geq p} \left( A^2-A^{-2} \right)^2 [lp-1]_{A^2} [n-lp+1]_{A^2} \rect{N}{n} . \]

Enfin, l'ajout de brins de bouclage ne change pas les calculs ci-dessus. D'où le résultat pour le coupon $\mathbf{B}_{i,n}$ avec $i \in \N$.
\end{proof}

\begin{Thm}
\label{Thm: bouclage}
Soient $n \in \N^*$, $l \in \N$ le quotient de $n$ dans sa division euclidienne par $p$, et $i \in \N^*$.
\begin{enumerate}[(i)]
	\item Si $n \leq p-1$ ou $n = -1 \mod p$, alors $\mathbf{B}_{i,n} = (-1)^i \left( A^{2(n+1)} + A^{-2(n+1)} \right)^i \rect{I}{n}$.
	\item Si $n \geq p$, alors :
	\begin{gather*}
	\mathbf{B}_{i,n} = \begin{multlined}[t]
		(-1)^i \left( A^{2(n+1)} + A^{-2(n+1)} \right)^i \; \rect{I}{n}
		- \left( A^2 - A^{-2} \right)^2 [lp-1]_{A^2} [n-lp+1]_{A^2} \\
		\times (-1)^{i-1} \left( A^{2(n+1)} + A^{-2(n+1)} \right)^{i-1} \sum_{k=0}^{i-1} \left( \frac{ A^{2(2lp-n-1)} + A^{-2(2lp-1-n)} }{ A^{2(n+1)} + A^{-2(n+1)} } \right)^k
		  \rect{N}{n} ,
		\end{multlined} \\
	\mathbf{B}_{i,n} \cdot \rect{N}{n} =(-1)^i \left( A^{2(2lp-n-1)} + A^{-2(2lp-1-n)} \right)^i \; \rect{N}{n} .
	\end{gather*}
\end{enumerate}
\end{Thm}

\begin{proof}
L'assertion $(i)$ s'obtient par une récurrence immédiate sur $i \in \N^*$ grâce à l'assertion $(i)$ du lemme \ref{Lemme: bouclage}.

On suppose désormais que $n \geq p$. D'après le lemme \ref{Lemme: bouclage}, on a :
\begin{gather*}
\mathbf{B}_{i,n} = \begin{multlined}[t]
	- \left( A^{2(n+1)} + A^{-2(n+1)} \right) \mathbf{B}_{i-1,n} \tag{1} \\
	- \left( A^2 - A^{-2} \right)^2 [lp-1]_{A^2} [n-lp+1]_{A^2} \mathbf{B}_{i-1,n} \cdot \rect{N}{n} .
	\end{multlined}
\end{gather*}
L'expression du coupon $\mathbf{B}_{i,n} \cdot \rect{N}{n}$ (i.e. le produit en pile \eqref{Eqn: produit en pile} de $\mathbf{B}_{i,n}$ et $\rect{N}{n}$) se déduit $(1)$ grâce aux relations :
\[ \rect{I}{n} \cdot \rect{N}{n} = \rect{N}{n} = \rect{N}{n} \cdot \rect{I}{n}, \qquad \rect{N}{n} \cdot \rect{N}{n} \overset{\ref{Prop: JWG}}{=} - \frac{ [lp]_{A^2} }{ [lp-1]_{A^2} } \rect{N}{n} . \]
On obtient ainsi :
\begin{align*}
\mathbf{B}_{i,n} \cdot \rect{N}{n} &= \begin{multlined}[t]
	- \left( A^{2(n+1)} + A^{-2(n+1)} \right) \mathbf{B}_{i-1,n} \cdot \rect{N}{n} \\
		+ \left( A^2 - A^{-2} \right)^2 [lp]_{A^2} [n-lp+1]_{A^2} \mathbf{B}_{i-1,n} \cdot \rect{N}{n} 
		\end{multlined} \\
	&= \begin{multlined}[t]
		- \left( A^{2(n+1)} + A^{-2(n+1)} \right) \mathbf{B}_{i-1,n} \cdot \rect{N}{n} \\
		+ \left( A^{2lp} - A^{-2(lp)} \right) \left( A^{2(n-lp+1)} - A^{-2(n-lp+1)} \right) \mathbf{B}_{i-1,n} \cdot \rect{N}{n} 
		\end{multlined} \\ 
	&= - \left( A^{2(2lp-n-1)} + A^{-2(2lp-1-n)} \right) \mathbf{B}_{i-1,n} \cdot \rect{N}{n} .
\end{align*}
Par une récurrence immédiate, on en déduit que :
\[ \mathbf{B}_{i,n} = (-1)^i \left( A^{2(2lp-n-1)} + A^{-2(2lp-1-n)} \right)^i \rect{N}{n} . \]
En injectant ce résultat dans l'équation $(1)$, on obtient :
\begin{gather*}
\mathbf{B}_{i,n} = \begin{multlined}[t]
	- \left( A^{2(n+1)} + A^{-2(n+1)} \right) \mathbf{B}_{i-1,n}
	- \left( A^2 - A^{-2} \right)^2 [lp-1]_{A^2} [n-lp+1]_{A^2} \\
	\times (-1)^{i-1} \left( A^{2(2lp-n-1)} + A^{-2(2lp-1-n)} \right)^{i-1} \rect{N}{n} .
	\end{multlined}
\end{gather*}
On en déduit que :
\[ \mathbf{B}_{i,n} = \begin{multlined}[t]
	(-1)^i  \left( A^{2(n+1)} + A^{-2(n+1)} \right)^i \; \rect{I}{n} \\
	- \left( A^2 - A^{-2} \right)^2 [lp-1]_{A^2} [n-lp+1]_{A^2} \Sigma_{i,n} \; \rect{N}{n} , 
	\end{multlined} \]
où :
\begin{align*}
\Sigma_{i,n} &= \sum_{k=1}^i (-1)^{k-1} \left( A^{2(n+1)} + A^{-2(n+1)} \right)^{k-1} (-1)^{i-k} \left( A^{2(2lp-n-1)} + A^{-2(2lp-1-n)} \right)^{i-k} \\
	&= \sum_{k=0}^{i-1} (-1)^{i-k-1} \left( A^{2(n+1)} + A^{-2(n+1)} \right)^{i-k-1} (-1)^k \left( A^{2(2lp-n-1)} + A^{-2(2lp-1-n)} \right)^k \\
	&= (-1)^{i-1} \left( A^{2(n+1)} + A^{-2(n+1)} \right)^{i-1} \sum_{k=0}^{i-1} \left( \frac{ A^{2(2lp-n-1)} + A^{-2(2lp-1-n)} }{ A^{2(n+1)} + A^{-2(n+1)} } \right)^k
\end{align*}
D'où le résultat.
\end{proof}

\section{Représentation modulaire sur l'espace d'écheveaux du tore solide}
\label{section: SL2-rep topo}

Avec les calculs d'écheveaux de la section \ref{section: SL2-rep topo}, on construit les prémisses d'un analogue topologique de la représentation de $\SL_2(\Z)$ sur le centre $\Zf$ de $\Uq$ (cf. la proposition \ref{Prop: SL2-rep} et le théorème \ref{Thm: SL2-rep}). Pour faciliter les écritures, on note :
\begin{equation} 
\label{Eqn: coupons JWGq}
\index{S @$\mathbf{S}^+_s$} \index{I @$\mathbf{I}^\pm_s$} \index{N7 @$\mathbf{N}^\pm_s$}
\begin{aligned}
\forall s \in \{1,...,p\} \qquad 
	&\mathbf{S}^+_s := \rect{\bar{f}_{s-1}}{} = \overline{\rect{I}{s-1}}, \\
	&\mathbf{I}_s^+ := \rect{\bar{f}_{2p-s-1}}{} = \overline{\rect{I}{2p-s-1}}, \\
	&\mathbf{I}_s^- := \rect{\bar{f}_{3p-s-1}}{} = \overline{\rect{I}{3p-s-1}}, \\
\forall s \in \{1,...,p-1\} \qquad 
	&\mathbf{N}_s^+ := \rect{\bar{f}_{2p-s-1}'}{} = \overline{\rect{N}{2p-s-1}}, \\
	&\mathbf{N}_s^- := \rect{\bar{f}_{3p-s-1}'}{} = \overline{\rect{N}{3p-s-1}}, \\
\end{aligned}
\end{equation}
où $(f_n)_{n\in \N^*}$ et $(f_n')_{n\in \N^*}$ désignent respectivement les idempotents et nilpotents de Jones-Wenzl évaluables. Ainsi, d'après le théorème \ref{Thm: POIs rep fonda}, on a les correspondances suivantes entre les classes d'écheveaux coloriés et les $\Uq$-modules à gauche.
\[ \begin{array}{ccll} 
	\text{\'Echeveaux} && \text{Représentations} & \\
	\mathbf{S}^+_s & \longleftrightarrow & \X^+(s), & 1 \leq s \leq p, \\
	\mathbf{I}^+_s & \longleftrightarrow & \PIM^+(s), & 1 \leq s \leq p, \\
	\mathbf{N}^+_s & \longleftrightarrow & \PIM^+(s) / \Rad \left( \PIM^+(s) \right) \cong \X^+(s), & 1 \leq s \leq p-1, \\
	\mathbf{I}^-_s & \longleftrightarrow & 2 \PIM^-(s), & 1 \leq s \leq p, \\
	\mathbf{N}^-_s & \longleftrightarrow & 2 \left( \PIM^-(s) / \Rad \left( \PIM^-(s) \right) \right) \cong 2 \X^-(s), & 1 \leq s \leq p-1.
\end{array} \] 
D'après la sous-section \ref{subsection: centre1}, une partie de cette correspondance se prolonge avec la base canonique $\{ e_s \; ; \; 0 \leq s \leq p \} \cup \{ w^\pm_s \; ; \; 1  \leq s \leq p-1 \}$ de $\Zf$.
\begin{equation}
\label{Eqn: analogies}
\begin{array}{cccccl} 
	\text{\'Echeveaux} && \text{Représentations} && \text{\'Eléments} \\
	\mathbf{I}^+_p & \longleftrightarrow & \X^+(p) & \longleftrightarrow & e_p, & \\
	\mathbf{I}^+_s + \mathbf{I}^-_{p-s} & \longleftrightarrow & \PIM^+(s) \oplus (2) \PIM^-(p-s) & \longleftrightarrow & e_s, & 1 \leq s \leq p-1, \\
	\mathbf{N}^+_s & \longleftrightarrow & \X^+(s) & \longleftrightarrow & \frac{w^+_s}{[s]}, & 1 \leq s \leq p-1, \\
	\mathbf{N}^-_{p-s} & \longleftrightarrow & (2) \X^-(p-s) & \longleftrightarrow & - \frac{w^-_s}{[s]}, & 1 \leq s \leq p-1, \\
	\mathbf{I}^-_p & \longleftrightarrow & (2) \X^-(p) & \longleftrightarrow & e_0. &
\end{array}
\end{equation} 
Le choix de la normalisation des termes $\mathbf{N}^\pm_s$, $1 \leq s \leq p-1$, permettra d'identifier l'action de la vrille négative avec celle de $\cT$ \eqref{Eqn: SL2-rep}. Enfin, on réalisera le morphisme de Drinfeld $\bchi$ \eqref{Eqn: Drinfeld1} par l'action de bouclage sur une couleur de type "identité". Cela offre une réalisation partielle de l'action de $\cS \cong \bchi \circ \hbphi^{-1}$ (cf. la remarque \ref{Rem: SL2-rep1}) pour laquelle il manque l'interprétation géométrique du morphisme de Radford $\hbphi$ \eqref{Eqn: Radford1} lié aux intégrales.

\subsection{Evaluation dans l'espace d'écheveaux du tore solide}

On considère désormais le $\Z[q^{\frac{1}{2}}, q^{-\frac{1}{2}}]$-module d'écheveaux $K(\bar{D} \times \mathbb{S}^1, 0)$ \index{K3@ $K(M, 2n)$} défini de manière analogue à $K_A(\bar{D} \times \mathbb{S}^1, 0)$ en remplaçant $A^2$ par $q = e^\frac{i \pi}{p}$ (cf. la section \ref{section: echeveaux}). Les espaces d'écheveaux générique $K_A(\bar{D} \times \mathbb{S}^1, 0)$ et évalué $K(\bar{D} \times \mathbb{S}^1, 0)$ sont liés par un morphisme d'évaluation surjectif :
\[ \insertion{a}{}{Tore} \longmapsto \insertion{\bar{a}}{}{Tore} \]
tel que celui des algèbres de Temperley-Lieb génériques et évaluées (cf. la sous-section \ref{subsection: TL}). Avec ces considérations, les formules des théorèmes \ref{Thm: enlacement}, \ref{Thm: vrille} et \ref{Thm: bouclage} sont toutes évaluables. Après évaluation de $A^2$ en $q$, on obtient leurs analogues dans l'espace d'écheveaux $K(\bar{D} \times \mathbb{S}^1, 0)$.

\begin{Cor}
\label{Cor: enlacement}
Soient $n \in \N^*$, $l \in \N$ le quotient de $n$ dans sa division euclidienne par $p$, et $i \in \{1,...,n-1\}$.
\begin{enumerate}[(i)]
	\item Si $n \leq p-1$ ou $n = -1 \mod p$, alors $\overline{\mathbf{E}_{i,n-i}^{\pm 1}} = q^{\pm(n-i)i} \; \overline{\rect{I}{n}}$.
	\item[(ii')] Si $n \geq p$ et $i = 1 < n$, alors :
	\begin{gather*}
	\overline{\mathbf{E}_{1,n-1}^{\pm 1}} = q^{\pm(n-1)} \; \overline{\rect{I}{n}} \pm q^{\mp 3} \left( q - q^{-1} \right) [n+1] \; \overline{\rect{N}{n}} , \\
	\overline{\mathbf{E}_{1,n-1}^{\pm 1}} \cdot \overline{\rect{N}{n}} = q^{\pm(n-1)} \; \overline{\rect{N}{n}}.
	\end{gather*}
	\item[(iii)] Si $n \geq p$ et $lp-1 \leq i < n$, alors :
	\begin{gather*}
	\overline{\mathbf{E}_{i,n-i}^{\pm 1}} = q^{\pm(n-i)i} \; \overline{\rect{I}{n}} \mp q^{\pm(n-i)i} \left( q - q^{-1} \right) (n-i) \; \overline{\rect{N}{n}} , \\
	\overline{\mathbf{E}_{i,n-i}^{\pm 1}} \cdot \overline{\rect{N}{n}} = q^{\pm(n-i)i} \; \overline{\rect{N}{n}} .
	\end{gather*}
\end{enumerate}
\end{Cor}

\begin{Cor}
\label{Cor: vrille}
Soient $n \in \N^*$ et $l \in \N$ le quotient de $n$ dans sa division euclidienne par $p$.
\begin{enumerate}[(i)]
	\item Si $n \leq p-1$ ou $n = -1 \mod p$, alors $\overline{\mathbf{T}_n^{\pm 1}} = (-1)^n q^{\pm\frac{n(n+2)}{2}} \; \overline{\rect{I}{n}}$.
	\item Si $n \geq p$, alors :
	\begin{gather*}
	\overline{\mathbf{T}_n^{\pm 1}} = (-1)^n q^{\pm\frac{n(n+2)}{2}} \; \overline{\rect{I}{n}} \mp (-1)^n q^{\pm\frac{n(n+2)}{2}} \left( q - q^{-1} \right) (n-lp+1) \; \overline{\rect{N}{n}} , \\
	\overline{\mathbf{T}_n^{\pm 1}} \cdot \overline{\rect{N}{n}} = (-1)^n q^{\pm \frac{n(n+2)}{2}} \; \overline{\rect{N}{n}} .
	\end{gather*}
\end{enumerate}
\end{Cor}

\begin{Cor}
\label{Cor: bouclage}
Soient $n \in \N^*$, $l \in \N$ le quotient de $n$ dans sa division euclidienne par $p$, et $i \in \N^*$.
\begin{enumerate}[(i)]
	\item Si $n \leq p-1$ ou $n = -1 \mod p$, alors $\overline{\mathbf{B}_{i,n}} = (-1)^i \left( q^{(n+1)} + q^{-(n+1)} \right)^i \; \overline{\rect{I}{n}}$.
	\item Si $n \geq p$, alors :
	\begin{gather*}
	\overline{\mathbf{B}_{i,n}} = \begin{multlined}[t]
		(-1)^i \left( q^{(n+1)} + q^{-(n+1)} \right)^i \; \overline{\rect{I}{n}}  \\
		+ (-1)^{i-1} \left( q^{(n+1)} q^{-(n+1)} \right)^{i-1} i \left( q - q^{-1} \right)^2 [n+1] \; \overline{\rect{N}{n}} ,
		\end{multlined} \\
	\overline{\mathbf{B}_{i,n}} \cdot \overline{\rect{N}{n}} = (-1)^i \left( q^{(n+1)} + q^{-(n+1)} \right)^i \; \overline{\rect{N}{n}} .
	\end{gather*}
\end{enumerate}
\end{Cor}

\subsection{Analogies topologiques et perspectives}

Pour tout $n \in \N$, on considère l'application linéaire : \index{Theta4 @$\Theta_n$}
\begin{equation}
\label{Eqn: Theta}
\Theta_{n} : \left\{ \begin{array}{l}
	\TL_n(q) \longrightarrow \C(A) \otimes_{\Z[q^{\frac{1}{2}},q^{-\frac{1}{2}}]} K(\bar{D} \times \mathbb{S}^1, 0) \\
	\bar{a} \longmapsto \insertion{\bar{a}}{n}{Tore}
	\end{array} \right. \hspace{-0.5cm} .
\end{equation}

Avec les notations \eqref{Eqn: coupons JWGq}, les corollaires \ref{Cor: vrille} et \ref{Cor: bouclage} se reformulent en la :

\begin{Prop} 
\label{Prop: SL2-rep topo}
\begin{enumerate}[(i)]
	\item L'action de la vrille négative $\mathbf{T}^{-1}$ vérifie :
	\begin{gather*}
		\overline{\mathbf{T}_{2p-s-1}^{-1} + \mathbf{T}_{2p-s-1}^{-1}} = \begin{multlined}[t]
			(-1)^{s-1} q^{-\frac{s^2-1}{2}} \left( \mathbf{I}^+_s + \mathbf{I}^+_{p-s} \right) \\
			+ \delta_{1 \leq s \leq p-1} (-1)^{s-1}  (q-q^{-1}) q^{-\frac{s^2-1}{2}}
			\left( (p-s) \mathbf{N}^+_s + s \mathbf{N}^-_{p-s} \right), \\
			1 \leq s \leq p, 
			\end{multlined} \\
		\overline{\mathbf{T}_{2p-s-1} \cdot \mathbf{N}_s^+ } = (-1)^{s-1} q^{-\frac{s^2-1}{2}} \mathbf{N}^+_s, \qquad 1 \leq s \leq p-1, \\
		\overline{\mathbf{T}_{2p+s-1} \cdot \mathbf{N}_{p-s}^+ } = (-1)^{s-1} q^{-\frac{s^2-1}{2}} \mathbf{N}^-_{p-s}, \qquad 1 \leq s \leq p-1,
	\end{gather*}
	\item L'action du bouclage $\mathbf{B}_R$ par un polynôme $R(\alpha) \in \C[\alpha]$ vérifie :
	\begin{gather*}
	\overline{ \sum_{s=1}^p \mathbf{B}_{R,2p-s-1} + \mathbf{B}_{R,2p+s-1} } = \begin{multlined}[t]
		R(-\widehat{\beta}_s)) \left(\mathbf{I}_s^+ + \mathbf{I}_{p-s}^+ \right) \\
		- (q-q^{-1})^2 R'(-\widehat{\beta}_s) [s] \left( \mathbf{N}_s^+ - \mathbf{N}_{p-s}^- \right)
		\end{multlined}
		\end{gather*}
	où, pour tout $s \in \N$, $\widehat{\beta}_s := q^s+q^{-s}$.
\end{enumerate}
\end{Prop}

L'assertion $(i)$ de la proposition \ref{Prop: SL2-rep topo} est à comparer avec l'assertion $(i)$ de la proposition \ref{Prop: SL2-rep}. Sous les identifications \ref{Eqn: analogies}, on obtient les mêmes formules lorsque $\delta=1$ (cf. les éléments d'enrubannement possibles $\mathbf{v}_\delta$ \eqref{Eqn: enrubannement}). L'action de $\cT$ \eqref{Eqn: SL2-rep} s'interprète donc par l'action de la vrille négative, qui est le résultat d'un \emph{homéomorphisme de Dehn} négatif le long d'une courbe qui borde un disque (cf. par exemple \cite[§ 12]{PS97}).

L'assertion $(ii)$ de la proposition \ref{Prop: SL2-rep topo} est à comparer avec l'assertion $(ii)$ de la proposition \ref{Prop: SL2-rep}. En choisissant judicieusement les polynômes $R(x)$ :
\[ \begin{array}{cccl} 
	\text{Polynômes} && \text{\'Eléments} & \\
	(-1)^p U_{p}(x) & \longleftrightarrow & e_p, & \\
	(-1)^{s-1} [s] \; U_{s}(x) & \longleftrightarrow & \frac{w^+_s}{[s]} , & 1 \leq s \leq p-1, \\
	- (-1)^{s-1} [s] \; \frac{1}{2} \left(U_{2p-s} - U_{s} \right)(x) & \longleftrightarrow & - \frac{w^-_s}{[s]} , & 1 \leq s \leq p-1, \\
	(-1)^{2p-1} \; \frac{1}{2} U_{2p}(x) & \longleftrightarrow & e_0, &
\end{array} \]
sous les identifications \ref{Eqn: analogies}, on obtient encore les mêmes formules lorsque $\delta=1$ à un scalaire $\widehat{\zeta}:= \frac{(-1)^p}{\zeta 2p ([p-1]!)^2}$ près. Quitte à choisir convenablement une intégrale (cf. la proposition \ref{Prop: integrales} et la remarque \ref{Rem: SL2-rep1}), on peut supposer que $\zeta = \frac{(-1)^p}{2p ([p-1]!)^2}$ et $\widehat{\zeta}=1$.

Compte-tenu des propositions \ref{Prop: insertion des JWG} et \ref{Prop: traces des JWG}, le choix de ces polynômes semble lié à une \emph{trace modifiée} sur les classes d'écheveaux $\mathbf{I}^\pm_s$ et $\mathbf{N}^\pm_s$, $1 \leq s \leq p$ :
\[ \begin{array}{cccl} 
	\text{Polynômes} && \text{Traces} & \\
	(-1)^p \; U_{p}(x) & \longleftrightarrow & - \frac{\tr \left( \Theta_{p-1} (\mathbf{I}^+_p) \right)}{[p]} \Theta_{p-1} (\mathbf{I}^+_p), & \\
	(-1)^{s-1} [s] \; U_{s}(x) & \longleftrightarrow & \frac{\tr \left( \Theta_{2p-s-1} (\mathbf{N}^+_s) \right)}{[p]} \frac{\Theta_{2p-s-1} (\mathbf{N}^+_s)}{[p]} , & 1 \leq s \leq p-1, \\
	- (-1)^{s-1} [s] \; U_{2p-s}(x) & \longleftrightarrow & \frac{\tr \left( \Theta_{2p+s-1} (\mathbf{N}^-_{p-s}) \right)}{[2p]} \frac{\Theta_{2p+s-1} (\mathbf{N}^-_{p-s})}{[2p]} , & 1 \leq s \leq p-1, \\
	(-1)^{2p-1} \; U_{2p}(x) & \longleftrightarrow & \frac{\tr \left( \Theta_{2p-1} (\mathbf{I}^-_p) \right)}{[2p]} \Theta_{2p-1} (\mathbf{I}^-_p). &
\end{array} \]
D'après le corollaire \ref{Cor: Drinfeld}, avec une telle trace modifiée, le morphisme de Drinfeld $\bchi^1$ \eqref{Eqn: Drinfeld2} s'interprèterait par l'action du bouclage, qui est le résultat du produit de Kauffman de deux entrelacs enrubannés plongés dans des tores solides complémentaires dans la sphère de dimension 3 (cf. \cite[§ 1]{BHMV92}). On obtiendrait ainsi une partie de l'interprétation topologique de $\cS = \bchi \circ \hbphi^{-1}$ (cf. la remarque \ref{Rem: SL2-rep1}) ; il manquerait l'interprétation du morphisme de Radford $\hbphi^1$ \eqref{Eqn: Radford2} lié aux intégrales.

Pour la représentation de $\SL_2(\Z)$ obtenue par la TQFT de \cite{RT91}, les morphismes $\cT'$ et $\cS'$ correspondants s'interprètent respectivement par les actions de la vrille négative et du bouclage via les idempotents de Jones-Wenzl usuels (cf. par exemple \cite[§ II.3, § XII]{Tur94}). C'est pourquoi, après interprétation totale de $\cS$, on espère ainsi définir une nouvelle représentation de $\SL_2(\Z)$ sur l'espace d'écheveaux du tore solide qui étende celle de \cite{RT91} conformément au théorème \ref{Thm: SL2-rep}.

Enfin, on espère également que les idempotents et nilpotents de Jones-Wenzl évaluables permettent de construire de nouveaux invariants de 3-variété à la manière de \cite{Lic92}. Cette perspective soulève les problèmes ouverts suivants.

\begin{Pb}
Définir une trace modifiée, analogue à celle de \cite{GPMT09}, sur l'espace d'écheveaux du tore solide qui ne s'annule pas sur les idempotents (et nilpotents) de Jones-Wenzl évaluables.
\end{Pb}

\begin{Pb}
Définir une couleur de Kirby à partir des idempotents (et nilpotents) de Jones-Wenzl évaluables à la manière de \cite{Lic91}.
\end{Pb}

\begin{Pb}
Etudier les coupons :
\[ \begin{tikzbox}
	\draw (-0.75,0.5) -- (-0.75,0) ;
	\draw (0.75,0.5) -- (0.75,0) ;
	\draw[thick] (-1.25,0) node[above] {$\scriptstyle a$} -- (-0.25,0) ;
	\draw[thick] (0.25,0) -- (1.25,0) node[above] {$\scriptstyle b$} ;
	\draw (-0.5,0) to[bend right=90] (0.5,0) ; 
	\draw (-1,0) to[out=-90, in=90] (-0.25,-1.5) ;
	\draw (1,0) to[out=-90, in=90] (0.25,-1.5) ;
	\draw[thick] (-0.5,-1.5) -- (0.5,-1.5) node[below] {$\scriptstyle j$} ;
	\draw (0,-1.5)  -- (0,-2) ;
\end{tikzbox} \; ; \quad a,b,j \in \N^*, \]
correspondant aux opérateurs de Clebsch-Gordan dans $\Rep^{fd}_s$, et définir des systèmes de couleurs admissibles pour étendre le coloriage par les idempotents (et nilpotents) de Jones-Wenzl évaluables du tore aux corps à anses à la manière de \cite{Lic93}.
\end{Pb}

\begin{Pb}
Avec les idempotents (et nilpotents) de Jones-Wenzl évaluables, est-il possible de retrouver les invariants construits dans \cite{CGPM14} pour la classe cohomologique nulle, à la manière de \cite{BHMV92} ? 
\end{Pb}

\renewcommand{\indexname}{Index des notations}
\printindex

\bibliographystyle{alpha-fr}
\bibliography{Biblio}

\end{document}